\documentclass[a4paper, 10pt]{amsart}
\usepackage{amsthm}
\usepackage{amsfonts}
\usepackage{latexsym, amssymb}
\newcommand{\field}[1]{\mathbb{#1}}
\newcommand{\R}{\field{R}}

\newcommand{\Z}{\field{Z}}

\def\x{\xi}

\def\D\theta ij#1{\dis \frac{\partial #1}{\partial \theta_i^j}}

\def\sqr{{\hskip1pt\vcenter{\vbox{\hrule height.4pt
\hbox{\vrule width.4pt height4pt\kern4pt
\vrule width.4pt}
\hrule height.4pt}}}}

\def\cal{\mathcal}

\def\trans{\ifmmode{\rm \frown\mkern-16.8mu \vert
\mkern8mu}\else{$\frown\mkern-16.8mu\vert\mkern8mu$}\fi\relax}
\def\ap{\rightarrow}

\def\dis{\displaystyle}

\def\D{\Delta}
\def\dis{\displaystyle}
\def\a{\alpha}
\def\b{\beta}
\def\g{\gamma}
\def\G{\Gamma}
\def\t{\tau} \def\d{\delta}

\def\l{\lambda}
\def\L{\Lambda} \def\n{\nu}

\def\s{\sigma}
\def\S{\Sigma}

\def\so{\underline}
\def\O{\Omega}
 
\def\T{\mathbf{T}}
\def\V{\mathbf{V}}
\def\H{\mathbf{H}}

\def\p{\partial}

\def\o{\omega}

\newtheorem{The}{Theorem}[subsection]
 \newtheorem{Pro}{Proposition}[subsection]
\newtheorem{Def}{Definition}[subsection]

\theoremstyle{Bold}
\newtheorem{Rem}[The]{Remark}

\newtheorem{Cor}{Corollary}[subsection]
\newtheorem{Lem}{Lemma}[subsection]
\newtheorem{Ex}{Example}[subsection]
\newtheorem{Exs}{Examples}[subsection]

\numberwithin{equation}{section}

\title{ Geometrical structures on the prolongation of  a  pre-Lie algebroid  on fibered manifolds    and  application to foliated anchored bundle  }

\begin{document}
\maketitle
\begin{center}
\author{   F. Pelletier\footnote
{Laboratoire de math\'ematiques, Universit\'e de Savoie, Campus scientifique 73 376 Le Bourget du Lac Cedex \& ANR G\'eom\'etrie de Finsler et applications} }
\end{center}
\date{}

%\begin{document}
%\maketitle

%In the present paper we study the properties of dual structures on the prolongations of a Lie algebroid. We introduce the dynamical covariant derivative on Lie algebroids and prove that the nonlinear connection induced by a regular Lagrangian is compatible with the metric and symplectic structures. The notions of mechanical structure and semi-Hamiltonian section are introduced on the prolongation of the Lie algebroid to its dual bundle and their properties are investigated. Finally, we prove the equivalence between the metric nonlinear connection and semi-Hamiltonian section, using the Legendre transformation induced by a regular Hamiltonian

\begin{abstract}  A pre-Lie algebroid is an anchored bundle provided with an almost Lie bracket such that the anchor is compatible with the Lie bracket of vector fields.  We firstly   show  how most geometrical structures intensively studied  in the framework of Lie algebroid  can easily be extended  in the  pre-Lie algebroid context. The principal purpose of this paper is to show that how all  these results only depend of  the foliated  structure and do not depend  of the pre-Lie algebroid structure  that we can put on a  foliated anchored bundle.  As application,  we  obtain a  Finsler connection on a foliated anchored bundle which induces the  classical Finsler connection on each leaf and  a similar  result for the Chern connection.
  \end{abstract}

  \bigskip
 {\bf  Mathematics Subject Classication 2000}: 17B66, 53C05, 93B05.


%\documentclass[slidestop,compress,mathserif]{beamer}                %Version ÔøΩcran
%\documentclass[10pt,handout]{beamer}      % Version imprimable

%\usepackage[francais]{babel}
%\usepackage[applemac]{inputenc}
%\usepackage[T1]{fontenc}
%\usepackage{beamerthemesplit}
%\definecolor{MidnightBlue}{rgb}{0.098,0.098,0.439}

%\mode<handout>
%{\usepackage{pgfpages}
%\pgfpagesuselayout{4 on 1}[a4paper,border shrink=3mm,landscape]
%}
%\usetheme{CambridgeUS}
%\usetheme{Berkeley}
%\usetheme{Antibes}
%\usecolortheme{dolphin}
%\setbeamerfont{frametitle}{series=\bfseries}
%\setbeamercolor*{frametitle}{fg=MidnightBlue}

%\newcommand{\R}{\field{R}}

%\begin{document}

%\frame{\titlepage}
%%%%%%%%%%%%%%%%%%%%%%%%%%%%%%%%%%%%%%%%%%%%%%%%%
\tableofcontents
%%%%%%%%%%%%%%%%%%%%%%%%%%%%%%%%%%%%%%%%%%%%%%

\section{Introduction}

%%%%%%%%%%%%%%%%%%%%%%%%%%%%%%%%%%%%%%%%%%%%%%%%%%%%%%%%%%%%%%%%%%%%%%%%%%%%%%%%%%

The geometry of second order differential equations on a manifold is closely related to the geometry of nonlinear connections. If  $S$ is semispray  there exists a canonical nonlinear connection ${\cal N}_S$ associated with $S$. Such a construction  were firstly introduced by M. Crampin \cite{C} and J. Grifone \cite{Gr}. On the other hand if $g$ is a metric on the vertical tangent bundle, we can associate to the pair $(S, g)$ a metric nonlinear connection. Using this nonlinear connection ${\cal N}_S$, we determine the whole
family of metric nonlinear connections that correspond to the pair $(S, g)$. Moreover under regularity  conditions for a Lagrangian $L$ we can associate a symplectic form on the tangent bundle and a metric on the vertical tangent.  Then, in this case,  the canonical nonlinear connection, which can be associated with
the Euler-Lagrange equations, is the unique nonlinear connection which is compatible with the metric and symplectic structure (\cite{Bu}).
Generalization of these results in the context of the prolongation of a Lie  algebroid on itself or its dual  has been the subject of an intense activity of research (Among a full of papers  see for instance \cite{We}, \cite{Lib}, \cite{CMM}, \cite{GLMM}, \cite{Va} and references inside).\\ %The reader can find in \cite{LPo4} a  nice and complete survey of these results and also the most important references on this subject. Moreover 

We show that   these results can be generalized in the {\it framework of the prolongation  over a fibered manifold of a pre-Lie algebroid }  {\it i.e.}  an anchored bundle provided with an almost Lie bracket such that the anchor is compatible with the Lie bracket of vector fields (see section \ref{ALbracket} for a definition Remark \ref{almquasi} about this terminology)).  %Note that comparable results on the prolongations of a Lie algebroid over this algebroid or on its dual have been proved by L. Popescu in a series of papers \cite{LPo2}, \cite{LPo3}, \cite{LPo4} and \cite{LPo5}

%More precisely,  an anchored bundle  is a vector bundle $\t:{\cal A}\ap M$ provided with a morphism $\rho$ (called {\it anchor})  from a vector bundle $\t:{\cal A}\ap M$ into $TM\ap M$. When the module of vector fields generated by the module $\{\rho(s) :  s \textrm{ section of } {\cal A}\} $ is involutive, the distribution $\rho({\cal A})$ is integrable and $({\cal A},M,\rho)$ is called a {\it foliated anchored bundle}. In this situation,   $\cal A$  can be provided with a pre-Lie algebroid structure.

More precisely, 
as it is well known, given an  anchored bundle $({\cal A}, M,\rho)$  and  $\pi:{\cal M}\ap M$  a fibered manifold, we can defined the prolongation of $\cal A$ {\it i.e.} a vector bundle  $\T{\cal M}\ap {\cal M}$ and an anchor $\hat{\rho}: \T{\cal M}\ap T{\cal M}.$ If moreover $({\cal A},M,\rho,[\;,\;]_{\mathcal{A}})$ is a pre-Lie algebroid, the almost Lie bracket $[\;,\;]_{\cal A}$ on $\cal A$  can be prolonged to an almost bracket $[\;,\;]_{\cal P}$ on $\T{\cal M}$ so that $( \T{\cal M},{\cal M},\hat{\rho},[\;,\;]_{\cal P})$ is again a pre-Lie algebroid.
%Now as in the framework  of Lie algebroid ({\it cf.} \cite{HiMa}), given a pre-Lie algebroid $({\cal A},M,\rho, [\;,\;]_{\cal A})$ and a fibered manifold $\cal M$ on $M$
% we can associate a kind of tangent bundle $\T{\cal M}$ on $\cal M$ and an anchor $\hat{\rho}:\T{\cal M}\ap T{\cal M}$  whose range also defines a foliation of $\cal M$. Moreover, the almost Lie bracket $[\;,\;]_{\cal A}$ on $\cal A$  can be prolonged to an almost bracket $[\;,\;]_{\cal P}$ on $\T{\cal M}$ such that $\hat{\rho}[\;,\;]_{\cal P}=[\hat{\rho}(.),\hat{\rho}(.)]_{T{\cal M}}$.  
Note that for any pre-Lie algebroid $({\cal A},M,\rho,[\;,\;]_{\mathcal{A}})$ the range $\rho(\mathcal{A})$ defines a Stefan-Sussmann foliation on $M$ and then we get also a  Stefan-Sussmann foliation on $\cal M$.\\
In this situation we can look for Euler section, almost tangent and almost cotangent structures, semisprays, nonlinear connections and their relative properties. As in classical Lie algebroid framework we can generalize the notion of nonlinear connection associated with a semispray on $\T{\cal M}$ and also associate to a regular (eventually local) Lagrangian $\cal L$ on $\cal M$ a unique nonlinear connection ${\cal N}_{\cal L}$ which is compatible with respect to the metric and a symplectic structure also canonically  associated with $\cal L$. Moreover, when $\cal M$ is an open subset of $\cal A$, then  $\cal L$ induces a regular Lagrangian $L_N$ on any leaf $N$ of the foliation defined $\rho({\cal A})$ and ${\cal N}_{\cal L}$ induces on $N$ the unique nonlinear connection ${\cal N}_{\cal L}$ which is compatible with respect to the metric and
symplectic structure   associated with ${L_N}$. In fact all these data on each leaf $N$ do not depend on the  initial choice of the almost Lie bracket $[\;,\;]_{\cal A}$ on $\cal A$ (such that $({\cal A},M,\rho,[\;,\;]_{\cal A})$ is a pre-Lie algebroid). Moreover, a curve is a locally minimizing geodesic for the (global) Lagrangian $L$ if and only if it is a locally minimizing geodesic in some leaf $N$ of $L_N$.\\
 In the particular case where ${\cal L}=\dis\frac{1}{2}{\cal F}^2$ where $\cal F$ is a (partial) Finsler metric on $\cal M$,  these connections ${\cal N}_{\cal L}$  and ${\cal N}_{L_N}$ are the Finsler connections on $\cal M$ and on $N$ respectively. Moreover, we can also define a Chern connection associated with $\cal L$ which also induces the Chern connection associated with $L_N$ on $N$.\\

  Now to an anchored bundle  $({\cal A}, M,\rho)$ is associated  the module of vector fields generated by the module $\{\rho(s) :  s \textrm{ section of } {\cal A}\} $. When this module  is involutive it generates an integrable distribution $\rho({\cal A})$   in the sense of Stefan-Sussmann  and then $({\cal A},M,\rho)$ is called a {\it foliated anchored bundle}. We show that $({\cal A},M,\rho)$ foliated anchored bundle if and only if $\cal A$  can be provided with a pre-Lie algebroid structure.
  This implies that {\it the previous results (in particular the results about Finsler geometry) depend only of the foliated anchored bundle}  $({\cal A},M,\rho)$ and so are independent of the pre-Lie structure $({\cal A},M,\rho,[\;,\;]_{\mathcal{A}})$ we can put on $({\cal A},M,\rho)$.\\

 A brief of outline  of this paper is as follows. Each section begins with a short  abstract  of the context and  results developed inside.  
 
  the second  section is devoted to the classical framework of the prolongation of a pre-Lie algebroid over  a fibered manifold and the notion of Euler section.
   
  We give results about relations between integrable distributions and  foliated anchored bundles and pre-Lie algebroids in section 3. We look for  the specific case of the prolongation of a foliated anchored bundle at the end of this section .
  
    The  section 4 is devoted to the introduction of geometrical objets like almost tangent structures, almost cotangent structures, semisprays and the relative properties between these data.
    
     We develop   the framework of nonlinear connection and semispray  in the context of pre-Lie algebroid  in the section 5 and then we give some links between nonlinear connections, semisprays and almost tangent structures or Euler sections.
     
      The essential results about Lagrangian metric connections and semispays are contained in Section 6. In particular, in this section  one can find the characterization of the unique Lagrangian metric connection associated with a semi-Hamiltonian almost tangent structure (see Theorem \ref{SLagMet}) which generalizes such  a classical result  about regular Lagrangian ( see for instance \cite{LPo4}). We end this section by an application to mechanical systems.
      
      In   section 6   we discuss the problem of minimization of a convex Lagrangian  and look for characterization of extremals and minimizers in terms of Hamiltonian framework. This section also contains  the essential results about  the induced structure of extremals and nonlinear connection on each leaf of the foliation defined by the anchor (see Theorem \ref{extremalN}). 
      
      The previous results are applied to the context of  partial Finsler geometry in the last section.

%%%%%%%%%%%%%%%%%%%%%%%%%%%%%%%%%%%%%%%%%%%%%%%%
%%%%%%%%%%%%%%%%%%%%%%%%%%%%%%%%%%%%%%%%%%%%%%%%
\section{Prolongation of   a pre-Lie algebroid over a fibered manifold}
%%%%%%%%%%%%%%%%%%%%%%%%%%%%%%%%%%%%%%%%%%%%%%%%%%%
%%%%%%%%%%%%%%%%%%%%%%%%%%%%%%%%%%%%%%%%%%%%%%%%%%%
In this section we recall the context and the essential properties about almost Lie algebroids   in reference to  a plentiful literature about Lie algebroid structure (see for example \cite{CMM}, \cite{GLMM}, \cite{GMM}, \cite{LMM}, \cite{Marl}, \cite{Marr}, \cite{PoPo} and all references inside these papers). \\

%%%%%%%%%%%%%%%%%%%%%%%%%%%%%%%%%%%%%%
\subsection{Almost Lie bracket }\label{ALbracket}${}$\\
%%%%%%%%%%%%%%%%%%%%%%%%%%%%%%%%%%%%%%%%%%%%%%%%%%%%%%%%%%%%
%%%%%%%%%%%%%%%%%%%%%%%%%%%%%%%%%%%%%%%%%%
%\subsubsection{Structures}
%%%%%%%%%%%%%%%%%%%%%%%%%%%%%%%%%%%%%%%%%%%%
%%%%%%%%%%%%%%%%%%%%%%%%%%%%%%%%%%%%%%%%%%%%%
%In this section we recall the context and the essential properties about almost Lie algebroid   in reference to  a plentiful literature about Lie algebroid structure (see for example \cite{CMM}, \cite{GLMM}, \cite{GMM}, \cite{LMM}, \cite{Marl}, \cite{Marr}, \cite{PoPo} and all references inside these papers). \\
${}\;\;\;\;$Let  $M$ be a paracompact connected manifold   of dimension $n$. We denote by $C^{\infty}(M)$ the ring of   $C^{\infty}$ functions on  $M$ and by
    $\Xi(M)$ the  $C^{\infty}(M)$-module of smooth vector fields on $M$.
  We consider     a vector bundle $\tau:{\cal A} \ap M$ over $M$  of rank $k$ and
 $\Xi({\cal A})$ the   $C^{\infty}(M)$-module of sections of  $\cal A$. \\
% {\bf Throughout this work we always assume that $0<k\leq n$} .\\
 An {\bf $\bf{\cal A}$-tensor of type $(h,l)$}  will be a smooth section of the associated $\bigotimes^h{\cal A}\bigotimes^l{\cal A}^*$.\\

An  { \it  anchor  } is 	a bundle morphism $\rho:{\cal A}\ap TM$.  In this case the triple $({\cal A},M,\rho)$ is called an {\it  anchored bundle}.
Given an anchored bundle $({\cal A},M,\rho)$, a curve $\bar{c}:I\subset \R\ap {\cal A}$ is called an {\it  admissible curve} if there exists a curve $c=I\ap M$ such that:

$ \tau (c(t))=c(t)$  and $\dot{c}(t)=\rho (\bar{c}(t))$ for any $t\in I$.\\

An {\it almost Lie bracket} on an anchored bundle $({\cal A},M,\rho)$ is a bilinear map
   $[\;,\;]:\Xi({\cal A})\times \Xi({\cal A}): \ap \Xi({\cal A})$
  which satisfies the following properties:
\begin{itemize}
 \item $[\;,\;]$ is antisymmetric;
 \item the Leibnitz property:
$ [X,fY]=f . [X,Y]+df(\rho(X)).Y ,\;\;\forall X,Y\in\Xi({\cal A})$.
 \end{itemize}

A {\bf Lie bracket } is an  almost Lie bracket which  satisfies:
\begin{itemize}
\item the Jacobi identity:
 $ [X,[[Y,Z]]_{\cal A}+[Y,[[Z,X]]+[Z,[[X,Y]]=0,\;\;\forall  X,Y, Z\in\Xi({\cal A})$.
\end{itemize}
An {\bf almost Lie algebroid} is an anchored bundle $({\cal A},M,\rho)$ provided with an almost Lie bracket $[\;,\;]$. Such a structure is denoted by $({\cal A},M,\rho,[\;,\;])$. When $[\;,\;]$ is in fact a Lie bracket  the associated structure $({\cal A},M,\rho,[\;,\;])$ is called a {\bf Lie algebroid}.
In this case,  $\rho:\Xi({\cal A})\ap \Xi(M)$  is a Lie algebra morphism and, in particular, $\rho$ is compatible with the brackets on $\cal A$  and $TM$ {\it i.e.} $[\rho X,\rho Y]=\rho([X,Y]_{\cal A})$. Note that in general such a compatibility does not imply the Jacobi identity  for  $[\;,\;]_{\cal A}$ (take $\rho\equiv 0$ for instance). If we have  $[\rho X,\rho Y]=\rho([X,Y]_{\cal A})$ for all sections $X,Y\in\Xi({\cal A})$ we  will say that  $\rho$ is a {\it Lie morphism}. In this case $( {\cal A}, M, \rho),[\;,\;]_{\cal A})$ is called a {\bf   pre-Lie algebroid}. \\

\begin{Rem}\label{almquasi}${}$
\begin{enumerate}
\item  If $[\;,\;]$ is an almost Lie bracket on $\cal A$, for any skew-symmetric $\cal A$-tensor $\Theta$ of type $(2,1)$,  $[\;,\;]'=[\;,\;]+\Theta$ is also an almost Lie bracket on $\cal A$.
 \item Since the terminology  of almost Poisson bracket  seems generally adopted in the most recent papers on nonholonomic mechanics,  in our work, we have used the definition of an almost Lie algebroid given in \cite{LMM}.  Moreover  many classical  geometrical objets  like  "almost tangent structure",  "almost cotangent structure", ...  need only an almost Lie bracket in the framework of mechanics. It is the reason why this terminology  seems us to be well adapted. 
 However, on the one hand, in  \cite{PPo},  the term "algebroid" will designate  an almost algebroid such that the anchor is compatible with the Lie bracket of vector fields and  on the other hand, in \cite{PoPo} or in \cite{GrJo}  such a structure  is called an almost Lie algebroid.  In a recent preprint \cite{Yu} such a structure is called a pre-lie algebroid. It is the reason why we adopt the terminology of "pre-Lie algebroid"  to avoid any previous ambiguity.
 \end{enumerate}
  \end{Rem}

On an anchored bundle  $({\cal A},M,\rho)$, there always exists an almost Lie bracket:
let  ${\nabla}$ be a linear connection  on $\cal A$  then we get an almost Lie bracket $[\;,\;]_\nabla$ defined by

$$[\;,\;]_\nabla=\nabla_{\rho X} Y-\nabla_{\rho Y} X.$$

More generally,  an {\bf  ${\cal A}$ linear connection} $\nabla$ on  an anchored  vector bundle  $({\cal A},M,\rho)$  is a

\noindent $\R$-bilinear map $\nabla:\Xi({\cal A})\times \Xi({\cal A})\ap \Xi({\cal A})$ such that

$\nabla _{f X }Y=f\nabla _{X }Y$ and $\nabla _{ X }f Y=df(\rho X)Y+f\nabla _{X }Y$

\noindent for all functions $f$ and all $X, Y \in \Xi ({\cal A})$.

  Given a $\cal A$ linear connection on $\cal A$ for any $X,Y\in \Xi({\cal A})$,   again we get an associated  almost Lie bracket $[\;,\;]_\nabla$ defined by
  $$[X,Y]_\nabla=\nabla_X Y-\nabla_Y X.$$

 \begin{Pro}\label{setALB} ${}$\\
  The set $\mathfrak{C}$ (resp $\mathfrak{L}$)  of  $\cal A$ linear connections (resp. of almost Lie brackets) on an anchored bundle $({\cal A},M,\rho)$ has a structure of affine space modeled on the vector space of tensors $\cal A$-tensors  of type $(2,1)$ (resp. tensors $\cal A$-tensors antisymmetric of type $(2,1)$). Moreover the map $\mathfrak{T}:\nabla\ap [\;,\;]_\nabla$ is an affine  surjection $\mathfrak{C}$  to $\mathfrak{L}$.
  \end{Pro}

  \begin{proof}${}$\\
  The first affirmation is clear. Fix some $\cal A$ linear connection $\nabla_0$ and set  $[\;,\;]_0=[\;,\;]_{\nabla_0}$ the associated almost  Lie bracket  on $({\cal A},M,\rho)$. Given any $\nabla\in \mathfrak{L}$ we set $D=\nabla-\nabla_0$, then $\mathfrak{T}(\nabla)=[\;,\;]_\nabla$ is characterized by
 $$ [X,Y]_\nabla=[X,Y]_0+D(X,Y)-D(Y,X).$$ Therefore $\mathfrak{T}$ is an affine map. Now, given any almost bracket $[\;,\;]$  set $T=[\;,\;]-[\;,\;]_0$ then $\nabla=\nabla_0+\dis\frac{1}{2}T$ is a linear connection such that $\mathfrak{T}(\nabla)=[\;,\;].$
  \end{proof}

Let  $(x^1,\dots,x^n)$ be  a system of  local coordinates on $M$ defined on a chart domain $U$
 and  $\{e_1,\dots ,e_k\}$  a local basis of $\cal A$ on the same domain (after restriction if necessary). Then
$u=(x,a)\in{\cal A}$ can be written   $a=y^\alpha e_\alpha$
and we have:
 $$\rho(e_\alpha)=\rho ^i_\alpha \displaystyle\frac{\partial}{\partial x^i}\;\; \textrm{ and }\;\;
 [e_\alpha,e_\beta]=C_{\alpha\beta}^\gamma e_\gamma.$$

We denote by
 $\tau^*:{\cal A}^*\ap M$ the dual bundle of $\tau:{\cal A}\ap M$,
 and  $\rho^*:T^*M\ap {\cal A}^*$ the  transposed morphism.  This gives rise to the {\it almost differential}  of a function $f:M\ap \R$: $d^{\cal A}f=\rho^*\circ df$.

 The {\it almost exterior differential   } $d^{\cal A}\omega\in \Lambda^{k+1}{\cal A}^*$ is characterized by: \\
$d^{\cal A}\omega(X_0,\cdots, X_k)= \dis\sum_{j=1}^k(-1)^j{\cal L}^{\cal A}_{X_j}  \omega((X_0,\cdots, ,\hat{X}_j,\cdots,X_k))$\\
${}\;\;\;\;\;\;\;\;\;\;\;\;\;\;\;\;\;\;\;\;\;\;\;\;\;\;\;\;\;\;\;\;\;\;\;\;\;\;\;\;\;\;\;\;-\dis\sum_{0\leq i<j\leq k}(-1)^{i+j}\omega([X_i,X_j]_{\cal A}X_0,\cdots,\hat{X}_i,\cdots,\hat{X}_{j},\cdots X_k)$, \\
for $k>0$ and  for $k=0$ it is the almost  differential of a function.\\
Moreover   $({\cal A},M,\rho,[\;,\;]_{\cal A})$ is a Lie algebroid if and only if
$d^{\cal A}\circ d^{\cal A}=0.$  Of course this no more true if the Lie bracket does not satisfy the Jacobi identity. However  $({\cal A},M,\rho,[\;,\;]_{\cal A})$ is a pre algebroid if and only if $d^{\cal A}\circ \rho^*df=0$ for all $f\in C^\infty(M)$.\\

 As classically, the  {\it inner product } $i_a\o \in \Lambda^{k-1}{\cal A}^*$ is the contraction of $\o$ by $a\in{\cal A}$. The {\it Lie derivative } along  $X\in\Xi({\cal A})$ is then defined by ${\cal L}^{\cal A}_X=i_X\circ d^{\cal A}+ d^{\cal A}\circ i_X$. \\

 As in classical differential geometry on a manifold, we can define the Fr\^{o}licher-Nijenhuis bracket and
 the Schouten bracket  for ${\cal A}$-tensors  (see \cite{GrUr}).\\

 %%%%%%%%%%%%%%%%%%%%%%%%%%%%%%%%%%%%%%%%%%%%%%%%%%%%%%%%%%%%%%%%%%%%%%%%%%%%%%
\subsection{Almost Poisson bracket and almost Lie bracket }\label{Adual}${}$\\
%%%%%%%%%%%%%%%%%%%%%%%%%%%%%%%%%%%%%%%%%%%%%%%%%%%%%%%%%%%%%%%%%%%%%%%%%%% %%%%%%%%%%%%%%%%%%%%%%%%%%%%%%%%%%%%%%%%%%%%%%%%%%%%%%%%
${}\;\;\;\;$ An {\it almost Poisson bracket} is a map
  $\{\;,\;\}:C^\infty(M)\times C^\infty(M)\ap C^\infty(M) $  such that
   \begin{enumerate}
 \item[(i)] $\{\;,\;\}$  $\R-$bilinear,
 \item[(ii)] Liebnitz property:
 $\{f,gh\}=g.\{f,h\}+h.\{f,g\}.$
  \end{enumerate}
  If  moreover we have :
   \begin{enumerate}
 \item[(iii)] Jacobi identity:
 $\{f,\{g,h\}\}+\{g,\{h,f\}\}+\{h,\{f,g\}\}=0,$
 \end{enumerate}
 then  $\{\;,\;\}$ is called a {\bf  Poisson bracket}.\\

   The datum  $\{\;,\;\}$ is equivalent to the existence of a bi-vector  $P\in \Lambda^2T^*M$  \textrm{ characterized by  } $P(df,dg)=\{f,g\}$ which is  called {\bf  almost Poisson bi-vector}. Such a bi-vector can also defined by a
   morphism $P^\flat:T^*M\ap TM$ \textrm{ such  }
 $$\{f,g\}=<dg,P^{\flat} df>=-<df,P^\flat dg>=P(df,dg).$$
 called the {\bf  almost Poisson map}.

\noindent The almost Poisson bi-vector  $P$ defined a Poisson bracket if and only if
 the Schouten bracket satisfies $[P^\flat,P^\flat]=0,$

where
 $ [P^\flat,P^\flat](\omega,\omega')=P^{\flat}(L_{P^\flat\omega}\omega'-L_{P^\flat\omega'}\omega+d<\omega,P^\flat\omega'>+[P^\flat\omega,P^\flat\omega']).$\\

 Given an almost Poisson bracket $\{\;,\;\}$ on $M$,  recall that the  {\bf   Hamiltonian field}  of a smooth  map $h:{\cal A}^*\ap \R$
 is  the vector field $X_h=P^\flat dh$.

 Consider a vector bundle  $\tau:{\cal A}\ap M$  and
$\tau^*: {\cal A}^*\ap M$  its dual bundle:

$\bullet $ a function $f$ on  ${\cal A}^*$ is called {\bf  linear}  if its  restriction to each  fiber ${\cal A}^*_x=(\tau^*)^{-1}(x)$ is a linear form;

$\bullet$ an almost Poisson bracket $\{\;,\;\}$ on  ${\cal A}^*$ is called  {\bf linear} if, for any pair   $(f,g)$ of linear functions on ${\cal A}^*$, then
$\{f,g\}$ is also a linear function.\\

 If  $\{\;,\;\}$  is linear on   ${\cal A}^*$, we also say that its associated almost Poisson  bi-vector $P$ or  morphism $P^\flat$ is  linear.\\

Let  $\{e_\alpha\}$ be a local basis of  ${\cal A}$ and  $(x^i, u^\alpha)$\footnote{  index i,j,k,l... vary from $1$ to $n$ and greek index $\a, \b,\g....$ vary from $1$ to $k$ and we use the Einstein convention of summation} the associated coordinates. We denote by  $\{\epsilon^\alpha\}$ the associated dual basis on  ${\cal A}^*$ and  $(x^i,\eta_\alpha)$ the associated coordinates on ${\cal A}^*$. Therefore, in local coordinates, an almost linear Poisson bi-vecteur $P$  on  ${\cal A}^*$ can be written
$$P=\frac{1}{2}C_{\alpha\beta}^\gamma\eta_\gamma \frac{\partial}{\partial \eta_\alpha}\wedge\frac{\partial}{\partial \eta_\beta}+ \rho_\alpha^i\frac{\partial}{\partial \eta_\alpha}\wedge\frac{\partial}{\partial x^i}.$$

For any $s\in \Xi({\cal A})$ we denote by  $\Phi_s:{\cal A}^*\ap \R$ the map given by :
$\Phi_s(\sigma)=\sigma(s).$\\
Note that such a function $\Phi_s$  is a linear function on ${\cal A}^*$. \\

\noindent {\it The relation between almost linear   Poisson bracket  and almost  Lie algebroid is given by (see for instance  \cite{GLMM}, \cite{LMM}, \cite{Marl} or  \cite{PoPo}):}\\
\begin{The}
Consider  a vector bundle $\tau:{\cal A}\ap M$   and $\rho:{\cal A}\ap TM$ an anchor. The datum of an almost linear Poisson bracket $\{\;,\;\}$ on ${\cal A}$  is equivalent to the datum of an   almost linear Poisson bi-vector  $P$ on ${\cal A}^*$ characterized by the following relations:
\begin{itemize}
\item $\Phi_{[s,s']_{\cal A}}=P(d\Phi_s,d\Phi_{s'})$  $\forall s,s'\in \Xi({\cal A})$, and
\item $P(d\Phi_s,d(f\circ\tau^*))=df(\rho(s))\circ\tau^*$ $\forall s\in\Xi({\cal A})$, $\forall f\in C^\infty(M)$.
\end{itemize}
\end{The}

\bigskip
In {\it local coordinates} we have:\\
  $\rho = \rho_\alpha^i\dis\frac{\partial}{\partial x^i}$
   and  $[e_\alpha,e_\beta]_{\cal A}=C_{\alpha\beta}^\gamma e_\gamma$
if and only if
$P=\dis\frac{1}{2}C_{\alpha\beta}^\gamma\eta_\gamma \frac{\partial}{\partial \eta_\alpha}\wedge\frac{\partial}{\partial \eta_\beta}+ \rho_\alpha^i\frac{\partial}{\partial  \eta_\alpha }\wedge\frac{\partial}{\partial x^i}$.\\

Given a linear bi-Poisson vector on ${\cal A}^*$ associated with an almost Lie bracket $[\;,\;]_{\cal A}$ on $\cal A$, the Hamiltonian vector field ${\cal X}_h$ of a smooth map $h$ has the following local decomposition:

\begin{eqnarray}\label{locVH}
{\cal X}_h =\rho_\a^i\dis\frac{\p h}{\p \eta^\a}\frac{\p}{\p x^i}-(C_{\a\b}^\g \eta_\g\dis\frac{\p h}{\p \eta^\b}+\rho_\a^i\dis\frac{\p h}{\p x^i})\frac{\p}{\p \eta_\a}.
\end{eqnarray}\\

Given two almost Lie algebroids $({\cal A},M,\rho,[\;,\;])_{\cal A}$ and $(\bar{\cal A},M,\bar{\rho},[\;,\;]_{\bar{\cal A}})$, set    $\t^*:{\cal A}\ap M$ and $\bar{\t}^*:\bar{\cal A}^*\ap M$ their associated dual bundles respectively. We denote by $\{\;,\;\}_{\cal A}$ and $\{\;,\;\}_{\bar{\cal A}}$ the Poisson brackets on ${\cal A}^*$ and $\bar{\cal A}^*$ associated with $[\;,\;]$ and $\overline{[\;,\;]}$ respectively. A morphism bundle $\digamma:{\cal A}^*\ap \bar{\cal A}^*$ is a {\it linear almost  Poisson morphism} if we have:
$$\{\bar{\phi}\circ\digamma,\bar{\psi}\circ\digamma\}_{\cal A}=\{\bar{\phi},\bar{\psi}\}_{\bar{\cal A}}\circ \digamma$$
for all smooth functions $\bar{\phi}$ and $\bar{\psi}$ on $\bar{\cal A}^*$.\\
\noindent In this case, assume that we have two smooth functions $h$ and $ \bar{h}$ on ${\cal A}^*$ and $\bar{\cal A}^*$ respectively such that
$$\bar{h}\circ \digamma=h.$$

Then, if  ${\cal X}_h$ and $\bar{\cal X}_{\bar{h}}$ are the respective Hamiltonian vector fields of $h$ and $\bar{h}$, we then have
$$T_\eta\digamma({\cal X}_h)=\bar{\cal X}_{\bar{h}}(\digamma(\eta)).$$
In particular, if $c:I\subset \R\ap {\cal A}$ is an integral curve of ${\cal X}_h$, then $\digamma\circ c : I\ap \tilde{\cal A}^*$ is an integral curve of  $\bar{\cal X}_{\bar{h}}.$
\bigskip
Now we have the following characterization of a pre-Lie algebroid

\begin{Pro}\label{characalgebroid}${}$\\
The almost Lie algebroid $({\cal A},M,\rho,[\;,\;]_{\cal A})$ is a pre-Lie algebroid if and only if $\rho^*:T^*M\ap {\cal A}^*$ is a linear almost  Poisson morphism.
\end{Pro}

\begin{proof}${}$\\
Clearly it is sufficient to prove the result locally. Consider a chart $(U,x^i)$ over $M$ such that ${\cal A}$ is trivializable over $U$. Given a local basis $\{e_\a\}$ of $\cal A$ and $\{\epsilon^\alpha\}$ its dual basis. If $(x^i,\eta_\b)$ is the associated coordinate system on ${\cal A}^*$, recall that the linear almost Poisson bivector $P_{\cal A}$ associated with the almost bracket $[\;,\;]_{\cal A}$ has a local decomposition of type:
$$P_{\cal A}=\frac{1}{2}C_{\alpha\beta}^\gamma\eta_\gamma \frac{\partial}{\partial \eta_\alpha}\wedge\frac{\partial}{\partial \eta_\beta}+ \rho_\alpha^i\frac{\partial}{\partial \eta_\alpha}\wedge\frac{\partial}{\partial x^i}.$$
Therefore the associated almost Poisson bracket is characterized by:
$$\{x^i,x^j\}_{\cal A}=0,\;\;\;\; \{\eta_\a,x^i\}_{\cal A}=\rho_\a^i,\;\;\;\; \{\eta_\a,\eta_\b\}_{\cal A}=\dis\frac{1}{2}C_{\a\b}^\g \eta_\g.$$
On the other hand, if $(x^i,\xi_i)$ is the coordinate system on $T^*M$, the canonical Poisson bracket is characterized by
$$\{x^i,x^j\}=0,\;\;\;\; \{\xi_i,x^j\}=\d_i^j,\;\;\;\; \{\xi_i,\xi_j\}=0.$$
Therefore on the one hand we have

$\{x^i\circ \rho^*,x^i\circ \rho^*\}=0,\;\;\;\;  \{\eta_\a\circ\rho^*,x^j\circ\rho^*\}=\{\rho_\a^i\xi_i,x^j\}=\rho_\a^i,$\\
${}\hfill\textrm{ and }\{\eta_\a\circ\rho^*,\eta_\a\circ\rho^*\}=\{\rho_\a^i\xi_i,\rho_\b^i\xi_j\}=(\rho_\b^i\dis\frac{\p \rho_\a^l}{\p x^i}-\rho_\a^i\dis\frac{\p \rho_\b^l}{\p x^i})\xi_l.\;\;\;\;\;\;\;$\\
On the other hand we have

$\{x^i,x^j\}_{\cal A}\circ\rho^*=0,\;\;\;\; \{\eta_\a,x^i\}_{\cal A}\circ\rho^*=\rho_\a^i,\;\;\;\; \{\eta_\a,\eta_\b\}_{\cal A}\circ\rho^*=C_{\a\b}^\g \rho_\g^l\xi_l.$\\
Therefore $\rho^*$ is a linear Poisson morphism if and only  if
$$(\rho_\b^i\dis\frac{\p \rho_\a^l}{\p x^i}-\rho_\a^i\dis\frac{\p \rho_\b^l}{\p x^i}-C_{\a\b}^\g \rho_\g^l)\xi_l=0$$
for all indices $l=1,\dots,n$ and  $1\leq \a<\b\leq k$.

This condition is equivalent to:
$$(\rho_\b^i\dis\frac{\p \rho_\a^l}{\p x^i}-\rho_\a^i\dis\frac{\p \rho_\b^l}{\p x^i})-C_{\a\b}^\g \rho_\g^l=0$$
for all $1\leq\a<\b\leq k$  and $l=1,\dots,n$.
But this last condition is exactly
$$[\rho(e_\a),\rho(e_\b)]=\rho([e_\a,e_\b]_{\cal A}$$
for all $1\leq\a<\b\leq k$. This ends the proof.\\
\end{proof}

% %%%%%%%%%%%%%%%%%%%%%%%%%%%%%%%%%%%%%%%%%%%%%%%%%%%%%%%%%%%%%%%%%%%%%%%%%%%%%
\subsection{Fibered manifold and Euler vector field}\label{verteuler}${}$\\
%%%%%%%%%%%%%%%%%%%%%%%%%%%%%%%%%%%%%%%%%%%%%%%%%%%%%%%%%%%%%%%%%%%%%%%%%%%%%%%%
%%%%%%%%%%%%%%%%%%%%%%%%%%%%%%%%%%%%%%%%%%%%%%%%%%%%%%%%%%%%%%%%%%%%%%%%%%%%%%%%
${}\;\;\;\;$ In the whole paper $\pi:{\cal M}\ap M$ will be a  fibered  manifold of dimension $n+h$ with a connected typical fiber.  Then this structure gives rise to a foliation on $\cal M$ whose leaves are the fiber of $\pi$. It follows that  we have an atlas $(U_\l, \Phi_\l)$ on $\cal M$ such that the transition functions

$$\Psi_{\l\mu}=\Phi_\mu\circ\Phi_\l^{-1}$$
are diffeomorphisms  of $\R^n\times \R^h$ whose type is
\begin{eqnarray}\label{chgtfib}
\Psi_{\l\mu}(x,y)=(\Psi^1_{\l\mu}(x),\Psi^2(x,y)).
\end{eqnarray}

The fibration is called {\bf locally affine}  or  {\bf  locally linear} if there exists a sub-atlas whose transition functions are of type:

$$\Psi_{\l\mu}(x,y)=(\Psi^1_{\l\mu}(x),\Psi^2(x,y))$$ where $\Psi^2(x,y)$ is an affine  or  linear transformations  of  $\R^h$ respectively  of type
\begin{center}
 $\{\Phi^A_B(x)y^B+\phi^B(x)\}_{A=1,\cdots, h}$
or  $\{\Phi^A_B(x)y^B\}_{A=1,\cdots, h}$.
\end{center}

On the other hand, an Euler vector field $C$ on $\cal M$ is a global vector field on $\cal M$ which is tangent to the fiber of $\pi$ ({\it i.e.}  {\it vertical}) and such that the flow of $C$ is an infinitesimal  homothety
on each fiber. This last property is equivalent to the existence of a local coordinate system $(x^i, y^A)$\footnote{ index $A,B,C,...$ vary from $1$ to $h$} (compatible with $\pi$)  such that $C$ can be written:
\begin{eqnarray}\label{EC}
C=y^A\dis\frac{\p}{\p y^A}.
\end{eqnarray}

Note that, if the fibration is locally linear, according to the associated atlas, (\ref{EC}) gives rise to a global vector field which is an Euler vector field. Conversely,
if such a vector field exists we have then:

\begin{Pro}\label{Euler}(see for example \cite{Va2}, \cite{Va3} \cite{ThSc})${}$
\begin{enumerate}
\item[(i)]   if there exists an Euler vector field on $\cal M$,    this   fibration is a  locally linear foliation.
\item [(ii)] On $\cal M$ there exists a sub-atlas such that the transition functions are of type
\begin{eqnarray}\label{Eulerchgt}
\Psi_{\l\mu}(x,y)=(\Psi^1_{\l\mu}(x),Id+\phi(x))
\end{eqnarray}
if and only if there exists an Euler vector field on $\cal M$.
\end{enumerate}
\end{Pro}
 In this last  case, we will say that  the fibration is {\it  locally  linear.} More generally, if the foliation defined by the fibration is locally affine, we will say that the fibration is {\it locally affine}.\\
 Therefore  a   locally linear fibration structure is characterized by  a choice of an Euler vector field $C$ on $\cal M$.
The context of  Proposition \ref{Euler} part (ii) means that if $(x^i,y^A)$ and $(\bar{x}^j,\bar{y}^B)$ are two coordinate systems compatible with the sub-atlas  of type (ii) which are defined on a common  open set then we have
\begin{eqnarray}\label{loceuler}
 C=(\bar{y}^A+\bar{\psi}^A(\bar{x}))\dis\frac{\p}{\p \bar{y}^A}=(y^A+\psi^A(x))\dis\frac{\p}{\p y^A}.
\end{eqnarray}
 Such an atlas is  also characterized by the choice of the Euler section $C$.\\

In general,  when the fibration is locally affine, the existence of an Euler vector field is not true and the obstruction to such an existence is characterized by a cohomology class  $E(\pi)$ in the "foliated cohomology" of $\cal M$ (see for instance \cite{Va2}). The nullity of this obstruction is equivalent to the existence of a sub-atlas on $\cal M$ whose  transition functions  are of type (\ref{Eulerchgt}).

\begin{Ex}\label{C1}${}$\\
Given  any vector bundle $\pi:{\cal M}\ap M$   over $M$ we have clearly an atlas with transition function of type (\ref{Eulerchgt}) and  there always exists an Euler vector field  $C$ on $\cal M$ according to Proposition \ref{Euler}. Moreover, any open submanifold of such a bundle, which is fibered on $M$, has the same property.
\end{Ex}

\begin{Ex}\label{C2}${}$ \\
The Hopf manifold $\mathbb{H}^{h}$of dimension $h$ is the quotient of $\R^{h}\setminus\{0\}$ by the action of the group of contractions $z\ap \l z$ on $\R^h$ with $0<\l<1$. It is well known that  $\mathbb{H}^{h}$ is diffeomorphic to $\mathbb{S}^{h-1}\times \mathbb{S}^1$   and    $\mathbb{H}^{h}$   is a radial manifold: {\it  i.e.}  there exists an Euler field  on the whole manifold (see \cite{FWH} for instance). On the other hand, consider  the  Hopf fibration $\mathbb{S}^{2n+1} \ap \mathbb{ CP}^n \equiv \mathbb{S}^{2n}$ whose fiber is $\mathbb{S}^1$.  Thus we obtain  a fibration $\pi:{\cal M}=\mathbb{S}^{h-1}\times \mathbb{S}^{2n+1}\ap\mathbb{S}^{2n}$ equipped with an Euler vector field. More precisely, let  $T$ be a unit vector field tangent to the fibers of   the  Hopf fibration $\mathbb{S}^{2n+1} \ap  \mathbb{S}^{2n}$ and $ M_i$  be the $i^{th}$ meridian vector field on $\mathbb{S}^{h-1}\subset\R^h$. Then  if $(y_1,\cdots, y_h)$ are the canonical coordinates in $\R^h$, we get a family  $\{ {\cal V}_i=M_i+y^iT\}_{i=1,\cdots h}$ of  (global) vector fields on $\mathbb{S}^{h-1}\times \mathbb{S}^{2n+1}$ which are tangent to the fibers (see \cite{Pa}). ${\cal C}=y^i{\cal V}_i$ is a (global) Euler vector field and  since $[{\cal V}_i{\cal V}_j]=y^i{\cal V}_j-y^j{\cal V}_i$ (see \cite{Pa}) we have
\begin{eqnarray}\label{compChpf}
[{\cal C},{\cal V}_i]=-{\cal V}_i
\end{eqnarray}
 Note that, on such a manifold, the affine structure on each fiber of  $\mathbb{S}^{2n+1} \ap  \mathbb{S}^{2n}$ is not complete  (see \cite{ThSc} or \cite{Va2}).
\end{Ex}

\begin{Ex}\label{C3}${}$\\
Let  $\mathbb{T}^{n+h}$ be the torus of dimension $n+h$ obtained as quotient of $\R^{n+h}$ by the canonical action of $\Z^{n+k}$ on $\R^{n+h}$ by translation. Then the canonical fibration $\pi: \mathbb{T}^{n+h}\ap \mathbb{T}^{n}$ is an affine fibration. Unfortunately, there exists no Euler vector field on this fibration.  Indeed, the  cohomology obstruction $E(\pi)\not=0$ (see \cite{Va3}).\\
\end{Ex}

%%%%%%%%%%%%%%%%%%%%%%%%%%%%%%%%%%%%%%%%%%%%%%%%%%%%%%%%%%%%%%%%%%%%%%%%%%%%%%%%%%%
\subsection{Prolongation of  an anchored bundle over a fibered  manifold  }\label{prolongM}${}$\\
%%%%%%%%%%%%%%%%%%%%%%%%%%%%%%%%%%%%%%%%%%%%%%%%%%%%%%%%%%%%%%%%%%%%%%%%%%%%%%%%%%%%%%%
 %\subsubsection{The prolongation  space}
 %%%%%%%%%%%%%%%%%%%%%%%%%%%%%%%%%%%%%%%%%%%%%%
 %%%%%%%%%%%%%%%%%%%%%%%%%%%%%%%%%%%%%%%%%%%%%%%%%%%%%%
%\subsection{ Foliated anchored bundle and quasi-algebroid structure}
%%%%%%%%%%%%%%%%%%%%%%%%%%%%%%%%%%%%%%%%%%%%%%%%%%%%

${}\;\;\;\;$ Consider a fibered  connected manifold $\pi:{\cal M}\ap M$   over $M$ of dimension $n+h$ and  $({\cal A},M,\rho)$  an anchored bundle of rank $k$. {\it Throughout this work we always assume that $0<h\leq k$}.\\

  For  $m\in \pi^{-1}(x)$ we set
$${\T}^{\cal A}_m{\cal M}=\{(b,v)\in{\cal A}_x\times T_m{\cal M} \textrm{ such that  } \rho(b)=T\pi(v)\}.$$
An element of ${\T}^{\cal A}_m{\cal M}$ will be denoted $(m,b,v)$.\\
The {\bf  ${\cal A}$-prolongation} over  ${\cal M}$ is the bundle
$\hat{\pi}:{\T}^{\cal A}{\cal M}=\dis\bigcup_{m\in{\cal M}}{\T}^{\cal A}_m{\cal M}\ap {\cal M}$ defined by \\$\hat{\pi}(m,b,v)=m$.\\
We  consider the following  maps:

\noindent $\bullet$ $\hat{ \rho}: {\T}^{\cal A}{\cal M}\ap T{\cal M}$  defined by $\hat{\rho}(m,b,v)=(m,v);$

\noindent$\bullet$  $ \pi_{\cal A}: {\T}^{\cal A}{\cal M}\ap {\cal A}$ defined by  $\pi_{\cal A}(m,b,v)=b;$.

\noindent if $\tilde{\tau}: \tilde{\cal A}\ap {\cal M}$ is the pull-back of the bundle $\t:{\cal A}\ap M$  over $\pi:{\cal M}\ap M$, then $\tilde{\pi}:\tilde{\cal A}\ap {\cal A} $ is defined by $\tilde{\pi}(m,b)=b;$

\noindent $\bullet$    $ {\pi}_{\tilde  {\cal A}}: {\T}^{\cal A}{\cal M}\ap \tilde{\cal A}$   defined by  ${\pi}_{\tilde{\cal A}}(m,b,v)=(m,b);$

\noindent$\bullet$  if  ${\V}^{\cal A}{\cal M}=\ker \pi_{\cal A}$,  then  $\hat{\pi}_{| {\bf V}^{\cal A}{\cal M}}:{\bf V}^{\cal A}{\cal M}\ap {\cal M}$ the vertical bundle associated with $\hat{\pi}:{\T}^{\cal A}{\cal M}\ap{\cal M}$.\\

We have then the following commutative diagrams:

$ \begin{array}{ccccc}
 &\hat{ \rho}&\\
{\T}^{\cal A}{\cal M} & \longrightarrow  &T{\cal M}  \\
\pi_{\cal A} \;  \Big\downarrow &  & \Big\downarrow \;T\pi\\
\;\;\;{\cal A} &  \longrightarrow &   TM & \\
& \rho&
\end{array}\;\;\;\;\;\;
 \begin{array}{ccccc}
 &\pi_{{\cal A}}&\\
{\T}^{\cal A}{\cal M}& \longrightarrow  &{\cal A} \\
\hat{\pi} \;  \Big\downarrow &  & \Big\downarrow \;\tau \\
{\cal M} &  \longrightarrow &   M  \\
& \pi&
\end{array}\;\;\;\;\;\;
  \begin{array}{ccccc}
 &\tilde{\pi}&\\
\tilde{\cal A}& \longrightarrow  &{\cal A} \\
\tilde{\tau} \;  \Big\downarrow &  & \Big\downarrow \;\tau \\
{\cal M} &  \longrightarrow & { M}   \\
& \pi &
\end{array}\;\;\;\;\;\;
 \begin{array}{ccccc}
 &{\pi}_{\tilde{\cal A}}&\\
{\T}^{\cal A}{\cal M}& \longrightarrow  &\tilde{\cal A} \\
\hat{\pi} \;  \Big\downarrow &  & \Big\downarrow \;\tilde{\tau} \\
{\cal M} &  \longrightarrow & {\cal M}   \\
& Id &
\end{array}
$

\noindent Therefore  the maps $\hat{\rho}$,  $\pi_{\cal A}$ and $\hat{\pi}_{\cal A}$  are vector bundle morphisms   over  $\rho$ and $\pi$, respectively, and  ${\pi}_{\tilde{A}}$ is a vector bundle  morphism over ${\cal M}$.  Note that we have the following properties:

$(\T^{\cal A}{\cal M},{\cal M},\hat{\rho})$  is an anchored bundle;

 $\pi_{\cal A}$, ${\pi}_{\tilde{\cal A}}$ and  $\tilde{\pi}$ are surjective and $\pi_{\cal A}=\tilde{\pi}\circ {\pi}_{\tilde{\cal A}}$;

$\tilde{\pi}$ is an  isomorphism in restriction to each fiber;

  $\hat{\rho}_{| {\bf V}{\cal M}} :{\bf V}{\cal M}\ap  VT{\cal M}$  is an isomorphism and ${\bf V}{\cal M}$ is also the kernel of ${\pi}_{\tilde{\cal A}}$ . \\

  {\bf When  the anchored bundle $({\cal A},M,\rho)$ is fixed, we  simply denote $\T{\cal M}$ and $\V{\cal M}$ the bundles $\T^{\cal A}{\cal M}$ and $\V^{\cal A}{\cal M}$}

\bigskip

Fix a system of local coordinates $(x^i,y^A)$ on ${\cal M}$  compatible with $\pi$ (that is  $T\pi(\dis\frac{\p}{\p y^A})=0$) and a  local basis $(e_\a)$ on $\cal A$.
We set  ${\cal X}_\alpha(m)=(m,e_\alpha(\pi(m)),\rho_\alpha^i\dis \frac{\partial}{\partial x^i}(m))$ and
${\cal V}_A(m)=(m,0,\dis\frac{\partial}{\partial y^A}(m)))$. Then  $\{{\cal X}_\alpha, {\cal V}_A\}$ is a local basis for ${\T}{\cal M}$  and moreover  ${\cal V}_A$ is  vertical.\\

If   $z=(m,b,v)\in{\T}{\cal M}$ ,  with
 $b=u^\alpha e_\alpha(\tau(m))$ and
  $v=\rho_\alpha^i u^\alpha \dis\frac{\partial}{\partial x^i}(m)+ v^A\dis\frac{\partial}{\partial y^A}(m)$
 then $z$ can also be written $z=u^\alpha{\cal X}_\alpha(m)+v^A {\cal V}_A(m)$.
It follows that   $(x^i,y^A,u^\alpha,v^A)$ are coordinates on  ${\T}{\cal M}$
and  we have:\\  $\hat{\rho}({\cal X}_\alpha)=\rho_\alpha^i\dis\frac{\partial}{\partial x^i}$,  $\pi_{\cal A}({\cal X}_\a)=e_\a$
  and  $\hat{\rho}({\cal V}_A)=\dis\frac{\partial}{\partial y^A}.$\\
{\bf Such a basis $\{{\cal X}_\a,{\cal V}_A\}$ will be called an associated basis to $(x^i,y^A)$ and $\{e_\a\}$} or simply a {\bf canonical basis}.

 The {\it dual basis} of  $\{{\cal X}_\alpha,{\cal V}_A\}$ is denoted by $\{{\cal X}^\alpha,{\cal V}^A\}$  then  we have:

${\cal X}^\alpha=(\pi_{\cal A})^*\epsilon^\alpha$,  ${\cal V}^A=(0,dy^A)=(\hat{\pi})^*dy^A$, and
$(\hat{\rho})^*dx^i=\rho_\alpha^i{\cal X}^\alpha \;\; \textrm{ and }\;\;  (\hat{\rho})^*dy^A={\cal V}^A$.

\noindent  A local section $\omega$ de $({\T}{\cal M})^*$ can then  be written
$\omega=\eta_\alpha{\cal X}^\alpha+\nu_A{\cal V}^A$. This implies that
 $(x^i,y^A,\eta_\alpha,\nu_A)$ are coordinates on $({\T} {\cal M})^*$.\\

\bigskip
 When ${\cal M}={\cal A}$ (resp. ${\cal M}={\cal A}^*$), we simply denote ${\T}{\cal A}$ (resp. $ {\T}{\cal A}^*$) the corresponding prolongation of ${\cal A}$  over ${\cal A}$ (resp. over ${\cal A}^*$). Of course if ${\cal A}=TM$ we get ${\T}{\cal A}=TTM$ and  ${\T}{\cal A}^*=TT^{*}M$. On $\T{\cal A}$,
 recall that  the vertical lift $s^V$  of a section  $s$  of ${\cal A}$ is the section $s^V$ of  ${\T}{\cal A}$ defined by $s^V(a)=\dis\frac{d}{dt}(a+ts(a))_{| t=0}$. This allows us to define the {\it  Liouville section} $C$  by $C(a,b,v)=(a, 0, \hat{b}^V(a))$ and also the vertical endomorphism $J$ defined by $J(a,b,v)=(a,0,\hat{b}^V(a))$ where $\hat{b}$ is the constant section $\hat{b}(x)=b$. We have im $J=\ker J={\V}{\cal A}$ \textit{i.e.}  the vertical bundle of  ${\T}{\cal A}$, in particular we have  $J^2=0$ .

 On ${\T}{\cal A}^*$  the  Liouville form $\theta$ is characterized by $\theta(X)=\zeta(T\widehat{\tau^*}(X))$ for any $\zeta \in {\cal A}^*$ and  $X\in {\T}{\cal A}^*$ where $\widehat{\tau^*}:\T{\cal A}^* \ap {\cal A}^*$ is the projection of the $\cal A$ prolongation over ${\cal A}^*$.
Given a  coordinate system  $(x^i,\eta_\a)$ on ${\cal A}^*$  ({\it cf.} subsection \ref{Adual}), we will denote by ${\cal P}_\a$ the vertical section such that, if $\hat{\rho}:{\T}{\cal A}^*\ap T{\cal A}^*$ is the anchor, $\hat{\rho}({\cal P}_\a)=\dis\frac{\p}{\p \eta_\a}$ and  again  $\hat{\rho}({\cal X}_\a)=\rho_\a^i\dis\frac{\p}{\p x^i}$. It follows that $\{{\cal X}_\a,{\cal P}_\b\}$ is  a local canonical basis of $\T{\cal A}^*$,  and we obtain  an associated coordinate system denoted $(x^i,\eta_\a, u^\a,\nu_\a)$. The associated dual basis is then denoted $\{{\cal X}^\a,{\cal P}_\b\}$ and the associated coordinate system will be denoted $(x^i,\eta_\a, \xi_\b,v^\a)$. With these notations, the Liouville form is written $\theta=\eta_\a{\cal X}^\a$ .\\

Consider a section $\cal X$ of $\T{\cal M}$ defined on an open set $V$. A  curve $c:I\subset\R\ap V$ is called {\it an integral curve of ${\cal X}$} if we have:
$$\dot{c}(t)=\hat{\rho}({\cal X})(c(t)).$$

A {\it  morphism of anchored bundle} $({\cal A},M,\rho )$ and  $({\cal A}',M',\rho' )$ is a morphism bundle $ \Phi:{\cal A}\ap {\cal A}'$ between $\tau:{\cal A}\ap M$ and $\tau':{\cal A}'\ap M'$,  over a map $\phi:M\ap M'$ such that
$$\T\phi\circ \rho=\rho'\circ \Phi.$$

 Given  $\Psi:{\cal M}\ap {\cal M}'$  a bundle morphism  $\Psi:{\cal M}\ap {\cal M}'$ between fibered manifolds $\pi:{\cal M}\ap M$  and $\pi':{\cal M}'\ap M'$  over $\phi$ and we get a map ${\T}\Psi$ from the prolongation ${\T}^{\cal A}{\cal M}$ to ${\T}^{{\cal A}'}{\cal M}'$ characterized by (see \cite{CMM})
$${\T}\Psi(m,b,v)=(\Psi(m),\Phi(b),T_m\Psi(v).$$
%Note that, if $\xi$ is the canonical Liouville form on the cotangent bundle $T^*M$,  then according to $\rho^*:T^*M\ap {\cal A}^*$,  we have

Let  $\cal E$ be a subbundle of $\T{\cal M}$ which contains $\V{\cal M}$. A $k$-form $\o$  on ${\cal E}$ (\textit{i.e.} a smooth section of the bundle $\L^k{\cal E}^*$) is called {\it semi-basic} if $i_{\cal X}\o=0$ for any vertical section $\cal X$. In the same way, a  vector  valued $k$-form $\o$ on $\cal E$ ({\it i.e.}  a smooth section of the bundle $\L^k{\cal E}^*\otimes {\cal E}$) is called {\it semi-basic} if  $\o$ takes its values in $\V{\cal E}$ and if  $i_{\cal X}\o=0$ when $\cal X$ is vertical. More generally, any tensor $ \Theta$ of type  $(k,l)$ on $\cal E$ (\textit{i.e.} a smooth section of the bundle $\bigotimes^k{\cal E}^*\bigotimes^l {\cal E}$) is called semi-basic  if  $\Theta$ is a section of  $\otimes^h{\cal E}^*\otimes^l (\V{\cal M})$ and $\Theta({\cal X}_1,\cdots,{\cal X}_h)$ belongs to $\bigotimes^l (\V{\cal M})$ for any section ${\cal X}_1,\cdots,{\cal X}_k$ of $\cal E$ and this value is $0$ as soon as one of these section is vertical.

%%%%%%%%%%%%%%%%%%%%%%%%%%%%%%%%%%%%%%%%%%%%%%%%%%%%%%%%%%%%%%%%%%%%%%%%%%%%%%%%
\subsection{Prolongation of the Lie bracket}\label{prolbrac} ${}$\\
%%%%%%%%%%%%%%%%%%%%%%%%%%%%%%%%%%%%%%%%%%%%%%%%%%%%%%%%%%%%%%%%%%%%%%%%%%%%%%%%
%%%%%%%%%%%%%%%%%%%%%%%%%%%%%%%%%%%%%%%%%%%%%%%%%%%%%%%%%%%%%%%%%%%%%%%%%%%%%%%%
{\it In this subsection we assume that the rank $k$ of the anchored bundle $({\cal A},M,\rho)$ is smaller than $n$.}\\

Among all sections of   $\hat{\pi}:{\T}{\cal M}\ap {\cal M}$ we will consider particular sections:

  \begin{Def}\label{Mproj}${}$\\
A section ${\cal Z}$ of  $\hat{\pi}:{\T}{\cal M}\ap {\cal M}$ is {\it projectable}  if  there  exists $s\in\Xi({\cal A})$  such that
$${\pi}_{\cal A}\circ{\cal Z}=s\circ \tau.$$
  \end{Def}

 Therefore  ${\cal Z} $  is projectable if and only if there exists a vector field  $Z$ on ${\cal M}$ and $s\in\Xi({\cal A})$ such that
$${\cal Z}(p)=(m,s\circ\pi(m),Z(m)) \textrm{ with   } T\pi(Z)(p)=\rho\circ s\circ \pi(p).$$
  In particular, in  any canonical  local basis $\{{\cal X}_\a,{\cal V}_A\}$,  each  ${\cal X}_\a$ and ${\cal V}_A$ are (local)  projectable sections.
  Therefore, locally any  section  can be decomposed  as a linear functional combination of  projectable sections and  vertical sections.

\noindent Moreover,  if ${\cal Z}$ is  projectable and ${\cal V}$ is a vertical  section then,  $p\ap (p,0,[\hat{ \rho}({\cal Z(p)}),\hat{ \rho}({\cal V(p)}])$ is a vertical section of ${\T}{\cal M}\ap{\cal M}$.
It follows that we can define the Lie bracket $[{\cal Z},{\cal V}]$   as the section $(0,[\hat{\rho}({\cal Z}),\hat{\rho}({\cal V})])$. \\

{\it Assume now that  $({\cal A},M,\rho,[\;,\;]_{\cal A})$ is a pre-Lie algebroid}. On the anchored bundle  $({\T}{\cal M}, {\cal M}, \hat{\rho})$  there exists a natural Lie bracket $[\;,\;]_{\cal P}$\footnote{  this almost Lie  bracket  is denoted with index   $\cal P$   to recall that is it a prolongation of the almost Lie bracket on $\cal A$}  which is a natural prolongation of the given Lie bracket $[\;,\;]_{\cal A}$ in the sense that for projectable sections ${\cal Z}_1=(s_1\circ \pi,Z_1)$   and ${\cal Z}_2=(s_2\circ \pi,Z_2)$  we have
$$[{\cal Z}_1,{\cal Z}_2]_{\cal P}= ([s_1,s_2]_{\cal A}\circ\pi, [Z_1,Z_2]).$$
Since any section of  ${\T}{\cal M}$ can be locally written as a linear combination of projectable
sections, the definition of the Lie bracket $[\;,\;]_{\cal P}$ for arbitrary sections of ${\T}{\cal M}$ follows. Locally,  this bracket is characterized by:
\begin{eqnarray}\label{locbracalg}
[{\cal X}_\a,{\cal X}_\b]_{\cal P}=C_{\a\b}^\g{\cal X}_\g,\;\;\;\;\; [{\cal X}_\a,{\cal V}_B]_{\cal M}=0, \textrm{ and }   [{\cal V}_A,{\cal V}_B]_{\cal P}=0.
\end{eqnarray}

Moreover we obtain an   algebroid structure  $(\T{\cal M},{\cal M},\hat{\rho},[\;,\;]_{\cal P})$. To this structure  we can associate an almost differential denoted $d^{\cal P}$. If $\{{\cal X}_\a,{\cal V}_\a\}$ is a local canonical basis and $\{{\cal X}^\a,{\cal V}A\}$ is the associated dual basis, according to (\ref{locbracalg}),  we have:
\begin{eqnarray}\label{locdifalg}
 \;\;\;\;\;dx^i=\rho_\a^i{\cal X}^\a,\;\;\;\;\; d^{\cal P}{\cal X}^\g=-\dis\frac{1}{2}C_{\a\b}^\g{\cal X}^\a\wedge{\cal X}^\b ,\;\;\;\;\; dy^A={\cal V}^A,\;\;\;\;\; d^{\cal P}{\cal V}^A=0.
\end{eqnarray}

 \begin{Rem}\label{dPsbasic}${}$
 \begin{enumerate}
\item A $1$-form $\eta$ on $\T{\cal M}$ is semi-basic if and only if  in any canonical dual basis $\{{\cal X}^\a,{\cal V}^B\}$ of $\T{\cal M}$,  we have
$$\eta=\eta_\a{\cal X}^a.$$
In this case, according to (\ref{locdifalg}) $d^{\cal P}\eta$ has a decomposition of type:
$$d^{\cal P}\eta=\dis\frac{\p \eta_\a}{\p y^A}{\cal V}^A\wedge {\cal X}^\a+\dis\frac{1}{2}(\dis\frac{\p \eta_\b}{\p x^i}\rho_\g^i-\dis\frac{\p \eta_\g}{\p x^i}\rho_\b^i-\eta_\a C_{\b\g}^\a){\cal X}^\b\wedge{\cal X}^\g. $$
Thus   $d^{\cal P}\eta$ only depends on the bracket $[\;,\;]_{\cal A}$. Therefore,  if $\eta$ is semi-basic,  the condition "$d^{\cal P}\eta$ is semi-basic" depends only on the bracket $[\;,\;]_{\cal A}$. \\

\item  on $\T{\cal A}^*$  the Liouville form $\theta$ is semi basic, so  the canonical $2$ form $\O=-d^{\cal P}\theta$ depends on the choice of the bracket $[\;,\;]_{\cal A}$. In fact,  $\O$ has locally the following decomposition:
  \begin{eqnarray}\label{Oloc}
\O={\cal X}^\a\wedge{\cal P}_\a+\dis\frac{1}{2}\eta_\g C_{\a\b}^\a{\cal X}^\a\wedge{\cal X}^\b.
\end{eqnarray}
\noindent In particular, $\O$ is symplectic and the {\it  vertical bundle $\V{\cal A}^*$ is Lagrangian}.\\

\item Consider  a Hamiltonian vector field ${\cal X}_h$   on ${\cal A}^*$ associated with a smooth map $h: {\cal A}^*\ap \R$ (relative to the linear Poisson bracket associated with  $[\;,\;]_{\cal A}$). Moreover, consider the equation  $i_{\cal X}\O=d^{\cal P}h$ has a unique solution denoted $\overrightarrow{h}$ and according to (\ref {Oloc}), we have
$$\overrightarrow{h}=\dis\frac{\p h}{\p \eta_\a}{\cal X}_\a-(\rho_\a^i\frac{\p h}{\p x^i}+\eta_\g C_{\a\b}^\g\frac{\p h}{\p \eta_\b}){\cal P}^\a.$$
According to the local decomposition (\ref {locVH}), we have
$$\hat{\rho}(\overrightarrow{h})={\cal X}_h.$$
\end{enumerate}
\end{Rem}

%%%%%%%%%%%%%%%%%%%%%%%%%%%%%%%%%%%%%%%%%%%%%%%%%%%%%%%%%%%%%%%%%%%%%%%%%%%%%%%%%%%%%%%%%%
\section{Foliated anchored bundle and pre-Lie algebroid }\label{prolfoliabundle}
 %%%%%%%%%%%%%%%%%%%%%%%%%%%%%%%%%%%%%%%%%%%%%%%%%%%%%%%%%%%%%%%%%%%%%%%%%%%%%%%%%%%%%%%%
 %%%%%%%%%%%%%%%%%%%%%%%%%%%%%%%%%%%%%%%%%%%%%%%%%%%%%%%%%%%%%%%%%%%%%%%%%%%%%%%%%%%%%%%%
\subsection{Integrability and involutivity of distributions}${}$\\
 %%%%%%%%%%%%%%%%%%%%%%%%%%%%%%%%%%%%%%%%%%%%%%%%%%%%%%%%%%%%%%%%%%%%%%%%%%%%%%%%%%%%%%%%%%%
 In this subsection we recall the context used in  \cite{St},  \cite{Su} and \cite{Ba}.\\

A distribution ${\cal D}$ on a manifold $M$ is an assignment  $x\mapsto {\cal D}_x$ where ${\cal D}_x$ is a vector subspace of $T_xM$. A  local vector field $X$ on $M$ is a smooth  vector field  defined on an open set $O$  denoted  Dom$ (X)$. A local vector field $X$ is tangent to a distribution $\cal D$ if for any $x\in$Dom$(X)$ then $X(x)$ belongs to ${\cal D}_x$. We denote by $\Xi({\cal D})$ the sets of all  vector fields which are tangent to $\cal D$.  Note that $\Xi({\cal D})$ has a structure of module over the ring ${\cal C}^\infty(M)$.
A distribution is called {\it smooth} if for any $x\in M$, ${\cal D}_x$ is generated by the set $\{X(x), X\in \Xi({\cal D})\}$. \\
Any module  $\mathbf{A}$ of  smooth vector  fields  on $M$ generates a smooth distribution $\cal D_{\mathbf{A}}$ defined by  ${\cal D}_x: = $span$\{X(x),:\; X\in \mathbf{A}\}$. Of course such a module is contained in $\Xi(\cal D_{\mathbf{A}})$; in particular, such a distribution is smooth. But  we can have $\mathbf{A}\not=\Xi(\cal D_{\mathbf{A}})$ as the following example shows:

\begin{Ex}\label{balan}\cite{Ba}${}$\\
On $M=\R^2$ we consider the vector fields
$$X_1=(x^2+y^2)\dis\frac{\partial}{\partial x}\;\; X_2=(x^4+y^4)\dis\frac{\partial}{\partial y}$$
We denote by $\mathbf{A}$ the module generated by $X_1$ and $X_2$. The distribution $\cal D$ generated by $\mathbf{A}$  is such that ${\cal D}_{(x,y)}=T_{(x,y)}\R^2$ and  ${\cal D}_{(0,0)}=\{(0,0)\}$. Now we have
$$[X_1,X_2]=4x^2(x^2+y^2)\dis\frac{\partial}{\partial y}-2y(x^4+y^4)\dis\frac{\partial}{\partial x}$$
Therefore $[X_1,X_2]$ is tangent to $\cal D$. But it is easy to see that $[X_1,X_2]$ does not belong to $\mathbf{A}$
\end{Ex}

Given  an anchored bundle $({\cal A},M,\rho)$  on $M$, we denote by $\underline{\rho}:\Xi({\cal A})\ap \Xi(M)$ the morphism of modules induced by $\rho$. Then  $\underline{\rho}(\Xi({\cal A}))$ is a module of vector fields which generates the distribution $\rho({\cal A})$ but in general we have $\underline{\rho}(\Xi({\cal A}))\not=\Xi(\rho({\cal A})$ as the Example \ref{balan} shows.  A relation between modules of vector fields and anchored bundles is given by the following result of  \cite{Os} (pp. 122-123):

\begin{Pro}\label{Os}\cite{Os}${}$\\
For any smooth manifold M and any integer $k\geq O$ there is a one-to-one correspondence between smooth anchored  vector bundles $({\cal A},M,\rho)$ of rank $k$ such that $\underline{\rho}:\Xi({\cal A})\ap \Xi(M)$ is injective and isomorphism classes of  locally free modules of finite rank $k$ on $M$.\\
\end{Pro}

A distribution is called  {\it  punctually integrable  at $x$}  in $ M$ if there exists  a submanifold $P$ of $M$ such that $x$ belongs to $N$ and $T_yP={\cal D}_y$ for all $y\in P$. In this case, $P$ is called an integral manifold of $\cal D$. The distribution is called {\it integrable} if it is punctually integrable at each point of $M$ and  we say that $\cal D$ is  {\it  Stefan-Sussman integrable}. In this case we have a {\it global integrability} property in the following sense:

consider the binary relation
$$x{\cal R}y \textrm { iff there exists an integral manifold } P \textrm{  of  } {\cal D} \textrm{ such that } x ,y \in  P);$$
then $\cal R$ is an equivalence relation and the equivalence class $N(x)$ of $x$ has a natural structure of connected  manifold whose dimension is dim$({\cal D}_x)$;  if $\iota: N(x)\ap M$ is the natural inclusion then $(N,\iota)$ is an {\it immersed submanifold} of $M$.\\
Moreover  $N(x)$, is a maximal integral manifold of ${\cal D}$  in the following way: for any integral manifold $P$ of $\cal D$, such that $P\cap N(x)$ is not empty  then $P\subset N(x)$. \\
The set of all the equivalence classes is called the {\it foliation } defined by $\cal D$ and any equivalence class $N$ is called a leaf of the foliation. \\

A smooth distribution $\cal D$ is called {\it involutive} if for any two vector fields $X$ and $Y$ which are tangent to $\cal D$ the usual Lie bracket $[X,Y]$ is also tangent to $\cal D$. When $\cal D$ is  regular {\it i.e.}  $\cal D$ is a subbundle of $TM$ the Frobenius Theorem asserts that  $\cal D$ is integrable. However, in the general case of a smooth distribution it is well known that this result is no more true (see \cite{St} and \cite{Su} for counterexamples).  Nevertheless if $\cal D$ is integrable then $\cal D$ is involutive ({\it cf.} \cite{Ba} Proposition 2.3). \\

 A module  $\mathbf{ A}$ of  smooth vector  fields  on $M$ is called {\it involutive} if for any two vector fields $X$ and $Y$of $\mathbf{A}$ the usual Lie bracket $[X,Y]$  belongs to $\mathbf{A}$. Note that when $\cal D$  is generated by a module $\mathbf{A}$ and if $\cal D$ is involutive it does not implies that $\mathbf{A}$ is involutive ({\it cf.} Example \ref{balan}). However we have the following result of \cite{Ba} Proposition 2.6 

\begin{Pro}\label{balan2.6}\cite{Ba}${}$\\
If a smooth distribution $\cal D$ is integrable then  the module $\Xi({\cal D})$ is involutive.
\end{Pro}

For an anchored bundle we have:

\begin{Pro}\label{involanchor}${}$\\
Let  $({\cal A},M,\rho)$ be an anchored bundle on $M$. 
\begin{enumerate}
\item If $\rho({\cal A})$ is integrable and  $\underline{\rho}(\Xi({\cal A}))=\Xi(\rho({\cal A})$  then $\underline{\rho}(\Xi({\cal A}))$ is involutive. 
\item if $\underline{\rho}(\Xi({\cal A}))$ is involutive then  $\rho({\cal A})$ is integrable
\end{enumerate}
\end{Pro}

\begin{proof}${}$\\
Point (1) is a corollary of Proposition \ref{balan2.6}. Point (2) is a consequence of Theorem 4.2 of \cite{Su}.

\end{proof}

 %%%%%%%%%%%%%%%%%%%%%%%%%%%%%%%%%%%%%%%%%%%%%%%%%%%%%%%%%%%%%%%%%%%%%%%%%%%%%%%%%%%%%%%%
 \subsection{Foliated anchored bundle }${}$\\
 %%%%%%%%%%%%%%%%%%%%%%%%%%%%%%%%%%%%%%%%%%%%%%%%%%%%%%%%%%%%%%%%%%%%%%%%%%%%%%%%%%%%%%%%%%%
Given   an anchored $({\cal A},M,\rho)$  we consider:
  
  \begin{Def}${}$\\
  The anchored bundle $({\cal A},M,\rho)$ is called foliated if the module $\underline{\rho}(\Xi({\cal A}))$ is involutive.
  \end{Def}
  Note that if $({\cal A},M,\rho)$  is a foliated anchored bundle, according to Proposition \ref{involanchor} the distribution $\rho({\cal A})$ is integrable. However the converse is not true according to Example \ref{balan}: \\if we consider the anchored bundle $(T\R^2,\R^2,\rho)$ on $\R^2$ with $\rho(\dis\frac{\p}{\p x})=X_1$ and $\rho(\dis\frac{\p}{\p y})=X_2$, then the distribution $\rho(T\R^2)$ is integrable (since it is punctually integrable), but $[\rho(\dis\frac{\p}{\p x}),\rho(\dis\frac{\p}{\p y})]$ does not belongs to the module $\underline{\rho}(\Xi(T\R^2))$.\\
  
  Now, as we have seen in the previous section if $({\cal A},M,\rho)$  is a Lie algebroid or pre-Lie algebroid then $\rho$ induces a Lie morphism and so $({\cal A},M,\rho)$ is a foliated anchored bundle. We have the following general examples:
  
  \begin{Pro}\label{conform}${}$
  \begin{enumerate}
  \item If $({\cal A},M,\rho)$ is a foliated anchored bundle, so is the anchored bundle $({\cal A},M,\Phi\rho)$ for any smooth function $\Phi$ on $M$.
 \item If $({\cal A},M,\rho)$  is a foliated anchored bundle
  then given any vector bundle ${\cal A}'$ on $M$ we have a structure of foliated anchored bundle on ${\cal A}\oplus {\cal A}'$.
 \item if $\mathbf{A}$ is a locally free involutive module of smooth vector fields of rank $k$, there exists an anchored bundle $({\cal A},M,\rho)$ such that $\underline{\rho}(\Xi({\cal A}))=\mathbf{A}$ and which is a foliated anchored bundle.
  \end{enumerate}
  \end{Pro}
  \begin{proof}${}$\\
  For point (1), first of all note that $\Phi\rho:{\cal A}\ap TM$ is a bundle morphism. Now if $s$ and $s'$ are two sections of $\cal A$ we have
  \begin{eqnarray}\label{crochet}
  [\Phi\rho(s),\Phi\rho(s')]=\Phi (\Phi[\rho(s),\rho(s')]+d\Phi(\rho(s))\rho(s')-d\Phi(\rho(s')\rho(s))
\end{eqnarray}
 
 As $({\cal A},M,\rho)$ is a foliated anchored bundle, it follows that $\Phi[\rho(s),\rho(s')]$ belongs to the module $\underline{\rho}(\Xi({\cal A}))$. Therefore the second members of (\ref{crochet}) belongs to $\Phi\underline{\rho}(\Xi({\cal A}))$. This ends the proof.\\
 
 \noindent For point (2) take $\rho':{\cal A}\oplus {\cal A}'\ap TM$ defined by $\rho'_{| {\cal A}}=\rho$ and $\rho'_{| {\cal A}'}=0$ and the bracket $[\;,\;]'$ characterized by: 

 $[X,Y]'=[X,Y]$  for $X$ and $Y$ sections of $\cal A$;
 
 $[X,X']'=0 $ for sections $X$ of $\cal A$ and $X'$ of $\cal K$;
 
 $[X',Y']'=[X',Y']_{\cal K}$ where  $[\;,\;]_{\cal K}$ is any almost Lie bracket on $\cal K$.\\
 
For  Point (3) according to  Proposition \ref{balan2.6}, there exists an anchored bundle $({\cal A},M,\rho)$ such that  $\underline{\rho}(\Xi({\cal A}))=\mathbf{A}$. Now, since $\mathbf{A}$ is locally free we must have $\mathbf{A}=\Xi(\rho({\cal A}))$. Finally as $\mathbf{A}$ is involutive   $({\cal A},M,\rho)$ is a foliated anchored bundle.\\
  \end{proof}
  
  The link between pre-Lie algebroid and foliated anchored bundle is given in the first part of the following result:\\

  \begin{Pro}\label{foliatedquasi}${}$
  \begin{enumerate}
 \item[(i)] The anchored bundle $({\cal A},M,\rho)$ is a foliated anchored bundle if an only if there is an almost Lie bracket $[\;,\;]_{\cal A}$ such that $({\cal A},M,\rho,[\;,\;]_{\cal A})$ is a pre-Lie algebroid.
 \item[(ii)] Let $N$ be a leaf of a foliated anchored bundle  $({\cal A},M,\rho)$. If ${\cal A}_N=\tau^{-1}(N)$ then ${\cal K}_N={\cal A}_N\cap \ker \rho$ is a subbundle of ${\cal A}_N$ and the the restriction of $\rho$  to ${\cal A}_N$ induces  a quotient morphism $\rho_N: {\cal A}_N/{\cal K}_N\ap TM$ which  is an isomorphism onto $TN$. Moreover, given any almost Lie bracket $[\;,\;]_{\cal A}$ such that  $({\cal A},M,\rho,[\;,\;]_{\cal A})$ is a pre-Lie algebroid then $[\;,\;]_{\cal A}$ induces an Lie bracket $[\;,\;]_N$ on the anchored bundle $( {\cal A}_N/{\cal K}_N,N,\rho_N)$ which is independent of such a choice of almost bracket $[\;,\;]_{\cal A}$. In particular $\rho_N$ gives rise to an isomorphism between the Lie algebra $(\Xi({\cal A}_N/{\cal K}_N), [\;,\;]_N)$ and $(\Xi(N),[\;,\;])$.\\
   \end{enumerate}
  \end{Pro}

  \begin{Rem}\label{toutcbracalg}${}$\\
  Let $[\;,\;]_1$ and $[\;,\;]_2$ be two almost brackets such that $({\cal A},M,\rho),[\;,\;]_k)$ is a pre-Lie algebroid for $k=1,2$. If $\Theta=[\;,\;]_1-[\;,\;]_2$ then the ${\cal A}$-tensor $\Theta $ takes values in $\ker \rho$ and conversely.\\
  \end{Rem}

\begin{proof}${}$\\
For Part (i) we only have to prove that if $({\cal A},M,\rho)$ is foliated anchored bundle,  there exists an almost bracket $[\;,\;]_{\cal A}$ on $({\cal A},M,\rho)$ such that $({\cal A},M,\rho,[\;,\;]_{\cal A})$ is a pre-Lie algebroid.
 First of all, since $M$ is paracompact,  there always exists a $\cal A$-linear connection $\nabla$ on an anchored bundle $({\cal A},M,\rho)$.  We denote by $[\;,\;]_\nabla$ the associated almost Lie bracket. (see Section \ref{ALbracket}). Therefore, locally over an open $U$ we have:

 $[e_\a,e_\b]=C_{\a,\b}^\g e_\g,\;\;\;\;\rho(e_\a)=\rho_\a^i\dis\frac{\p}{\p x^i}, \textrm{ and } [\rho(e_\a),\rho(e_\b)]=(\rho_\a^i\dis\frac{\p\rho_\b^j}{\p x^i}-\rho_\b^i\dis\frac{\p\rho_\a^j}{\p x^i})\dis\frac{\p}{\p x^j}$.

 Since the module $\underline{\rho}(\Xi({\cal A})$ is involutive, there exists a family of functions $B_{\a\b}^\g$ on $U$ such that
 \begin{eqnarray}\label{locbrac}
[\rho(e_\a),\rho(e_\b)]=B_{\a\b}^\g \rho_\g^j \dis\frac{\p}{\p x^j}.
\end{eqnarray}
It follows that $\Theta_{\a\b}^\g=B_{\a\b}^\g-C_{\a\b}^\g$ are the components of an antisymmetric $\cal A$-tensor  $\Theta_U$ of type $(2,1)$   over $U$  such that
\begin{eqnarray} \label{locabrac}
[\rho(e_\a),\rho(e_\b)]=\rho\left([e_\a,e_\b]+\Theta_U(e_\a,e_\b)\right).
\end{eqnarray}
Denote by  $[\;,\;]'_U=[\;,\;]+\Theta_U(\;,\;)$   %the almost Lie bracket  associated with the $\cal A$-connection $\nabla_{| U}+\Theta_U$ (see Proposition \ref{setALB}).
By construction $[\;,\;]'_U$ is an almost Lie bracket ({\it cf.} Remarque \ref{almquasi} (1)) which satisfies  the relation
\begin{eqnarray}\label{locmorph}
[\rho(X),\rho(Y)]_{| U}=\rho[X,Y]'_{U}.
\end{eqnarray}
Now, choose an open cover
$\{U_k\}_{k\in K}$  of  $M$ by open sets associated to a partition of unity $\{\phi_k\}_{k\in K}$, such
that $\cal A$  is trivializable on each open $ U_k$. Now, given an open set $U_k$,  there
exists an antisymmetric $\cal A$-tensor  of type $(2,1)$  $\Theta_k$ over $U_k$  which satisfies the Equation (\ref{locmorph}) for some choice of a local basis of $\cal A$ on $U_k$.
 Then $\Theta=\dis\sum_{k\in K}\phi_k \Theta_k$ is an antisymmetric $\cal A$-tensor  of type $(2,1)$ over $M$. It follows that %$\nabla'=\nabla+\Theta$ is a
  %$\cal A$-linear connection  and denote
  $[\;,\;]'=[\;,\;]_\nabla+\Theta$ is an almost  Lie bracket ({\it cf.} Remarque \ref{almquasi} (1)). It remains to show that $\rho$ induces a morphism of Lie algebras between  the Lie algebras $(\Xi({\cal A}), [\;,\;]_{\nabla'}$ and $\Xi(M),[\;,\;])$.  Therefore we have $[\;,\;]'=[\;,\;]_\nabla+\dis\sum_{k\in K}\phi_k \Theta_k=\dis\sum_{k\in K}\phi_k[\;,\;]'_{U_k}$. Thus  we have $[\rho(X),\rho(Y)]=\dis\sum_{k\in K}\phi_k [\rho(X),\rho(Y)]_{| U_k}$. Now, according to the equation (\ref{locmorph}) on each $U_k$ we obtain $[\rho(X),\rho(Y)]_{| U_k}=\rho[X,Y]'_{U_k}$. It follows that we get
  $[\rho(X),\rho(Y)]=\rho[X,Y]_{\nabla'}$. Thus $({\cal A},M,\rho,[\;,\;]')$ is a pre-Lie algebroid.\\

  According to the notations introduced in Part (ii), if $N$ is a leaf of $\rho({\cal A})$, it is  clear that  ${\cal K}_N$ is a subbundle of ${\cal A}_N$  and $\rho_N: {\cal A}_N/{\cal K}_N\ap TM$ is  an isomorphism onto $TN$. There exists a subbundle  ${\cal H}_N$ of ${\cal A}_N$ such that ${\cal A}_N={\cal H}_N\oplus{\cal K}_N$. Therefore the restriction of $\rho$ to ${\cal H}_N$ is also an isomorphism on $TN$.  Now fix some almost bracket $[\;,\;]_{\cal A}$ such that $({\cal A},M,\rho),[\;,\;]_{\cal A})$ is a pre-Lie algebroid. Since $M$ is paracompact, so is $N$ and therefore  an almost Lie bracket on $M$ or $N$  is completely characterized by its value on locally sections defined on some open subsets. Consider a coordinate system $(x^i)$ on an open $U$ such that $(x^1,\cdots,x^q)$ is a coordinate system on $U\cap N$. After restricting $U$ if necessary, there exists a basis $\{e_\a\}$ such that $\{e_1,\cdots,e_q\}$ and $\{e_{q+1},\cdots, e_k\}$ is a local basis of ${\cal H}_N$ and ${\cal K}_N$, respectively. Since  $\rho_N$ is of constant rank $q$ over $U\cap N$, therefore  after restricting $U$ if necessary, we can choose $\{e_1,\cdots,e_q\}$ such that $\rho(e_\a)=\dis\frac{\p}{\p x^\a}$ for $\a=1,\cdots,q$ on $U$. It follows that, if we set
  $\rho(e_\a)=\rho_\a^i\dis\frac{\p}{\p x^i}$, then $\rho_\a^\b=\d_\a^\b$ for $\a,\b=1,\cdots,q$, $\rho_\a^j=0$ on  $x^{q+1}=\cdots=x^n=0$ for all $\a=1,\cdots, q$ and $j=q+1,\cdots,n$. Consequently  all  brackets $[\rho(e_\a), \rho(e_\b)]$ vanish on  $x^{q+1}=\cdots=x^n=0$ for all $\a,\b=1,\cdots q$. Since $\rho$ is a Lie morphism, it follows that we have $\rho([e_\a,e_\b]_{\cal A}=0$ on $U\cap N$. This implies that  $[e_\a,e_\b]_{\cal A}$ belongs to ${\cal K}_N$ over $U\cap N$. \\

    Note that for any section $X$ of ${\cal A}_N$ and $Y$ of ${\cal K}_N$ we have  $[\rho(X),\rho(Y)]=0=\rho([X,Y]_{\cal A})$. Therefore $[X,Y]_{\cal A}$ belongs to ${\cal K}_N$.\\

 Let  $[X,Y]_{{\cal H}_N}$ be the projection of $[X,Y]_{\cal A}$ over ${\cal H}_N$ parallel to ${\cal K}_N$ for any sections $X$ and $Y$ of ${\cal H}_N$. Now given any section $\bar{X}$ and $\bar{Y}$ of  ${\cal A}_N/{\cal K}_N$ there exists  two sections  $X$ and $Y$ of ${\cal H}_N$ whose canonical projections are $\bar{X}$ and $\bar{Y}$ respectively. We have the decompositions $X=\dis\sum_{\a=1}^q f^\a e_\a$ and $Y=\dis\sum_{\a=1}^q g^\b e_\b$. It follows that the $[X,Y]$ has a can be written      $[X,Y]_{\cal A}=\dis\sum_{\a,\b=1}^q f^\a g^\b[e_\a, e_\b]+[X,Y]_{{\cal H}_N}$. Since all brackets $ [e_\a, e_\b]_{\cal A}$ belongs to ${\cal K}_{N}$, The projection of $[X,Y]_{\cal A}$ on ${\cal A}_N/{\cal K}_N$ is equal to the projection of $[X,Y]_{{\cal H}_N}$;

 Now any other sections $X'$ and $Y'$ of ${\cal A}_N$ whose l projection on ${\cal A}_N/{\cal K}_N$ is also $\bar{X}$ and $\bar{Y} $ respectively, we have $X'=X+X"$ and $Y'=Y+Y"$ where $X"$ and $Y"$ belong to ${\cal K}_N$. Since we have
$[X',Y']_{\cal A}=[X',Y"]_{\cal A}+[X",Y']_{\cal A}+[X,Y]_{\cal A}$  and since $[X',Y"]_{\cal A}$ and $[X",Y']_{\cal A}$ belong to ${\cal K}_N$ the projection of $[X',Y']_{\cal A}$ on ${\cal A}_N/{\cal K}_N$ is equal to the projection of $[X,Y]_{\cal A}$ on this quotient which is also the projection on ${\cal A}_N/{\cal K}_N$  of $[X,Y]_{{\cal A}}$.
  In this way by canonical projection of $[\;,\;]_{\cal A}$ we get an almost Lie bracket $[\;,\;]_N$ on ${\cal A}_N/{\cal K}_N$. Note from our construction, if $\bar{e}_\a$ is the canonical projection of $e_\a$ over  ${\cal A}_N/{\cal K}_N$, then $\{\bar{e}_\a\}_{\a=1,\cdots, q}$ is a local basis and we have
 $[\bar{e}_\a ,\bar{e}_\b]=0$ for $\a,\b=1,\cdots, q$. This clearly implies that $[\;,\;]_N$ satisfies the Jacobi identity. \\
 Finally according to Remark \ref{toutcbracalg} any other  almost Lie bracket $[\;,\;]'_{\cal A}$ such that $({\cal A},M,\rho,[\;,\;]'_{\cal A})$ is a pre-Lie algebroid and gives rise to the same bracket $[\;,\;]_N$ by projection on ${\cal A}_N/{\cal K}_N$. This ends the proof of Part (ii).
  \end{proof}

 %%%%%%%%%%%%%%%%%%%%%%%%%%%%%%%%%%%%%%%%%%%%%%%%%%%%%%%%%%%%%%%%%
 \subsection{Prolongation of foliated anchored bundle over a fibered manifold}
 %%%%%%%%%%%%%%%%%%%%%%%%%%%%%%%%%%%%%%%%%%%%%%%%%%%%%%%%%%%%%%%%%%%%%%
 We consider a fibered manifold $\pi:{\cal M}\ap M$ of dimension $n+h$ and  $(\T{\cal M},{\cal M},\hat{\rho},[\;,\;]_{\cal P})$ the associated pre-Lie algebroid as built in subsection \ref{cal E}. Therefore the distribution $\hat{\rho}(\T{\cal M})$ is integrable. \\
Fix some leaf $N$ of the integrable foliation $\rho({\cal A})$ and consider ${\cal M}_N=\pi^{-1}(N)$ Then $ \pi_N=\pi_{| {\cal M}_N}: {\cal M}_N\ap N$ is a fibered manifold of dimension $q+h$ if dim$(N)=q$. We denote by $\T{\cal M}_N$ the restriction of the bundle $\T{\cal M}$ over ${\cal M}_N$ and we set
$${\bf K}{\cal M}_N=\{(m,b,v)\in \T{\cal M}_N,\;|\; \hat{\rho}(m,b,v)=0\}=\{(m,b,0)\in \T{\cal M}_N\}.$$
Note that $\T{\cal M}_N$ is a bundle over ${\cal M}_N$ of rank $h+k$ and ${\bf K}{\cal M}_N$    is a subbundle of corank $q+h$. It follows that $\hat{\rho}$ induces an injective   quotient morphism $\hat{\rho}_N:\T{\cal M}_N/{\bf K}{\cal M}_N\ap T{\cal M}_N$. We can easily obtain:

\begin{Pro}\label{prolonfolit}${}$
\begin{enumerate}
\item[(i)] ${\cal M}_N$ is a leaf of the foliation defined by the integrable distribution $\hat{\rho}(\T{\cal M})$ and $\hat{\rho}_N$ is an isomorphism from $\T{\cal M}_N/\bf{K}{\cal M}_N$ onto $T{\cal M}_N$.
\item[(ii)] The almost  bracket $[\;,\;]_{\cal P}$  induces by projection on $\T{\cal M}_N/\bf{K}{\cal M}_N$ an almost bracket $[\;,\;]_{{\cal M}_N}$  which is independent on the choice of the almost bracket $[\;,\;]_{\cal A}$ such that $({\cal A},M,\rho,[\;,\;]_{\cal A})$ is a pre-Lie algebroid. Moreover $[\;,\;]_{{\cal M}_N}$ satisfies the Jacobi identity and $\hat{\rho}_N$ induces an isomorphism of Lie algebra between $(\Xi(\T{\cal M}_N/\bf{K}{\cal M}_N), [\;,\;]_{{\cal M}_N})$ and $(\Xi({\cal M}_N),[\;,\;])$.
\end{enumerate}
\end{Pro}

\begin{Rem}\label{spbasis}${}$\\
Fix some leaf $N$ a leaf of the  foliation defined by the integrable distribution $\rho({\cal A})$. According to the proof of Proposition \ref{foliatedquasi},  consider a coordinate system $(U, (x^1,\cdots,x^n))$ of $M$  such that

  $\bullet$ $U\cap N,(x^1,\cdots,x^q)$   is a coordinate system of $N$ and

  $\bullet$  there exists a local basis $\{e_1,\cdots,e_k\}$ over $U$ so that $\{e_{q+1},\cdots,e_k\}$ is a local basis  of ${\cal K}_N$ over $U\cap N$ with $\rho(e_\a)=\dis\frac{\p}{\p x^\a}$ and $[{e}_\a,{e}_\b]_{\cal A}=\sum_{\g=q+1}^kC_{\a\b}^\g e_\g$  for $\a=1,\cdots q$.\\

Now  if $(x^i,y^A)$ is a coordinate system defined on some open subset ${\cal U}$ of $\pi^{-1}(U)$, let $\{{\cal X}_\a,{\cal V}_A\}$ be the associated basis of $\T{\cal M}$. Then ${\cal U}\cap {\cal M}_N,(x^1,\cdots,x^q,y^1,\cdots,y^k)$ is a coordinate system on $\T{\cal M}_N$ and over this open set $\{{\cal X}_{q+1},\cdots,{\cal X}_k\}$ is a local basis of $\bf{K}{\cal M}_N$ such that
 $[{\cal X}_\a,{\cal X}_\b]_{\cal P}=\sum_{\g=q+1}^kC_{\a\b}^\g {\cal X}_\g $ for $\a,\b=1,\cdots,q$. Moreover if $\bar {\cal X}_\a$ is the projection of ${\cal X}_\a$ on $\T{\cal M}_N/\bf{K}{\cal M}_N$ we also have $\hat{\rho}_N(\bar{\cal X}_\a)= \dis\frac{\p}{\p x^\a}$ for $\a=1,\cdots q$.\\

\end{Rem}

 We fix some leaf $N$ of the foliation defined by $\rho({\cal A})$ of dimension $q$  and consider the associated fibered manifold $\tau_N:{\cal A}_N\ap N$. %According to subsection \ref{prolfoliabundle}, $J({\bf K} {\cal A}_N)$ is a subbundle of $\V{\cal A}_N$ of rank $k-q$

 .%Let $\T^L{\cal M}_N$ be the subbundle of $\T{\cal M}_N$ which is the orthogonal of $({\K} {\cal M}_N)\oplus J({\K} {\cal M}_N)$. We have:

 \begin{Pro}\label{reducleaf}${}$
 \begin{enumerate}
   \item[(i)] Let $J$ be the vertical endomorphism in $\T{\cal A}$ (see Subsection \ref{prolfoliabundle}). Then  $J({\bf K} {\cal A}_N)$ is a subbundle of $\V{\cal A}_N$ of rank $k-q$. The anchor $\hat{\rho}$ induces   an isomorphism $\hat{\rho}_{TN}$ from $\left(\T{\cal A}_N\right)/ \left({\bf K} {\cal A}_N\oplus(J({\bf K} {\cal A}_N)\right)$ onto $T(TN)$ over $\rho_N: {\cal A}_N/{\cal K}_N\ap TN$.
    \item[(ii)] the Lie bracket $[\;,\;]_{\cal P}$ on $\T{\cal A}$ induces  by projection on the previous quotient bundle an almost bracket $[\;,\;]_{TN}$  which is independent of the choice of the almost bracket $[\;,\;]_{\cal A}$ such that that $({\cal A},M,\rho,[\;,\;]_{\cal A})$ is a pre-Lie algebroid. In fact  $[\;,\;]_{TN}$ satisfies the Jacobi identity and $\hat{\rho}_{TN}$ induces an isomorphism of Lie algebra between $(\Xi\left(\left(\T{\cal A}_N\right)/ \left({\bf K} {\cal A}_N\oplus J({\bf K} {\cal A}_N)\right)\right), [\;,\;]_{TN})$ and $(\Xi({T(TN)}),[\;,\;])$.

   \end{enumerate}
   \end{Pro}

\begin{proof}${}$\\
Consider  a leaf  $N$    of the foliation defined $\rho({\cal A})$ and $q=$dim$N$.  Since the morphism $\rho:{\cal A}_N\ap TN$ is surjective, so is $T\rho: T{\cal A}_N\ap T(TN)$. Therefore its kernel is a subbundle $\ker T\rho$ of rank $k-q$ over ${\cal A}_N$. %Note that $T{\cal M}_N=(T{\cal A}_N)_{{|\cal M}_N}$.

Fix   some chart  $(U, (x^1,\cdots,x^n))$ such that $(U\cap N,(x^1,\cdots,x^q))$ is a chart for $N$ and $\cal A$ is trivializable over $U$. We can  choose  $(U, (x^1,\cdots,x^n))$  and a local basis $\{{e}_\a\}$ of $\cal A$ such that $\rho(e_\a)=\dis\frac{\p}{\p x^\a}$ for $\a=1,\cdots,q$ over $U\cap N$ (see Remark \ref{spbasis}). Denote by $(x^1,\cdots,x^n,y^1,\cdots,y^k)$ the induced coordinates system on ${\cal U}=\tau_N^{-1}(U)$. From this construction, $\dis\frac{\p}{\p y^{q+1}},\cdots,\dis\frac{\p}{\p y^{k}}$ is a basis of $\ker T\rho$ over $\cal U$.\\

Now let  $(x^1,\cdots,x^q,\dot{x}^1,\dots,\dot{x}^q)$ be the coordinate system on $TN$ associate to $(U\cap N,(x^1,\cdots,x^q))$. By  composition of the projection $\bar{\tau}_N:{\cal A}_N\ap {\cal A}_N/{\cal K}_N$   the  local coordinate system $(x^1,\cdots,x^n,y^1,\cdots,y^k)$ gives rise to a local coordinate system $(\bar{x}^1,\cdots,\bar{x}^q,\bar{y}^1,\cdots \bar{y}^q)$ of $ {\cal A}_N/{\cal K}_N$ on the open $\bar{U}=\bar{\tau}_N(U)$. From the properties of the basis $\{e_\a\}$ it follows that the local expression $\rho_N:{\cal A}_N/{\cal K}_N\ap TN$ is $x^\a=\bar{x}^\a$ and $\dot{x}^\a=\bar{y}^\a$ for $\a=1,\cdots,q$.

  Moreover according to Remark \ref{spbasis} let $\{{\cal X}_\a,{\cal V}_\b\}$ be a basis of $\T{\cal A}_N$ over  $\cal U$, associated with the previous data. This basis have the following properties:

$\bullet$  $\hat{\rho}({\cal X}_\a)=\dis\frac{\p}{\p x^\a}\circ \rho$ for $\a=1,\cdots,q$,

$\bullet$  $\hat{ \rho}({{\cal V}_\a})=\dis\frac{\p}{\p y^\a}\circ\rho$ for all $\a=1,\cdots,k$,

 $\bullet$ $[{\cal X}_\a,{\cal X}_\b]_{\cal P}=\sum_{\g=q+1}^kC_{\a\b}^\g{\cal X}^\g$ for all $\a,\b=1,\cdots,q$,

$\bullet$  $\{{\cal X}_{q+1},\cdots,{\cal X}_k\}$ is a local basis of $\bf{K}{\cal A}_N$.\\

 From the properties of $J$ (see Subsection \ref{prolfoliabundle})  we have then  $J({\cal X}_\a)={\cal V}_\a$ for any $\a=1,\cdots,k$. It follows that $\{{\cal V}_{q+1},\cdots,{\cal V}_k\}$ is a local basis of $J({\bf K} {\cal M}_N)$. In particular $J({\bf K} {\cal M}_N)$ is a subbundle of  $\V{\cal A}_N$ of rank $k-q$. We now consider the morphism $T\rho\circ \hat{\rho}:\T{\cal A}_N\ap T(TN)$. The kernel of this morphism is ${\bf K} {\cal A}_N\oplus J({\bf K} {\cal A}_N)$. Therefore we obtain a quotient morphism from  $ \T{\cal A}_N/ \left({\bf K} {\cal A}_N \oplus J({\bf K} {\cal A}_N)\right)$ to $T(TN)$  over $\rho:{\cal A}_N\ap TN$ which is an isomorphism between fibers. It follows that we obtain an isomorphism $\hat{\rho}_{TN}$ from   $\T{\cal M}_N/ \left({\bf K} {\cal M}_N\oplus J({\bf K} {\cal M}_N)\right)$  to the pull back $\rho^*T(TN)$.
%On the other hand we  have a canonical isomorphism between the quotient bundles :

%  $\T{\cal M}_N/ \left({\bf K} {\cal M}_N\oplus J({\bf K} {\cal M}_N)\right)$ and

%  $\left\{\T{\cal M}_N/ {\bf K} {\cal M}_N\right\}/\left\{\left({\bf K} {\cal M}_N\oplus J({\bf K} {\cal M}_N)\right)/{\bf K} {\cal M}_N\right\}$.

 \noindent Using the same arguments as in the proof of Proposition \ref{foliatedquasi}, by projection on $\T{\cal M}_N/ \left({\bf K} {\cal M}_N\oplus J({\bf K} {\cal M}_N)\right)$ the bracket $[\;,\;]_P$ induces a  Lie bracket $[\;,\;]_{TN}$ on  $\T{\cal M}_N/ \left({\bf K} {\cal M}_N\oplus J({\bf K} {\cal M}_N)\right)$ which does not depend on the choice of $[\;,\;]_{\cal A}$ such that  $({\cal A},M,\rho,[\;,\;]_{\cal A})$ is a pre-Lie algebroid

Now,  by projection on  $\T{\cal M}_N/ \left({\bf K} {\cal M}_N\oplus J({\bf K} {\cal M}_N)\right)$ we obtain a basis \\$\{{\bar{\cal X}}_1, \cdots,{\bar{\cal X}}_q,{\bar{\cal V}}_1, \cdots,{\bar{\cal V}}_q\}$ such that
$ \hat{ \rho}_{TN}({\bar{\cal X}}_\a)=\dis\frac{\p}{\p x^\a}\circ\rho$   and $ \hat{ \rho}_{TN}({\bar{\cal V}}_\a)=\dis\frac{\p}{\p y^\a}\circ \rho$ for all $\a=1,\cdots,q$ and of course we have

 $[\bar{\cal X}_\a,\bar{\cal X}_\b]_{TN}=[\bar{\cal X}_\a,\bar{\cal V}_\b]_{TN}=[\bar{\cal X}_\a,\bar{\cal V}_\b]_{TN}=0$ for all $\a,\b=1,\cdots, q$.\\

To end the proof it remains to show that the bundle   $\T{\cal M}_N/ \left({\bf K} {\cal M}_N\oplus J({\bf K} {\cal M}_N)\right)\ap {\cal A}_N$ induces a  bundle    ${\T{\cal M}_N/ \left({\bf K} {\cal M}_N\oplus J({\bf K} {\cal M}_N)\right)} \ap {\cal A}_N/{\cal K}_N$   which is isomorphic to the bundle  $T(TN)\ap TN$.\\

 Indeed with the previous notations, over ${\cal U}=\tau_N^{-1}(U)$ we have a trivialization of  $\T{\cal M}_N/ \left({\bf K} {\cal M}_N\right)$ given by the basis $\{{\bar{\cal X}}_1, \cdots,{\bar{\cal X}}_q,{\bar{\cal V}}_1, \cdots,{\bar{\cal V}}_q\}$.  Now for $(x,y)\in{\cal U}$  and $\a=1,\cdots,q$ we have  $\bar{\cal X}_\a(x,y)=((x,y), e_\a(x),\dis\frac{\p}{\p x^\a}(x)) $ and $\bar{\cal V}_\a=((x,y),0,\dis\frac{\p}{\p y^\a}(x,y))$ on $\cal U$ (see subsection \ref{prolongM}). It follows that $\bar{\cal X}_\a$ and $\bar{\cal V}_\a$ are projectable on $\bar{U}$.  Moreover the bracket $[\;,\;]_{TN}$ in
 restriction to section over $U$ is clearly also projectable on $\bar{U}$.
  Consider  an other chart  $(\acute{U}, (\acute{x}^1,\cdots,\acute{x}^n))$ and $\{\acute{e}_\a\}$ a basis of ${\cal A}$ with  the same properties as  $(U, (x^1,\cdots,x^n))$ and $\{e_\a\}$ and such that $U\cap\acute{U}\not=\emptyset$. Then on this intersection, we have a field of matrix $(A_\a^\b)$ with $A_\a^\b=0$ for $q<\a\leq k$ and $1\leq\b\leq q$ on $N$  and such that ${e}_\a=A_\a^\b \acute{e}_\b$. If  $\{\acute{\cal X}_\a,\acute{\cal V}_\a\}$ is  the basis of $\T{\cal A}$ associated to the data $(\acute{U}, (\acute{x}^1,\cdots,\acute{x}^n))$ and $\{\acute{e}_\a\}$, we have the relations:
 \begin{eqnarray}\label{chgebase}
 \begin{cases}
\acute{x}^i=\acute{x}^i(x^1,\cdots,x^q) \textrm{ and }\acute{x}^i=\acute{x}^i(x^1,\cdots,x^q) \textrm{ for } i=1,\cdots,q \textrm{ on } N\cr
\acute{y}^\b=A_\a^\b y^\a\hfill{}\cr
{\cal X}_\a=A_\a^\b \acute{\cal X}_\b+\rho_\a^i\dis\frac{\p A_\g^\b}{\p x^i}y^\g\acute{\cal V}_\b\hfill{}\cr
 {\cal V}_\a=A_\a^\b\acute{\cal V}_\b.\hfill{}\cr
  \end{cases}.
\end{eqnarray}
Over  $\tau^{-1}(U\cap \acute{U}\cap N)\subset {\cal A}_N$, we have the following relations

 $\rho_\a^i=0$ for $q<\a\leq k$ and  $\rho_\a^i=\d_\a^i$,

 $A_\a^\b=0$ for $q<\a,\b\leq k$.

 \noindent Thus for $\a=1,\dots,q$,  we obtain
%\begin{eqnarray}\label{chgebase}
% \begin{cases}
%\acute{x}^i=\acute{x}^i(x^1,\cdots,x^q) \textrm{ and }\acute{x}^i=\acute{x}^i(x^1,\cdots,x^q) \textrm{ for } i=1,\cdots,q \textrm{ on } N\cr
%\acute{y}^\b=\sum_{\b=1}^q A_\a^\b y^\a\hfill{}\cr
%{\cal X}_\a=\sum_{\b=1}^q A_\a^\b \acute{\cal X}_\b+\sum_{\b,\g=1}^q \dis\frac{\p A_\g^\b}{\p x^\a}y^\g\acute{\cal V}_\b\hfill{}\cr
% {\cal V}_\a=\sum_{\b=1}^q A_\a^\b\acute{\cal V}_\b\hfill{}\cr
%  \end{cases}.
%\end{eqnarray}

$\bar{\cal X}_\a(x,y)=(A_\a^\b\acute{e}_\b, A_\a^\b\dis\frac{\p}{\p \acute{x}^\b}+\dis\frac{\p A_\g^\b}{\p x^\a} \dis\frac{\p}{\p \acute{y}^\b)}=(e_\a,\dis\frac{\p}{\p {x}^\a})$,

 $\bar{\cal V}_\a(x,y)=(0,A_\a^\b\dis\frac{\p}{\p \acute{y}^\b})=(0,\dis\frac{\p}{\p y^\a}(x,y))$.

 Therefore by projection on the intersection $\bar{U}\cap \bar{\acute{U}}$  we get the compatibility  of  local triviality  with  the equivalent relation  associated with  the quotient  ${\cal A}_N\ap {\cal A}_N/{\cal K}_N$ and  compatible with the isomorphism between   ${\cal A}_N/{\cal K}_N$  and $T(TN)\ap TN$

  \end{proof}

%\begin{Pro}\label{induce connection}

%%%%%%%%%%%%%%%%%%%%%%%%%%%%%%%%%%%%%%%%%%%%%%%%%%%%%%%%%%%%%%%%%%%%%%%%%%%%%%%%%
\section{Geometrical objets on the prolongation over  a fibered manifold}
%%%%%%%%%%%%%%%%%%%%%%%%%%%%%%%%%%%%%%%%%%%%%%%%%%%%%%%%%%%%%%%%%%%%%%%%%%
  {\it  From now on,  the almost Lie algebroid $({\cal A},M,\rho,[\;,\;]_{\cal A})$ is fixed, and  we assume that its rank $k$. We assume that the typical fiber of the fibered manifold $\pi:{\cal M}\ap M$ is connected and is also of dimension $ k$. }\\

 {\rm First of all in this section we introduce the definition and the  properties of almost tangent and cotangent structures. Then we give a non classical definition of a semispray and we study  the compatibility   between these previous almost structures and  a semispray on the one hand and with Euler section and semispray on the other hand.
%%%%%%%%%%%%%%%%%%%%%%%%%%%%%%%%%%%%%%%%%%%%%%%%%%%%%%%%%%%%%%%%%%%%%%%%%%%%%%%%%%%%%%%%%%%%%%%%

%%%%%%%%%%%%%%%%%%%%%%%%%%%%%%%%%%%%%%%%%%%%%%%%%%%%%%%%%%%%%%%%%%%%%%%%%%%%%%%%%%%%%%
%%%%%%%%%%%%%%%%%%%%%%%%%%%%%%%%%%%%%%%%%%%%%%%%%%%%%%%%%%%%%%%%%%%%%%%%%%%%%%%%%%%%%
\subsection{Almost tangent structure and almost cotangent structure }\label{cal E}${}$\\
%%%%%%%%%%%%%%%%%%%%%%%%%%%%%%%%%%%%%%%%%%%%%%%%%%%%%%%%%%%%%%%%%%%%%%%%%%%%%%%%%%%%%%

  According to the original  terminology  of \cite{ClGo1}, \cite{ClGo2}, \cite{LPo3},  \cite{Va3} and \cite{ThSc} we introduce:}

 \begin{Def}\label{ATC}${}$
  \begin{enumerate}
 \item[(i)] An  almost tangent structure on $\T{\cal M}$ is an endomorphism $\cal J$ of $\T{\cal M}$  such that  $\textup{   im }{\cal  J}=\V{\cal M}=\ker {\cal  J}$.
 \item[(ii)] An almost  cotangent structure on $\T{\cal M}$ is a $2$-form $\O\in \L^2(\T{\cal M})^*$ of  maximal rank $2k$, such that each vertical space $\V_m{\cal M}$ is Lagrangian for all $m\in{\cal M}$.
\item[(iii)] An almost tangent  $\cal J$ and almost   cotangent structure  $\O$  on $\T{\cal M}$    are compatible if we have
$$\O({\cal J}{\cal X},{\cal Y})+ \O({\cal X},{\cal J}{\cal Y})\equiv 0 \textrm{ for all sections } {\cal X}, {\cal Y}.$$
\end{enumerate}
\end{Def}

An  almost tangent structure $\cal J$ on $\T{\cal M}$ is  then locally given by a   field of matrix $({\cal J})=({\cal J}_{a}^{\b})$ of maximal rank $k$ such that:
\begin{eqnarray}\label{locJ1}
{\cal J}({\cal Z}_a)= {\cal J}_{a}^B{\cal V}_\b.
\end{eqnarray}
 where  $({\cal J})=({\cal J}_{\a}^{\b})$  is a field of invertible matrices.\\

An almost cotangent structure $ \O$ on $\T{\cal M}$ is locally given by a pair of $k$-square matrix $(\o)=(\o_{\a\b})$ which is  skew symmetric,  and  $(\bar{\o})= (\bar{\o}_{\a\b})$ which is  of rank $k$,  such that we have
\begin{eqnarray}\label{O}
\O({\cal X}_\a,{\cal V}_\b)=\bar{\o}_{\a\b} \;\;\;\; \O({\cal X}_\a,{\cal X}_\b)={\o}_{\a\b} \;\;\;\; \O({\cal V}_\a,{\cal V}_\b)=0.
\end{eqnarray}

The following result contains the principal properties of such structures:

\begin{Pro}\label{TO}${}$
\begin{enumerate}
\item[(i)] On $\T{\cal M}$   there  exists an almost tangent   structure if and only if there exists an almost cotangent   structure.
\item[(ii)] An almost tangent  structure $\cal J$  and  an almost cotangent  structure $\O$ on $\T{\cal M}$ are compatible  if and only if  the matrices $({\cal J})$ and  $ (\O)$  associated with the local decomposition (\ref{locJ1}) of $\cal J$ and  (\ref{O}) of  $\O$) respectively satisfy the relation:
$$(\O)({\cal J})=({\cal J})^{t}(\O)^t,$$
where $(.)^t$ denote the transpose matrix  of $(.)$.
\item[(iii)] If  an almost tangent  structure $\cal J$ and an almost cotangent   structure $\O$  are compatible, then
$g({\cal J}{\cal X},{\cal J}{\cal Y})=\O({\cal J}{\cal X},{\cal Y})$ defines a pseudo-Riemannian metric on $\V{\cal M}$.
On the other hand, given  any Whitney decomposition $\T{\cal M}=\V{\cal M} \oplus H$  and let  $\Pi_{ H}$ be  the associated projection on $H$.   Then
$g_{ H}(\Pi_{ H}{\cal X},\Pi_{ H}{\cal Y})=\O({\cal J}{\cal X},{\cal Y})$ defines also a pseudo-Riemannian metric on ${ H}$.
\item[(iv)]  Conversely, assume that we have   a pseudo-Riemannian metric $g$ on $\V{\cal M}$ and let  $\cal J$ or  $\O$ be  an   almost tangent   or a cotangent structure  on $\T{\cal M}$ respectively. Then there exists a unique almost cotangent  $\O$ or   a unique almost tangent structure $\cal J$  on $\T{\cal M}$ respectively  such that $\cal J$ and $\O$ are compatible and which are linked by the relation:
$$g({\cal J}{\cal X},{\cal J}{\cal Y})=\O({\cal X},{\cal J}{\cal Y})$$
for all sections ${\cal X}$ and ${\cal Y}$ of $\T {\cal M}$.

\end{enumerate}

\end{Pro}

\begin{proof}[ proof of Proposition \ref{TO}]${}$\\

 \noindent\so{\it Part (i)}\footnote{the proof is analogue to the proof of the corresponding  result of \cite{ThSc} but we shall refer to these arguments  in the following}.

 Let  $g$ be a Riemannian metric on $\T{\cal M}$ and choose an almost tangent structure ${\cal J}$ on $\T{\cal M}$ and denote by  $H$   the orthogonal  $(V{\cal M})^\perp$ of    $V{\cal M}$)   relatively to $g$.   Therefore  ${\cal J}$ induces an isomorphism ${\cal J}_g$ between  the bundles  $ H$ and  $\V {\cal M}$. Then we can associate to ${\cal J} $ an almost complex structure $\cal I$ on $\T{\cal M}=H\oplus V{\cal M}$  defined in the following way:

${\cal I}({\cal Z})={\cal J}_g({\cal Z}) \textrm{ for } {\cal Z} \in { K}$, and

${\cal I}({\cal Z})=-({\cal J}_g)^{-1}({\cal Z})  \textrm{ for } {\cal Z} \in V{\cal M}$.

Then we define

$\O({\cal X},{\cal Y})=g({\cal X},{\cal I}{\cal Y})-g({\cal Y},{\cal I}{\cal X}) \textrm{ for sections }  {\cal X}, {\cal Y} \textrm{ of } \T{\cal M}$.

%$\O({\cal X},{\cal Y})=0 \textrm{ for section }  {\cal X} \textrm{ of } K $ and any section ${\cal Y}$ of $\T{\cal M}$.

Conversely let  $\O$ be  an almost cotangent structure on $\T{\cal M}$.
 Denote on the one hand by $\O^\flat$ the  classical  musical  isomorphism   between $\T{\cal M}$  and  $\T{\cal M}^*$,  and on the other hand by $g^\sharp$ the  classical  musical  isomorphism   between $\T{\cal M}^*$  and  $\T{\cal M}$  induced by  $\O$  and  $g$  respectively.
Then the almost tangent structure ${\cal J}$ is defined by:

${\cal J}_{| V{\cal M}}=0\;\;$
$\;\;{\cal J}_{| { H}}= g^\sharp \circ   (\O^\flat)_{| H}\;\;$ and
$\;\;{\cal J}_{| { K}}=0.$\\

\noindent\so{\it Part (ii)}:
  Again let  $\O^\flat $ be  the musical morphism from $\T{\cal M}$ to $\T{\cal M}^*$ induced by $\O$.  Then  $\O$ and $\cal J$ are compatible if and only if $\O^\flat\circ {\cal J}=-{\cal J}^{*}\circ\O^\flat$.  According to the local decomposition of  $\cal J$ and  $\O$, we have

$({\cal J})=\begin{pmatrix} (0)&(0)\cr
					(J)&(0)
					\end{pmatrix}$
 and
							$(O)=\begin{pmatrix} ((\bar{\O})&-(\O)^t\cr
 			   								 (\O)&(0)\cr
			    								\end{pmatrix}$
 associated with the local decomposition of  $\O$. As $\O$ is antisymmetric, we must have $(O)^t=-(O)$.   Then the matrices of 											
  the morphisms $\O^\flat\circ {\cal J}$  and  $-{\cal J}^{*}\circ\O^\flat$ are  respectively  $(O)(J)$  and  $-(J)^t(O)=(J)^{t}(O)^t$, which ends the proof.\\

\noindent\so{\it Part (iii)}:
First of all,  from the compatibility of $\cal J$ and $\O$, the definition of $g$ is symmetric in ${\cal X}$ and $\cal Y$. Now if ${\cal J}{\cal X}={\cal J}{\cal X}'$ then ${\cal X}-{\cal X}'$ is vertical, and since  $\V{\cal M}$ is Lagrangian it follows that $g$ is well defined. Moreover, in local coordinates the matrix  of $g$ is   $(\bar{\O})(J)=(J)^t(\bar{\O})^t$
 which is symmetric and  of maximal rank $k$.\\
The proof  of the last part  is  analogue to the previous one.\\

\noindent \so{\it Part (iv)}:
Assume that $\O$ is given and denote by   $j:V{\cal M}\ap \T{\cal M}$  the canonical inclusion  and $\cal H$ a complementary of $\V{\cal M}$ in $\T{\cal M}$. Then we can define $\cal J$  as in the proof of part (i),  and  it is easy to see that $\O$ and $\cal J$ are compatible.

Now assume that $\cal J$ is given. By similar arguments used in the proof of part  (i), and  with the same notations we can associate to ${\cal J}$  an almost complex structure $\cal I$ on $\T{\cal M}$.  Now, since  $\bar{\cal J}={\cal J}_{| {\cal H}}$ is an isomorphism from ${\cal H}$ to ${\V}{\cal M}$, we define a pseudo-Riemannian metric $\bar{g}$  on $\V{\cal M}$  by $\bar{g}({\cal X},{\cal Y})=g_{\cal H}({\bar{\cal J}}^{-1}{\cal X},{\bar{\cal J}}^{-1}{\cal X})$.  These two pseudo-Riemannian metrics  induce a pseudo-Riemannian metric on $\T{\cal M}$ such that   ${ \cal H}$ and  $\V{\cal M}$ are orthogonal.  Then, we can define $\O$ in the same way as in the proof of part (i). By an elementary calculus we can see that $\O$ and $\cal J$ are compatible.\\
\end{proof}

 As in the case of almost tangent manifold (\cite{Va3},  \cite{ThSc}) the existence of an almost tangent  structure on $\T{\cal M}$  implies some topological conditions on $\cal M$:

 \begin{Pro}\label{vertaff}${}$\\
If there exists an almost tangent  or cotangent structure on $\T{\cal M}$  then  the fibration $\pi :{\cal M}\ap M$ is a   locally affine fibration (see Subsection \ref{verteuler}).
 \end{Pro}

 The proof of   this Proposition is an adaption, to our context,  of the proof of Theorem 3.3 in  \cite{ThSc}

 \begin{proof} [skecth of the proof ]${}$\\
 %We consider any subbundle  $\cal E$ of ${\T}{\cal M}$ which satisfies the assumptions of Proposition \ref{vertaff}.
  Let  $\cal H$ be a complementary subbundle of $\V{\cal M}$ in $\T{\cal M}$.
 Then $\pi_{{\cal A}}({\cal H})={\cal A}$. Moreover, $(\pi_{\cal A})_{| {\cal H}}$ is an isomorphism from the fiber over $m$ of $\cal H$ onto the fiber over $\pi(m)$ of ${\cal  A}$.
  We assume that there exists an almost tangent structure $\cal J$  on $\T{\cal M}$.
 Now any section $  X$ of $\cal A$ gives rises to some  section $\cal  X$ of $\cal H$  and then a unique vertical section ${\cal X}^v={\cal J}({\cal X})$ which does not depend on the choice of $\cal X$.
  Fix some $x_0\in M$ and consider the fiber $F_{x_0}=\pi^{-1}(x_0)$ and some $m_0\in F_{x_0}$. Choose  a Riemannian metric $g$ on ${ H}$ and a local section $\s$ of ${\cal M}$ defined on a neighbourhood $U$ of $x_0$ such that ${\s}(x_0)=m_0$. Then, for any section $X$ and $Y$ of ${H}$, $g_{\s(x)}(X^v,Y^v)$ on each fiber $F_x=\pi^{-1}(x)$ only depends  on $x$. We end the proof by using the same arguments as in the proof of Theorem 3.3 of \cite{ThSc}.\\
 Now, we assume that we have an almost cotangent structure $\O$ on $\T{\cal M}$. Choose a Riemaniann metric  $g$  on ${\T}{\cal M}$. Then, according to the proof of  Proposition \ref{TO} part (i), to the pair $(g,\O)$ we can associate  an almost tangent  structure $\cal J$ on  $\T{\cal M}$.
 \end{proof}

\begin{Rem}\label{affJO}${}$\\
 According to the proof of  the previous Proposition, in this context, we have an atlas of chart $\{U_\l,([x_\l]^i, [y_\l]^\a)$  and a system of basis $\{[e_\l]_\a\}$ of $ H$   such that the associated   vertical basis  is given by $[{\cal V}_\l]_\a=([e_\l]_\a)^v$.
It follows that  the transition functions are of type:
 $$[x_\l]^i=[\Psi_{\l\mu}]^i([x_\mu]^j)\;\;[ y_\l]^\a=[\Phi_{\l\mu}]^\a_\b([x_\mu]^j)[y_\mu]^\b+[\phi_{\l\mu}]^\b([x_\mu]^j)\;\; [ e_\l]_\a=[G_{\l\mu}]_\a^\b([x_\mu]^j)[e_\mu]_\b,$$
    where the matrix $([\Phi_{\l\mu}]^\a_\b)$ is the inverse of the matrix $([G_{\l\mu}]_\a^\b )$ .\\
If $ \O$ is an almost cotangent structure on $\T{\cal M}$,  according to the proof of Proposition \ref{affJO}, in this context again,  given any section $\hat{\cal X}$ of $ H$, we can define a vertical lift ${\cal X}^{*v}$  in $(\V{\cal M})^*$  by ${\cal X}^{*v}=\O^\flat({\cal X})$. Therefore  we have an atlas of charts $\{U_\l,([x_\l]^i, [y_\l]^\a)$  and a system of basis $\{[e_\l]_\a\}$ of $ H$ such that for the associated vertical dual  basis $[{\cal W}_\l]^\a=([e_\l]_\a)^{*v}$. It follows that the transition functions are of type:
$$[x_\l]^i=[\Psi_{\l\mu}]^i([x_\mu]^j)\;\;[ z_\l]^\a=[\Phi_{\l\mu}]^\a_\b([x_\mu]^j)[z_\mu]^\b+[\phi_{\l\mu}]^\b([x_\mu]^j)\;\; [ e_\l]_\a=[G_{\l\mu}]_\a^\b([x_\mu]^j)[e_\mu]_\b $$
    where the matrix $([\Phi_{\l\mu}]^\a_\b)$ is equal to the matrix $([G_{\l\mu}]_\a^\b )$ . Note that, as in the proof of Proposition \ref{vertaff},  if ${\cal J}=g^\sharp \circ   (\O^\flat)$  we will have $g^\sharp([{\cal W}_\l]^\a)=[{\cal V}_\l]_\a$.   Therefore, in general, these locally affine fibration structures are not compatible as "locally affine fibration structures ", but, according to the relation $g^\sharp([{\cal W}_\l]^\a)=[{\cal V}_\l]_\a$ we  can consider that $g$ is  a "link" between both these structures.
\end{Rem}

{\rm For the illustration of the  existence of almost tangent or  cotangent structures, we mention the following classical cases (see more original  Examples in subsection \ref{ssprays}):\\

On ${\T}{\cal A}$, the vertical endomorphism $J$ is a canonical almost tangent structure  and on ${\T}{\cal A}^*$ the almost canonical symplectic form $\o=d^{{\T}{\cal A}^*}\theta$ is an almost cotangent structure. In particular, when ${\cal A}=TM$, we get the canonical almost tangent  and  cotangent structure on $TM$ and  $T^*M$ respectively. Moreover, when ${\cal A}$ is a Lie algebroid,  the vertical endomorphorm $J$  is "integrable" on ${\T}{\cal A}$  and the canonical  symplectic form $\o$ on ${\T}{\cal A}^*$ is exact. More generally we introduce :

\begin{Def}\label{int}${}$\\
 An almost tangent structure $\cal J$ on $\T{\cal M}$ is called {\it integrable}  if its Nijenhuis tensor ${\cal N}_{\cal J}$ vanishes:
\begin{eqnarray}\label{nijen}
{\cal N}_{\cal J}({\cal X},{\cal Y})=[{\cal J}{\cal X},{\cal J}{\cal Y}]_{\cal P}-{\cal J}[{\cal J}{\cal X},{\cal Y}]_{\cal P}-{\cal J}[{\cal X},{\cal J}{\cal Y}]_{\cal P}\equiv 0.
\end{eqnarray}
 \end{Def}

% When  $\pi_{{\cal A}}({\cal E})$ is a subbundle of $\cal A$, then,  locally,
 we have the following characterization of an integrable almost tangent structure :
\begin{Pro}\label{intloc}${}$\\
Assume that  $\cal J$ is an almost tangent structure on $\T{\cal M}$ and
 %According subsection   \ref{subbun}
 denote by $(\cal J_\a^\b)$ its associated matrix in some  canonical local   basis $\{{\cal X}_\a,{\cal V}_\b\}$  of $\T{\cal M}$ around $m\in {\cal M}$.% (see According subsection   \ref{subbun}).
 Then the following properties are equivalent:
\begin{enumerate}
\item[(i)] ${\cal N}_{\cal J}=0$;
\item[ii)] around  any $m\in {\cal M}$ there exists a canonical local   basis $\{{\cal X}_\a,{\cal V}_\b\}$ of $\T{\cal M}$ such that, with  the previous notations if $(\tilde{J}_\g^\d)=(\cal J_{\a}^{\b})^{-1}$,
  we have $\dis\frac{\p \tilde{J}_\a^\b}{\p y^\g}=\dis\frac{\p \tilde{J}_\g^B}{\p y^\a}$;
  \item[(iii)]for each $m\in {\cal M}$ there exists a coordinate system $(\bar{x}^i,\bar{y}^\a)$ around $m$ compatible with $\pi$ and a basis $\{\bar{e}_\a\}$ of $\cal A$, such, in the corresponding local basis  $\{\bar{\cal X}_\a,\bar{\cal V}_\b\}$  of ${\T}{\cal M}$ we have  ${\cal J}\bar{\cal X}_\a=\bar{\cal V}_\a$.
\end{enumerate}
\end{Pro}

\begin{proof}${}$\\
According to (\ref {nijen}), and the properties of the bracket of the  canonical  basis $\{{\cal X}_\a,{\cal V}_\b\}$ we have:

${\cal N}_{\cal J}({\cal X}_\a,{\cal X}_\b)=({\cal J}_\a^\g \dis\frac{\p {\cal J}_\b^\d}{\p y^\g}-{\cal J}_\b^\g\dis\frac{\p {\cal J}_{\a}^{\d}}{\p y^{\g}}){\cal V}_{\d}$ and
${\cal N}_{\cal J}({\cal X}_\a,{\cal V}_\b)={\cal N}_{{\cal J}}({\cal V}_{\a},{\cal V}_{\b})=0$.\\
Therefore,  ${\cal N}_{{\cal J}}\equiv 0$ is equivalent to ${\cal J}_{\a}^{\g}\dis\frac{\p {\cal J}_{\b}^{\d}}{\p y^{\g}}-{\cal J}_{\a}^{\g}\dis\frac{\p {\cal J}_{\a}^{\d}}{\p y^{\g}}=0$ and this last relation is equivalent to
$\dis\frac{\p \tilde{J}_\a^\b}{\p y^\g}=\dis\frac{\p \tilde{J}_\g^\b}{\p y^\a}$ (see \cite{HP}).

Assume that ${\cal N}_{\cal J}\equiv 0$. We set $\bar{\cal V}_\a={\cal J}{\cal X}_\a$. As we have seen previously, it follows that $[\bar{\cal V}_\a,\bar{\cal V}_\b]_{\cal P}=0$. From the properties of $[\;,\;]_{\cal P}$, the  vector fields $\{\hat{\rho}(\bar{\cal V}_\a)\}$ on ${\cal M}$  commute for the usual Lie bracket on $\cal M$ . Therefore, from the coordinates $(x^i,y^\a)$ on $\cal M$, we can obtain a new coordinate system $(\bar{x}^i,\bar{y}^\a)$ of type $\bar{x}^i=x^i$ and $\bar{y}^\a=\phi^\a(x,y)$ such that
$\hat{\rho}(\bar{\cal V}_\a)=\dis\frac{\p}{\p \bar{y}^\a}={\cal J}_\a^\b\dis\frac{\p}{\p {y}^\a}.$
 Moreover, with this system  $(\bar{x}^i,\bar{y}^\a)$ and the initial  basis $\{e_\a\}$ of $\cal A$ one can associate  a local adapted basis $\{\bar{\cal X}_\a,\bar{\cal V}_\b\}$. If we write $y^\b=\psi^\b(\bar{x},\bar{y})$,  we have
 $\rho(e_\a)=\bar{\rho}_\a^i\dis\frac{\p}{\p \bar{x}^i}$ and so $\bar {\cal X}_\a={\cal X}_\a+ \dis\frac{\p \psi^\b }{\p \bar{x}^i}\bar{\rho}_\a^i{\cal V}_\b$. Thus we obtain:
$${\cal J}\bar{\cal X}_\a={\cal J}({\cal X}_\a+ \dis\frac{\p \psi^\b }{\p \bar{x}^i}\bar{\rho}_\a^i{\cal V}_\a)=\bar{\cal V}_\a.$$
The converse is elementary.\\
\end{proof}

\begin{Rem}\label{atalsJ}${}$\\
If $\cal J$ is integrable,   in the context of Proposition \ref{intloc} (iv), let  $(x^i,y^\a)$ be a coordinate system around some point   $m\in {\cal M}$ compatible with $\pi$  and a local basis $\{e_\a\}$ of $\cal A$  such that the associated basis $\{{\cal X}_\a,{\cal V}_\a\}$ satisfies
\begin{eqnarray}\label{adpJ}
{\cal J}{\cal X}_\a={\cal V}_\b.
\end{eqnarray}
Therefore, we have a covering of $\cal M$ by chart domains, each one being associated with a  coordinate system as previously, and a  local basis  $\{e_\a\}$ such  that the corresponding  basis of ${\T}{\cal M}$ satisfies  the property (\ref{adpJ}).
 It follows that we have the following  type of transition functions between such coordinate systems and local basis of $\cal A$
\begin{eqnarray}
\bar{e}_\a= A_\a^\b(x) e_\b,\;\;\bar{x}^i=\Psi^i(x),\;\; \bar{y}^\a=\bar{A}^\a_\b(x) y^\b+\phi^\a(\bar{x}) \; \textrm{ with } (\bar{A}^\a_\b)=(A_\a^\b)^{-1}.
\end{eqnarray}
This implies that  we obtain  a  particular sub-atlas  associated with $\cal J$  which is a   locally affine fibration. This structure  will be called the {\bf $\cal J$-locally affine fibration}.
\end{Rem}
\bigskip

Recall that on  ${\T}{\cal A}^*$, we have a canonical almost cotangent structure  $\O=-d^{{\cal A}^*}\theta$ where $\theta$ is the Liouville form. %  then $\O=d^{{\cal A}^*}\theta$ is an almost symplectic form %and moreover the vertical bundle $ {\V}{\cal A}^*$ is Lagrangian and then $\O$ is  the "canonical" almost cotangent structure on ${\T}{\cal A}^*$.
More generally we introduce

\begin{Def}\label{LiouvALO}${}$\\ We will say that $\O$ is an exact  or a locally exact almost  structure on $\T{\cal M}$ if there exists a global  semi-basic $1$-form $\theta$  or  a local  semi-basic $1$-form $\theta$ around each point of $\cal M$ respectively,     such that $\O=-(d^{\cal P}\theta)$ is  symplectic.
\end{Def}

As $\theta$ vanishes on $\V{\cal M}$ and $\V{\cal M}$ is stable under the bracket $[\;,\;]_{\cal P}$, from Cartan formulae, it follows that an exact or  a locally exact almost structure on $\T{\cal M}$ is  {\it really} an almost cotangent structure on $\T{\cal M}$. Moreover, according to Remark \ref{dPsbasic},  this property depends only on the choice of the bracket $[\;,\;]_{\cal A}$
\bigskip
Classically, given a  regular Lagrangian $L:{\cal A}\ap \R$ (\textit{i.e.} $rank(\dis\frac{\p^2 L }{\p y^\a\p y^\b})=k$ in local coordinates) on a Lie algebroid, the Cartan  $1$-form is $\theta_L=J^*(d^{\cal P}L)$ and we obtain  an almost cotangent structure $\o_L=-d^{\cal P}\theta_L$ on ${\T}{\cal A}$ which is locally written:
\begin{eqnarray}\label{oL}
\o_L=\dis\frac{\p^2 L }{\p y^\a \p y^\b}{\cal X}^\a \wedge {\cal V}^\b+\dis\frac{1}{2}(\dis\frac{\p^2 L }{\p x^i \p y^\a}\rho^i_\b-\dis\frac{\p^2 L }{\p x^i\p y^\b}\rho^i_\a+\dis\frac{\p L }{\p  y^\g}C^\g_{\a\b}){\cal X}^\a\wedge{\cal X}^\b.
\end{eqnarray}

As in \cite{Va2}, we introduce:

\begin{Def}\label{locLag}${}$\\
Let  ${\cal J}$ be  an almost tangent and $\O$ an almost cotangent structure on   ${\T}{\cal M}$ respectively. We say that $({\cal M},{\cal J},\O)$ is  {\it locally Lagrangian}, if, for any $m\in {\cal M}$ there exists a local function $L$ on ${\cal M}$  such that locally  $\O=-d^{\cal P}\theta_L$  where $\theta_L={\cal J}^*(d^{\cal P}L)$.
\end{Def}
%\begin{Lem}\label{loclagP}${}$\\
%The property " $({\cal M},{\cal J},\O)$  (locally) Lagrangian" depends only of the choice of the bracket $[\;,\;]_{\cal A}$ and not of the choice of  $\cal P$
%\end{Lem}

%\begin{proof}${}$\\
%It is sufficient to prove the result locally. At first under the assumption of the definition, we have $h=k$.  According to our convention (Cf.  Subsection \ref{subbun}), consider a canonical local basis $\{{\cal X}_\a,{\cal V}_\b\}$ of ${\T}{\cal M}$.  So we have a decomposition:
%$${\cal J}{\cal X}_\a=J_\a^\b{\cal V}^\b.$$
%In the associated dual basis $\{{\cal X}^\a,{\cal V}^\b\}$ we have then ${\cal J}^*{\cal V}^\b=J_\a^\b{\cal X}^\a$ and ${\cal J}^*{\cal X}^\a=0$.
% Thus we get
%$$\theta_L=\dis\frac{\p L}{\p y^\a}J_\b^\a{\cal X}^\b.$$
% Finally, according to (\ref{locdifalg}), $d^{\cal P}\theta_L$ can be written:
%$$(\dis\frac{\p^2 L}{\p y^\a\p y^\g}J_\b^\a+\dis\frac{\p L}{\p y^\a}\dis\frac{\p J_\b^\a}{\p y^\g}){\cal V}^\g\wedge {\cal X}^\b+\left((\dis\frac{\p^2 L}{\p y^\a\p x^i}\rho_\g^i J_\b^\a+\dis\frac{\p L}{\p y^\a}\dis\frac{\p J_\b^\a}{\p x^i}\rho_\g^i){\cal X}^\g\right)\wedge{\cal X}^\a-\dis\frac{1}{2}\dis\frac{\p L}{\p y^\a}J_\b^\a C^\a_{\b\d}{\cal X}^\b\wedge {\cal X}^\g.$$
%Therefore, $d^{\cal P}\theta_L$ is independent of the choice of the projection $\cal P$ and thud depends only of the choice of $[\;,\;]_{\cal A}$.\\

%\end{proof}

Now, we have the following characterization of a locally Lagrangian almost cotangent structure (compare with \cite{Va2}):

\begin{Pro}\label{locLagch}${}$\\
Consider  an almost tangent structure ${\cal J}$   and an almost  cotangent  structure $\O$  on   ${\T}{\cal M}$ respectively. Assume that ${\cal J}$ is integrable. Then $({\cal M},{\cal J},\O)$ is   locally  Lagrangian
if and only if  ${\cal J}$ and $\O$ are compatible and $\O$ is exact.

\end{Pro}

\begin{proof}${}$\\
We first begin by an observation. Since ${\cal J}$ is integrable, according to Proposition \ref{intloc},  there exists a coordinate system $({x}^i,{y}^\a)$ compatible with $\pi$ and a basis $\{e_\a\}$, such that  in the corresponding local basis  $\{{\cal X}_\a,{\cal V}_\b\}$ we have
${\cal J}{\cal X}_\a={\cal V}_\b$. Given  a  function $L$ on $M$ defined on  the domain of the previous chart we set  $\theta_L={\cal J}^*(d^{\cal P}L)$.
 In the  dual basis $\{{\cal X}^\a,{\cal V}^\b\}$,  associated with the previous basis  $\{{\cal X}_\a,{\cal V}_\b\}$ the $2$-form   $-d^{\cal P}\theta_L$    can be written as in (\ref{oL}).   \\

 Now, assume that  $({\cal M},{\cal J},\O)$ is   locally Lagrangian. This implies that  there exists a local function $L$ such that  $\O=-d^{\cal P}\theta_L$ with $\theta_L =-{\cal J}^*d^{\cal P}L$ so $\theta$  vanishes on $\V{\cal M}$ and then, $\O$ is locally exact.  On the other hand, as $\cal J$ is integrable, from the previous observation, we have a coordinate system $(x^i, y^\a)$ and an associated   dual basis $\{{\cal X}^\a,{\cal V}^\b\}$  such that $\O$ can be written as in (\ref{oL}). Therefore we obtain \begin{eqnarray}\label{gJL}
\O({\cal X}_\a,{\cal J}{\cal X}_\b)=-\O({\cal J}{\cal X}_\a,{\cal X}_\b)=\dis\frac{\p^2 L }{\p y^\a \p y^\b}.
\end{eqnarray}

Moreover, as  ${\cal J}{\cal V}_\a=0$, it follows that  $\O$ and ${\cal J}$ are compatible  (cf Proposition \ref{TO}).\\

 Conversely, assume that  ${\cal J}$ and $\O$ are compatible and  let  $\theta$ be a local  $1$-form such that $\O=d^{\cal P}\theta$ and $\theta$  vanishes  on $\V{\cal M}$.
  In the previous basis we have the decomposition $\theta=\mu_\b{\cal X}^\b$, and therefore we obtain
    $$\O=\dis\frac{1}{2}({\cal L}^{\cal P}_{{\cal X}_\a}\mu_\b-{\cal L}^{\cal P}_{{\cal X}_\b}\mu_\a-\mu_\g C^\g_{\a\b}){\cal X}^\a\wedge{\cal X}^\b-{\cal L}^{\cal P}_{{\cal V}_\b}\mu_\a{\cal X}^\a\wedge{\cal V}^\b.$$
Let us notice that we have  ${\cal L}^{\cal P}_{{\cal V}_\b}\mu_\a=\dis\frac{\p \mu_\a}{\p y^\b}$.  If moreover, $\cal J$ and $\O$ are compatible,  then the matrix of general term ${\cal L}^{\cal P}_{{\cal V}_\b}\mu_\a$ must be symmetric and  so we get $\dis\frac{\p \mu_\a}{\p y^\b}=\dis\frac{\p \mu_\b}{\p y^\a}$. It follows that  there exists a local function $L$ such that
$\mu_\a=-\dis\frac{\p L}{\p y^\a}$, and we have  $\O=d^{\cal P}({\cal J}^*d^PL)$. We conclude that  $({\cal M},{\cal J},\O)$ is   locally  Lagrangian.
\end{proof}}

%%%%%%%%%%%%%%%%%%%%%%%%%%%%%%%%%%%%%%%%%%%%%%%%%%%%%%%%%%%%%%%%%%%%%%%%%%%%%%%%%%%%%%%
\subsection{Semispray  and almost tangent structure}\label{ssprays}${}$\\
%%%%%%%%%%%%%%%%%%%%%%%%%%%%%%%%%%%%%%%%%%%%%%%%%%%%%%%%%%%%%%%%%%%%%%%%%%%%%%%%%%%%%%
%%%%%%%%%%%%%%%%%%%%%%%%%%%%%%%%%%%%%%%%%%%%%%%%%%%%%%%%%%%%%%%%%%%%%%%%%%%%%%%%%%%%%%
${}\;\;\;\;${\rm In this section we will define a generalization of the notion of  semispray  on a Lie  algebroid. For this purpose, we need the following Lemma.

\begin{Lem}\label{dirfan}${}$
\begin{enumerate}
\item[(i)] Let  ${\cal Z}$ be a local  section of ${\T}{\cal M}$ around some point $m\in {\cal M}$ and $\{{\cal V}_\a\}$ a local basis of $\V{\cal M}$ around $m$. Then the vector space ${\cal B}_{\cal Z}(m)\subset {\T}{\cal M}$ spanned by the set
$$\left\{{\cal V}_1(m),\cdots,{\cal V}_k(m), [{\cal Z},{\cal V}_1]_{\cal P}(m),\cdots[{\cal Z},{\cal V}_k]_{\cal P}(m)\right\}$$
 in ${\T}_m{\cal M}$
 is independent of the choice of the basis $\{{\cal V}_\a\}$ and the choice of the bracket $[\;,\;]_{\cal A}$ .
 \item[(ii)] If  ${\cal Z}$  is a global section of ${\T}{\cal M}$ then the correspondence   ${\cal B}_{\cal Z}:m\ap {\cal B}_{\cal Z}(m)$ is a smooth distribution on ${\cal M}$   \end{enumerate}
 \end{Lem}

\begin{proof}${}$\\
First of all we can notice that we have
$[{\cal Z},f{\cal V}]_{\cal P}={\cal L}^{\cal P}_{\cal Z}(f){\cal V}+f[{\cal Z},{\cal V}]_{\cal P}$.
This implies that the vector space ${\cal B }_{\cal Z}(m)$ is independent of the choice of the basis $\{{\cal V}_\a\}$ and also that ${\cal B}_{\cal Z}$ is a smooth distribution. Now, given a canonical basis $\left\{{\cal X}_\a,{\cal V}_\b\right\}$ of $\T{\cal M}$ associated with a coordinate system $(x^i,y^\a)$ we have a decomposition of type   ${\cal Z}={\cal Z}^\a{\cal X}_\a+\bar{\cal Z}^\b{\cal V}_\b.$  Then modulo the vertical bundle we have $[{\cal Z},{\cal V}_\g]_{\cal P}\equiv -\dis\frac{\p {\cal Z}^\a}{\p y^\g}{\cal X}_\a.$

Therefore ${\cal B}_{\cal Z}$ is locally generated by $\{ \dis\frac{\p {\cal Z}^\a}{\p y^\g}{\cal X}_\a, {\cal V}_\g\}_{\g=1,\dots,k}$ and so ${\cal B}_{\cal Z}(m)$ does not depend on the choice of the bracket $[\;,\;]_{\cal A}$.

\end{proof}
%Now, we have (compare with \cite{Fou}) :

\begin{Def}\label{fansection}${}$\\
A global section ${\cal S}$ of ${\T}{\cal M}$ is called a   semispray  on ${\cal M}$ if the  associated distribution ${\cal B}_{\cal S}$ is equal to $\T{\cal M}$
\end{Def}

   \begin{Rem} \label{Schangebrac}${}$\\
Consider two almost brackets $[\;,\;]_1$ and $[\;,\;]_2$ on $\cal A$ such that $({\cal A},M, \rho, [\;,\;]_k)$ is a pre-Lie algebroid for $k=1,2$. According to Lemma \ref{dirfan},  any semispray $\cal S$  relative to $({\cal A},M, \rho, [\;,\;]_1)$
is also a semispray relative to $({\cal A},M, \rho, [\;,\;]_2)$
\end{Rem}

The Definition \ref{fansection} is then  justified by the following result:

\begin{Pro} \label{JS}${}$
\begin{enumerate}
\item[(i)] For any semispray ${\cal S}$ on ${\cal M}$ there exists a unique almost structure ${\cal J}$ on $\T{\cal M}$ associated with ${\cal S}$, characterized by
${\cal J}[{\cal S},{\cal V}]_{\cal P}=-{\cal V}$
for all vertical section $\cal V$.\\
Moreover, for any vertical section ${\cal V}$  of ${\T}{\cal M}$ or any non zero function $f$ on $ M$, the sections ${\cal S}'={\cal S}+{\cal V}$ and   ${\cal S}^{''}=(f\circ\pi){\cal S}$  are  also  semisprays  on $\cal M$ %with the following properties:
%(a)  ${\cal B}_{\cal S}={\cal B}_{{\cal S}'}$ (resp. ${\cal B}_{\cal S}={\cal B}_{{\cal S}^{''}}$);
%(b)
and  the almost structures respectively  associated with ${\cal S}'$ and ${\cal S}"$ are ${\cal J}$  and   $(f\circ\pi){\cal J}$.
\item[(ii)]  Set  ${\cal C}={\cal J}{\cal S}$. Then we have the following:

%(a)   $\pi_{\cal A}({\cal E})$ is a subbundle of $\cal A$;

(a)  $\cal J$ is integrable   and we have  $({\cal L}^{\cal P}_ {\cal C}{\cal J})=-{\cal J}$\footnote{Recall that, if $\cal I$ is an endomorphism of ${\T}{\cal M}$ the Lie derivative ${\cal L}^{\cal P}_{\cal S}{\cal I}$ is given by:
$${\cal L}^{\cal P}_{\cal S}{\cal I }({\cal X})=[{\cal S},{\cal I}{\cal X}]_{\cal P}-{\cal I}[{\cal S},{\cal X}]_{\cal P}$$};

\indent (b)  for any $m\in {\cal M}$,  there exists a system of coordinates $(x^i,y^\a)$, compatible with
$\pi$ and a basis $\{e_\a\}$ of $\cal A$ such that, in the associated   basis $\{{\cal X}_\a,{\cal V}_\b\}$ of ${\T}{\cal M}$   around $m\in {\cal M}$, we have the following properties:
\begin{eqnarray}\label{adaptJS}
\hat{\rho}({\cal V}_\a)=\dis\frac{\p}{\p y^\a}\;\;\;\;\;\;{\cal J}{\cal X}_\a={\cal V}_\a\;\;\;\;\;\;\;{\cal S}=y^\a {\cal X}_\a+{\cal S}^\b{\cal V}_\b \;\;\;\;\;\; {\cal C}=y^\a{\cal V}_\a.
\end{eqnarray}
\end{enumerate}
\end{Pro}

\begin{Rem}\label{lieJ}${}$\\
 Let   ${J}$ be  the almost tangent structure  associated with a  semispray $\cal S$ defined in Proposition \ref{JS}. From  the Property (b) of  Part (ii), we have
$${\cal L}^{\cal P}_{\cal S}{\cal J}({\cal V})=-{\cal J}[{\cal S},{\cal V}]_{\cal P}= {\cal V}$$
for all vertical section $\cal V$.\\
\end{Rem}

According to the previous Remark we have:

\begin{Def}\label{ATSasso}${}$\\
An almost tangent structure $\cal J$ on $\T{\cal M}$  is called an almost tangent  structure compatible with a semispray  $\cal S$
if we have  $[{\cal L}^{\cal P}_{\cal S}{\cal J}]_{| \V{\cal M}}=Id_{| \V{\cal M}}$.\\
\end{Def}
 Under the assumption of part (ii) of Proposition \ref{JS},   we have a global vertical section ${\cal C}={\cal J}{\cal S}$   so that,   in some appropriate local coordinates  around some $m\in{\cal M}$,   on each  vertical fiber of $T{\cal M}$,  the vector field $C=\hat{\rho}({\cal C})$ is an infinitesimal homothety.  It follows that  $C$ is an Euler vector field on the fibration $\pi:{\cal M}\ap M$ (see subsection \ref{verteuler}). In particular, this  fibration  is a locally linear fibration. Therefore, naturally we introduce

 \begin{Def}\label{alEuler}${}$\\
 A global vertical section $\cal C$ is called an   Euler section on $\cal M$, if  around any point $m\in{\cal M}$, there exists a coordinate system $(x ^i,y^\a)$ compatible with $\pi$
 such that in the associated basis $\{{\cal V}_\a\}$ of $\V{\cal M}$  we have ${\cal C}=y^\a{\cal V}_\a$
 \end{Def}

 It is clear that if there exists an  Euler section $C$  on $\cal M$, then, since  $\hat{\rho}$ is an isomorphism from $\V{\cal M}$ to $\V T{\cal M}$, there exists an  Euler vector field  $\cal C$ on $\cal M$ such that $\hat{\rho}({\cal C})=C$ and so  we have an Euler section on $\cal M$ if and only if we have a locally linear fibration structure on $\cal M$ (see Proposition \ref{Euler}).
Now, in some sense, we have the following converse of part (ii) of Proposition \ref{JS}:

 \begin{Pro}\label{existS}${}$\\
Assume that there exists an almost tangent structure ${\cal J}$ on ${\T}{\cal M}$ and an Euler section  ${\cal C}$ on $\cal M$ such that  $({\cal L}^{\cal P}_ {\cal C}{\cal J})=-{\cal J}$. Then
   $\cal J$ is integrable and   there exists a   semispray  ${\cal S}$,  such that  $\cal J$ is compatible with $\cal S$.
 Finally,  any  other section ${\cal S}'$ , such that ${\cal J}{\cal S}'={\cal C}$ can be written   ${\cal S}'={\cal S}+{\cal V}$ for some   vertical section $\cal V$, and is  a  semispray, ${\cal S}'$ has  the same properties as ${\cal S}$. In particular, the set of    semisprays $\cal S$ such that ${\cal J}{\cal S}=\cal C$ has an affine structure.
  \end{Pro}
 This result leads to the following  Definition:

\begin{Def}\label{SEJ}${}$\\
An almost tangent structure $\cal J$ on $\T{\cal M}$ is called  compatible with an Euler section $\cal C$ on $\cal M$, if we have
$({\cal L}^{\cal P}_ {\cal C}{\cal J})=-{\cal J}$.
\end{Def}

    Now we  shall give some examples   which illustrate the context of Proposition \ref{JS} and Proposition \ref{existS}. Note that this context also occurs on compact fibered manifolds (see  Example \ref{JOS6}).

\begin{Ex}\label{exJOS1}${}$\\
 On ${ \T}{\cal A}$  the Liouville section $C$  is an   Euler section,  the vertical endomorphism is an integrable  almost tangent structure, and we have ${\cal L}^{\cal P}_{\cal C}J=-J$. In this context, classically a  semispray  is a section ${\cal S}$ such that $J{\cal S}={\cal C}$. From Proposition \ref{existS}, ${\cal S}$ satisfies the assumptions of Definition \ref{fansection}. If ${\cal M}$ is an open submanifold of ${\cal A}$ so that  $\pi=\t_{| {\cal M}}:{\cal M}\ap M$ is a fibration, all these data induce on ${\T}{\cal M}$ an  Euler section, an integrable almost structure, and   a semispray .\\
 \end{Ex}
 \begin{Ex}\label{exJOS2}${}$\\
Let  $({\cal A},M,\rho [\;,\;]_{\cal A})$ be a Lie algebroid    and ${\cal A}^*$ the dual of $\cal A$. On ${\T}{\cal A}^*$ we have a canonical almost cotangent structure, so  there exists an almost tangent structure ${\cal J}$. On the other hand, we also have an   Euler section $\cal C$ which is compatible with $\cal J$ (cf \cite{LPo2}). It follows that  the assumptions of Proposition \ref{existS} are satisfied, and then,  there always exists  a  semispray  on ${\cal A}^*$.  Regular Hamiltonian vector fields on ${\cal A}$ are particular  semisprays  in this context (for a complete description see \cite{LPo2}). Again when ${\cal M}$ is an open submanifold of ${\cal A}^*$ such that $\pi=\t^*_{| {\cal M}}:{\cal M}\ap M$ is a fibration, such structures  can be induced from global ones on ${\cal A}^*$.\\
 \end{Ex}

\begin{Ex}\label{JOS6}${}$\\
According to Example \ref{C2}, consider the fibration $\pi :\mathbb{S}^{h-1}\times\mathbb{S}^{2n+1}\ap \mathbb{S}^{2n}$. We have seen that there exists an Euler vector field $\cal C$ on this fibration. On the other hand consider the Torus action $\psi:\mathbb{T}^n\times \mathbb{CP}^n\ap \mathbb{CP}^n$. Then classically we get a Lie algebroid structure $({\cal A}=\mathbb{CP}^n\times \R^n, \Psi,\mathbb{CP}^n,[\,;\;])$.
Then,  for $h=n$  the manifold ${\cal M}=\mathbb{S}^{n-1}\times\mathbb{S}^{2n+1}$ is  a fibered manifold $\pi:{\cal M}\ap \mathbb{S}^{2n}$. Then    we can consider   prolongation
 $\pi^1:{\T}{\cal M}\ap{\cal M}$ of  the previous Lie algebroid.
Since  $\cal A$ is trivial, we have a global basis $\{{\cal X}_1, \cdots,{\cal X}_{n},{\cal V}_1,\cdots,{\cal V}_{n}\}$ (see also Example \ref{C2}). Thus the endomorphism ${\cal J}$ of ${\T}{\cal M}$ defined by ${\cal J}{\cal X}_\a={\cal V}_\a$ and ${\cal J}{\cal V}_\a=0$ is an almost tangent structure and according to (\ref{compChpf}), $\cal J$ is compatible with $\cal C$. Therefore, according to Point (ii) of Proposition \ref{JS}  the assumptions of Proposition \ref{existS} are satisfied,  and then, there always exists  a   semispray on $\cal M$.  Note that on   the usual tangent bundle ${ T}(\mathbb{S}^{2n-1}\times\mathbb{S}^{2n+1})$  we also have an almost tangent structure and an Euler vector field  and so $\mathbb{S}^{2n-1}\times\mathbb{S}^{2n+1}$ is  an almost tangent manifold in the sense of \cite{Va3}.\\
In fact, this construction can be extended to a larger context:\\  Let  $\hat{\pi}:\hat{M}\ap M$ be a $\mathbb{S}^1$-bundle over $M$ and assume that we have a right action of a Lie group $G$ on $M$.  Thus we get a Lie algebroid  $(M\times \cal G,\Psi, M, [\;,\;])$  where  $\cal G$ is the Lie algebra of $G$. Then we also have  a fibered manifold $\pi:{\cal M}=\mathbb{S}^{h-1}\times \hat{M}\ap M$ ($h=$dim $\cal G$) and on the prolongation ${\T}{\cal M}$ of  the previous algebroid, we have an Euler section and an almost tangent structure.\\
\end{Ex}

\bigskip
We now end this subsection by the proof of the previous Propositions.

\begin{proof}[Proof of Proposition \ref{JS}]${}$\\
 Denote by $[\V{\cal M}]_m$ and $[\T{\cal M}]_m$ the respective fiber of $\V{\cal M}$ and $\T{\cal M}$ over $m\in {\cal M} $. According  to the proof of Lemma \ref{dirfan} the correspondence ${\cal V}_A\ap  -Cl([{\cal S},{\cal V}_A]_{\cal P})$ from $[\V{\cal M}]_m$  to $[\T{\cal M}]_m/ [\V{\cal M}]_m$, where $Cl(\;)$ denotes the equivalent class in the quotient space, is well defined. If we choose another   basis $\{\bar{\cal X}_\a,\bar{\cal V}_\b \}$, we have a field $(\L_\b^\g)$ of invertible matrices such that $\bar{\cal V}_\b=\L_\b^\g{\cal V}_\g$ and we have

    $[{\cal S}, \bar{\cal V}_\b]_{\cal P}=d \L_\b^\b(\hat{\rho}({\cal S})){\cal V}_\g +\L_\b^\g[{\cal S},{\cal V}_\a]_{\cal P}.$

 \noindent This implies  that   we can define a linear map $\hat{\cal J}_m:[\T{\cal M}]_m\ap[ \T{\cal M}]_m/[\V{\cal M}]_m$  by
 $$\hat{\cal J}(m)({\cal V}_\a)=-Cl([{\cal S},{\cal V}_\a]_{\cal P}).$$
  This map gives rise to an isomorphism $\hat{\cal J}$ between  $\V{\cal M}$  and  $\T{\cal M}/\V{\cal M}$.
     If $\varpi :V{\cal M}\ap \T{\cal M}/\V{\cal M}$    is the natural projection we set ${\cal J}=\varpi\circ \hat{\cal J}^{-1}$. By construction, ${\cal J}$ is an almost tangent structure. \\

    \noindent Recall that
 ${\cal L}^{\cal P}_{\cal S}{\cal J}({\cal V})=-{\cal J}[{\cal S},{\cal V}]_{\cal P}$ for all vertical sections $\cal V$. Therefore, from the construction of ${\cal J}$, we then have: \\ ${\cal J}[{\cal S}, {\cal V}]_{\cal P}=-{\cal V}=- {\cal L}^{\cal P}_{\cal S}{\cal J}({\cal V})$.\\
 Now,  let ${\cal J}'$ be another almost tangent structure such that
 ${\cal L}^{\cal P}_{\cal S}{\cal J}'({\cal V})={\cal V}$ for all vertical sections ${\cal V}$. Consider   any  basis $\{{\cal V}_\a\}$ of $\V{\cal M}$ associated with a coordinate system $(x^i,y^\a)$. Then if we set ${\cal Y}_\a=-[{\cal S},{\cal V}_\a]_{\cal P}$, then $\{{\cal Y}_A,{\cal V}_B\}$ is a basis of $\T{\cal M}$ and we have ${\cal J}{\cal Y}_\a={\cal V}_\a$.
 From our assumption on ${\cal J}'$, we have
 $$ {\cal V}_\a={\cal L}^{\cal P}_{\cal S}{\cal J}'({\cal V}_\a)=-{\cal J}'[{\cal S},{\cal V}_\a]_{\cal P}={\cal J}'{\cal Y}_\a,$$

 and then  ${\cal J}$ is unique.\\

 Now consider ${\cal S}'={\cal S}+{\cal V}$.  In an adapted basis $\{{\cal Y}_\a,{\cal V}_\b\}$, we have
 $$[{\cal S}',{\cal V}_\b]_{\cal P}=[{\cal S},{\cal V}_\b]_{\cal P}+[{\cal V},{\cal V}_\b]_{\cal P}$$
 But $[{\cal V},{\cal V}_\b]_{\cal P}$ is always vertical, so, in ${\T}{\cal M}/V{\T}{\cal M}$ we have
 $$\varpi [{\cal S}',{\cal V}_\b]_{\cal P}=\varpi [{\cal S},{\cal V}_\b]_{\cal P}.$$
 It follows that ${\cal S}'$ is also a  semispray  which has the announced properties.\\

 Finally, consider ${\cal S}{''}=(f\circ \pi){\cal S}$. In a local  basis $\{{\cal X}_\a,{\cal V}_\b\}$
we then have:
$$ [{\cal S}^{''},{\cal V}_\b]_{\cal P}= -(\dis\frac{\p (f\circ \pi)}{\p y^\b}){\cal S}+(f\circ \pi )[{\cal S},{\cal V}_\b]_{\cal P}.$$
 But, from our  assumption, $\dis\frac{\p (f\circ \pi)}{\p y^\b}=0$. As $f$ is a nonzero function,  it follows that ${\cal S}^{''}$ is also a  semispray  which has the announced properties.\\
This ends the proof of Part  (i).\\

  From now on, we consider  the context of  the assumptions of part (ii).\\
  Fix a local basis $\{{\cal X}_\a,{\cal V}_\b\}$   associated with some coordinate system $(x^i,y^\b)$ and a local basis $\{e_\a\}$ of $\cal A$ around some point $m\in \cal M$. In this basis, $\cal S$ can be written  ${\cal S}=S^\a{\cal X}_\a+{\cal S}^\b{\cal V}_\b$
  and then, from the construction of  $\cal J$,  we get
    $${\cal J}[{\cal S},{\cal V}_\b]_{\cal P}=-{\cal J}(\dis\frac{\p S^\a}{\p y^\b}{\cal X}_\a)=-{\cal V}_\b.$$
 Therefore, the matrix     $(\dis\frac{\p S^\a}{\p y^\b})$ is of maximal rank $k$. Therefore  there exists a coordinate system $(\tilde{x}^i=x^i, \tilde{y}^\a)$ (compatible with $\pi$),   and if $\Phi$ is the corresponding diffeomorphism, we have
 $(S^1\circ \Phi^{-1},\cdots,S^k\circ \Phi^{-1})=(\tilde{y}^1,\cdots,\tilde{y}^k)$.
  \noindent If  $\{\tilde{\cal X}_\b,\tilde{\cal V}_B\}$ is  the canonical basis associated with $(\tilde{x}^i,\tilde{y}^A)$ and the same basis $\{e_\a\}$ of $\cal A$, we have a decomposition
  ${\cal S}=\tilde{y}^\a\tilde{\cal X }_a+\tilde{\cal S}^\b\tilde{\cal V}_\b.$
  But according to the construction of $\cal J$, in these coordinates we have

  (i) ${\cal J}[{\cal S},\tilde{\cal V}_\b]_{\cal P}=-{\cal J}(\tilde{\cal X}_\b)=-\tilde{\cal V}_\b.$

 Therefore we obtain

  (ii) ${\cal C}={\cal J}{\cal S}=\tilde{y}^\b\tilde{\cal V}_\b.$

  It follows that $\cal J$ is integrable.
On the other hand, recall that :
\begin{eqnarray}\label{LCP}
{\cal L}^{\cal P}_{\cal C}{\cal J}({\cal X})=[{\cal C},{\cal J}{\cal X}]_{\cal P}-{\cal J}[{\cal C},{\cal X}]_{\cal P}.
\end{eqnarray}
From this expression we obtain  ${\cal L}^{\cal P}_{\cal C}\tilde{\cal J}(\tilde{\cal V}_\b)=0$ and  according to (i) and (ii), we have  $[{\cal C}, \tilde{\cal X}_\a]_{\cal P}=0$. Finally we obtain \\ ${\cal L}^{\cal P}_{\cal C}{\cal J}(\tilde{\cal X}_\a)=-\tilde{\cal V}_\a=-{\cal J}\tilde{\cal X}_\a$. This ends the proof of part (ii)\\
    \end{proof}

     \begin{proof}[Proof of Proposition \ref{existS}]${}$\\
    Consider a coordinate system $(x^i, y^\a)$ and an associated canonical  basis $\{{\cal X}_\a,{\cal V_\b}\}$ of $\T{\cal M}$ such that ${\cal C}=y^\a{\cal V}_\a$. We have
${\cal J}{\cal X}_\a=J_\a^\b{\cal V}_\b$. Since  the matrix $(J_\a^\b)$ is invertible, we can build a local basis $\{{\cal Y}_\a,{\cal V}_\b\}$ of $\T{\cal M}$  such that
${\cal J}{\cal Y}_\a={\cal V}_\a$.
As  ${\cal J}$ is compatible with $\cal C$, we have
   $[{\cal C},{\cal V}_\a]_{\cal P}-{\cal J}[{\cal C},{\cal Y}_\a]_{\cal P}=-{\cal V}_\a$ ({\it cf.} (\ref{LCP})).
 Therefore, we must have ${\cal J}[{\cal C},{\cal Y}_\a]_{\cal P}\equiv 0$ for all $\a=1,\cdots,k$.

\noindent Now in the previous local canonical  basis $\{{\cal X}_\a,{\cal V}_\b\}$ of ${\T}{\cal M}$   we have a  decomposition of type ${\cal Y}_\a=Y_\a^\b{\cal X}_\b$. It follows that  ${\cal J}[{\cal C},{\cal Y}_\a]_{\cal P}=y^\b\dis\frac{\p Y_\a^\g}{\p y^\b}J\g^\mu{\cal V}_\mu$. As  the matrix $(J_\g^\mu)$ is invertible, it follows that
 we obtain $y^\b\dis\frac{\p Y_\a^\g}{\p y^\b}\equiv 0$ for all $\a,\g=1,\cdots,k$ , on some neighbourhood of $m$.  This implies  that, each function $Y_\a^\g$  only depends on $x$ for all $\a=1\cdots, k$ and $\g=1,\cdots k$. \\

   Consider $\bar{e}_\a=Y_\a^\g e_\g$ for $\a=1,\cdots,k$. Since  the rank of the matrix $(Y_\a^\g)$ is $k$, we get a new basis $\{\bar{e}_\a\}$ of $\cal A$ and an associated canonical basis $\{\bar{\cal X}_\a,{\cal V}_\b\}$ of $\T{\cal M}$    such that ${\cal C}=y^\b{\cal V}_\b$. Moreover, by construction, we have ${\cal Y}_\a=\bar{\cal X}_\a$ so ${\cal J}\bar{\cal X}_\a={\cal V}_\b$. Therefore $ \cal J$ is integrable.

       Consider a subbundle ${\bf H}$  such that $\T{\cal M}=\V{\cal M}\oplus{\bf H}$ then, the restriction ${\cal J}_{\bf{ H}}$ of ${\cal J}$ to ${ \bf H}$ is an isomorphism. Consider ${\cal S}={\cal J}_{\bf{ H}}^{-1}{\cal C}$. By construction, we have ${\cal J}{\cal S}={\cal C}$.
    We have already proved that $\cal J$ is integrable, and, there exists a  canonical  basis $\{{\cal X}_\a,{\cal V}_\b\}$ of $\T{\cal M}$  such that ${\cal J}{\cal X}_\a={\cal V}_\a$  and ${\cal C}={y}^\a{\cal V}_\a$. Therefore in this basis we have :

${\cal S}={y}^\a{\cal X}_\a+{\cal S}^\b{\cal V}_\b$

Recall that $[{\cal X}_\a,{\cal V}_\b]_{\cal P}=0$.
Thus we obtain
\begin{eqnarray}\label{bracSV}
[{\cal S},{\cal V}_\g]_{\cal P}=-{\cal X}_\g-\dis\frac{\p {\cal S}^\b}{\p {y}^\b}{\cal V}_\b.
\end{eqnarray}
      \noindent  It follows that  $\cal S$ is a semispray. Moreover, we have
      ${\cal L}^{\cal P}_{\cal S}{\cal J}({\cal V}_\a)=-{\cal J}[{\cal S},{\cal V}_\a]_P={\cal J}{\cal V}_\a$. This implies that
       $\cal J$ is compatible with $\cal S$.

       The proof of the last part is left to the reader.
      \end{proof}

\begin{Rem}\label{locC}${}$\\
 From the proof of Proposition \ref{existS}, around any $m\in {\cal M}$, there exist coordinates $(x^i, y^\a)$ on $\cal M$ compatible with $\pi$ and a canonical  basis $\{{\cal X}_\a,{\cal V}_\b\}$ of $\T{ \cal M}$ such that ${\cal J}{\cal X}_\a={\cal V}_\a$ and   ${\cal C}=y^\a{\cal V}_\a$, and
each   semispray  ${\cal S}$ such that ${\cal J}{\cal S}={\cal C}$ can be written ${\cal S}=y^\a{\cal X}_\a+{\cal S}^\b{\cal V}_\b$.\\
  In  particular, in such a coordinate system, $C=\hat{\rho}({\cal C})$ generates locally a flow of homotheties on each fiber. Therefore the set of such charts when $m$ varies in $\cal M$ defines  a locally linear fibration structure  on $\cal M$   whose associated atlas is a sub-atlas of the atlas  associated with the affine structure on $M$ corresponding  to the existence of the almost tangent structure $\cal J$ (cf. Proposition \ref{vertaff} and Remark \ref{affJO}).\\
  Moreover the  atlas associated with  locally  linear fibration defined by the Euler vector field $C=\hat{\rho}({\cal C})$ is a sub-atlas  of  the atlas of the $\cal J$-locally affine fibration (cf Remark \ref{atalsJ}). Note that, when $\cal M$ is an open set of $\cal A$ such that $\pi_{| {\cal M}}$ is a fibration, the previous locally linear fibration is exactly the differentiable structure induced  by the natural differential structure of vector bundle of   $\cal A$.
  \end{Rem}

%%%%%%%%%%%%%%%%%%%%%%%%%%%%%%%%%%%%%%%%%%%%%%%%%%%%%%%%%%%%%%%%%%%%%%%%%%%%%
\subsection{Semispray  and almost cotangent  structure}${}\label{SOC}$\\
%%%%%%%%%%%%%%%%%%%%%%%%%%%%%%%%%%%%%%%%%%%%%%%%%%%%%%%%%%%%%%%%%%%%%%%%%%%%%
%%%%%%%%%%%%%%%%%%%%%%%%%%%%%%%%%%%%%%%%%%%%%%%%%%%%%%%%%%%%%%%%%%%%%%%%%%%%%
${}\;\;\;\;${\it   We  will introduce the notion of "regular semi-Hamiltonian"   which will appear as the dual notion of semispray  via an almost cotangent structure.} \\

In this subsection  we assume that there exists a locally affine fibration structure  on $\cal M$.
For instance, if there exists  an almost tangent  structure we can choose the associated  locally  affine structure  or the locally  affine structure  associated with a cotangent structure  on a  ${\T}{\cal M}$ ({\it cf.} Proposition \ref{TO} or Remark \ref{affJO}).
According to some fixed  locally affine fibration structure,  for any two coordinate systems $(x^i,y^\a)$ and  $(\bar{x}^i,\bar{y}^\b)$ on $\cal M$  whose domain are not disjoint, we have
\begin{eqnarray}\label{chgtaffine}
\bar{y}^\a=\Phi^\a_\b(x)y^\b+\phi^\b(x).
\end{eqnarray}

\begin{Lem}\label{Heta} ${}$\\ Let $\eta$ be  a $1$-form on  ${\T}{\cal M }$ such that $(d^{\cal P}\eta)_{| \V{\cal M}}=0$.  Around each point $m\in \cal M$ there exists a function $L$ such that $(d{L})_{| \V{\cal M}}=\eta_{| \V{\cal M}}$.
Moreover, the rank of the hessian $\dis\frac{\p^2 {L}}{\p y^\a \p y^\b}(m)$ is independent of the choice of such a function $L$ and the choice of the coordinate system compatible with the chosen  locally  affine fibration structure.
\end{Lem}

\begin{proof}${}$\\
Choose a local  coordinate system $(x^i,y^\a)$ around some point $m\in {\cal M}$ compatible with
the chosen   locally  affine fibration structure,
a basis $\{e_\a\}$ of $\cal A$, and let  $\{{\cal X}^\a,{\cal V}^\b\}$ be the associated dual basis. Then we have a local decomposition $\eta=\eta_\a{\cal X}^\a+\bar{\eta}_\b{\cal V}^\b$.
According to (\ref{locdifalg})  we get
 $$(d^{\cal P}\eta)_{| \V{\cal M}} =\dis\frac{1}{2}(\frac{\p\bar{ \eta_C}}{\p y^\b}-\dis\frac{\p \bar{ \eta}_\b}{\p y^\g}){\cal V}^\b\wedge{\cal V}^\g.$$
%First of all  it follows that  $(d^{\cal P}\eta)_{| \V{\cal M}}$ is independent of the choice of the projection $\cal P$.
 The condition $(d^{\cal P}\eta)_{| \V{\cal M}}=0$ is equivalent to $\dis\frac{\p\bar{ \eta_\b}}{\p y^\g}=\dis\frac{\p\bar{ \eta_\g}}{\p y^\b}$ for any index $\b,\g$. Therefore there exists a smooth function $L$ on a neighbourhood of $m$ such that $(d^{\cal P}{L})_{|  \V{\cal M}}=\eta_{| \V{\cal M}}$. Now if ${L}'$ has the same property, then in some neighbourhood of $m$, the function ${L}-{L}'$ is a basic function (\textit{i.e.} $L$ depends only  on the variables $x^i$).
  Then the last property is clear.\\
\end{proof}

On the tangent bundle, in \cite{Mi}, G. Mitric introduces the concept of "regular $1$-form" which was generalized  by
the notion of semi-Hamiltonian on the dual of a Lie algebroid introduced in \cite{LPo1}. This   motivates the following definition:

\begin{Def}\label{coS}${}$\\
 Given a locally affine fibration on $\cal M$, a $1$-form $\eta$ on  ${\T}{\cal M }$ will be called  a  {\it regular  semi-Hamiltonian } if $(d^{\cal P}\eta)_{| \V{\cal M}}=0$,
  and for any $m\in {\cal M}$, if  $L$ is a function such that  $(d^{\cal P}{L})_{| \V{\cal M}}=\eta_{| \V{\cal M}}$  around $m$, then the rank  of  the hessian $\dis\frac{\p^2 {L}}{\p y^\a \p y^\b}(m)$ is maximal in a coordinate system $(x^i,y^\a)$ compatible with the fixed locally affine fibration structure.
\end{Def}

We have seen that  the compatibility of $\cal J$ and $\cal C$ is crucial in the existence of semispray $\cal S$ such that ${\cal J}{\cal S}={\cal C}$.  In a dual way, if $\cal A$ is a Lie algebroid, on ${\T}{\cal A}^*$,  denote by $\theta$ the Liouville form and   $\O=-d^{\cal P}\theta$ the canonical symplectic form.   Then we have $\theta=- i_{\cal C}\O$ if $\cal C$ is the Euler vector field on ${\T}{\cal A}^*$  and in this case we have
${\cal L}_{\cal C}\O=-\O$ (see \cite{LPo1}). In fact we  have the following characterization

 ${\cal L}_{\cal C}\O=-\O\Longleftrightarrow  $ $\theta=- i_{\cal C}\O$ and $\O+d^{\cal P}\theta$ is semi-basic.

Unfortunately, in general, this equivalence is no longer  true when the Jacobi identity is not satisfied (\textit{i.e.} $d^{\cal P}\circ d^{\cal P}\not=0$). Therefore in our context,   for a notion of "compatibility" of an almost cotangent structure $\O$ and an Euler section,  we  adopt  a generalization of  the notion of "Liouville manifold"  defined in  \cite{Va1} which corresponds to the right properties of the previous equivalence as is given in the  Lemma \ref{equivcomp}.

\begin{Def}\label{liouv}${}$\\
 If  $\O$  is  an almost cotangent structure on $\T{\cal M}$,  then $\O$  is of Liouville type if  there exists a semi basic $1$-form $\theta$ on ${\T}{\cal M}$, whose kernel contains the kernel of ${\O}$ and   such that
  ${\O}+d^{\cal P}\theta$ is  semi-basic.   In this case $\theta$ will be called an {\bf almost Liouville form} associated with $\O$.
 \end{Def}

 The denomination " almost Liouville form" for  such $\theta$ is justified by the Part (ii) of  the following Proposition:

\begin{Pro}\label{equivcomp}${}$\\
Let  $\O$   be an almost cotangent structure on $\T{\cal M}$.
\begin{enumerate}
\item[(i)] In Definition \ref{liouv}, the condition "$\tilde{\O}+d^{\cal P}\theta$ is   semi-basic" only depends  on the choice of the bracket $[\;,\;]_{\cal A}$. Moreover $\theta$ and $\theta'$ are  two almost Liouville forms associated with  $\O$,  if and only if  there exists a global  section $\s$ of ${\cal A}^*$ such that $\theta'=\theta+\pi_{\cal A}^*\s$.
\item[(ii)] Assume that  $\O$ is  of Liouville type and $\eta$ be an almost Liouville form associated with $\O$. Then  there exists a unique  Euler section $\cal C$ such that  $\theta=- i_{\cal C}\O$. Moreover
 around each point $m\in {\cal M}$, there exists a coordinate system $(x^i,y^\a)$ compatible with the locally linear fibration  structure associated with $\cal C$ and an associated  canonical basis $\{{\cal X}_\a,{\cal V}_\b\}$ such that ${\cal C}=y^\a{\cal V}_\a$ and $\theta=y^\a{\cal X}^\a$ in the corresponding dual basis $\{{\cal X}^\a,{\cal V}^\b\}$.
\item[(iii)] if $\O=-d^{{\cal P}}\theta$ is any exact almost cotangent structure on $\T{\cal M}$   then $\O$ is  of Liouville type
\end{enumerate}
\end{Pro}

The proof of this Proposition will be given at the end of this paragraph.\\

\begin{Rem}\label{affO}${}$\\
Given an almost Liouville form $\eta$ associated with an almost cotangent structure $\O$ on $\T{\cal M}$, any $\theta'=\theta+\pi_{\cal A}^*(\s)$ is an almost Liouville form associated with $\O$. Now  if ${\cal C}$ and ${\cal C}'$ are the Euler sections associated with $\theta$  and $\theta'$ then ${\cal Z}={\cal C}'-{\cal C}$ is a vertical section such that  in any coordinate system $(x^i,y^\a)$ compatible with the locally linear fibration structure  defined by $\cal C$ we have:

$${\cal Z}=Z^\b(x){\cal V}_\b.$$

Therefore  all these Euler sections gives rise to  a natural sub-atlas of    locally affine fibration. This sub-atlas, will be called {\bf the locally affine fibration structure associated with $\O$}
\end{Rem}

Now, we are in situation  to give a link between   semisprays  and   regular semi-Hamiltonians in the following result.

\begin{The}\label{S-coS}${}$

\begin{enumerate}
\item[(i)]  Given an almost cotangent structure $\O$  on $\T{\cal M}$  and  $\eta$  a regular  semi-Hamiltonian,  we have the following properties:

\noindent  (a) there exists a canonical pseudo-Riemannian metric $g$ on $\V{\cal M}$ associated with $\eta$ whose locally matrix  is $(\dis\frac{\p^2 {L}}{\p y^\a \p y^\b})$  in a local basis $\{{\cal V}_\a\}$ of $V{\cal M}$  associated with some coordinate system $(x^i,y^\a)$ compatible with the chosen locally affine fibration structure.

\noindent (b) There exists a unique almost tangent structure ${\cal J}$ on ${\cal M}$ compatible with $\O$ and such that $g({\cal J},\;)=\O(\;,\;)$

\noindent (c)   The equation $i_{\cal S}\O=\eta$ defines a unique section of $\T{\cal M}$. Moreover,
on $\T{\cal M}$, there exists a   basis $\{{\cal Y}_{\a},{\cal V}_{\b}\}$ and a coordinate system $(x^i,y^\a)$ such that if $\{{\cal X}_\a,{\cal V}_\b\}$ is an associated canonical basis we have:

\noindent${\cal Y}_\a=Y_\a^\b{\cal X}^\b\;\;\;$ ${\cal J}{\cal Y}_{\a}={\cal V}_{\a}\;\;\;$ $\;\;\;{\cal S}=y^{\a}{\cal Y}_{\a}+{\cal S}^{\b}{\cal V}_{\b}\;\; \;\textrm { and }\;\;\; {\cal C}_\cal S= {\cal J}{\cal S}=y^{\a}{\cal V}_{\a}$.\\
\noindent In particular ${\cal C}_{\cal S}$ is an Euler section.
%\item[(ii)] Assume that   $\cal J$ is compatible with ${\cal C}_\cal S={\cal J}{\cal S}$  and moreover  assume that   and any bracket $[{\cal V},{\cal Y}]_{\cal P}$ is a section of $\cal E$ for  vertical section $\cal V$ and any section $\cal Y$ of $\cal E$. Then $\cal S$ is a almost  semispray and all the conclusions of Proposition \ref{existS} are true.  Moreover
% $\O$ is  of  Liouville type on $\cal E$ considered as $\T^E{\cal M}$ (cf Proposition \ref{subbA}).
\item[(ii)] Assume that    $\O$ is an almost  cotangent  structure of Liouville type on $\T{\cal M}$. Consider a regular semi Hamiltonian $\eta$ relative to the locally affine fibration structure assocated to $\O$ (see Remark \ref{affO}).Then,   according to the corresponding   data associated  to $\O$ and $\eta$ as build in Part (i), the almost tangent structure ${\cal J}$ is integrable and $\cal S$ is an almost  semispray.  Moreover,  $\cal J$ is compatible with  $\cal S$ and the Euler section ${\cal J}{\cal S}$. Finally, around each point $m\in {\cal M}$  there exists  coordinate system $(x^i,y^\a)$  (compatible with the chosen locally affine   fibration structure)  and an associated  canonical basis $\{{\cal X}_\a,{\cal V}_{\b}\}$ of $\T{\cal M}$  with the same properties of (c) .
\end{enumerate}
\end{The}

\bigskip
 \begin{Rem}\label{afflinO}${}$
\begin{enumerate}
\item From the property (c) of Part  (i) in Theorem \ref{S-coS}, and  according to Remark \ref{locC},  the Euler section ${\cal C}_\cal S$ gives rise to a locally  linear fibration which is compatible with the locally affine fibration associated to $\cal J$. However this last structure is not,   in general, compatible with the   chosen initial locally affine fibration structure. Again the link between these two structures is given by  $g_\eta$, as it is explained in Remark \ref{affJO}.
For example,  when $\cal M$ is the dual bundle ${\cal A}^*$ of $\cal A$,  we have on $\cal M$ the natural structure of vector bundle. This structure corresponds to the "locally affine fibration" associated with the canonical  Cartan form on ${\T}{\cal A}^*$.  The notion of regular semi-Hamiltonian on ${\cal A}^*$ is based on this structure (see Example \ref{exsmH} 1. or  \cite{LPo1} for  a more complete description).\\
 Moreover, in general the basis $\{{\cal Y}_\a,{\cal V}_\a\}$ built  is {\it not a canonical basis} of $\T{\cal M}$ associated with some coordinate system $(x^i,y^\a)$ on $\cal M$. Then,  this local decomposition does not  give  good "informations" on rank of the family  $[{\cal S},{\cal V}_A]_{\cal P}$,  although $\theta$ is a regular semi-Hamiltonian. Therefore, in general, without complementaries properties on $\O$, $\cal S$ will not be a semispray.  \\
\item In Part (ii) of Theorem  \ref{S-coS},  the almost tangent structure ${\cal J}$ is compatible with the Euler section ${\cal C}_{\cal S}={\cal J}{\cal S}$. Given some almost Liouville form $\theta$ associated with $\O$, we denote by ${\cal C}_\theta$ the Euler section defined by $i_{\cal C_\theta}\O=\theta$.  In general $\cal J$ is not compatible with the given Euler section ${\cal C}$. Again, in this context,  on $\cal M$ we have two distinct locally affine  fibration structures: one which is associated with $\O$ and the other which is associated with ${\cal C}_{\cal S}$. Again  the link between these  two structures is given by  $g_\eta$. %On the other hand,  according to Remark \ref{indcrochet},  the context of these results of oint (ii) are in accordance with the assumptions of Proposition \ref{existS}.
\end{enumerate}
 \end{Rem}

\bigskip

\begin{Exs}\label{exsmH}${}$\\
{\bf 1.}   If $\theta$ is the Liouville $1$-form on   ${\T}{\cal A}^*$  then the canonical Cartan form $\O=-d^{\cal P} \theta$ is a global exact almost structure which is of Liouville type. Given any open submanifold ${\cal M}$ of ${\cal A}^*$  fibered on $M$, then $\O$ in restriction to $\T{\cal M}$ has the same properties. Consider any regular Hamiltonian ${\cal H}$ on ${\cal M}$ \textit{i.e.} the vertical hessian $(\dis\frac{\p^2 {\cal H}}{\p y^{\a}\p y^{\b}})$, with respect to the canonical bundle structure on ${\cal A}^*$, has maximal rank $k$. Then $\cal H$ is clearly a regular   semi-Hamiltonian according to Definition \ref{coS}. Thus we get  an associated almost  semispray   $\cal S$ solution of the equation $i_{\cal S}\O=d^{\cal P}{\cal H}$ usually called the associated Hamiltonian section. \\
{\bf 2.} Let  ${\cal M}$ be an open submanifold of $\cal A$ fibered on $M$ and provided with the locally  affine fibration induced by the structure of vector bundle of $\cal A$. Consider a regular Lagrangian $L$ on ${\cal M}$ (\textit{i.e.} the hessian $(\dis\frac{\p^2 {\cal L}}{\p y^{\a}\p y^{\b}})$, with respect to  the previous  locally  affine fibration has maximal rank $k$). As classically, if ${ J}$ and  ${\cal C}$ are  the canonical almost tangent structure and Euler section on ${\T}{\cal M}$ respectively, we get an induced almost tangent structure and  almost  Euler section on ${\T}{\cal M}$ again denoted ${\cal J}$ and  ${\cal C}$ respectively. Then  $\O=-d^{\cal P}{\cal J}^*(d^{\cal P}L)$ is an  almost cotangent structure on ${\T}{\cal M}$ of Liouville type, and which is compatible with $\cal J$. The canonical almost  semispray  $S$ associated with the Hamiltonian ${\cal L}^{\cal P}_{\cal C}(L)-L$ gives rise to the associated Euler-Lagrange equations. \\

    {\bf 3.} Consider a subbundle   $\cal M$ of $\cal A$    as in the  context of Example \ref{exJOS2}. We endow $\cal M$ with the   linear structure induced by the natural bundle structure on $\cal M$.   Fix some Riemannian metric $g$ on ${\T}{\cal M}$. Then we get on ${\T}^{\cal M}{\cal M}\subset \T{\cal M}$ a well defined bracket $[\;,\;]_{\cal P}$  from the bracket $[\;,\;]_{\cal P}$ on ${\T}{\cal M}$. We consider a regular Lagrangian $ L$ on $\cal A$ . If $J$ is the vertical endomorphism, consider  $\O=-d^{\cal P}J^* (d^{\cal P}L)$  any exact  almost cotangent  structure on ${\T}{\cal A}$ which is compatible  with the canonical  almost structure $J$. Assume that the induced $2$-form $\O$ is symplectic (which occurs if $\O$ comes from a regular Lagrangian on $\cal A$). Then $\O$ is an exact almost cotangent structure on $\T{\cal M}$. Therefore all the conclusions of Theorem \ref{S-coS} are valid in this context. Note that, when we have a regular Lagrangian $L$ on $\cal A$ whose hessian  is positive definite, we have a canonical Riemannian metric on ${\T}{\cal M}$  and all the previous assumptions are satisfied in this case.\\ %(see Example  \ref{admisconstraint}).\\

\end{Exs}

Now assume that we have an integrable almost tangent structure $\cal J$ on ${\T}{\cal M}$ and an Euler section $\cal C$. From now on, we choose the locally linear fibration structure on $\cal M$ associated with this choice (see Remark \ref{locC}).  We end this subsection by a particular case of almost cotangent structure of Liouville type:
\begin{Def}\label{semiham}${}$\\
An almost cotangent structure $\O$ on ${\T}{\cal M}$ is {\it   semi-Hamiltonian} if there exists a regular  semi-Hamiltonian $\theta$ (according to the locally affine fibration structure)  such that
$$\O=-d^{\cal P}{\cal J}^*(\theta).$$
\end{Def}
 Let $L$ be a locally smooth function such that $d^{\cal P}{L}_{| \V{\cal M}}=\theta_{| \V{\cal M}}$. According to the choice of the  locally affine fibration structure,  we have a local basis $\{{\cal X}_\a,{\cal V}_\b\}$   such that:
\begin{eqnarray}\label{semihamO}
\begin{cases}
 {\cal J}{\cal X}_\a={\cal V}_\a, \textrm{ and }\hfill{}\cr
 \O=\dis\frac{\p^2 {L}}{\p y^\a \p y^\b}{\cal X}^\a\wedge{\cal V}^\b+\dis\frac{1}{2}(\frac{\p ^2{L}}{\p x^i \p y^\a}\rho_\b^i-\frac{\p ^2{L}}{\p x^i\p y^\b}\rho_\a^i+\frac{\p{L}}{\p y^\g} C_{\a\b}^\g){\cal X}^\a\wedge{\cal X}^\b.\cr
 \end{cases}
\end{eqnarray}
It follows that ${\cal J}^*\theta=\dis\frac{\p {L}}{\p y^\a }{\cal X}^\a$ is an almost Liouville form. In particular, we get an Euler section  ${\cal C}_\theta$ associated with $\theta$ via $\O$ but  ${\cal C}\not={\cal C}_\theta$  ({\it cf.} Remark \ref{OhamC}).
In Examples \ref{exsmH}, {\bf 1.} and  {\bf 2.},  the almost cotangent structures  are semi-Hamiltonian (see  also Example \ref{metricLagS})
\bigskip
\begin{Pro}\label{propsemiH}${}$\\
Assume that $\O$ is semi-Hamiltonian. Then we have:
\begin{enumerate}
\item[(i)]   ${\cal J}$ is compatible with $\O$,   we have   $g_\theta({\cal J},\;)=\O(\;,\;)$ and  $({\cal M},{\cal J},\O)$ is locally Lagrangian.  Therefore around any point $m\in{\cal M}$  we can  write
$\O=-d^{\cal P}{\cal J}^*d^{\cal P}L$;
\item[(ii)] Let  $\cal S$ be the section  defined by $i_{\cal S}\O={\cal L}^{\cal P}_{\cal C}\theta- \theta$. Then $\cal S$ is an almost semispray, we have  ${\cal J}{\cal S}={\cal C}$ and $\cal J$ is compatible with  $\cal S$ and with $\cal C$.
\item[(iii)]  Given any almost  semispray  $\cal S$ compatible with $\cal J$, then the $1$-form $i_{\cal S}\O$ is a semi-Hamitonian.
\end{enumerate}
\end{Pro}
\bigskip
\begin{Rem}\label{OhamC}${}$ \\
Before the Proposition \ref{propsemiH}, we have seen that if  $\O$ is semi-Hamitonian we have an  Euler section ${\cal C}_\theta$  associated with $\O$. According to the basis used in the previous proof and the local expression of $\O$, we have ${\cal C} _\theta=\dis\frac{\p L }{\p y^\a }{\cal X}_\a$. It follows that  ${\cal C}_\theta$ is  not equal to $\cal C$. In particular, the locally linear fibration structure defined by  ${\cal C}_\theta$ is not compatible with the locally linear fibration structure defined by  ${\cal C}$. Again, the link between the two structures is given by the pseudo-Riemannian metric  $g_\theta$ associated with $\theta$ such that $\O=-d^{\cal P}({\cal J}^*\theta)$ (cf. Remark \ref{affJO}).
\end{Rem}

\bigskip

\begin{proof}[Proof of Proposition \ref{equivcomp}]    ${}$\\
Consider a vertical section $\cal V$   and any section $\cal X$.  Since  $\theta$ is semi-basic, according to Remark \ref{dPsbasic}, $d^{\cal P}\theta$ only depends  on the choice of the bracket $[\;,\;]_{\cal A}$.  Therefore in Definition \ref{liouv}, the condition "$\tilde{\O}+d^{\cal P}\theta$ is   semi-basic" only depends  of  the choice of the bracket $[\;,\;]_{\cal A}$.\\

 Assume that there exists
  a semi-basic $1$-form $\theta$ such that  ${\O}+d^{\cal P}\theta$ is semi-basic.
 Choose any coordinate system $(x^i,y^\a)$ and a local basis $\{e_\a\}$ of $\cal A$ and consider the associated local basis $\{{\cal X}_\a,{\cal V}_\b\}$ of ${\T}{\cal M}$. In the corresponding dual basis $\{{\cal X}^\a,{\cal V}^\b\}$
 we have a decomposition of the type

$\theta=\theta_\a{\cal X}^\a$ and $\O=g_{\a \b}{\cal X}^\a\wedge{\cal V}^\b+\dis\frac{1}{2}\o_{\a\b}{\cal X}^\a\wedge {\cal X}^\b$

\noindent Therefore, under our assumption, we must have $g_{\a \b}= \dis\frac{\p  \eta_\a}{\p y^\b}$ and then
  the rank of the matrix $(g_{\a \b})$  is $k$. %Without loss of generality,   we can assume that the matrix $(g_{AB})$ is invertible.
%Note that given   any new system of coordinates $(\bar{x}^i,\bar{y}^A)$,  if $\{\bar{\cal X}_{A},\bar{\cal V}_{B}\}$ is the local canonical basis of ${\T}{\cal M}$ associated with $(\bar{x}^i,\bar{y}^A)$ and the previous basis $\{e_\a\}$ of $\cal A$, in the corresponding dual basis $\{\bar{\cal X}^\a,\bar{\cal V}^B\}$, we have $\bar{\cal X}^\a={\cal X}^{\a}=(\pi_{\cal A})^*\epsilon_{\a}$ (where $\{\epsilon_{\a}\}$ is the dual basis of $\{e_\a\}$).
It follows that there exists a new system of coordinates $(\bar{x}^i,\bar{y}^\a)$ so that we have
$\theta=\bar{y}^\a\bar{\cal X}^{\a}$.

\noindent In this context we get

${\O}=\bar{\cal X}^A\wedge \bar{\cal V}^\a+\dis\frac{1}{2}\bar{\o}_{\a\b}\bar{\cal X}^\a\wedge \bar{\cal X}^\b$.

\noindent Finally,  we have a unique vertical section $\cal C$  such that $i_{\cal C}\O=-\theta$ and we get

${\cal C}=\bar{y}^\a\hat{\cal X}_{\a}.$

\noindent This implies that ${\cal C}$ is an Euler section.\\

If   $\theta'$ is another almost Liouville form associated with $\O$ and set $\vartheta=\theta'-\theta$.  Therefore in the previous local basis we have a decomposition of the type:
$\vartheta=\vartheta_\a\hat{\cal X}^\a$.\\
As $d^{\cal P}\vartheta$ must be semi-basic, we must have $\dis\frac{\p\vartheta_\a}{\p \bar{y}^\b}\equiv 0$. Since  each component  $\vartheta_\a $ only depends on $x$, with the previous notations, on the chart domain, if we set $\s= \vartheta_\a{\epsilon}^\a$ where $\{{\epsilon}^\a\}$ is the dual basis associated with the chosen basis $\{e_\a\}$ of $\cal A$. Then $\pi^*_{\cal A}\s=\vartheta$. Since  $\vartheta$ is globally defined, it follows that, in this way, we can build a global section $\s$ such that $\pi_{\cal A}^*\s= \vartheta$.\\

Now, if $\O$ is exact, this means that there exists a semi-basic $1$-form $\theta$ such that $\O=-(d^{\cal P}\theta)$. Given a canonical basis $\{{\cal X}_\a,{\cal V}_\b\}$ of ${\T}{\cal M}$, in the corresponding dual basis, we can write

 $\theta=\theta_\a{\cal X}^a$ and then $d^{\cal P}\theta=\dis\frac{\p \theta_\a}{\p y^\b}{\cal V}^\b\wedge {\cal X}^{\a}+\dis\frac{1}{2}(\dis\frac{\p \theta_\b}{\p x^i}\rho_\a^i-\dis\frac{\p \theta_\a}{\p x^i}\rho_\b^i){\cal V}^\a\wedge {\cal X}^\b$.

 It follows that we obtain
 $\O+d^{\cal P}\theta=0$. Therefore $\O+d^{\cal P}\theta$ is semi-basic.

% According  the assumption on the kernel of $\theta$ and  the kernel of ${\O}$ the conditions of the Definition are satisfied.\\
\end{proof}

\begin{proof}[Proof of Theorem \ref{S-coS}]${}$\\
\indent
Choose   a local function $L$   (defined on an open $U$ of $\cal M$)  such that $d^{\cal P}{L}_{| \V{\cal M}}=\eta$. In affine coordinate systems,  according to (\ref{chgtaffine}) we have a matrix relation
$$(\dis\frac{\p^2{L}}{\p\bar{y}^\a\p\bar{y}^\b})=(\Phi^\a_\g)^t(\dis\frac{\p^2{L}}{\p{y}^\g\p y^\d})(\Phi^\a_\g).$$
 Since  the matrix $(\dis\frac{\p^2{L}}{\p{y}^\g\p y^\g})$ does not depend on the choice of such a function  $L$,   and is a symmetric matrix  of rank maximal $k$, we get a pseudo-Riemannian metric $g$ on $\V{\cal M}$.

 Denote by $g_{\eta}^\sharp$  the musical isomorphism induced by  $g$ from $ (\V{\cal M})^*$  to $\V{\cal M}$ and   $j:\V{\cal M}\ap \T{\cal M}$ the canonical inclusion. If $\O^\flat:\T{\cal M}\ap \T{\cal M}^*$ is the isomorphism associated with $\O$ and  let  ${\bf H}$  be the subbundle $\O^\flat(j^*({\V\cal M}^*)$. Then $\T{\cal M}={\bf H}\oplus \V{\cal M}$ and  the almost tangent structure ${\cal J}$ is defined by:

${\cal J}_{| \V{\cal M}}=0$ and ${\cal J}_{| {\bf  H}}=g^\sharp  \circ j^* \circ\O^\flat_{| {\bf H}}$.\\

 Consider   a canonical basis $\{{\cal X}_\a,{\cal V}_\b\}$  associated with a coordinate system compatible with the chosen  locally affine fibration structure on $\cal M$ and a choice of a basis $\{e_\a\}$  of $\cal A$.
% As in subsection \ref{not},  we consider
  %an adapted basis, $\{{\cal Y}_A,{\cal V}_B\}$ of $\cal E$  so that ${\cal Y}_A=Y_A^{\g}{\cal X}_\g$
% and let be   $\{{\cal Y}^A,{\cal V}^B\}$ the corresponding  dual basis (cf subsection \ref{not}).
 Then we have a decomposition of type:
$$\O=\dis\frac{1}{2}\O_{\a\b}{\cal X}^\a\wedge {\cal X}^\b+\bar{\O}_{\a\b}{\cal X}^\a\wedge {\cal V}^\b.$$

\noindent Since the matrix of general terms $(\bar{\O}_{\a\b})$ is invertible, consider the family of $1$-forms $\{{\cal Y}^\a=\bar{\O}_{\a\b}{\cal X}^\a\}$. We obtain a  dual  basis $\{{\cal Y}^\a,{\cal V}^\b\}$   so  that we have the decomposition
\begin{eqnarray}\label{darbO}
\O=\dis\frac{1}{2}\o_{\a\b}{\cal Y}^\a\wedge {\cal Y}^\b+{\cal Y}^\a\wedge {\cal V}^\a.
\end{eqnarray}
In such a basis we have  $\eta=\eta_\a{\cal Y}^\a+\bar{\eta}_\b{\cal V}^\b$. According to Lemma \ref{Heta}, let  ${L}$ be such that, locally,
$(d^{\cal P}{L})_{| \V{\cal M}}=\eta_{| \V{\cal M}}=\dis\frac{\p {L}}{\p y^\a}{\cal V}^\a$.
Since $\O$ is a symplectic form on $\T{\cal M}$, the equation $i_{\cal S}\O=\eta$ has a unique solution locally given by:
\begin{eqnarray}\label{semihamS}
{\cal S}=\dis\frac{\p {L}}{\p y^\a}{\cal Y}_\a+(\dis\frac{1}{2}\o_{\a\b}\dis\frac{\p {L}}{\p y^\a}-\eta_\b){\cal V}_\b.
\end{eqnarray}
According to the choice of the initial canonical basis $\{{\cal X}_\a,{\cal V}_\b\}$,    the matrix of  $g_\eta$  is  $(\dis\frac{\p^2 {L}}{\p y^\a\p y^\b})$  and we denote by $(g_\a^\b)$ the inverse of this matrix.  We then have:
${\cal J}{\cal Y}_\a=g_\a^\b{\cal V}_\b$. We  consider a new coordinate system $ (\bar{x}^i=x^i, \bar{y}^\a)$ such that
$\dis\frac{\p}{\p \bar{y}^\a}=g_A^B \dis\frac{\p}{\p {y}^\a}$.
Then we have $\bar{y}^A=\dis\frac{\p {L}}{\p y^\a}-\phi^\a$ where $\phi^\a$ is some function which only depends on $(x^i)$.
Of course, we have   ${\cal J}{\cal Y}_\a=\bar{\cal V}_\a$.
In this new coordinate system ${\cal S} $ has a decomposition of type
$${\cal S}= (\bar{y}^\a+\phi ^\a){\cal Y}_\a+S^\b\bar{\cal V}_\b.$$
\noindent Therefore we obtain  ${\cal J}{\cal S}=(\bar{y}^\a+\phi ^\a){\cal V}_\a$.  In particular, it follows that ${\cal C}_{\cal S}={\cal J}{\cal S}$ is an Euler section.
Now, consider the new coordinate system $(\hat{x}^i=\bar{x}^i=x^i,\hat{y}^\a=\bar{y}^\a+\phi^\a)$ on $\cal M$. The canonical basis associated with this coordinate system and the initial choice of the basis $\{e_\a\}$ is of type $\{{\cal X}_\a,\hat{\cal V}_\b\}$.  For this new system,  in the associated  basis $\{{\cal Y}_\a,\hat{\cal V}_B\}$ of $\T{\cal M}$,  we get the required decomposition  announced  in property (c) of Part (i).\\

 Assume that $\O$ is of Liouville type. Choose a canonical basis $\{{\cal X}_\a,{\cal V}_\b\}$   of $\T{\cal M}$  associated with a coordinate system $(x^i,y^\b)$  compatible with the  locally affine fibration structure  associated with $\O$, and to a local basis $\{e_\a\}$ of $\cal A$.  In this basis
the matrix associated with $g_\eta$  is  $(\dis\frac{\p^2 {L}}{\p y^\a\p y^\b})$  and, as in the proof of Part  (i),  denote by $(g_\a^\b)$ the inverse of this matrix.
According to Proposition \ref{equivcomp}, in  the corresponding dual basis $\{{\cal X}^\a,{\cal V}^\b\}$ we have

$\O=\dis\frac{\p^2 L}{\p y^\a \p y^\b}{\cal X}^\a\wedge{\cal V}^\b+\dis\frac{1}{2}\o_{\a\b}{\cal X}^\a\wedge{\cal X}^\b\;\;\;\;\;\;\; \eta=\eta_\a{\cal X}^\a+\dis\frac{\p L}{\p y^\b}{\cal V}\b$ and ${\cal J}{\cal X}_\a=g_\a^\b{\cal V}_∫$
%the matrix associated with $g$  is  $(\dis\frac{\p^2 {L}}{\p y^A\p y^C})$  and denote by $(g_A^B)$ the inverse of this matrix

By the same argument used in the proof of Part (i) with ${\cal Y}_\a={\cal X}_\a$ we  build a new coordinate system $(\hat{x}^i,\hat{y}^\a)$ such that in the canonical associated basis $\{\hat{\cal X}_\a,\hat{\cal V}_\b\}$ (associated with the same basis $\{e_A\}$ of $\cal A$) we have

${\cal J}\hat{\cal X}_\a=\hat{\cal V}_\a\;\;\;$  $\;\;\; {\cal S}=\hat{y}^\a\hat{\cal X}_\a+\hat{\cal S}^\b\hat{\cal V}_\b\;\;\;$.

\noindent  It follows that  ${\cal C}_{\cal S}={\cal J}{\cal S}=\hat{y}^\a\hat{\cal V}_\a$ and  $\cal J$ is integrable.

Moreover, from the last decomposition of $\cal S$, we have ${\cal J}[{\cal S},\hat{\cal V}_\a]_{\cal P}=-\hat{\cal V}_\b$.  It follows that $\cal S$ is a semispray and $\cal J$ is compatible with $\cal S$. By similar arguments we also obtain that $\cal J$ is compatible with ${\cal C}_{\cal S}$.

\end{proof}

\bigskip

\begin{proof}[Proof of Proposition \ref{propsemiH}]${}$\\
 In this proof we only consider coordinate systems $(x^i,y^\a)$  compatible with the locally linear fibration  structure associated with $\cal C$ and associated  adapted local basis $\{{\cal X}_\a,{\cal V}_\b\}$ of ${\T}{\cal M}$ such that ${\cal J}{\cal X}_\a={\cal V}_\b$ and ${\cal C}=y^\a{\cal V}_\a$ (see  Remark \ref{locC}).\\

Since  $\O$ is semi-Hamiltonian, from the local decomposition  (\ref{semihamO} )  it follows that  ${\cal J}$ is compatible with $\O$ and  $g_\theta ( {\cal J},\; ) = \O( \;, \;)$.
 Consider  local basis $\{{\cal X}_\a,{\cal V}_\b\}$ such that we have ${\cal C}=y^\a{\cal V}_\a$ and $\theta=\theta_\a{\cal X}^\a+\dis\frac{\p {L}}{\p y^\b}{\cal V}^\b$.
Consider the $1$-form ${\cal L}^{\cal P}_{\cal C}\theta$.
\noindent  Therefore we obtain:
\begin{eqnarray}\label{hamiltheta}
{\cal L}^{\cal P}_{\cal C}\theta=y^\b\dis\frac{\p \theta_\a}{\p y^\b} {\cal X}^\a+(\dis\frac{\p {L}}{\p y^\a}+y^\b\dis\frac{\p^2{L}}{\p y^\a\p y^\b}){\cal V}^\a.
\end{eqnarray}

\noindent It follows that we have

\begin{eqnarray}\label{semi-hamilsp}
{\cal L}^{\cal P}_{\cal C}\theta- \theta=y^\b\dis\frac{\p^2 {L}}{\p y^\a\p y^\b}{\cal V}^\a+(y^\b\dis\frac{\p \theta_\a}{\p y^\b}-\theta_\a){\cal X}^\a.
\end{eqnarray}

\noindent According to the decomposition (\ref{semihamO}),  the solution of $i_{\cal S}\O={\cal L}^{\cal P}_{\cal C}\theta- \theta$ has a decomposition of type
${\cal S}=y^\b{\cal X}^\b+{\cal S}^\b{\cal V}^\b$
with  ${\cal S}^\b$ characterized by:
\begin{eqnarray}\label{Let-et}
\dis\frac{\p^2{L}}{\p y^\a \p y^\b}{\cal S}^\b+y^\b(\dis\frac{\p \theta_\a}{\p y^\b}-\dis\frac{\p^2 {L}}{\p x^i\p y^\b} \rho_\a ^i+\dis\frac{\p^2 {L}}{\p x^i\p y^\a} \rho_\b ^i-\frac{\p {L}}{\p y^\g}C_{\a\b}^\g)=\theta_\a.
\end{eqnarray}

\noindent Since  locally  we have ${\cal J}{\cal X}_\a={\cal V}_\a$, from the previous decomposition of $\cal S$ we get

$[{\cal S},{\cal V}_\a]=-{\cal X}_\a-\dis\frac{\p {\cal S}^\b}{\p y^\a}{\cal V}_\b$

\noindent This implies that  $\cal S$ is a semispray and  $\cal J$  is compatible with $\cal S$. Finally since  ${\cal C}=y^\b{\cal V}_\b$, it follows that    $\cal J$ is compatible with $\cal C$, which ends the proof of Part  (ii).\\

Given any   semispray  $\cal S$ compatible with $\cal J$, from Proposition \ref{JS}, with  the previous notations,    we locally have  ${\cal S}=y^\a{\cal X}_\a+{\cal S}^\b{\cal V}_\b$. We deduce that  the "vertical" components of $i_{\cal S}\O=\eta$ is then
$$\dis\frac{\p^2 L }{\p y^\a \p y^\b}y^\a {\cal V}^\b.$$
Thus,  if we set ${\cal H}=\dis\frac{1}{2}{\cal L}^{\cal P}_{\cal C}L-L$, we obtain  $\dis\frac{\p^2 L }{\p y^\a \p y^\b}y^\a {\cal V}^\b=\dis\frac{\p {\cal H}}{\p y^\b}{\cal V}^\b$. It follows that $ \eta$ is semi-hamitonian.\\

\end{proof}

%%%%%%%%%%%%%%%%%%%%%%%%%%%%%%%%%%%%%%%%%%%%%%%%%%%%%%%%%%%%%%%%%%%%%%%%%%%%%%%%%%%%%%%%
%%%%%%%%%%%%%%%%%%%%%%%%%%%%%%%%%%%%%%%%%%%%%%%%%%%%%%%%%%%%%%%%%%%%%%%%%%%%%%%%%%%
\section{Nonlinear connections and semisprays}
%%%%%%%%%%%%%%%%%%%%%%%%%%%%%%%%%%%%%%%%%%%%%%%%%%%%%%%%%%%%%%%%%%%%%%%%%%%%%%%%%%
linear connections to our context and their links with almost tangent structures, Euler sections and semisprays.

%%%%%%%%%%%%%%%%%%%%%%%%%%%%%%%%%%%%%%%%%%%%%%%%%%%%%%%%%%%%%%%%%%%%%%%%%%%%%%%%%
\subsection{Non linear connection and almost tangent structure}${}$\\
%%%%%%%%%%%%%%%%%%%%%%%%%%%%%%%%%%%%%%%%%%%%%%%%%%%%%%%%%%%%%%%%%%%%%%%%%%%%%%%%%%
%%%%%%%%%%%%%%%%%%%%%%%%%%%%%%%%%%%%%%%%%%%%%%%%%%%%%%%%%%%%%%%%%%%%%%%%%%%%%%%%%%
${}\;\;\;\;$ Classically, a {\it nonlinear connection}  is  a decomposition of this bundle in a Whitney sum $\T{\cal M}=\V{\cal M}\oplus \H{\cal M}$. Such a   decomposition is equivalent to  the datum of an
endomorphism ${\cal N}$ of  $\T{\cal M}$ such that ${\cal N}^2=Id$   with $\V{\cal M}=\ker (Id+{\cal N})$ and $\H{\cal M}=\ker (Id-{\cal N})$ where $Id$ is the identity morphism of $\T{\cal M}$. We naturally get two projections:

$h_{\cal N} =\dis\frac{1}{2}(Id+{\cal N}): \T{\cal M}\ap \H{\cal M}$ and $v_{\cal N}=\dis\frac{1}{2}(Id-{\cal N}): \T{\cal M}\ap \V{\cal M}$

$v_{\cal N}$ and  $h_{\cal N}$ are called respectively  the {\it vertical}  and  {\it  horizontal} projector of ${\cal N}$.\\

%{\it Now we come back to to our context of a bundle $\cal E\subset{\T}{\cal M}$  which has the previous  properties:}\\
 \begin{Pro}\label{JG}${}$
  \begin{enumerate}
 \item[(i)] Assume that there exists an  almost structure ${\cal J}$ on $\T{\cal M}$. Then any  nonlinear connection ${\cal N}$ on $\T{\cal M}$ satisfies
   \begin{eqnarray}\label{relJG}
{\cal J }{\cal N} ={ \cal J} \text {  and    } {\cal N} { \cal J}=-{\cal J}.
\end{eqnarray}
Moreover,  an endomorphism ${\cal N}$ of $\T{\cal M}$   is a nonlinear connection if and only if ${\cal N}$  satisfies the relations (\ref{relJG}) .
\item[(ii)] Let  $ \Upsilon$ be a semi-basic vector  valued $1$-form  then ${\cal N}+ \Upsilon$ is a  nonlinear connection on $\T{\cal M}$. Conversely, given any nonlinear connection ${\cal N}'$ on $\T{\cal M}$, there exists a unique  semi-basic  vector valued   $1$-form $ \Upsilon$   such that ${\cal N}'={\cal N}+ \Upsilon$.
  \item[(iii)]  Consider   a   semispray  ${\cal S}$ on ${\cal M}$ and ${\cal J}$ the associated almost structure on  $\T{\cal M}$ (see Proposition \ref{JS}). We then have the following properties:\\

  \noindent $\bullet$  The endomorphism ${\cal N}_{\cal S}=-{\cal L}^{\cal P}_{\cal S}{\cal J}$  is   a nonlinear connection on $\T{\cal M}$, where ${\cal L}^{\cal P}$ is the Lie derivative for tensor associated with $[\;,\;]_{\cal P}$. Moreover, ${\cal N}_{\cal S}$ only depends on the choice of  the bracket $[\;,\;]_{\cal A}$.\\
 \noindent $\bullet$    The  vertical  projector $v_{\cal S}$ of ${\cal N}_S$  satisfies the following  relation, for any vertical section $\cal X$:
\begin{eqnarray}\label{vS}
v_{\cal S}[{\cal S},{\cal X}]_{\cal P}=-\dis\frac{1}{2}{\cal J}[{\cal S},[{\cal S},{\cal X}]_{\cal P}]_{\cal P}.
\end{eqnarray}
$\bullet$ the morphism $H_{\cal S}:\V{\cal M}\ap \T{\cal M}$ defined by
 \begin{eqnarray}\label{HS}
{\cal  H}_{\cal S}{\cal X}=-[{\cal S},{\cal X}]_{\cal P}+\dis\frac{1}{2}{\cal J}[{\cal S},[{\cal S},{\cal X}]_{\cal P}]_{\cal P}.
\end{eqnarray}
\noindent is an isomorphism onto the horizontal space $\H_{\cal S}$ associated with ${\cal N}_{\cal S}$ whose inverse is the restriction of  ${\cal J}$ to  $\H_{\cal S}$.
 \end{enumerate}
 \end{Pro}
 \smallskip
 \begin{Rem}${}$ \\
 Consider a  semispray $\cal S$ on $\T{\cal M}$ and the  associated tangent structure $\cal J$. For any  (local) section $s$ of $\cal A$, we can define a vertical lift $s^v:{\cal M} \ap \T{\cal M}$ by $s^v(m)={\cal J}(\s(\pi(m),X(m))$ where $X$ is any vector field on $\cal M$ such that $T\pi(V)=\rho(s)$ ({\it cf.} proof of proposition \ref{vertaff}). Given any two brackets $[\;,\;]_{\cal A}$ and $[\;,\;]'_{\cal A}$ on $\cal A$, the associated difference $\cal A$-tensor  $T=[\;,\;]'_{\cal A}-[\;,\;]_{\cal A}$   can  be lifted to a semi-basic tensor $T^v$  on $\T{\cal M}$  by:
$$T^v((m,b,v), (m,b',v'))=[T((\pi(m),b), (\pi(m),b'))]^v$$ for any $(m,b,v)$ and $(m,b',v')$ in $\T_m{\cal M}$.

  \noindent According to subsection \ref{prolbrac}, let   $[\;,\;]_{\cal P}$ and  $[\;,\;]'_{\cal P}$ ) be the almost  bracket prolongation $[\;,\;]_{\cal P}$ and $[\;,\;]'_{\cal P}$ on $\T{\cal M}$   associated with   $[\;,\;]_{\cal A}$ and  $[\;,\;]'_{\cal A}$, respectively. Then, between the nonlinear connections  ${\cal N}_{\cal  S}$  and  ${\cal N}'_{\cal  S}$  associated with  $\cal S$  and the almost bracket $[\;,\;]_{\cal P}$  and  $[\;,\;]'_{\cal P}$ respectively, we have the following relation
 $${\cal N}'_{\cal S}={\cal N}_{\cal S}+ i_{\cal S}T^v.$$
\end{Rem}
 \smallskip
 \begin{proof} [Proof of Proposition \ref{JG}]${}$\\
  We only prove the  announced properties  of $v_{\cal S}$ and ${\cal H}_{\cal S}$. The others properties are classical.

\noindent  For any vertical section ${\cal X}$,  from the definition of ${\cal N}_{\cal S}$ and the characteristic property of $\cal J$, we have

  $\;\;\;\;{\cal N}_{\cal S}[{\cal S},{\cal X}]_{\cal P}=[{\cal S},{\cal X}]_{\cal P}+{\cal J}[{\cal S},[{\cal S},{\cal X}]_{\cal P}]_{\cal P}$.

\noindent  Therefore, from the definition of $v_{\cal S}$ we get (\ref{vS}) and from the definition of ${\cal H}_{\cal S}$, the second member of (\ref{HS}) is horizontal. It remains to show that the kernel of $H_{\cal S}$ is $\{0\}$. Assume that  ${\cal H}_{\cal S}({\cal X})=0$. Consider a local canonical  basis $\{{\cal X}_\a,{\cal V}_\b\}$ of $\T{\cal M}$.  Now   if ${\cal X}=f_\a{\cal V}_\a$, we get

$[{\cal S},{\cal X}]_{\cal P}={\cal L}^{\cal P}_{\cal S}f_\a{\cal V}_\a+f_\a{\cal X}_\a$.

  \noindent Then, according to (\ref{HS}), we have ${\cal H}_{\cal S}({\cal X})=0$ if and only if  $[{\cal S},{\cal X}]_{\cal P}$ is vertical  and then we must have ${\cal X}=0$. Finally, according to we have
 $${\cal J}\circ{\cal H}_{\cal S}({\cal X})=-{\cal J}[{\cal S},{\cal X}]_{\cal P}={\cal X}.$$

 \end{proof}
 \bigskip \bigskip \bigskip
 In a  basis $\{{\cal Y}_\a,{\cal V}_\a\}$ such that ${\cal J}{\cal Y}_\a={\cal V}_\a$, a nonlinear connection  is characterized by a matrix of type
$$\begin{pmatrix}
Id&0\cr
-(2{\cal N}_\a^\b)&-Id\cr
\end{pmatrix}$$
and the general terms ${\cal N}_\a^\b$ are called the {\it coefficients} of the connection.  In particular, we have $h_{\cal N}({\cal Y}_\a)={\cal Y}_\a -{\cal N}_\a^\b{\cal V}_\b$. \\
Now, under  the assumptions of Proposition  \ref{existS}, we can choose  a canonical basis such that ${\cal C}=y^\a{\cal V}_\a$  (see Remark \ref{locC}) and in such a basis any  semispray  $\cal S$ can be written ${\cal S}=y^\a{\cal X}_\a+ {\cal S}^\b{\cal V}_\b$.  Therefore in a canonical basis $\{{\cal X}_\a,{\cal V}_\b\}$ which satisfies  the relations (\ref{adaptJS}),  the  connection ${\cal N}_{\cal S}$ associated with  a  semispray  $\cal S$ are given by:
\begin{eqnarray}\label{locGS}
{\cal N}_\a^\b=\dis\frac{1}{2}(-\frac{\p {\cal S}^\b}{\p y^\a}+ {C}_{\a\d}^\b y^\d ).
\end{eqnarray}

\bigskip
 \bigskip
 {\it We end this subsection by the notion of} {\bf  curvature} of a connection  which depends on the bracket $[\;,\;]_{\cal P}$. }

  Now, according to \cite{Gr} we have:

  \begin{Def}\label{courb}${}$\\
   Let ${\cal N}$ be a nonlinear connection on $\T{\cal M}$. The {\it curvature} of ${\cal N}$ is the vector valued   $2$-form given by the Fr\"{o}licher- Nijenhuis bracket
 $${\bf R}_{\cal N}=-\dis\frac{1}{2}[h_{\cal N},h_{\cal N}]_{\cal P}.$$
 \end{Def}
\bigskip

 \begin{Pro}\label{propR}${}$
 \begin{enumerate}
\item[(i)] The curvature ${\bf R}_{\cal N}$ also has  the following value:
 $$R_{\cal N}({\cal X},{\cal Y})=-v_{\cal N}[h_{\cal N}{\cal X},h_{\cal N}{\cal Y}]_{\cal P}={\bf R}_{\cal N}(h_{\cal N}{\cal X},h_{\cal N}{\cal Y}).$$ In particular, ${\bf R}_{\cal N}$ is a semi-basic  $2$-form
 \item[(ii)] Consider  the nonlinear connection ${\cal N}$   associated with an almost  semispray s ${\cal S}$. Then $\Phi=v_{\cal N}\circ {\cal L}^{\cal P}_{\cal S}v_{\cal N}$ is a morphism from $\T{\cal M}$ to $V{\cal M}$ such that $\Phi\circ\Phi=0$ and  we have
 $$\Phi{\cal X}=i_{\cal S}R_{{\cal N}}{\cal X}-v_{{\cal N}}[v_{{\cal N}}{\cal S},h_{{\cal N}}{\cal X}]_{\cal P}. $$
  \end{enumerate}
 \end{Pro}
 \bigskip
  \begin{Def}\label{jacob}${}$\\
  The endomorphism  $\phi$ is called the Jacobi endomorphism associated with $\cal S$
  \end{Def}
  \bigskip
 \bigskip

   \begin{proof}[Proof of Proposition \ref{propR}] :${}$\\
 We only prove the second Part. The first one is classical.

  If ${\cal N}$ is the nonlinear connection of at  semispray  $\cal S$, we have

  $v_{{\cal N}}\circ{\cal L}^{\cal P}_{\cal S}v_{{\cal N}}({\cal X})=v_{{\cal N}}[S,v_{{\cal N}}{\cal X}]_{\cal P}-v_{\cal N}[{\cal S},{\cal X}]_{\cal P}=v_{\cal N}[{\cal S},v_{\cal N}{\cal X}-X]_{\cal P}=-v_{\cal N}[{\cal S},h_{\cal N}{\cal X}]_{\cal P}$.

 Thus  $\Phi$ is a morphism from $\T{\cal M}$ to $\V{\cal M}$ and  with $\V{\cal M}\subset \ker \Phi$. From the last member of   the previous relation we get $\Phi{\cal X}=i_{\cal S}R_{{\cal N}}{\cal X}-v_{{\cal N}}[v_{{\cal N}}{\cal S},h_{{\cal N}}{\cal X}]_{\cal P} $.\\
  \end{proof}

 %%%%%%%%%%%%%%%%%%%%%%%%%%%%%%%%%%%%%%%%%%%%%%%%%%%%%%%%%%%%%%%%%%%%
 \subsection{Nonlinear connection and Euler section}\label{GC}${}$\\
 %%%%%%%%%%%%%%%%%%%%%%%%%%%%%%%%%%%%%%%%%%%%%%%%%%%%%%%%%%%%%%%%%%%
 %%%%%%%%%%%%%%%%%%%%%%%%%%%%%%%%%%%%%%%%%%%%%%%%%%%%%%%%%%%%%%%%%%
${}\;\;\;\;$ {\it In this subsection, we   assume that    there exists  an Euler section on $\cal M$, an almost tangent structure ${\cal J}$  on $ \T{\cal M}$ compatible    with some   semispray  $\cal S$    such that ${\cal J}{\cal S}={\cal C}$}.

%This situation occurs in particular in Example \ref{exJOS1}, Example \ref{exJOS2}  and more generally under the assumptions of Proposition \ref{JS} point (ii) or Proposition \ref{existS}.\\

\begin{Lem}\label{allG}${}$\\
Let  ${\cal N}$ be a nonlinear connection on $\T{\cal M}$
\begin{enumerate}
\item[(i)]  There exists a unique   semispray  $\cal S$ which  is horizontal called the {\bf canonical   semispray } of ${\cal N}$.
\item[(ii)] The connections ${\cal N}$ and    ${\cal N}'={\cal N}+ \Upsilon$  have the same canonical  semispray  $\cal S$ if and only if $ \Upsilon{\cal S}=0$.
\end{enumerate}
\end{Lem}

As classically we have:

  \begin{Def}\label{geod}${}$
 \begin{enumerate}
 \item Consider a  semispray  $\cal S$.  A curve  $\g:[a,b]\ap M$ is called an integral curve of $\cal S$ if there exists an integral curve  $\hat{\g}:[a,b]\ap {\cal M}$ of  the vector field  $\hat{\rho}({\cal S})$  on $\cal M$ such that  $\g=\pi\circ\hat{\g}$.
  \item Let  $\cal S$ be the canonical  semispray of a nonlinear connection   $\cal N$. A geodesic of $\cal N$ is an integral curve of $\cal S$
   \end{enumerate}
    \end{Def}

Consider the canonical  semispray $\cal S$ of a nonlinear  connection ${\cal N}$ on $\T{\cal M}$. Then,  around each point $m\in {\cal M}$, there exists a coordinate system $(x^i,y^\a)$, defined on a connected open neigbourhood $U$ of $m$ and a canonical basis $\{{\cal X}_\a,{\cal V}_\a\}$ of $\T{\cal M}$ (also defined on $U$) such that

 ${\cal J}{\cal X}_\a={\cal V}_\a,\;\;$ $\;\; {\cal C}=y^\a{\cal V}_\a\;\;$ and $\;\;{\cal S}=y^\a{\cal X}_\a+{\cal S}^\b{\cal V}_\b$.

% On the other hand, on $U$ we also have a decomposition

% ${\cal Y}_A=Y_A^\a{\cal X}_\a$ where $\{{\cal X}_\a,{\cal V}_A\}$ is a local basis of ${\ T}{\cal M}$  associated with $(x^i,y^A)$.

It follows that  on the open set $V=\pi(U)$ in $M$ the  integral curves of $\cal S$ or the geodesics of $\cal N$ (whose canonical semispray is $\cal S$) are locally characterized by the differential system of differential equations:
 \begin{eqnarray}
\label{eqgeod}
\begin{cases}
\dot{x}^i= \rho^i_\a y^\a \cr
\dot{y}^\a={\cal S}^\a.\cr
\end{cases}
\end{eqnarray}

 \bigskip

We end this subsection  by  the notions of torsion of a nonlinear connection  and homogeneous connections according to \cite{Gr}:

 \begin{Def} ${}$\\Let ${\cal N}$ be a nonlinear connection on $\T{\cal M}$.
 \begin{enumerate}
 \item[(i)]The {\bf week torsion } of ${\cal N}$ is the $2$- vector valued form:
 ${\bf t}_{\cal N}=\dis\frac{1}{2}[ {\cal J},{\cal N}]_{\cal P}$.
\item[(ii)] The {\bf  tension} of a connection ${\cal N}$ is the $1$-vector valued form:
$\mathbb{H}_{\cal N}=\dis\frac{1}{2}{\cal L}^{\cal P}_{\cal C}{\cal N}$
 \item[(iii)] The {\bf strong torsion} of ${\cal N}$ is the $1$-vector valued form :
 $\mathbb{ T}_{\cal N}=i_{\cal S}{\bf t}_{\cal N}-\mathbb{ H}_{\cal N}$.
 \item[(iv)] The  {\bf  tension} of a semispray  is the the section
 $\mathbb{ H}_{\cal S}={\cal S}-[{\cal C},{\cal S}]_{\cal P}.$\\
 \end{enumerate}
 \end{Def}

 As in the classical case (\cite{Gr}), the  vector valued forms  ${\bf t}_{\cal N}$, $\mathbb{ H}_{\cal N}$ and $\mathbb{ T}_{\cal N}$ are semi-basic. Moreover,   we have also the following results whose  proofs  are the same as  the proofs  of the corresponding results on the tangent bundle $TM$ found in \cite{Gr}:

  \begin{Pro}\label{Hprop}${}$
 \begin{enumerate}
\item[(i)] let  ${\cal N}$ be a nonlinear connection on $\T{\cal M}$. The following properties are equivalent:
 \begin{enumerate}
 \item[(a)]  $\cal S$ is the canonical   semispray  of ${\cal N}$;
 \item[(b)] $\mathbb{H}_{\cal N}({\cal S})=\mathbb{ H}_{\cal S}$;
 \item[(c)] $\mathbb{ T}_{\cal N}({\cal S})=\mathbb{ H}_{\cal S}$.
 \end{enumerate}
 \item[(ii)] The canonical  semispray  of a connection ${\cal N}_{\cal S}$, associated with $\cal S$, is $\dis\frac{1}{2}({\cal S}+[{\cal C},{\cal S}]_{\cal P})$. In particular, ${\cal S}$ is the canonical  semispray  of ${\cal N}_{\cal S}$ if and only if the tension $\mathbb{ H}_{\cal S}=0$
% \item[(iii)] Given two almost  semispray s ${\cal S}$ and ${\cal S}'$, there canonical connection ${\cal N}_{\cal S}$ and ${\cal N}_{{\cal S}'}$ coincide if and only if ${\bf H}_{\cal S}={\bf H}_{{\cal S}'}$
\end{enumerate}
 \end{Pro}
 \bigskip
 \begin{The}\label{same S}${}$\\
 Let    $\cal S$   be  a  semispray   and  $\mathbb{ T}$ a semi-basic vector $1$-form such that $\mathbb{ T}({\cal S})=\mathbb{H}_{\cal S}$. Then ${\cal N}_{\cal S}+\mathbb{T}$ is the  unique  nonlinear connection whose canonical   semispray  is $\cal S$ and its strong torsion is $\mathbb{ T}_{\cal N}=\mathbb{T}$.
 \end{The}
\smallskip
\begin{Cor}${}$\\ Two nonlinear  connections ${\cal N}$ and ${\cal N}'$ have the same canonical  semispray  and the same strong torsion, if and only if ${\cal N}={\cal N}'$
\end{Cor}
\begin{proof}
Assume that  ${\cal N}$ and ${\cal N}'$ have the same canonical almost  semispray  ${\cal S}$ and the same strong torsion  $\mathbb{T}$. From proposition \ref{Hprop} we have $\mathbb{ T}({\cal S})=\mathbb{ H}_{\cal S}$. Therefore from Theorem \ref{same S}, we have  $${\cal N}={\cal N}'={\cal N}_{\cal S}+\mathbb{ T}.$$
The converse is trivial.
\end{proof}
\smallskip
\begin{The}\label{Tzero}${}$\\
Assume that the bracket $[\;,\;]_{\cal A}$ satisfies the Jacobi identity.
\begin{enumerate}
\item  The weak torsion of any connection of type ${\cal N}_{\cal S}$ is zero;
\item For any nonlinear connection ${\cal N}$ then  $\mathbb{T}_{\cal N}=0$ is zero if and only if  ${\bf t}_{\cal N}=0$ and  $\mathbb{ H}_{\cal N}=0$.
\end{enumerate}
\end{The}

\bigskip
 \begin{Rem}\label{torsionzero}${}$\\ Consider   a nonlinear connection ${\cal N}$  and $\cal S$ its canonical   semispray. As in the the classical case of $TM$, if the tension $\mathbb{ H}_{\cal N}\not=0$, then  ${\cal N}$ is different from the connection ${\cal N}_{\cal S}$ associated with ${\cal S}$.  Moreover,   in this case,  the strong torsion of ${\cal N}$ is never zero.
 \end{Rem}

\smallskip
 Finally, we look  for  some properties of { \it homogeneous } nonlinear connections. According to the classical context of nonlinear connection ({\it cf.} \cite{Gr}) we have:\\

\begin{Def}\label{homog}${}$
\begin{enumerate}
\item[(i)] a function $f$ (resp. a $l$-form $\o$) is called $r$-homogeneous if ${\cal L}^{\cal P}_{\cal C}(f)=r.f$ (resp. ${\cal L}^{\cal P}_{\cal C}\o=r.\o$). A vector valued $l$-form $\o$ is called  $r$-homogeneous if $[{\cal C},\o]_{\cal P}=(r-1).\o$.
\item[(ii)] A  semispray  ${\cal S}$  is a spray if $\cal S$ is $2$-homogenous. .
%\item[(ii)]  An Ehresmann connection ${\cal N}$ on $\cal E$ is said homogeneous if its tension ${\bf H}_{{\cal N}}$ is zero.
\end{enumerate}
\end{Def}
\bigskip
The following properties are obtained in the same way as the classical case by an adapted proof or are direct consequences of the previous theorems:

$\bullet$ $\cal S$ is a spray if and only if  its tension $\mathbb{H}_{\cal S}$ is vanishes.

$\bullet$ A nonlinear connection ${\cal N}$ on $\T{\cal M}$ is $1$- homogeneous if  and only if its tension $\mathbb{ H}_{{\cal N}}$ vanishes.\\

\noindent We consider coordinate systems compatible with the locally linear structure defined by $\cal C$ (see Subsection \ref{verteuler}) \\

$\bullet$ in a  local canonical  basis $\{{\cal X}_\a,{\cal V}_\b\}$ associated with such a coordinate system we have:\\
%Let be $\{{\cal Y}_A,{\cal V}_B\}$ an adapted basis as described in Remark \ref{locC}. Then we have:\\

${}\;\;\;\;\;${\bf -} a  semispray  ${\cal S}$ is a spray if and only if in such a basis we can write
$${\cal S}=y^\a{\cal X}_\a+{\cal S}^\a_{\b\g}y^\b y^\g{\cal V}_\a;$$
%${}\;\;\;\;\;\;\;\;\;$ and ${\cal S}^A_{BC}$ depends only of $(x^i)$.\\

${}\;\;\;\;\;${\bf -}  a  nonlinear connection ${\cal N}$ is $1$-homogeneous if and only if the coefficient ${\cal N}_\a^\a$ of ${\cal N}$ can be written in the\\
${}\;\;\;\;\;\;\;\;\;\;$ following
 way:
$${\cal N}_\a^\b={\cal N}_{\a\g}^\b y^\g;$$
%${}\;\;\;\;\;\;\;\;\;\;\;$  and ${\cal N}_{AC}^B$ depends only of variables $(x^i)$.\\

$\bullet$  let  ${\cal N}_{\cal S}$ be the nonlinear connection associated with a semispray  $\cal S$. The following properties are\\
${}\;\;\;\;\;\;$ equivalent:

\begin{enumerate}
\item[(i)] $\cal S$ is $2$-homogeneous;
\item[(ii)] ${\cal N}_{\cal S}$ is $1$-homogeneous;
\item[(iii)] The canonical   semispray  of ${\cal N}_{\cal S}$ is $\cal S.$\\
\end{enumerate}

$\bullet$  if ${\cal N}_{\cal S}$ is the $1$-homogeneous  connection associated with a spray $\cal S$ then the strong torsion of ${\cal N}_S$ is zero. Moreover, ${\cal N}_{\cal S}$ is the unique $1$-homogeneous connection whose canonical   semispray  is $\cal S$ and whose strong torsion is zero.

%%%%%%%%%%%%%%%%%%%%%%%%%%%%%%%%%%%%%%%%%%%%%%%%%%%%%%%%%%%%%%%%%%%%%%%%%%%%%%%%%%%%%%%%%
\section{Lagrangian  metric connections  and semisprays }\label{lagcon semispray }
%%%%%%%%%%%%%%%%%%%%%%%%%%%%%%%%%%%%%%%%%%%%%%%%%%%%%%%%%%%%%%%%%%%%%%%%%%%%\
%%%%%%%%%%%%%%%%%%%%%%%%%%%%%%%%%%%%%%%%%%%%%%%%%%%%%%%%%%%%%%%%%%%%%%%%%%%%%%

Using the notion of Dynamical derivation on the vertical bundle as introduced in \cite{Bu}, we begin by  the context of metric connections. Then we look for the framework of Lagrangian connections. This section contains in particular the characterization of the unique Lagrangian metric connection associated with a semi-Hamiltonian almost tangent structure (see Theorem \ref{SLagMet}).  As in \cite{Gr}, for a  semispray associated with some semi-hamitonian we also build a canonical Lagrangian  nonlinear connection whose canonical semispray is the original one.
%%%%%%%%%%%%%%%%%%%%%%%%%%%%%%%%%%%%%%%%%%%%%%%%%%%%%%%%%%%%%%%%%%%%%%%%%%%%%%%%%%%%%%
\subsection{Dynamical derivation and metric connection}\label{DynDer}${}$\\
%%%%%%%%%%%%%%%%%%%%%%%%%%%%%%%%%%%%%%%%%%%%%%%%%%%%%%%%%%%%%%%%%%%%%%%%%%%%
${}\;\;\;\;$ {\it  In this subsection,  ${\cal S}$ is  a fixed  almost  semispray  on ${\cal M}$ and $\cal J$ is an almost tangent structure on $\T{\cal M}$ compatible with $\cal S$.}\\

{\rm We denote by $\Xi(\V{\cal M})$ the module of vertical sections of $\T{\cal M}$.   Following the arguments of \cite{Bu}, \cite{CMS}, \cite{F}, \cite{LPo2}, \cite{LPo3}, \cite{LPo4}  (among many others papers) we  introduce}

\begin{Def}\label{dyderiv}${}$\\
A dynamical derivation associated with $\cal S$ is a map $ D$ from $\Xi(\V{\cal M})$ to $\Xi(\V{\cal M})$ which satisfies the following properties:
\begin{enumerate}
\item[(i)] $ D({\cal X}+{\cal X}')=D({\cal X})+D({\cal X}')$ for all $\cal X$ and ${\cal X}'$ in $\Xi(\V{\cal M})$;
\item[(ii)] $D(f{\cal X})={\cal L}^{\cal P}_{\cal S}(f)+f D({\cal X})$ for all $\cal X$ in $\Xi(\V{\cal M})$ and any smooth function $f$ on $\cal M$.
\end{enumerate}
\end{Def}

Let  ${\cal N}$ be a nonlinear  connection on $\T{\cal M}$ and  $v_{\cal N}$ its  associated vertical projector.  We then have:

\begin{Pro}\label{alldyn}${}$
\begin{enumerate}
\item[(i)] The map $D_{\cal N}$ from $\Xi(\V{\cal M})$ into itself defined by:
$$D_{\cal N}({\cal X})=v_{\cal N}[{\cal S},{\cal X}]_{\cal P}$$

is a dynamical derivation associated with $\cal S$.
\item[(ii)]
If  $\{{\cal X}_\a,{\cal V}_\a\}$ is  a local canonical basis of ${\T}{\cal M}$ we set ${\cal Y}_\a=[{\cal S},{\cal V}_\a]_{\cal P}$. Then, $\{{\cal Y}_\a,{\cal V}_\b\}$ is  a local  basis  of $\T{\cal M}$, and,  if ${\cal N}_\a^\b$ are  the corresponding   coefficients of  a nonlinear connection  ${\cal N}$ on $\T{\cal M}$, then, we have
$$D_{\cal N}({\cal V}_\a)=-{\cal N}_\a^\b{\cal V}_\b.$$
\item[(iii)] There exists a coordinate system $(x^i,y^\a)$, and an associated  canonical basis $\{{\cal X}_\a,{\cal V}_\b\}$ such that, if ${\cal N}_\a^\b$ are  the corresponding   coefficients  of  a nonlinear connection  ${\cal N}$  we have:
\begin{eqnarray}\label{baseJSD}
{\cal J}{\cal X}_\a={\cal V}_\a, \;\;\;\;{\cal S}=y^\a{\cal X}_\a+{\cal S}^\b{\cal V}_\b \textrm{  and } \;\; D_{\cal N}({\cal V}_\a)=-({\cal N}_\a^\b+\dis\frac{\p{\cal S}^\b}{\p y^\a}){\cal V}_\b.
\end{eqnarray}

\item[(iv)] The set $\mathfrak{D}_{\cal S}$ of dynamical derivation associated with $\cal S$ has a natural affine structure. Moreover the map ${\cal N} \ap D_{\cal N}$ is an affine bijection between the  affine set  $\mathfrak{C}$ of nonlinear connections and  the set   $\mathfrak{D}_{\cal S}$.
\end{enumerate}
\end{Pro}

\begin{proof}${}$\\
The proof of Part  (i) is easy: From the definition, $D_{\cal N}$ is an endomorphism of $\Xi(\V{\cal M})$ which clearly satisfies property (ii) of Definition \ref{dyderiv}.\\
%At first we assume that ${\cal N}$ is the canonical Ehresmann connection associated with  $\cal S$.
For Part  (ii), consider  some basis $\{{\cal V}_\a\}$ of $\V{\cal M}$ associated with some coordinate system $(x^i,y^\a)$ compatible with $\pi$. According to the proof of Proposition \ref{JS}, there exists an adapted basis $\{{\cal Y}_\a,{\cal V}_\b\}$ such that if  ${\cal Y}_\a=[{\cal S},{\cal V}_\a]_{\cal P}$ and ${\cal J}{\cal Y}_\a={\cal V}_\a$

On the other hand, since   $v_{\cal N}({\cal Y}_\a)={\cal N}_\a^\b{\cal V}_\b$
%${\cal J}_{{\cal N}_{\cal S}}^{-1}({\cal X})={\cal X}^A({\cal Y}_A-\dis\frac{1}{2}(\frac{\p {\cal S}^B}{\p y^A}{\cal V}_B-y^D C_{AD}^B))$.
we get

\begin{eqnarray}\label{defDG}
D_{{\cal N}}({\cal V}_\a)=-{\cal N}_\a^\b{\cal V}_\b.
\end{eqnarray}

This ends the proof of Part (ii).\\

For  Part  (iii),  according to Proposition  \ref{JS}, around each point $m\in {\cal M}$, there exists a coordinate system $(x^i,y^\a)$ and an associated  canonical basis $\{{\cal X}_\a,{\cal V}_\b\}$ of $\T{\cal M}$ such that

${\cal J}{\cal X}_\a={\cal V}_\a$, $\;\;{\cal S}=y^\a{\cal X}_\a+{\cal S}^\b{\cal V}_\b$

\noindent According to the expression of $[\;,\;]_{\cal P}$ in such a basis we then have:

$[{\cal S},{\cal V}_\a]_{\cal P}=-{\cal X}_\a-\dis\frac{\p {\cal S}^\b}{\p y^\a}{\cal V}_\b$

\noindent Thus if   ${\cal N}_\a^\b$ are the local coefficient of ${\cal N}$ in this basis, we get

$D_{\cal N}({\cal V}_\a)=-({\cal N}_\a^\b+\dis\frac{\p{\cal S}^\b}{\p y^\a}){\cal V}_\b$

\noindent This ends the proof of Part (iii).\\

Let $D$ be any dynamical derivation  associated with $\cal S$. Denote by  $D_{\cal S}$  the dynamical derivation corresponding  to the canonical connection ${\cal N}_{\cal S}$ associated with $\cal S$. Then $\D=D-D_{\cal S}$ is a $\V{\cal M}$-tensor of type $(1,1)$. In particular, we get an affine structure on  $\mathfrak{D}_{\cal S}$.\\  If ${\cal J}$ is the almost tangent  structure on $\T{\cal M}$, we consider
$$\Upsilon=2\D{\cal J}.$$
Then $\Upsilon$ is a semi-basic $\T{\cal M}$-tensor.  Then ${\cal N}={\cal N}_{\cal S}+\Upsilon$ is a nonlinear  connection on $\T{\cal M}$ (see Proposition \ref{JG}). Moreover, in the adapted basis $\{{\cal Y}_\a,{\cal V}_\a\}$ of Part  (ii), from (\ref{defDG}) we have $D=D_{\cal N}$. On the other hand, if $D_{\cal N}=D_{{\cal N}'}$, by the same argument,    locally  ${\cal N}$ \and ${\cal N}'$ must have the same  coefficients and so ${\cal N}={\cal N}'$.\\
\end{proof}
{\rm It is easy to extend the action of a dynamical derivation $D$ to the module of $\V{\cal M}$-tensor   by requiring that $D$  preserves the tensor product. In particular, if $g$ is a $\V{\cal M}$-tensor of type $(2,0)$ which is semi-basic, we have:}
\begin{eqnarray}\label{Dg}
Dg({\cal X},{\cal X}')={\cal L}^{\cal P}_{\cal S}(g({\cal X},{\cal X}'))-g(D({\cal X}),{\cal X}')-g({\cal X},D({\cal X}')).
\end{eqnarray}

Now, let  $g$ be a pseudo-Riemannian metric on $\V{\cal M}$.  We have:

\begin{Def}(see \cite{Bu}, \cite{CMS}, \cite{F} or  \cite{LPo3}) ${}$\\
A nonlinear  connection ${\cal N}$ is called  $g$-metric (relatively to the almost  semispray  $\cal S$) if there exists a pseudo-Riemannian metric $g$  on $\V{\cal M}$ such that $D_{\cal N} g=0$.
\end{Def}

\so{\it This property has the classical following interpretation:}\\
${}\;\;\;\;$ {\rm Given a dynamical derivation $D$ associated with $\cal S$,  we can define the parallel transport along any integral curve $c:[0,T]\ap {\cal M}$ of $\cal S$:

a vertical section $\cal X$ along $c$ is parallel  if $D{\cal X}(c(s))=0$  for all $s\in [0,T]$. Now, for any $z\in \V{\cal M}_{c(0)}$ there exists a unique parallel vertical section $\cal X$ along $c$ such that ${\cal X}(c(0))=z$. We then get an isomorphism $\t_s$ from $\V{\cal M}_{c(0)}$ to $\V{\cal M}_{c(s)}$.

Then,  the property $D_{\cal N} g=0$ means that $\t_s$ is an isometry from $V{\cal M}_{c(0)}$ to $V{\cal M}_{c(s)}$ for any $s\in [0,T]$ and any integral curve $c$ of $\cal S$.\\

Assume that the pseudo-Riemannian $g$ on $\V{\cal M}$ is fixed. Thus we get an isomorphism $g^\flat$ from $\V{\cal M}$ to $(\V{\cal M})^*$. Then we have:}

\begin{Lem}\label{isogvert}${}$\\
The map ${g^2}: (\V{\cal M})^*\otimes V{\cal M}\ap \otimes^2 (\V{\cal M})^*$ defined by:
$$g^{2}(\D)({\cal X},{\cal X}')=<g^\flat\circ \D({\cal X}),{\cal X}'>$$
(where $<\;,\; >$ is the classical duality bracket) is an isomorphism.
\end{Lem}

\begin{proof}${}$\\
At first note that the respective fiber of each vector bundle $(V{\cal M})^*\otimes V{\cal M}$ and $\otimes^2 (V{\cal M})^*$ have the same dimension. On the other hand ${g^2}$ is clearly a bundle morphism. Assume that    $\D$ is in the kernel of ${g^2}$. It follows that, for any vertical section $\cal X$, we have
$<g^\flat\circ\D,{\cal X}>=0$ and we get $g^\flat\circ\D=0$. As $g^\flat$ is an isomorphism, finally $\D\equiv 0$.\\
\end{proof}

Denote by  $\mathbb{S}(\V{\cal M})$ and $ \mathbb{A}(\V{\cal M})$   the bundle of symmetric bilinear forms and skew symmetric bilinear forms on $\V{\cal M}$ respectively. Of course we have
\begin{eqnarray}\label{decompE}
\otimes^2 (\V{\cal M})^*=\mathbb{S}(\V{\cal M})\oplus  \mathbb{A}(\V{\cal M}).
\end{eqnarray}

We denote by $g^2_\mathbb{S}$ and  $g^2_ \mathbb{A}$ the  composition of $g^2$ with  the canonical projection of $\otimes^2 (\V{\cal M})^*$ onto $\mathbb{S}(\V{\cal M})$ and on  $ \mathbb{A}(\V{\cal M})$ respectively. Then we have

\begin{Pro}\label{connectmet}${}$
\begin{enumerate}
\item[(i)] If  ${\cal N}$ is  a nonlinear  connection, then ${\cal N}'={\cal N}-2((g^2)^{-1}\circ D_{{\cal N}_{\cal S}}g)\circ{\cal J}$ is a $g$-metric connection.
\item[(ii)] The set of $g$-metric connection is characterized by:
$$\{{\cal N}={\cal N}_{\cal S}- 2\D{\cal J}, \;\;\mathrm{ with }\;\; \D\in (g^2_\mathbb{S})^{-1}[D_{{\cal N}_{\cal S}}g]\;\}.$$

\end{enumerate}
\end{Pro}

\begin{Rem}${}$\\
The dynamical derivation previously defined corresponds to the {\it vertical part } of the analogous notion introduced in \cite{Bu} and  \cite{LPo3}. In particular, Part (i) of  Proposition \ref{connectmet}    corresponds to an intrinsic  version of Proposition 3.1 in  \cite{LPo2} and of  Theorem 1 in  \cite{LPo3}.
\end{Rem}
\begin{proof}${}$\\
For the proof of Part  (i) it is sufficient to show that $D_{{\cal N}'}g=0$. First of all we can note that $\D=(g^2)^{-1}\circ D_{{\cal N}_{\cal S}}g$ is a $\V{\cal M}$ tensor of type $(1,1)$. This implies that $D_{\cal N}+\D$ is a dynamical derivation which is associated with ${\cal N}'={\cal N}+2\D{\cal J}$. Recall that
 we have
 \begin{eqnarray}\label{equivmet}
 D_{{\cal N}'}g({\cal X},{\cal X}')=D_{{\cal N}_{\cal S}}g({\cal X},{\cal X}')-g(\D{\cal X},{\cal X}')-g({\cal X},\D{\cal X}').
 \end{eqnarray}
 On the one hand, we have
 $g^2(\D)({\cal X})=g(\D(.),{\cal X})$
 and on the other hand,
 $g^2(\D)=D_{\cal N} g$

Thus  the proof of Part  (i) is completed.\\

According to the proof of Proposition \ref{alldyn},  the dynamical derivation $D_{\cal N}$ associated with a nonlinear connection ${\cal N}$ can be written

$D_{\cal N}=D_{{\cal N}_{\cal S}}+\D$ if and only if ${\cal N}-{\cal N}_{\cal S}=2\D{\cal J}$.

Therefore ${\cal N}$ is $g$-metric if and only if
\begin{eqnarray}\label{equivmet}
D_{\cal N} g({\cal X},{\cal X}')=D_{{\cal N}_{\cal S}}g({\cal X},{\cal X}')-g(\D{\cal X},{\cal X}')-g({\cal X},\D{\cal X}')=0
\end{eqnarray}

for all vertical sections ${\cal X}$ and ${\cal X}'$.

On the other hand, we can note that
$$g^2_\mathbb{S}(\D)({\cal X},{\cal X}')=\dis\frac{1}{2}(g(\D{\cal X},{\cal X}')+g({\cal X},\D{\cal X}')).$$
 But $D_{{\cal N}_{\cal S}}g$ belongs to $\mathbb{S}(V{\cal M})$. It follows that (\ref{equivmet}) is equivalent to
$$\D\in (g^2_\mathbb{S})^{-1}[\dis\frac{1}{2}D_{{\cal N}_{\cal S}}g].$$

\end{proof}

\begin{Ex}\label{metricSconnet} (see \cite{LPo3})${}$\\
Assume that
 $\cal S$ satisfies the assumptions of Proposition \ref{existS}.
Given a local coordinate system $(x^i,y^\a)$ and an associated  basis $\{{\cal X}_\a, {\cal V}_\b\}$ such that ${\cal J}{\cal X}_\a={\cal V}_\a$ and ${\cal S}=y^\a{\cal X}_\a+{\cal S}^\b{\cal V}_\b$,  let $\{{\cal X}^\a, {\cal V}^\b\}$ be the associated dual basis. The pseudo-Riemannian metric $g$  can be written:  $g_{\a\b}{\cal V}^\a\odot {\cal V}^\b$. Recall that   the coefficients of  ${\cal N}_{\cal S}$ are then
$${\cal N}_\a^\b=\dis\frac{1}{2}(-\frac{\p {\cal S}^\b}{\p y^\a}+C_{\a\d}^\b y^\d).$$
Thus we have
$$D_{{\cal N}_{\cal S}}({\cal V}_\a)=-\dis\frac{1}{2}(\frac{\p {\cal S}^\b}{\p y^\a}+C_{\a\d}^\b y^\d)){\cal V}_\b$$
\begin{eqnarray}\label{DGSg}
g_{\a\b}=D_{{\cal N}_{\cal S}}g({\cal V}_\a,{\cal V}_\b)={\cal L}^{\cal P}_{\cal S}(g_{\a\b})+\dis\frac{1}{2}[(\frac{\p {\cal S}^d}{\p y^\a}+C_{\a\epsilon}^Cy^\epsilon)g_{\b\d}+(\frac{\p {\cal S}^\d}{\p y^\b}+C_{\b\epsilon}^\d y^\epsilon)g_{\a\d}].
\end{eqnarray}
The $g$-metric connection associated with ${\cal N}_{\cal S}$ which is given in Proposition \ref{connectmet} has the following local coefficients:
$$\dis\frac{1}{2}(-\frac{\p {\cal S}^\b}{\p y^\a}+C_{\a\d}^\b y^\d)-\dis\frac{1}{2}g^{\a\d}g_{\d\b/}.$$
\end{Ex}

%%%%%%%%%%%%%%%%%%%%%%%%%%%%%%%%%%%%%%%%%%%%%%%%%%%%%%%%%%%%%%%%%%%%%%%%%%%%%%
%%%%%%%%%%%%%%%%%%%%%%%%%%%%%%%%%%%%%%%%%%%%%%%%%%%%%%%%%%%%%%%%%%%%%%%%%%%%%%
\subsection{Lagrangian metric connections and   semisprays }\label{LAMeconn}${}$\\
%%%%%%%%%%%%%%%%%%%%%%%%%%%%%%%%%%%%%%%%%%%%%%%%%%%%%%%%%%%%%%%%%%%%%%%%%%%%%%
%%%%%%%%%%%%%%%%%%%%%%%%%%%%%%%%%%%%%%%%%%%%%%%%%%%%%%%%%%%%%%%%%%%%%%%%%%%%%%%%
${}\;\;\;\;$ {\it In this subsection, we assume that we have an almost cotangent  structure $\O$ on $\T{\cal M}$  which is compatible with the almost tangent structure $\cal J$ on $\T{\cal M}$ associated with  ${\cal S}$.}\\

We consider a canonical  basis $\{{\cal X}_\a,{\cal V}_\b\}$  of $\T{\cal M}$ which has the property (iii) of Proposition \ref{alldyn} and its  corresponding  dual basis $\{{\cal X}^\a,{\cal V}^\b\}$. According to (\ref{O}),  and Proposition \ref{TO}, we can write
\begin{eqnarray}\label{decompO}
\O=\dis\frac{1}{2}\o_{\a\b}{\cal X}^\a\wedge{\cal X}^\b+g_{\a\b}{\cal X}^\a\wedge{\cal V}^\b
\end{eqnarray}
where the matrix of general terms $\o_{\a\b}$ and  $g_{\a\b}$ are respectively  antisymmetric  and  symmetric. Moreover we get a pseudo-Riemannian metric $g$ on $\V{\cal M}$ defined by
$$g_\O({\cal X},{\cal Y})=\O({\cal Z},{\cal Y})$$
where ${\cal J}{\cal Z}={\cal X}$.  In particular we obtain
$g_\O({\cal V}_\a,{\cal V}_\b)=g_{\a\b}.$

On the other hand, we will say that a {\bf nonlinear connection ${\cal N}$ is Lagrangian }   if its horizontal space is Lagrangian.
%Clearly, this condition equivalent to the relation $$i_{\cal N}\O\equiv 0$$
Now we have:

\begin{The}\label{Lagmet}${}$\\
Let $\cal J$ and $\O$ be an almost tangent structure and an almost cotangent structure respectively  on $\T{\cal M}$ which are compatible and $\cal S$ a semispray.
\begin{enumerate}
\item[(i)] There exists a unique connection ${\cal N}$ which is Lagrangian (relative to $\O$) and $g_\O$-metric (relatively to $\cal S$) characterized by the following relation  :
$${\cal N}={\cal N}_{\cal S }-2(g^2)^{-1}[D_{{\cal N}_{\cal S}}g+\O({\cal H}_{\cal S},{\cal H}_{\cal S})]{\cal J}.$$
According to the previous notations,  in a canonical  basis $\{{\cal X}_\a,{\cal V}_\b\}$  of $\T{\cal M}$ which has the property (\ref{baseJSD})  the local coefficients of  ${\cal N}$ are characterized by:
\begin{equation}\label{decompolagmet}
g_{\b\g}{\cal N}_\a^\g=-\dis\frac{1}{2}[\o_{\a\b}+{\cal L}^{\cal P}_{\cal S}(g_{\a\b})+g_{\a\g}\frac{\p{\cal S}^\g}{\p y^\b}+g_{\b\g}\frac{\p{\cal S}^\g}{\p y^\a}]
\end{equation}
\item[(ii)]
The canonical connection ${\cal N}_S$   is Lagrangian (relative to $\O$) and $g_\O$-metric (relatively to $\cal S$) if and only if,  in an adapted canonical basis  associated with a coordinate system compatible with the locally linear fibration structure associated with $\cal C$, the  semispray $\cal S$ satisfies
\begin{eqnarray}\label{gmetrisS}
\o_{\a\b}+{\cal L}^{\cal P}_{\cal S}(g_{\a\b})+g_{\a\d}\frac{\p{\cal S}^\d}{\p y^\b}+g_{\b\g}{C}_{\a\mu}^\d y^\mu=0.
\end{eqnarray}
In this case, it is the unique nonlinear connection which is  Lagrangian  and $g_\O$-metric
\end{enumerate}
\end{The}

\begin{proof}[proof of Theorem \ref{Lagmet}]${}$\\
We simply denote  $g_\O$ by $g$. Since  we impose  that ${\cal N}$ is $g$-metric, from Proposition \ref{connectmet}, we must have
${\cal N}={\cal N}_{\cal S}+2 \D{\cal J}$ for some $\D$ in $(g^2_\mathbb{S})^{-1}[-D_{{\cal N}_{\cal S}}g]$.  Moreover, we have
$$h_{\cal N}=h_{{\cal N}_{\cal S}}+\D{\cal J}.$$ The connection ${\cal N}$ will be Lagrangian if and only if
$$\O(h_{\cal N}({\cal X}), h_{\cal N}({\cal X}'))=0$$
for all sections $\cal X$ and ${\cal X}'$ of $\T{\cal M}$.
This relation is then equivalent to:
\begin{eqnarray}\label{lagmet2}
\O(h_{{\cal N}_{\cal S}}({\cal X}), h_{{\cal N}_{\cal S}}({\cal X}'))=-[\O(\D{\cal J}({\cal X}), h_{{\cal N}_{\cal S}}({\cal X}'))\hfill\;\;\;\;\;\;\;\;\;\;\;\;\;\;\;\;\;\;\;\;\;\;\;\;\;\;\;\;\;\;\;\;\;\;\;\;\;\;\;\;\;\;\;\;{\;\;\;\;\;\;\;\;\;\;\;\;\;\;\;\;\;\;}\nonumber\\
{}\;\;\;\;\;\;\;\;\;\;\;\;\;\;\;\;\;\;\;\;\;\;\;\;\;\;\;\;\;\;\;\;\;\;\;\;\;\;\;\;\;\;\;\;\;\;\;\;\;\;\;\;\;\;\;\;\;\;\;\;\;\;\;\;\;\;\;\;\;\;\;\;\hfill+\O(h_{{\cal N}_{\cal S}}({\cal X}), \D{\cal J}({\cal X}'))+\O(\D{\cal J}({\cal X}), \D{\cal J}({\cal X}'))].
\end{eqnarray}

But, as $\D$ is an endomorphism of $\V{\cal M}$ the last term in the previous relation is zero and, on the other hand we have:

$\O(\D{\cal J}({\cal X}), h_{{\cal N}_{\cal S}}({\cal X}'))-\O(\D{\cal J}({\cal X}'),h_{{\cal N}_{\cal S}}({\cal X}))=g(\D{\cal J}({\cal X}), {\cal J}h_{{\cal N}_{\cal S}}({\cal X}'))-g(\D{\cal J}({\cal X}'),{\cal J}h_{{\cal N}_{\cal S}}({\cal X})$

$\;\;\;\;\;\;\;\;\;\;\;\;\;\;\;\;\;\;\;\;\;\;\;\;\;\;\;\;\;\;\;\;\;\;\;\;\;\;\;\;\;\;\;\;\;\;\;\;\;\;\;\;\;\;\;\;\;\;\;\;\;\;\;\;\;\;\;\;\;\;=g(\D{\cal J}({\cal X}), {\cal J}({\cal X}'))-g(\D{\cal J}({\cal X}'),{\cal J}({\cal X}))$

$\;\;\;\;\;\;\;\;\;\;\;\;\;\;\;\;\;\;\;\;\;\;\;\;\;\;\;\;\;\;\;\;\;\;\;\;\;\;\;\;\;\;\;\;\;\;\;\;\;\;\;\;\;\;\;\;\;\;\;\;\;\;\;\;\;\;\;\;\;\;=2 g^2_{\mathbb{A}}(\D)({\cal J}({\cal X}),{\cal J}({\cal X}'))$

(According to the decomposition (\ref{decompE})).

Now, from Part (iii) of  Proposition \ref{JG}  we have an isomorphism ${\cal H}_{\cal S}$ from $\V{\cal M}$ to the horizontal bundle of ${\cal N}_{\cal S}$ which is the inverse of the restriction of $\cal J$ to this bundle.

It follows that (\ref{lagmet2}) is equivalent to:
$$\O( {\cal H}_{\cal S}({\cal X}),{\cal H}_{\cal S}({\cal X}'))=-2 g^2_{\mathbb{A}}(\D)({\cal X}),{\cal X}')$$
for any vertical section ${\cal X}$ and ${\cal X}'$.

Finally, as  $g^2(\D)$ must satify the unique decomposition
\begin{eqnarray}\label{lagmet3}
g^2(\D)=g_{\mathbb{S}}^2(\D)+g_{\mathbb{A}}^2(\D)=-\dis\frac{1}{2}[D_{{\cal N}_{\cal S}}g+\O({ \cal H}_{\cal S},{ \cal H}_{\cal S})].
\end{eqnarray}
Since $g^2$ is an isomorphism, it follows that ${\cal N}={\cal N}_{\cal S}+2\D{\cal J}$ must be the unique nonlinear connection $g_\O$-metric and Lagrangian.

Assume that ${\cal N}$ is $g_\O$-metric and Lagrangian. Therefore, in a canonical basis $\{{\cal X}_\a,{\cal V}_\b\}$, we must have
\begin{eqnarray}\label{DGg01}
\O(h_{\cal N}({\cal X}_\a),h_{\cal N}({\cal X}_\b))=\o_{\a\b}-g_{\a\g}{\cal N}_\b^\g+g_{\b\g}{\cal N}_\a^\g=0.
\end{eqnarray}
On the other hand we also have:
\begin{eqnarray}\label{Oh0}
D_{\cal N} g({\cal V}_\a,{\cal V}_\b)={\cal L}^{\cal P}_{\cal S}(g_{\a\b})+ g_{\a\g}(({\cal N}_\b^\g+\dis\frac{\p {\cal S}^\g}{\p y^\b})+g_{\b\g}(({\cal N}_\a^\g+\dis\frac{\p {\cal S}^\g}{\p y^\a})=0.
\end{eqnarray}
By addition of these last  relations, we obtain  the expression  of local  (\ref{decompolagmet}) given in the Theorem which ends the proof of point (i).\\

In a local basis which satisfies (\ref{adaptJS}),  according to (\ref{locGS}),
 the local coefficients of ${\cal N}_S$ are:
\begin{eqnarray}\label{GSloc}
{\cal N}_\a^\b=\dis\frac{1}{2}(-\frac{\p {\cal S}^\b}{\p y^\a}+{C}_{\a\d}^\b y^\d).
\end{eqnarray}

From (\ref{decompolagmet}) we easily obtain the relation (\ref{gmetrisS}), which ends the proof of Part (ii).\\

\end{proof}

\begin{Rem}\label{symG} ${}$\\
% the arguments used in the proof of Theorem 3.5 and Theorem 3.6 of \cite{LPo3}. \\
 Consider  a canonical basis $\{{\cal X}_\a,{\cal V}_\b\}$      of $\T{\cal M}$ (as in point (ii) of Proposition \ref{alldyn})
 and  set ${\cal N}_{\a\b}=g_{\a\g}{\cal N}_\b^\g$. The symmetric  and  antisymmetric part of ${\cal N}_{\a\b}$ are respectively $\dis\frac{1}{2}({\cal N}_{\a\b}+{\cal N}_{\b\a})$ of ${\cal N}_{\a\b}$ and $\dis\frac{1}{2}({\cal N}_{\a\b}-{\cal N}_{\b\a})$ and  will be denoted ${\cal N}_{\a\b}^\mathbb{S}$ and  ${\cal N}_{\a\b}^\mathbb{A}$. \\

According to (\ref{DGg01}),  the equation $\O(h_{{\cal N}},h_{{\cal N}}) \equiv 0$ is equivalent to
 \begin{eqnarray}\label{GamaA}
{\cal N}_{\a\b}^\mathbb{A}=\dis\frac{1}{2}\o_{\a\b}.
\end{eqnarray}

According to  (\ref{Oh0}), the equation $D_{{\cal N}}g=0$ is equivalent to
 \begin{eqnarray}\label{GamaS}
{\cal N}_{\a\b}^\mathbb{S}=-\dis\frac{1}{2}[{\cal L}^{\cal P}_{\cal S}(g_{\a\b})+ g_{\a\g}\dis\frac{\p {\cal S}^\g}{\p y^\b}+g_{\b\g}\dis\frac{\p {\cal S}^\g}{\p y^\a}].
\end{eqnarray}
Moreover, from (\ref{Oh0}) and (\ref{GSloc})    the symmetric parts of the local coefficients of  ${\cal N}_{\cal S}$ are given by

\begin{eqnarray}\label{GSS}
{\cal N}_{\a\b}^\mathbb{S}=\dis\frac{1}{2}[{\cal L}^{\cal P}_{\cal S}(g_{\a\b})+( g_{\a\d}{\a}^\mu_{\b\mu}+g_{\b\d}{C}^\d_{\a\mu})y^\mu].
\end{eqnarray}\\
\end{Rem}

%%%%%%%%%%%%%%%%%%%%%%%%%%%%%%%%%%%%%%%%%%%%%%%%%%%%%%%%%%%%%%%%%%%%%%%%%
\subsection{Lagrangian connection and semi-Hamiltonian almost cotangent structure }\label{lagsemiH}${}$\\
%%%%%%%%%%%%%%%%%%%%%%%%%%%%%%%%%%%%%%%%%%%%%%%%%%%%%%%%%%%%%%%%%%%%%%%%%
Consider an almost tangent structure $\cal J$   on ${\T}{M}$ compatible with the semispray $\cal S$. Consider a semi-Hamiltonian almost cotangent structure $\O$ on ${\T}{\cal M}$. From Proposition \ref{propsemiH} it follows that  $({\cal M},{\cal J},\O)$ is locally Lagrangian.  Therefore around any point $m\in{\cal M}$  we have
$\O=-d^{\cal P}{\cal J}^*d^{\cal P}L$ for some local function $L$.

\begin{The}\label{SLagMet}${}$\\
  The connection  ${\cal N}_{\cal S}$ is a Lagrangian and $g_{\O}$-metric  connection if and only, around any point $m\in{\cal M}$,   there exists a semi-basic $1$-form $\xi$
 such that $d^{\cal P}\xi$ is semi-basic and which satisfies
  $$i_{\cal S}\O=d^{\cal P}({\cal L}^{\cal P}_{\cal C}(L)-L)+\xi$$
  where $L$ is a local function around $m$ such that $\O=-d^{\cal P}{\cal J}^*d^{\cal P}L$.\\
  In this case,  ${\cal N}_{\cal S}$ is the unique  $g_{\O}$-metric and  Lagrangian connection.
\end{The}

\begin{Def}\label{SlocLag}${}$\\
Let  $\O$ be a semi-Hamiltonian almost structure on ${\T}{\cal M}$. An almost  semispray $\cal S$ on $\cal M$ is called  locally Lagrangian if  we have $i_{\cal S}\O=d^{\cal P}({\cal L}^{\cal P}_{\cal C}(L)-L)+\xi$ where $L$ is a local function  such that $\O=-d^{\cal P}{\cal J}^*d^{\cal P}L$ and $\xi$ and $d^{\cal P}\xi$  are semi-basic.\\
\end{Def}
%According to this Definition,  part (i) of  Theorem \ref{SLagMet} can be reformulated in the following version:

%{\it The connection  ${\cal N}_{\cal S}$ is $g_{\O}$-metric and  Lagrangian connection if an only if

\bigskip
\begin{Exs}\label{metricLagS}${}$

{\bf 1.} Consider  a regular Lagrangian $L$ on $\cal M$,   $\O_L=-d^{\cal P}{\cal J}^*(d^{\cal P}L)$  the associated almost cotangent structure and $\cal S$ the associated   semispray, then we get that the corresponding connection ${\cal N}_{\cal S}$ is the unique   $g_{\O_L} $-metric and Lagrangian connection. When $\cal M$ is a Lie algebroid, we recover the essential result of \cite{LPo2}.\\

 {\bf 2.} Let  $\eta$  be a regular  semi-Hamiltonian on ${\cal A}^*$. The canonical symplectic form $\O=-d^{\cal P}\theta$ where $\theta$ is the canonical Liouville  form is a semi-Hamiltonian almost structure on ${\T}{\cal A}^*$. Note that  around any point $m\in{\cal M}$ we can write $\O=-d^{\cal P}{\cal J}^*d^{\cal P}L$ with $L(x,y)=\dis\frac{1}{2}(y^\a)^2$. Denote by  $\cal S$  the  semispray defined by $i_{\cal S}\O=\eta$. Then, the connection ${\cal N}_{\cal S}$ is $g_\O$ metric and Lagrangian, if and only, around each  $m\in {\cal M}$, there exists  local semi-basic $1$-form $\xi$ such that $d^{\cal P}\xi$ is semi-basic  such that locally:
$$\eta=d^{\cal P}({\cal L}^{\cal P}_{\cal C}(L)-L)+\xi.$$
%where $\cal Z$ is the vertical lift of $Z$.\\

 {\bf 3.} More generally,    consider  a regular  semi-Hamiltonian $\eta$ on ${\cal A}^*$, $\O$  a semi-Hamitonian almost structure on ${\T}{\cal A}^*$ , and $\cal S$ the associated  semispray. Around any point $m\in {\cal A}^*$ we write $\O=-d^{\cal P}{\cal J}^*d^{\cal P}L$. Then ${\cal N}_{\cal S}$ is $g_\O$ metric and Lagrangian if and only if, around $m$,  there exists a local semi-basic $1$-form $\xi$ such that $d^{\cal P}\xi$ is semi-basic  such that locally:
 $$\eta=d^{\cal P}({\cal L}^{\cal P}_{\cal C}(L)-L)+\xi.$$\\
\end{Exs}

\begin{proof}[Proof of theorem \ref{SLagMet}]${}$\\
%As $\O$ is semi-Hamiltonian, then  $\O$ is exact and    from Theorem \ref{S-coS}, there exists an integrable almost tangent  structure ${\cal J}$ on ${\T}{\cal M}$  which is compatible with $\O$ and such that  $g ( \;,\; ) = \O( \;, {\cal J })$. So $({\cal M},{\cal J},\O)$ is locally Lagrangian, and there
According to (\ref{semihamO} )  we have a local coordinates system $(x^i,y^\a)$,  an associated  basis $\{{\cal X}_\a,{\cal V}_\b\}$ and its  dual basis  $\{{\cal X}^\a,{\cal V}^\b\}$, such that  :

$\bullet\;\;{\cal J}{\cal X}_\a={\cal V}_\a$, and \\

$\bullet\;\;\O=\dis\frac{\p^2 L }{\p y^\a \p y^\b}{\cal X}^\a \wedge {\cal V}^\b+\dis\frac{1}{2}(\dis\frac{\p^2 L }{\p x^i \p y^\a}\rho^i_\b-\dis\frac{\p^2 L }{\p x^i\p y^\b}\rho^i_\a+\dis\frac{\p L }{\p  y^\g}C^\g_{\a\b}){\cal X}^\a\wedge{\cal X}^\b$.\\

We will use the arguments of the proof of Theorem 3.5 and Theorem 3.6 of \cite{LPo2}. \\

As in Remark \ref{symG}, we consider the symmetric ${\cal N}_{\a\b}^\mathbb{S}$ (resp. antisymmetric ${\cal N}_{\a\b}^\mathbb{A}$)) part   of ${\cal N}_{\a\b}$. Then
taking into account the particular value of $\o_{\a\b}$, the Equations (\ref{GamaA}) and (\ref{GSS}) are equivalent to
\begin{eqnarray}\label{GL}
{\cal N}_{\a\b}=\dis\frac{1}{2}(\dis\frac{\p^2 L }{\p x^i \p y^\a}\rho^i_\b-\dis\frac{\p^2 L }{\p x^i\p y^\b}\rho^i_\a+\dis\frac{\p L }{\p  y^\g}C^\g_{\a\b}+{\cal L}^{\cal P}_{\cal S}(g_{\a\b})+g_{\b\g}C_{\a\d}^\g y^\d+g_{\a\g}C_{\b\d}^\g y^\d)
\end{eqnarray}
where $g_{\a\b}=\dis\frac{\p^2 L}{\p y^\a \p y^\b}$.\\

On the other  hand, as classically,  the solution of the equation $i_{\bar{\cal S}}\O=d^{P}({\cal L}^{P}_{\cal C}(L)-L)$ is characterized by (see \cite{CLMM})
\begin{eqnarray}\label{iSO}
\begin{cases}
\bar{\cal S}=y^\a{\cal X}_\a+\bar{\cal S}^\b{\cal V}_\b \cr
g_{\a\b}\bar{\cal S}^\b+\dis\frac{\p^2 L}{\p x^i\p y^\a} \rho_\d ^i y^\d+\frac{\p L}{\p y^\g}C_{\a\d}^\g y^\d-\rho_\a^i\frac{\p L}{\p x^i}=0.\cr
\end{cases}
\end{eqnarray}

Therefore if we denote by $\bar{{\cal N}}_\a^\b$  the local coefficients  of ${\cal N}_{\bar{\cal S}}$  then $\bar{{\cal N}}_{\a\b}=g_{\a\g}\bar{{\cal N}}_\a^\g$ are exactly the second member of (\ref{GL}).  This implies that  (\ref {GL}) is equivalent to
$$\dis\frac{\p{\cal S}^\b}{\p y^\a}=\dis\frac{\p\bar{\cal S}^\b}{\p y^\a}.$$
This last relation is equivalent to the existence of a local  function $\phi^\b$ on $ M$ such that ${\cal S}^\b=\bar{\cal S}^\b+\phi^\b$. If we define  ${\cal Z}=-\phi^\b{\cal V}_\b$ then we have
$$i_{\cal S}\O=d^{\cal P}({\cal L}_{\cal C}^{\cal P}L- L)+{\O}^\flat({\cal Z}).$$

\noindent Thus we have

$\xi={\O}^\flat({\cal Z})=\dis\frac{\p^2 L }{\p y^\a \p y^\b}\phi^\b{\cal X}^\a$ for all $\a,\b=1,\cdots,k$.

\noindent First of all, note that $\xi$ is semi-basic. On the other hand, from  (\ref{iSO}) we must have ${\cal N}_{\cal S}={\cal N}_{\bar{\cal S}}$ if and only if

${\cal L}^{\cal P}_{\cal Z}(g_{\a\b})=0$ for all $\a,\b=1,\cdots,k$.

\noindent In fact, this last relation can be written:

$\dis\phi^\g\frac{\p}{\p y^\g}\left(\frac{\p^2 L }{\p y^\a \p y^\b}\right)=0$ for all $\a,\b=1,\cdots,k$.

\noindent But $d^{\cal P}\xi$ is semi basic if and only if

$\dis\frac{\p }{\p y^\b}\left(\dis\frac{\p^2 L }{\p y^\a \p y^\g}\phi^\g\right){\cal V}^\b\wedge{\cal X}^\a=\dis\frac{\p^3 L }{\p y^\a  \p y^\b \p y^\g}\phi^\g{\cal V}^\b\wedge{\cal X}^\a\equiv 0$ for all $\a,\b=1,\cdots,k$.

Now,  in this case, from Theorem \ref{Lagmet},  ${\cal N}_{\cal S}$ is the unique $g_\O$-metric and Lagrangian connection.\\

%Assume that $\cal S$ is a spray. According to the previous proof, ${\cal N}_{\cal S}$  is $g_\O$-metric and Lagrangian, if there exists a local  function $\phi^\b$ on $ M$ such that ${\cal S}^\b=\bar{\cal S}^\b+\phi^\b$ where $\bar{\cal S}$  satisfies (\ref{iSO}). But we have
%$${\cal S}^\b=S_{\a\g}^\b(x)y^\a y^\g$$
%So we must have $\phi^\b=0$  which ends the proof of part (iv).\\

\end{proof}

%%%%%%%%%%%%%%%%%%%%%%%%%%%%%%%%%%%%%%%%%%%%%%%%%%%%%%%%%%%%%%%%%%%%%%%%%%%%%%%%%%%%%
%%%%%%%%%%%%%%%%%%%%%%%%%%%%%%%%%%%%%%%%%%%%%%%%%%%%%%%%%%%%%%%%%%%%%%%%%%%%%%%%%%%%%
\subsection{Lagrangian connection and canonical semispray }\label{ConsLag}${}$\\
%%%%%%%%%%%%%%%%%%%%%%%%%%%%%%%%%%%%%%%%%%%%%%%%%%%%%%%%%%%%%%%%%%%%%%%%%%%%%%%%%%%%%
${}\;\;\;\;$ { We have seen in subsection \ref{GC} that to each nonlinear  connection ${\cal N}$ on $\T{\cal M}$ is associated a unique   semispray  $\cal S$ which is horizontal. But in general the canonical  semispray  associated with a connection of type ${\cal N}_{\cal S}$ is not $\cal S$. Given some   semispray  $\cal S$,  we will at first  look for condition under which we can associate a  Lagrangian nonlinear  connection ${\cal N}$ whose canonical  semispray  is precisely $\cal S$. When $\cal S$ is the  semispray associated with some semi-hamitonian, we will give a characterization of such connection.}\\

We consider an almost tangent  structure ${\cal J}$ and  an almost cotangent structure   and  $\O$ on $\T{\cal M}$ respectively which are compatible and we denote by $g$ the associated  pseudo-Riemannian  metric.

\begin{The}\label{exisLagScan} ${}$\\
   Let  $\cal S$ be a  semispray  on $\cal M$.
\begin{enumerate}
\item[(i)] The set  ${\cal M}_0$ of points where $g({\cal C},{\cal C})\not=0$ is an open dense subset of $\cal M$ whose complementary is at least a co-dimensional  one subset. On $\T{\cal M}_{| {\cal M}_0}$,   there exists a canonical Lagrangiannonlinear   connection ${\cal N}$ whose  canonical    semispray  is $\cal S$ which characterized by

if  $\Upsilon={\cal N}-{\cal N}_{\cal S}$ then
\begin{eqnarray}\label{canconnlag}
g(\Upsilon{\cal X},{\cal J}{\cal Y})=\O({\cal H}_{\cal S}({\cal J}{\cal X}),{\cal H}_{\cal S}({\cal J}{\cal Y}))+i_{\cal C}\O \odot \o({\cal X},{\cal Y})
\end{eqnarray}

where $\o$ is given by (\ref{ocanoique}) and $\odot$ is the symmetric product.
\item[(ii)] If  ${\cal N}$ is a Lagrangian nonlinear connection,  the following properties are equivalent:
\begin{enumerate}
\item[(a)] $\cal S$ is the canonical  semispray    of ${\cal N}$
\item[(b)] there exists a $1$-form $\eta$ such that $i_{\cal S}\O-\eta$ is semi-basic and  $i_{h_{\cal N}}\eta=0$ where  $h_{\cal N}$ is the horizontal projector associated with  ${\cal N}$.
\item[(c)] $i_{h_{\cal N}}(i_{\cal S}\O)=0$
\end{enumerate}
\end{enumerate}
In any one of the previous situations ${\Upsilon}={\cal N}-{\cal N}_{\cal S}$ is the strong torsion of ${\cal N}$
\end{The}
\begin{Rem}\label{lagrangemetric}${}$\\
Classically, given a Lagrangian $L$ on a manifold $M$, we get a canonical symplectic form $\O$ on $TM$ and  $\cal S$ the    semispray  associated with $L$ \textit{i.e.}  $i_{\cal S}\O=d{\cal H}$ where ${\cal H}={\cal L}_{\cal C}L-L$, then  a  nonlinear  connection on $TM$ is called {\it conservative} if ${\cal N}$ is Lagrangian and
$$i_{h_{\cal N}}d{\cal H}=0.$$
In our case, the equivalence (a) and (c) can be seen as a generalization of such a result in our context.\\
\end{Rem}
\begin{Def}\label{cons}${}$\\
Let  $\O$ be an almost cotangent structure on $\T{\cal M}$ and $\eta$ a semi-Hamitonian on $\cal M$. A  nonlinear  connection ${\cal N}$  on $\T{\cal M}$ is called conservative (for $\eta$) if ${\cal N}$ is Lagangian and if $i_{h_{\cal N}}(\eta)=0$.\\
\end{Def}

\begin{Ex}\label{sprayHamilton}${}$\\
Assume that $\O$ is a semi-hamitonian almost structure on  ${\T}{\cal M}$ and consider a spray $S$   on $\T{\cal M}$. Then $\cal S$ is the canonical spray of ${\cal N}_{\cal S}$ (see Proposition \ref{Hprop}). On the other hand ${\cal N}_{\cal S}$ is $g_\O$-metric and Lagrangian if and only if $\cal S$ is locally Lagrangian (see Theorem \ref{SLagMet} and Definition \ref{SlocLag}). It follows that ${\cal N}_{\cal S}$ is the unique  $g_\O$-metric  Lagrangian connection whose canonical spray is $\cal S$. In fact, ${\cal N}_{\cal S}$ is conservative. We get a natural  generalization  of classical results about Lagrangian spray on the tangent bundle  (in particular Finslerian manifold) as we have seen in Remark \ref{lagrangemetric}.
\end{Ex}

\begin{proof}[Proof of Theorem \ref{exisLagScan}]${}$\\
We will use the context of subsection \ref{DynDer}. Consider the associated nonlinear  connection ${\cal N}_{\cal S}$. According to (\ref{lagmet2}), the      nonlinear  connection  $\tilde{{\cal N}}={\cal N}_{\cal S}- \D{\cal J}$ where $g_{\mathbb{A}}^2(\D)=\O({\cal H}_{\cal S},{\cal H}_{\cal S})$ is Lagrangian.

Now we  must find a semi-basic vector valued $1$-form $\Upsilon$ such that ${\cal N}=\tilde{\cal N}+\Upsilon$ satisfies
(1) ${\cal N}({\cal S})={\cal S}$, and

(2) if $D_{\cal N}$ is the dynamical derivation associated with ${\cal N}$ (relatively to $\cal S$), then $g^2({\cal N}-{\cal N}_{\cal S})$ belongs to $\mathbb{S}(V{\cal M})$.\\

Assume that we have built such a $\Upsilon$.  On the one hand,  from Property  (2) and  according to the proof of Theorem \ref{Lagmet}, then ${\cal N}$ is also Lagrangian.  On the other hand, from Property  (1),  ${\cal S}$ is the canonical   semispray  of ${\cal N}$ which  ends  the construction of ${\cal N}$ in Part (i).\\

We will prove the two previous properties ( 1) and (2).\\

Property (1) is  equivalent to $\Upsilon({\cal S})={\cal S}-\tilde{{\cal N}}({\cal S})=2v_{\tilde{{\cal N}}}({\cal S})={\cal S}^*$.

From Proposition \ref{alldyn}, there exists a unique $\V{\cal M}$-tensor $\bar{\D}$ such that $\Upsilon=\bar{\D}\circ{\cal J}$.
Assume that we have   found  a symmetric $2$-form $\Theta$ such that

$\Theta({\cal X},{\cal Y})=g(\Upsilon {\cal X},{\cal J }{\cal Y})$ and   $\Theta({\cal S},{\cal Y})=g({\cal S}^*,{\cal J}{\cal Y})$.\\
We use the same arguments as in \cite{Gr}, chapter II section 8. Therefore  Properties (1) and (2) will be true if there exists a $2$-form   $\Theta$ of type

$$\Theta({\cal X},{\cal Y})=i_{\cal C}\O\odot \o$$ where $\o$ is a  semi-basic scalar $1$-form and $\odot $ is the  symmetric product fulfilling the condition

$(i_{\cal C}\O\odot \o)({\cal S},{\cal Y})=g({\cal S}^*,{\cal J}{\cal Y}).$\\

 \noindent This last condition is equivalent to

 $i_{\cal C}\O({\cal S}) \o({\cal Y})+i_{\cal C}\O({\cal Y}) \o({\cal S})=i_{{\cal S}^*}\O({\cal Y}),$

  \noindent and finally we get the equivalent condition:

\begin{eqnarray}\label{lastcond}
g({\cal C},{\cal C})\o=\o({\cal S})i_{\cal C}\O -i_{{\cal S}^*}\O.
\end{eqnarray}

\noindent On the other hand,  the relation(\ref{lastcond}) also gives:

$2 g({\cal C},{\cal C})\o({\cal S})=-i_{{\cal S}^*}\O({\cal S})=g({\cal C},{\cal S}^*)$

Finally, to prove Properties (1) and (2), it remains to find such a semi-basic $1$-form $\o$ on ${\cal M}_0$.\\

\noindent Now  outside the subset $\S=\{m\in{\cal M} \textrm{ such that } g({\cal C}(m),{\cal C}(m))=0\}$ we can define
\begin{eqnarray}\label{ocanoique}
\o=\dis\frac{1}{g({\cal C},{\cal C})}[-i_{{\cal S}^*}\O+\dis\frac{g({\cal C},{\cal S}^*)}{2 g({\cal C},{\cal C})}i_{\cal C}\O].
\end{eqnarray}

\noindent It follows that $\o$ is well defined and  moreover,  since $\cal C$ and ${\cal S}^*$ are  vertical sections and  since $V{\cal M}$ is Lagrangian, it follows that  $\o$ is semi-basic.\\

Moreover note that since ${\Upsilon}={\cal N}-{\cal N}_{\cal S}$,  by construction, we have ${\Upsilon}({\cal S})={\cal S}-{\cal N}_{\cal S}({\cal S})=\bf{H}_{\cal S}$ and so from Theorem \ref{same S},  ${\Upsilon}$ is the strong torsion of ${\cal N}$. \\

\noindent In order to end the proof of Part (i) in Theorem \ref{exisLagScan}  we must show that $\S={\cal M}\setminus{\cal M}_0$ is locally contained in a submanifold of codimension at most $1$. But locally
we have ${g({\cal C},{\cal C})}=g_{\a\b}y^\a y^\b$. Therefore the local equation of $\S$ is $ g_{\a\b}y^\a y^\b=0$. As the matrix $(g_{\a\b})$ is invertible we can easily see that   the differential of this equation must not be identically zero on $\S$.\\

 In the context of  Point (ii), assume   that $\cal S$ is the canonical  semispray   of a Lagrangian connection ${\cal N}$ and denote simply by $h$  the horizontal projector of ${\cal N}$. Then we have:

$i_h(i_{\cal S}\O)({\cal X})=\O({\cal S}, h{\cal X})$ but from our assumption, $h({\cal S})={\cal S}$  and  ${\cal N}$ is Lagrangian, so $i_h(i_{\cal S}\O)=0$.\\
\noindent Now assume that  $i_h(i_{\cal S}\O)=0$. Then,  for $\eta=i_{\cal S}\O$, trivially  $i_{\cal S}\O - \eta$  is semi-basic and by assumption $i_{h}\eta=0$.

\noindent Now assume that there exists a $1$-form $\eta$ such that $i_{\cal S}\O-\eta$ is semi-basic and  $i_{h}\eta=0$. Since $i_{\cal S}\O-\eta$ is semi-basic we have

$i_h(i_{\cal S}\O-\eta)=i_{\cal S}\O-\eta$.

\noindent But ${\cal N}$ is Lagrangian thus, if $v$ is the vertical projector associated with ${\cal N}$, we have   $i_{\cal S}\O-\eta=\O(v{\cal S},\;)-\eta$. On the other hand, since  $i_h\eta=0$ we obtain
$\O(v{\cal S},({\cal X})-\eta({\cal X})= \O(v{\cal S},v{\cal X})-\eta{\cal X}+\O(v{\cal S},h{\cal X})=\O(v{\cal S},{\cal X})$.

\noindent It follows that $\eta=0$ and then $i_h(i_{\cal S}\O)=0$.

\noindent Finally, it remains to show that if  $i_h(i_{\cal S}\O)=0$ then $h{\cal S}={\cal S}$.

\noindent But  we have seen that $i_{\cal S}\O=\O(v{\cal S},\;)$ so if $i_h(i_{\cal S}\O)=0$. This implies  that

$\O(v{\cal S},h{\cal X})=-g(v{\cal S},{\cal J}{\cal X})=0$ for any section $\cal X$. But as $g$ is non degenerate, it follows that $v{\cal S}=0$.\\

\end{proof}

%%%%%%%%%%%%%%%%%%%%%%%%%%%%%%%%%%%%%%%%%%%%%
%%%%%%%%%%%%%%%%%%%%%%%%%%%%%%%%%%%%%%%%%%%%%%%%%%%%%%%%%%%%%%%%%%%%%%%%%%%
\subsection{Application to mechanical systems}\label{NHmeca}${}$\\
%%%%%%%%%%%%%%%%%%%%%%%%%%%%%%%%%%%%%%%%%%%%%%%%%%%%%%%%%%%%%%%%%%%%%%%%%%%%%%%%%%%%%%%
%%%%%%%%%%%%%%%%%%%%%%%%%%%%%%%%%%%%%%%%%%%%%%%%%%%%%%%%%%%%%%%%%%%%%%%%%%%%%%%%%%%%%%%
%${}\;\;\;\;$  {\it In this subsection we adapt to our context, the concept  mechanical system  (see for instance \cite{Go} or  \cite{Mi}).}\\
%\subsection{$\cal A$-mechanical systems}\label{NHmeca}${}$\\
%%%%%%%%%%%%%%%%%%%%%%%%%%%%%%%%%%%%%%%%%%%%%%%%%%%%%%%%%%%%%%%%%%%%%%%%%%%%%%%%%%%%%%%
%%%%%%%%%%%%%%%%%%%%%%%%%%%%%%%%%%%%%%%%%%%%%%%%%%%%%%%%%%%%%%%%%%%%%%%%%%%%%%%%%%%%%%%
%${}\;\;\;\;$  {\it In this subsection we adapt to our context, the concept  mechanical system  (see for instance \cite{Go} or  \cite{Mi}).}\\
When ${\cal M}=TM$, classically a mechanical system    is a triple $(M,L,\o)$ where $L: TM\ap \R$ is a regular Lagrangian whose "vertical" hessian is positive definite and $\o$ is a semi-basic $1$-form. Then the $1$-form  $\theta=d({\cal L}_CL-L)+\o$ is a semi-Hamitonian. Conversely, if $\theta$ is a semi-hamitonian, such that its associated pseudo-riemannian metric is Riemannian,  there exists a (global) smooth function $L$ and a semi-basic $1$-form  $\o$ such that $\theta=d({\cal L}_CL-L)+\o$ (cf \cite{Mi}). In this subsection we adapt the concept of mechanical system to our context.\\

Consider a locally affine fibration $\pi:{\cal M}\ap M$. First of all recall that a regular Lagrangian on $\cal M$ is a smooth map $L:{\cal M}\ap \R$ such that, around each point $m\in {\cal M}$, there exists a coordinate system $(x^i,y^\a)$ (compatible with the locally affine fibration structure) such that the Hessian $(\dis\frac{\p^2 L}{\p y^\a \p y^\b} )$ is an invertible matrix. In these condition, we get a pseudo-Riemannian metric $g_L$ on the vertical bundle $\V{\cal M}$. %We consider an  algebroid $( {\cal A},M,\rho,[\;,\;]_{\cal A})$ of rank $k$,  and we assume that  dim ${\cal M}=n+k$.

\begin{Def}\label{mecha}${}$\\
  %Let be $( {\cal A},M,\rho,[\;,\;]_P)$ an almost algebroid.
  A {\it ${\cal A}$-mechanical system } is a  triple $({\cal M},L,\o)$ such that:
\begin{enumerate}
\item[-]    $\pi:{\cal M}\ap M$ is a locally affine  fibration;
%\item[-]  there exists on $\cal M$ an  is an  Euler section $\cal C$ on ${\cal M}$   %which is a non-zero section of $V{\T}{\cal M}$;
\item[-] $L$ is a regular Lagrangian on $\cal M$ (according to the locally affine fibration structure on $\cal M$)  such that the associated metric $g_L$ is a Riemannian metric;
\item[-] $\o$ is a semi-basic $1$-form.
%that  is  $\o=(\pi^2)^*\varepsilon$ where $\varepsilon$ is a section of ${\cal A}^*$. \\
\end{enumerate}
\end{Def}
\begin{Exs}\label{Exmecha}${}$\\
{\bf 1.} If ${\cal M}$ is the tangent bundle of a manifold $M$   the previous definition gives rise to the classical definition of a mechanical system on $M$ (see \cite{Go} or  \cite{Mi}).

{\bf 2.} Assume that  ${\cal M}$ is the bundle ${\cal A} $, or more generally an open submanifold of ${\cal A} $ provided with its natural linear fibration structure. Now, given any Riemannian metric $g$ on $\cal A$ we get  ${\cal A}$-mechanical system $({\cal M}, g, 0)$.

{\bf 3.}  Given any $1$-homogeneous Lagrangian $L:\cal A\ap \R^+$  such that $L(x,u)=0$ if and only $u=0$  and such that $L^2$ is a non-degenerate Lagrangian, then we get a $\cal A$-mechanical system $({\cal A}, L^2,0) $,  which corresponds to the situation of singular  "sub-Finslerian manifold"  studied in  \cite{HPo}(see also  section \ref{Partialfins}).

{\bf 4.} According to the Example \ref{JOS6},   the manifold ${\cal M}=\mathbb{S}^{n-1}\times\mathbb{S}^{2n+1}$ is  a fibered manifold $\pi:{\cal M}\ap \mathbb{S}^{2n}$. A positive $2$-homogeneous  Lagrangian on $\cal M$ gives rise to a $\cal A$-mechanical system $({\cal M}, L,0) $ (see \cite{Va3}).  Other examples of this type on almost tangent manifold can be also found in  \cite{Va3}. In the same way when $\pi:\cal M\ap M$ is a Lagrangian submersion,   the relation between Lagrangian foliations and Riemaniann foliations in \cite{PoPo2}  can be seen as a $\cal A$-mechanical system $({\cal M}, L,0) $ problem.

.% In particular, if $g$ is a Riemannian metric on   $\cal A$  the case  $L(x,u)=\sqrt{g_x(u,u)}$ is a $1$ homogeneous  Lagrangian  and we get a mechanical system mechanical system $(\stackrel{ \circ }{\cal A}, L^2,0) $ corresponding  the situation of singular sub-Riemannian manifold (see  subsection \ref{subfinsler}).

%{\bf 4.} Consider a $2$-homogeneous hamiltonian  on ${\cal A}^*$  which is nondegenerate. Then the sphere bundle\\ $\mathbb{S}{\cal H}=\{(x,u)\textrm{ such that } {\cal H}(x,u)=1\} $ is diffeomorphic to the homogenous bundle $\mathbb{H}{\cal A}$ of $TM$. Then the mechanical system $(\mathbb{S}{\cal H},{\cal H}_{| \mathbb{S}{\cal H}},0)$ induces a naturally a $\cal A$-mechanical system $(\mathbb{\cal H}{\cal A}, \hat{\cal H},0)$.
\end{Exs}

To  a $\cal A$-mechanical system $({\cal M},L,\o)$  we can associate the $1$-form $\theta= d^{\cal P}({\cal L}^{\cal P}_{\cal C}L-L)+\o$ which   is a  regular semi-Hamitonian. When ${\cal A}={\cal M}$, by same arguments as used in Theorem 4 in \cite{Mi} we can show that the converse is true.  \\

From now on, assume moreover that there exists  an integrable almost structure $\cal J$ on  ${\T}{\cal M}$ which is compatible with an Euler section $\cal C$ such that ${\cal C}\not=0$ on $\cal M$. In general  if the  typical  fiber of $\pi:{\cal M}\ap M$ is not simply connected, for any semi-Hamitonian $\theta$,  locally, there exists a smooth function $L$ such that $\theta= d^{\cal P}({\cal L}^{\cal P}_{\cal C}L-L)+\o$  but not globally.  It follows that any semi-Hamiltonian $\theta$ on $\cal M$ gives rises to a "locally" ${\cal A}$-mechanical system $(U,L,\o)$ where $\o=\theta- \left(d^{\cal P}({\cal L}^{\cal P}_{\cal C}L-L)\right)$ is semi-basic (see Theorem 4 in \cite{Mi} ). Therefore we introduce:

\begin{Def}\label{geneAmeca}${}$\\
An $\cal A$-generalized mechanical system is a pair $({\cal M},\theta)$ where $\theta$ is a regular semi-Hamitonian (relative to the  locally affine fibration structure on $\cal M$  defined by $\cal J$) and such that  its associated  metric $g_\theta$ is riemannian.
\end{Def}

 We consider an {\it  fixed }  $\cal A$-generalized mechanical system $({\cal M},\theta)$.\\
Under the previous assumptions,
the $2$-form
 $\O=-d^{\cal P}({\cal J}^*\theta)$
is a semi-Hamitonian almost cotangent structure, compatible with $\cal J$ ({\it cf.} Proposition \ref{locLagch}). Moreover, ${\cal J}^*\theta$ is a generalized Liouville form ({\it cf.} subsection \ref{cal E}). According to Proposition \ref{propsemiH},  $({\cal M},{\cal J},\O)$ is  locally Lagrangian and moreover we have
$g_L({\cal J},\;)=\O(\;,\;)$  (see  (\ref{gJL})).\\
Again, from Proposition \ref{propsemiH},   we  get a well defined    semispray  $\cal S$ so that $i_{\cal S}\O={\cal L}^{\cal P}_{\cal C}\theta-\theta$  and ${\cal J}$ is compatible with $\cal S$ and   we have ${\cal J}{\cal S}={\cal C}$.\\
Now, according to Theorem \ref{S-coS}, there exists a local coordinate system $(x^i,y^\a)$  and an  local basis $\{{\cal X}_\a,{\cal V}_\b\}$  of ${\T}{\cal M}$ such that

${\cal J}{\cal X}_\a={\cal V}_\a,\;\;$  $\;\;{\cal C}=y^\a{\cal V}_\a\;\;$ and ${\cal S}=y^\a{\cal X}_\a+ {\cal S}^\b{\cal V}_\b$.

%with  ${\cal S}^\b$ characterized by (cf (\ref{iSO})):

%$\dis\frac{\p^2L}{\p y^\a \p y^\b}{\cal S}^\b+\dis\frac{\p^2 L}{\p x^i\p y^\a} \rho_\d ^i y^\d+\frac{\p L}{\p y^\g}C_{\a\d}^\g y^\d-\rho_\a^i\frac{\p L}{\p x^i}-\o_\a$
 and, if  locally $\theta=\theta_\a{\cal X}^\a+\dis\frac{\p {L}}{\p y^\b}{\cal V}^\b$ we have the relation (cf (\ref{Let-et}))
 \begin{eqnarray}\label{EuLageq}
\dis\frac{\p^2{L}}{\p y^\a \p y^\b}{\cal S}^\b+y^\d(\dis\frac{\p \theta_\a}{\p y^\d}-\dis\frac{\p^2 {L}}{\p x^i\p y^\d} \rho_\a ^i+\dis\frac{\p^2 {L}}{\p x^i\p y^\a} \rho_\d ^i-\frac{\p {L}}{\p y^\g}C_{\a\b}^\g)=\theta_\a.
\end{eqnarray}

%${\cal S}=\dis\frac{\p {\cal H}}{\p y^A}{\cal Y}_A+(\dis\frac{1}{2}\o_{AB}\dis\frac{\p {\cal H}}{\p y^A}-\theta_B){\cal V}_B$
When ${\cal M}={\cal A}$  and $(M,L,\o)$ is a classical mechanical system for $\theta=d^{\cal P}L+\o$, the previous relation are  the classical Euler-lagrange equations.
   Therefore  each integral curves of the vector field $\hat{\rho}({\cal S})$ is locally a solution of a differential system  of type (\ref{eqgeod})
   \begin{Def}\label{semSEleq}${}$\\
   ${\cal S}$ is called the  semispray associated with $({\cal M},\theta)$ and the equations (\ref{EuLageq}) are called the Euler Lagrange Equations of $({\cal M},\theta)$
   \end{Def}
   We then have the following results:
\begin{The}\label{mecasyst}${}$\\
 Let  $({\cal M},\theta)$ be a $\cal A$-generalized mechanical system and $\cal S$ its associated  semispray.
\begin{enumerate}
\item[(i)] the nonlinear connection ${\cal N}_S$ is the unique $g_L$-metric and Lagrangian connection if and only if, around any $m\in{\cal M}$,  there exists  a local semi-basic  $1$-form $\o$ such that $d^{\cal P}\o$ is semi-basic and locally we have:
$$\theta=d^{\cal P}({\cal L}^{\cal P}_{\cal C}(L)-L)+\o.$$
\item[(ii)] there exists a canonical conservative connection ${\cal N}_0$ for ${\cal L}^{\cal P}_{\cal C}\theta-\theta$. In particular, the geodesics of ${\cal N}_0$ are the solutions of the Euler Lagrange equations of $({\cal M},\theta)$.
\item[(iii)] Assume that  $ \theta=d^{\cal P}L$ where $L$ is a $2$-homogeneous regular Lagrangian. Then ${\cal N}_{\cal S}$ is the unique $g_L$-metric and conservative connection relative to $\cal S$ and we  have ${\cal N}_{\cal S}={\cal N}_0$
\end{enumerate}
\end{The}
\begin{proof}[Proof of Theorem \ref{mecasyst}]${}$\\
As $\O$ is semi-Hamitonian, point (i) is a consequence of Theorem \ref{SLagMet}.\\
Point (ii) is a consequence of Point (i) in  Theorem \ref{exisLagScan}. \\
Point (iii) is also a consequence of Theorem \ref{SLagMet} (see Example \ref{sprayHamilton}).\\
\end{proof}

     %%%%%%%%%%%%%%%%%%%%%%%%%%%%%%%%%%%%%%%%%%%%%%%%%%%%%%%%%%%%%%%%%%%%%%%%%%%%%%%%%%%
\section{Extremals and connection associated with partial  hyperregular Lagrangian on a foliated anchored bundle}\label{hyperlag}
%%%%%%%%%%%%%%%%%%%%%%%%%%%%%%%%%%%%%%%%%%%%%
%%%%%%%%%%%%%%%%%%%%%%%%%%%%%%%%%%%%%%%%%%%%%%%%%%%%%%%%%%%%%%%%%%%%%%%%%%%
The  first purpose of this section is  to give characterizations of local minimizers of a partial convex   hyperregular Lagrangian. The second purpose is to give  a link between, on the one hand,  the canonical Lagrangian metric nonlinear connection associated with a partial  hyperregular Lagrangian and a pre-Lie algebroid structure, and, on the other hand,  the induced structures on each leaf of the foliation defined by the anchored bundle (see Theorem \ref{extremalN}).
%%%%%%%%%%%%%%%%%%%%%%%%%%%%%%%%%%%%%%%%%%%%%%%%%%
%\section { Canonical connection and extremal of a Partial hyperregular Lagrangian on a foliated anchored bundle }\label{extremales}
%%%%%%%%%%%%%%%%%%%%%%%%%%%%%%%%%%%%%%%%%%%%%%%%%%
%%%%%%%%%%%%%%%%%%%%%%%%%%%%%%%%%%%%%%%%%%%%%%%%%%%%%%%%%%%%%%%%%%%%%%%%
\subsection{ Extremals of   a  partial convex hyperregular Lagrangian on a foliated anchored bundle  }\label{normalext}${}$\\
%%%%%%%%%%%%%%%%%%%%%%%%%%%%%%%%%%%%%%%%%%%%%%%%%%%%%%%%%%%%%%%%%%%%%%%%
\noindent Consider an anchored bundle  $({\cal A},M,\rho)$.

\begin{Def}\label{partL}${}$\\
 A  {\bf partial  Lagrangian} on $({\cal A},M,\rho)$ is a function $L: {\cal M}\ap \R$ defined on
an open manifold $\cal M$ of $\cal A$ which is fibered on $M$.% such that  the associated metric $g_L$ is Riemannian.  \\
\end{Def}

%\begin{Exs}\label{partL}${}$
%\begin{enumerate}
%\item[(1)] Any  Lagrangian $L$ on $A$ is a partial Lagrangian defined on ${\cal M}={\cal A}$. In particular this situation occurs for any Riemannian metric on $\cal A$
%\item[(2)] If the   metric $g_L$ associated with a Lagrangian $L$ on ${\cal A}$  is pseudo-Riemannian, then the set of point ${\cal M}_+$ and ${\cal M}_-$ on which $g_L$ is positive definite and negative definite are open subsets of $\cal A$. Then the restriction of $L$ to ${\cal M}_+$ and of $-L$ on ${\cal M}_-$ are partial Lagrangian on $\cal A$.
%\item[(3)]  A finsler metric $F$ on $\cal A$  (see \cite{HPo} or subsection \ref{subfinrap}) gives rise to a partial Lagrangian $L=\dis\frac{1}{2}F^2$ on the complementary ${\cal M}$ of the zero section of $\cal A$.\\
%\end{enumerate}
%\end{Exs}

 If $\pi:{\cal M}\ap M$ is the projection,  given any $x_0\in M$, there exists a coordinate system $(x^i, y^\a)$ on $\cal A$ associated with a local basis $\{e_\a\}$ on an open neighbourhood $V$ of $x_0$ in $M$ such that $(x^i,y^\a)$ induces a coordinate system on $\pi^{-1}(U)$. Moreover this coordinate system induces a diffeomorphism from   $\pi^{-1}(U)$ onto an open set of type $U\times V\subset\R^n\times \R^k$
\begin{Def}\label{Madmiss}${}$\\
An absolutely continuous curve $c$ from an interval $I$ of $\R$ to M is called ${\cal M}$-admissible if there exists a curve $\bar{c}:I\ap {\cal M}$  such that
$c(t)=\pi (\bar{c}(t)) \textrm{ and } \dot{c}(t)=\rho(\bar{c}(t) \textit{ a.e. } \forall t\in I.$
The curve $\bar{c}$ will be called a $\cal M$-lift of $c$
\end{Def}

\begin{Rem}\label{charactMadmiss}${}$\\
A  $H^1$-curve \footnote{\textit{i.e} the curve $\dot{\bar{c}} :I\ap \T{\cal M}$ is absolutely continuous} $\bar{c}:I\ap{\cal M}$ gives rise to a ${\cal M}$-admissible curve $c=\pi\circ \bar{c}$ if and only if $\bar{c}$ is  an admissible absolutely continuous  curve ${\cal M}$  such that $\bar{c}$ is also an admissible curve according to the anchored bundle $(\T{\cal M},{\cal M},\hat{\rho})$.
\end{Rem}

Given two points $x_0$ et $x_1$ in $M$ we denote by
${\cal C}(x_0,x_1,{\cal M})$ the set of all ${\cal M}$-admissible $H^1$ curves $c:[a,b]\ap M$ which join $x_0$ to $x_1$ \textit{i.e.}  such that $c(a)=x_0$ and $c(b)=x_1$.\\

A first problem is to find conditions under which ${\cal C}(x_0,x_1, {\cal M})\not=\emptyset$.\\

\noindent According to the famous  Sussmann's accessibility Theorem, we have a foliation on $\cal M$ such  any two points of $\cal M$ can be joined by an admissible curve (relative to $\T{\cal M}$) if and only if they belongs to the same leaf of this foliation. Therefore Remark \ref{charactMadmiss} implies that ${\cal C}(x_0,x_1, {\cal M})\not=\emptyset$ if and only if $x_0$ and $x_1$ belongs to the projection by $\pi$ of some leaf of this foliation.\\

Assume now that $({\cal A},M,\rho)$ is a foliated anchored bundle. According to Proposition \ref{foliatedquasi} and Proposition \ref{prolonfolit} the foliation given is Sussmann's accessibility Theorem in $\cal M$ is the foliation defined by the integrable distribution $\hat{\rho}({\T{\cal M}})$. Any leaf of this foliation is of type $\pi^{-1}(N)$ where $N$ is a leaf of the foliation defined by $\rho({\cal A})$. Therefore ${\cal C}(x_0,x_1, {\cal M})\not=\emptyset$ if and only if  $x_0$ and $x_1$ belongs to a same integral manifold of $\rho({\cal A})$.

%{\it In the sequel we assume that ${\cal C}(x_0,x_1, {\cal M})\not=\emptyset$}.\\

Now  we consider a  partial Lagrangian  $L:{\cal M} \ap \R$.  A classical optimal problem is to find the $\cal M$-admissible curves which minimizes the functional
   ${\cal I}_L(\bar{c})=\dis \int_a^b
L(\bar{c}(t))dt $:
\begin{equation}\label{pbopti}
 \inf \{ {\cal I}_L(\bar{c}) \;|\;\pi\circ \bar{c} \in {\cal C}(x_0,x_1, {\cal M})\}
\end{equation}
First of all notice that  the value   ${\cal I}_L(\g)$  is invariant
by translation on the parametrization of $\bar{c}$. We will denote by  ${\cal
C}(x_0,x_1, {\cal M}, T)$ the set  $\cal M$-admissible curves of class $H^1$ which are defined on an interval  of type $[a,a+T]$ and we will look for
the problem of minimization of ${\cal I}_L$  on the set   ${\cal
C}(x_0, x_1, {\cal M},T)$. Assume that  ${\cal I}_L$ has a  minimum on ${\cal C}(x_0,x_1,{\cal M},T)$.

\begin{Def}\label{minimisante} ${\;}$
\begin{enumerate}
\item A  curve ${c}\in {\cal C}(x_0,x_1,{\cal M},T)$  is called a minimizing curve or simply a {\bf minimizer}  in time $T$  if there exists a ${\cal M}$-lift  $\bar{c}$ of $c$ such that the minimum of  ${\cal I}_L$  on ${\cal C}(x_0,x_1,{\cal M},T)$ is ${\cal I}_L(\bar{c})$
\item We will say that   ${c}\in {\cal C}(x_0,x_1,{\cal M}, T)$, defined on an interval   $[a,a+T]$ is
a {\bf local minimizer} for the functional ${\cal I}_L$ if for any $t_0\in I$, there exists an open $ U$ in $ M$  and a $\cal M$-lift $\bar{c}$ such that $\bar{c}(t_0)$ belongs to $\pi^{-1}( U)$ and  for any interval
 $[t_1,t_2]\subset [a,a+T]$ such that $\bar{c}(t)\in \pi^{-1}({ U})$ for  $t \in [t_1,t_2]$, then the
restriction of $c_{|  [t_1,t_2]}$ is a minimizer in  ${\cal C}(\g(t_1),\g(t_2), {\cal U},t_2-t_1)$.
\end{enumerate}
\end{Def}
\indent It is clear that a minimizer is a local minimizer. Therefore it is natural to look for sufficient conditions under which a curve in ${\cal C}(x_0,x_1,{\cal M},T)$ is a local minimizer. Since the  problem of finding local minimizers  is a local problem, without loss of generality we may assume that $ U$ is a chart domain in $\cal M$ compatible with the fibration $\t: {\cal A}\ap M$.
Consider a canonical coordinate system $(x^i,y^\a)$ on $\cal M$  naturally associated with a local basis $\{e_\a\}$ of $\cal A$ on the chart domain $U$ of the coordinate system $(x^i)$ in $M$.
 With these notations, our problem  gives rise to the following optimal control problem:\\

 Find a curve $c$ which  is a solution of the differential system $\dot{x}^i=\rho^1_\a y^\a, i=1\dots,n$ which minimizes the functional $\dis \int_{t_1}^{t_2}
L(x(s),y(s))ds$.

 Therefore, by application of the Maximum Principle,  for $\n\in \R$ consider the Hamiltonian:

$H_\n:\pi^{-1}(U)\times_U T¨^*U\equiv U\times \R^k\times \R^n\ap \R$ defined by
\begin{eqnarray}\label{hamilPMP}
H_\n(x,y,\xi)=\xi_j\rho_\a^j y^\a -\n L(x,y).
\end{eqnarray}
  If a $\cal M$-admissible curve $c:[a,a+T]\ap U$ is minimizing  then there exists a lift $\xi:[0,T]\ap \R^n$ over $c$  and $\n\geq 0$ such that:

  \begin{eqnarray}\label{eqhamiltonienne}
\left\{\begin{array} [c]{l}
\dot{c}(t)=\dis\frac{\partial H_{\nu}}{\partial p}(c(t),y(t),\xi(t))\\
\dot{\xi}(t)=-\dis\frac{\partial H_{\nu}}{\partial x}(c(t),y(t),\xi(t))\\
H_\n(c(t),y(t),\xi(t))=\sup\{H_\n(c(t),z,\xi(t))\;|\; z\in V\} \textrm{ if } \pi^{-1}(U)\equiv U\times V.
\end{array}
\right.
\end{eqnarray}\\
Taking into account  the Equation (\ref{hamilPMP}), this condition is equivalent to:\\
$(c,y,\xi)$ is a solution of the constrained  differential system
 \begin{eqnarray}\label{eqhamiltonienne2}
\left\{\begin{array} [c]{l}
\dot{x}^i=\rho_\a^i y^\a\\
\dot{\xi}_i=- \xi_j\dis\frac{\partial \rho_\a^j}{\partial x^i}y^\a+\n\dis\frac{\partial L}{\partial x^i}(x^i,y^\a)\\
\rho_\b^j\xi_j=\n \dis\frac{\partial L}{\partial y^\b}(x^i,y^\a).\\
\end{array}
\right.  \label{eqhamiltonienne}
\end{eqnarray}\\

 On the other hand any curve $(c, y,\xi): [0,T]\ap \pi^{-1}(U)\times _UT^*U$  which satisfies Equations
        (\ref{eqhamiltonienne2})  is called a {\bf bi-extremal}. According to the value of $\n$
we obtain two  types  of bi-extremals:
\begin{enumerate}
\item if  $\n=0$ but  $\xi(t)\not=0$, the  bi-extremal $(c,y,\xi)$ called  \textbf{ abnormal bi-extremal }
 (such cuves depends only on $\cal M$).
\item If $\nu\not=0  $ (which can be normalized by $\n=1)$ the  bi-extremal $(c,y, \xi)$
is called  \textbf{ normal bi-extremal}.
\end{enumerate}
\begin{Def}${}$\label{biextremale}${}$\\
Consider   a $\cal M$-admissible curve $c :[0,T]\ap M$:\\
(1) we say that   $c$ is a normal extremal if for any $t_0\in [0,T] $ there exists a coordinate system $(x^i,y^\a)$  in $\cal M$ over   chart domain $U$ around $c(t_0)$ and a lift  $\check{c}=(c,y,\xi)$ of $c$  in ${\cal M}\times_M\ T^*M$ over $U$
such that  $\check{c}$ is a  normal bi-extremal;\\
(2) We say that  $c$ is   a strictly abnormal extremal for any $t_0\in [0,T] $, if for any coordinate system $(x^i,y^\a)$  in $\cal M$ over   chart domain $U$ around $c(t_0)$, and any lift  $\check{c}=(c, y,\xi)$ of $c$ in ${\cal M}\times_M\ T^*M$ over $U$ which is a solution of the differential system
(\ref{eqhamiltonienne}),  then $\check{c}$ is an abnormal  bi-extremal.
\end{Def}

\begin{Rem}\label{strictanor}${}$\\
 A local minimizer  $c$ gives rise to a
  bi-extremal $(\bar{c}, \xi)$, but it is well known in control theory
 that there exists local minimizer which are strictly abnormal extremals  (see for instance \cite{LS}). However, we will see that in our context  there is no strictly abnormal  local minimizer (see Remark \ref{noabnor}).\\
\end{Rem}
  A direct consequence of  the constrained differential system (\ref{eqhamiltonienne2}) is
\begin{Lem}\label{abnorm}${}$
 \begin{enumerate}
 \item A bi-extremal $(c,y,\xi)$ is  abnormal  if an only if  $\rho^*(c,\xi)=0$.
 \item Given a normal bi-extremal $(c,y,\xi)$ then $(c,y,\xi+\mu)$ is also a normal bi-extremal if an only if $(c,y,\mu)$ is an abnormal bi-extremal.
 \end{enumerate}
 \end{Lem}

%%%%%%%%%%%%%%%%%%%%%%%%%%%%%%%%%%%%%%%%%%%%%%%%%%%%%%%%%%%%%%%%%%%%%%%%
\subsection{Hamiltonian approach of  extremals for a partial convex hyperregular Lagrangian}\label{normalext}
%%%%%%%%%%%%%%%%%%%%%%%%%%%%%%%%%%%%%%%%%%%%%%%%%%%%%%%%%%%%%%%%%%%%%%%%
%We will now look for extremal of a particular type of partial Lagrangian:
%\begin{Def}\label{partL}${}$\\
 %A  {\bf partial convex  Lagrangian} on $({\cal A},M,\rho)$ is a partial Lagrangian  $L: {\cal M}\ap \R$  such that  the associated metric $g_L$ is Riemannian.  \\
%\end{Def}
%\begin{Exs}\label{partL}${}$
%\begin{enumerate}
%\item[(1)] Any  Lagrangian $L$ on $A$ is a partial Lagrangian defined on ${\cal M}={\cal A}$. In particular this situation occurs for any Riemannian metric on $\cal A$
%\item[(2)] If the   metric $g_L$ associated with a Lagrangian $L$ on ${\cal A}$  is pseudo-Riemannian, then the set of point ${\cal M}_+$ and ${\cal M}_-$ on which $g_L$ is positive definite and negative definite are open subsets of $\cal A$. Then the restriction of $L$ to ${\cal M}_+$ and of $-L$ on ${\cal M}_-$ are partial Lagrangian on $\cal A$.
%\item[(3)]  A finsler metric $F$ on $\cal A$  (see \cite{HPo} or subsection \ref{subfinrap}) gives rise to a partial Lagrangian $L=\dis\frac{1}{2}F^2$ on the complementary ${\cal M}$ of the zero section of $\cal A$.\\
%\end{enumerate}
%\end{Exs}

We fix some partial %convex
Lagrangian $L:{\cal M}\ap \R$.  We consider  local coordinates $(x^i,\xi_j)$ on $T^*M$,  $(x^i,y^\a)$ on $\cal M$, and  $(x^i,\eta_\a)$ on ${\cal A}^*$ over a chart domain of the coordinate system  $(x^i)$ on $M$.  We denote by $\L_L:{\cal M}\ap {\cal A}^*$ the associated Legendre transformation {\it i.e.}   locally $\L_L$ is $\eta_\a=\dis\frac{\p L}{\p y^\a}$ in local canonical coordinates .

Choose an almost Lie bracket $[\;,\;]_{\cal A}$ such that $({\cal A},M,\rho,[\;,\;]_{\cal A})$ is a pre-Lie algebroid. If $J$ is the vertical endomorphism on $ \cal A$,  recall that on $\T{\cal M}$ we have a canonical symplectic form $\O_L=-d^{\cal P}J^*(dL)$ which only depends on $[\;,\;]_{\cal A}$. We also have a  semispray ${\cal S}_L$ associated with $\theta$ characterized by $i_{{\cal S}_L}\O_L=d^{\cal P}({\cal L}^P_CL-L)$ where $C$ is the canonical Euler section induced on $\cal M$.  (see Equation (\ref{iSO}))
  \begin{eqnarray}\label{SL}
\begin{cases}
{\cal S}_L=y^\a{\cal X}_\a+\bar{\cal S}^\b{\cal V}_\b \cr
g_{\a\b}{\cal S}^\b+\dis\frac{\p^2 L}{\p x^i\p y^\a} \rho_\d ^i y^\d+\frac{\p L}{\p y^\g}C_{\a\d}^\g y^\d-\rho_\a^i\frac{\p L}{\p x^i}=0\cr
\end{cases}
\end{eqnarray}
where $g_{\a\b}=\dis\frac{\p^2 L}{\p y^\a\p y^\b}$. In particular, $S_L$ depends only on $[\;,\;]_{\cal A}$ too.\\

\begin{Def}\label{hypereg}${}$
\begin{enumerate}
 \item A   partial hyperregular  Lagrangian on $({\cal A},M,\rho)$ is a partial Lagrangian  $L: {\cal M}\ap \R$  such that the Legendre transformation $\L_L$ is injective.
 \item A   partial convex  Lagrangian on $({\cal A},M,\rho)$ is a partial Lagrangian  $L: {\cal M}\ap \R$  such that the associated metric $g_L$ is Riemannian.  \\
  \end{enumerate}
\end{Def}

Moreover we  assume that $L$ is  hyperregular. Therefore the range ${\cal M}^{*}=\L_{L}({\cal M})$ is an open submanifold of ${\cal A}^{*}$ and $\L_{L}$ is a diffeomorphism from $\cal M$ to ${\cal M}^{*}$.  Denote by   $H^{\a}$ the local components of $\mathfrak{ H}_{L}=\L_{L}^{-1}$ \textit{i.e.}
  $\dis\frac{\p L}{\p y^\a}(x^i,H^{\a}(x^{i},\eta_{\b}))=\eta_{\b}$. Let   $\rho^*:T^*M\ap {\cal A}^*$ be  the transpose morphism of $\rho: {\cal A}\ap TM$.

Now  consider the bundle $\check{p}_*:{\cal M}\times_M T^*M\ap M$. Note that  $\check{p}_*:{\cal M}\times_MT^*M\ap {\cal M}$ is also the pull-back of ${p}_*:T^*M\ap M$ over $\pi:{\cal M}\ap M$ and let $\check{\pi} :{\cal M}\times_MT^*M \ap T^*M$ be  the canonical morphism which is an isomorphism in restriction to each fiber. Finally   the map $(x,y,\xi)\ap (x,\xi)$ is also  fibration of $\check{\pi}_*:{\cal M}\times_MT^*M\ap T^*M$.
 We consider the map $\mathfrak{g}:T^{*}M\ap {\cal M}\times_MT^*M$ defined by

$\mathfrak{g}(x,\xi)=(x,\mathfrak{ H}_{L}(x,\rho^*(x,\xi),\xi),$ and ${\cal G}=\left\{(x,y,\xi)\in :{\cal M}\times_MT^*M\;|\; \rho^*(x,\xi)=\L_L(x,y)\right\}$

\begin{Pro}\label{cal G}${}$\\
$\cal G$ is a submanifold of ${\cal M}\times_MT^*M$ of dimension $2n$ and $ \mathfrak{g}$ is a diffeomorphism from $T^*M$ to $\cal G$. We have $\check{p}_*({\cal G})=\L_L^{-1}\circ\rho^*(T^*M)$. \end{Pro}

\begin{proof} ${}$\\
Considering  the fixed  local coordinate system, we get a coordinate system $(x^i,y^\a,\xi_j)$ on ${\cal M}\times_M T^*M$. It follows that locally $\cal G$ is defined by the equations:
\begin{eqnarray}\label{eqSigma}
\rho_\a^j\xi_j-\dis\frac{\p L}{\p y^\a}(x^i,y^\a)=0\;\; \a=1,\dots,k.
\end{eqnarray}
Since $L$ is a regular Lagrangian, the  previous equations are independent. %From the implicit function Theorem, it follows also that the restriction $since $\mathfrak{g}_*(\tilde{X}_{H_L})=X_{{H}^*_L}$ from Theorem \ref{locmin} part (ii)_*$ to $\G$ is a fibration on $\cal M$. \\
Locally we have $\mathfrak{g}(x^i,\xi_j)=(x^i,\xi_j,H^\a(x^i,\rho_\b^j\xi_j))$. Since $\dis\frac{\p L}{\p y^\a}(x^i,H^\a(x,\eta_\b)=\eta_\a$, we have \\$\dis\frac{\p L}{\p y^\a}(x^i,H^\a(x,\rho_\b^j\xi_j)=\rho_\a^j\xi_j$. Therefore $\mathfrak{g}(T^{*}M)\subset {\cal G}$. On the other hand  if $(x,y,\xi)$ belongs to $\cal G$, we have $\mathfrak{ H}_{L}(x,\rho^*(x,\xi))=\rho^*(x,\xi)$ so  $\mathfrak{g}(T^{*}M)= {\cal G}$. Finally $\mathfrak{g}$ is clearly  injective and the rank of its Jacobian matrix  is $2n$. Therefore $\mathfrak{g}$ is a diffeomorphism. The last part is clear. \\
\end{proof}

On $ {\cal M}\times_MT^*M$ consider the Hamiltonian $\check{H}_L(x,y,\xi)=<\xi,\rho(x,y)>- L(x,y)$. Locally, we have $\check{H}_L(x^i,\xi_j,y^\a)=\rho_\a^j\xi_j y^\a-L(x^i,y^\a)$. According to the  Equation (\ref{hamilPMP})   we locally have $H_1=\check{H}_L$. Moreover, on $\cal G$ we have a symplectic form $\check{\O}$ such that $\mathfrak{g}^*\check{\O}=\O$. Let  $\check{X}_{H_L}$ be the Hamiltonian vector field associated with $(\check{H}_L)_{| {\cal G}}$. % Then we have:

\begin{The}\label{locmin}${}$\\
Consider a partial hyperregular Langrangian  $L$ on $\cal M$ and a $H^1$ curve $c:I\ap M$. Then we have the following properties:
\begin{enumerate}
\item[(i)] $c$ is a normal extremal of $L$ if and only if there exists a curve $\check{c}:I\ap {\cal G}$ such that $\check{c}$ is an integral curve of $\check{X}_{H_L}$ and $\check{p}\circ\check {c}=c$, and in particular $c$ is smooth. Such a curve will be called a ${\bf {\cal G}}${\bf -lift } of $c$.
%\item [(ii)] If $({\cal A},M,\rho,[\;,\;]_{\cal A})$ is an almost Lie algebroid which  is trivial on some open set $O$ of $M$, then $T\tilde{p}_*(\tilde{X}_{H_L})=S_L$ over $\pi^{-1}(O)$.
%  Conversely,  for any normal extremal $ c:I\ap M$ of  $L$ which is injective, there exists an open neighbourhood $O$ of $c(I)$ and an almost bracket $[\;,\;]_{\cal A}$ such that he almost Lie algebroid  $({\cal A},M,\rho,[\;,\;]_{\cal A})$  is trivial on $O$ and such that $c$ is an integral curve of the associated semispray $S_L$. In this case  we again have $T\tilde{p}_*(\tilde{X}_{H_L})=S_L$ over $\pi^{-1}(O)$.
\item[(ii)] We  have  $T\check{p}_*(\check{X}_{H_L})={\cal S}_L$ over $\check{p}_*(\cal G)$. In particular, the restriction of ${\cal S}_L$ to $\tilde{p}_*(\cal G)$ is independent of the choice of the almost  Lie bracket  $[\;,\;]_{\cal A}$ and each normal extremal is contained in one (and only one) leaf of the distribution $\rho({\cal A})$ and is a projection of an integral curve of the restriction of ${\cal S}_L$ to $\check{p}_*(\cal G)$.
%\item[(iii)] Let $N$ be a leaf of the integrable distribution $\rho({\cal A})$ and $L_N$ and  the restriction of $L$ to $\pi^{-1}(N)\subset {\cal M}={\cal M}_{| N}$. Then over
\item[(iii)] Moreover if $L$ is convex, any  integral curve $\check{c}$ of $\check{X}_{H_L}$ projects on a local minimizer of $L$.

%In fact any curve $c:I\ap M$ is a local minimizer of $L$ if and only if it is the projection of an integral curve of the restriction of ${\cal S}_L$ to $\tilde{p}_*(\cal G)$ .

\end{enumerate}
\end{The}

\begin{proof}${}$\\
For the proof of Point (i)it is a local problem.  But we have seen  that $H_1=\check{H}_L$. Therefore, by construction,  locally  $\check{\pi}(\cal G)$ is precisely   the set where $\dis\frac{\p H_\n}{\p y^\a}=0$ for $\n=1$. If follows that   if $c$ is a projection of an integral curve $\check{c}=(c,y,\xi)$ then $c$ is a normal extremal of $L$. Conversely, let $c:I\ap M$ be  a normal extremal. Therefore, for any $t_0\in I$ there exists an open neighbourhood $U$ around $c(t_0)$ and  a lift $ \check{c}=(c,y,\xi)$ which is a solution of the differential system (\ref{eqhamiltonienne2}) for $\n=1$; therefore $\check{c}$ is a local integral curve of $\check{X}_{H_L}$. % According to Lemma \ref{abnorm},

 On the other hand, according to the constrained $\dis\frac{\p H_\n}{\p y}=0$,  for $\n=1$, and using the previous notations locally we have  $H_1(x^i,\xi_j)= \rho_\a^j\xi_j H^\a(x,\rho_\b^j\xi_j)- L(x^i,H^\a(x^i,\rho_\a^j\xi_j ))$. Therefore, for $\n=1$ any solution of the constrained system (\ref{eqhamiltonienne}) is an integral curve of $\mathfrak{g}^{-1}_*\check{X}_{H_L}$. This ends the proof of Point (i).\\

For point (ii) we need the local equations of an integral curve of $\tilde{X}_{H_L}$. On the one hand, if $X_{H_1}$ is the Hamiltonian vector field of $H_1$ on $T^*M$, we have already seen that    $\mathfrak{g}_*X_{H_1}=\tilde{X}_{H_L}$. On the other hand, the integral curves of $X_{H_1}$ are locally the solutions of the following differential system:
\begin{eqnarray}\label{solH1}
\begin{cases}
\dot{x}^i=\rho_\a^i H^\a(x^i, \rho_\b^j\xi_j)\\
\dot{\xi}_i=\dis\frac{\p L}{\p x^i}(x^i,H^\a(x^i,\rho_\b^j\xi_j))- \xi_j\dis\frac{\p \rho_\a^j}{\p x^i} H^\a(x^i, \rho_\b^j\xi_j).\cr
\end{cases}.
\end{eqnarray}
From the equation $y^\a=H^\a(x^i, \rho_\b^j\xi_j)$ of $\cal G$ in ${\cal M}\times_MT^*M$ we get (see Section 1.4 in  \cite{PoPo})
\begin{eqnarray}\label{equcomp}
\dot{y}^\a=g^{\a\b}\left(\rho_\b^i\dis\frac{\p L}{\p x^i}-\dis\frac{\p^2L}{\p x^i\p y^\b}\rho_\g^i  y^\g+
\xi_i(\dis\frac{\p \rho_\b^i}{\p x^j}\rho_\g^j-\dis\frac{\p \rho_\g^i}{\p x^j}\rho_\b^j)y^\g\right).
\end{eqnarray}
Now from the choice of the local basis $\{e_\a\}$ of $\cal A$ we have $[e_\a,e_\b]_{\cal A}=C_{\a\b}^\g e_\g$. On the other hand we set
\begin{eqnarray}\label{rho1}
[\rho(e_\a),\rho(e_\b)]-\rho([e_\a,e_\b]_{\cal A})=\rho_{\a\b}^i\dis\frac{\p}{\p x^i}.
\end{eqnarray}
Therefore  we have
$\rho_{\a\b}^i=\dis\frac{\p \rho_\b^i}{\p x^j}\rho_\a^j-\dis\frac{\p \rho_\a^i}{\p x^j}\rho_\b^j-\rho_\g^i C_{\a\b}^\g$, and it follows that Equation (\ref{equcomp}) can be written ((see Section 1.4 in  \cite{PoPo}):
\begin{eqnarray}\label{eqcompC}
\dot{y}^\a=g^{\a\b}\left(\rho_\b^i\dis\frac{\p L}{\p x^i}-\dis\frac{\p^2L}{\p x^i\p y^\b}\rho_\g^i  y^\g+
\xi_i\rho_{\g\b}^i y^\g +\dis\frac{\p L}{\p y^\d} C_{\g\b}^\d y^\g\right).
\end{eqnarray}
Recall that on $\cal G$ we have $\xi_l\rho_\b^l(x^i)=\dis\frac{\p L}{\p y^\b}(x^i,y^\a)$. It follows that we obtain
\begin{eqnarray}\label{rhoL}
\dis\frac{\p \rho_\b^l}{\p x^i}\xi_l=\dis\frac{\p^2 L}{\p x^i \p y^\b}(x^i, y^\a).
\end{eqnarray}
Therefore  on $\cal G$, the integral curve of $\tilde{X}_{H_L}$ are locally the solution of the differential system
\begin{eqnarray}\label{solH-L}
\begin{cases}
\dot{x}^i=\rho_\a^i y^\a\\
\dot{\xi}_i=\dis\frac{\p L}{\p x^i}- \dis\frac{\p^2 L}{\p x^i \p y^\a}y^\a\cr
\dot{y}^\a=g^{\a\b}\left(\rho_\b^i\dis\frac{\p L}{\p x^i}-\dis\frac{\p^2L}{\p x^i\p y^\b}\rho_\g^i  y^\g+
\xi_i\rho_{\b\g}^i y^\g +\dis\frac{\p L}{\p y^\d} C_{\g\b}^\d y^\g\right).
\end{cases}.
\end{eqnarray}
\smallskip
The proof of the first part of Point (iii)   can be found in the appendix of \cite{FaPe}.

\end{proof}

%%%%%%%%%%%%%%%%%%%%%%%%%%%%%%%%%%%%%%%%%%%%%%%%%%%%%%%%%%%%%%%%%%%%%%%%%%%%%%%%%%%%%%
%%%%%%%%%%%%%%%%%%%%%%%%%%%%%%%%%%%%%%%%%%%%%%%%%%%%%%%%%%%%%%%%%%%%%%%%%%%%%%%%%%%%%%
\subsection{Extremals of  a partial convex hyperregular Lagrangian and reduction to a leaf}\label{restleaf}
%%%%%%%%%%%%%%%%%%%%%%%%%%%%%%%%%%%%%%%%%%%%%%%%%%%%%%%%%%%%%%%%%%%%%%%%%%%%%%%%%%%%%%
%%%%%%%%%%%%%%%%%%%%%%%%%%%%%%%%%%%%%%%%%%%%%%%%%%%%%%%%%%%%%%%%%%%%%%%%%%%%%%%%%%%%%
The purpose of this  Subsection is the following characterizations of the local minimizers of a partial convex hyperregular Lagrangian

\begin{The}\label{extremalN}${}$
\begin{enumerate}
\item Let $L:{\cal M}\ap \R$ be a partial convex hyperregular Lagrangian on $\cal A$. For each leaf $N$ of the foliation defined by $\rho({\cal A})$, the Lagrangian $L$ induces  a canonical partial convex  hyperregular  Lagrangian  $L_N$ defined on an open set $\bar{\cal M}_N=\rho({\cal M}_N)$ of $TN$ such that each integral curve of  $\hat{\rho}(S_L)$ projects on $M$ onto  one and only one integral curve of $S_{L_N}$ if this projection is contained in the leaf $N$.
\item Denote by ${\cal N}_{S_L}$ and ${\cal N}_{S_{L_N}}$ the nonlinear connections associated to the semisprays $S_L$ and $S_{L_N}$ respectively. Then via the projection  $\T{\cal A}_N\ap \left(\T{\cal A}_N\right)/ \left({\bf K} {\cal A}_N\oplus(J({\bf K} {\cal A}_N)\right)\equiv T(TN)$ ({\it cf.}  Proposition \ref{reducleaf}),
${\cal N}_{S_L}$ gives rise to a natural  endomorphism of $T\bar{\cal M}_N$ which is exactly ${\cal N}_{S_{L_N}}$.
\item  Assume that $\cal L$ is also convex.  Any $H^1$-curve $c:I\ap M$ is  contained in a leaf $N$ and  we have the following equivalence:
\begin{enumerate}
\item[(i)] $c$ is a normal extremal of $L$;
\item[(ii)] $c$ is the projection on $M$  of some integral curve $\bar{c}$ of $\hat{\rho}(S_L)$;
\item[(iii)] $c$ is a local minimizer of $L$;
\item[(iv)] $c$ is a local minimizer of  $L_N$;
\item[(v)] $c$ is an integral curve of $S_{L_N}$.
\end{enumerate}
Moreover in anyone of these situation, $c$ is smooth.
\end{enumerate}
\end{The}

\begin{Rem}\label{noabnor}${}$\\
The equivalence \textrm{ (i)} and   \textrm{ (iii)} in Point (3)  means that there does not exist strictly abnormal local minimizers when $ L$ is convex.
\end{Rem}

The proof of this Theorem needs the following auxiliary results which will be also used in the next Subsection:

 \begin{Lem}\label{dualham}${}$
  \begin{enumerate}
\item[(i)] The vector field $\hat{\rho}(S_L)$ is tangent to any leaf ${\cal M}_N=\pi^{-1}(N)\subset {\cal M}$ where $N$ is a maximal integral manifold of   $\rho({\cal A})$.
\item[(ii)] Consider  the canonical symplectic form $\O_{\cal A}$  on $\T{\cal A}^*$ (see Point (2) in Remark \ref{dPsbasic}). Then we have  $\O_L=\L_L^*\O_{\cal A}$. Consider the Hamiltonian    $H_L=({\cal L}^{\cal P}_CL-L)\circ \mathfrak{H}_L$  on ${\cal M}^*$ and $X_{H_L}$ its Hamiltonian field  (relative to $\O_{L}$) on ${\cal M}^*$ we also have $\T\L_L({\cal S}_L)=X_{H_L}\circ\L_L$.
\item[(iii)]  Let $\O$ be the canonical symplectic form on $T^*M$ and ${H}^*_L=H_L\circ\rho^*$. If $X_{{H}^*_L}$ is the associated hamiltonian field we have $T\rho^*(X_{{H}^*_L})=X_{H_L}\circ\rho^*$. Moreover we also have $\mathfrak{g}_*(\check{X}_{H_L})=X_{{H}^*_L}$
\item[(iv)]% The restriction of $L$ to ${\cal M}_N$ induces a natural partial Lagrangian $L_N$ on $TN$ whose whose associated Hamiltonian on $T^*N$  (as in Point (ii)) is nothing else but the projection of $H^*_L$ on $T^*N\equiv T^*N/\ker\rho^*$.
If $(TN)^0$  is the annulator of $TN$ in $T^*M_{| N}$,  the Hamiltonian $H^*_L$ induces by projection a   Hamiltonian $H^*_{L_N}$ on $T^*N\equiv T^*M/(TN)^0$ with the following property:

an ${\cal M}$-admissible  curve $\bar{c}:I\ap {\cal M}_N$ is an integral curve of  $\hat{\rho}(S_L)$  if and only if there exists a curve $\check{c}:I\ap TN$  which is an integral curve of the hamiltonian field of $H^*_{L_N}$  and such that  $\bar{c}$  and $\check{c}$ have the same projection $c:I\ap N$ on $N$.
 \end{enumerate}
 \end{Lem}
 \smallskip
\noindent Fix some leaf $N$ of the foliation defined by $\rho({\cal A})$. We set ${\cal M}_N^*={\cal A}^*_N\cap {\cal M^*}$  and $\bar{\cal M}_N=\rho({\cal M}_N)$.  Denote by $\bar{\rho}^*:T^*M_{| N}/\ker\rho^*\equiv T^*N\ap {\cal A}^*_N$ the isomorphism bundle induced by $\rho^*: T^*M_{| N}\ap{\cal A}_N^*$ and  consider the  subset  $\bar{\cal M}^*_N =(\bar{\rho}^*)^{-1}({\cal M}^*_N)$ of $T^*N$  and $\bar{\cal M}_N=\rho({\cal M}_N)$ of $TN$. We then have:

\begin{Lem}\label{reducN}${}$
\begin{enumerate}
\item[(i)]  $\check{p}_*({\cal G})\cap {\cal M}_N=\mathfrak{H}\circ\bar{\rho}^*(\bar{\cal M}^*_N)$,  $\bar{\mathfrak{H}}=\rho\circ \mathfrak{H}\circ \bar{\rho}^*$ induces a diffeomorphism from $\bar{\cal M}^*_N$ onto $\bar{\cal M}_N$  and $\rho$ induces a diffeomorphism from $\mathfrak{H}\circ\bar{\rho}^*(\bar{\cal M}^*_N)$  onto $\bar{\cal M}_N$.
%The vector field $\hat{\rho}(S_L)$ is tangent to any leaf ${\cal M}_N=\pi^{-1}(N)\subset {\cal M}$ where $N$ is a maximal integral manifold of   $\rho({\cal A})$.
\item[(ii)]  $L$ induces on $N$ a partial convex hyperregular Lagrangian $L_N:\bar{\cal M}_N\ap \R$ on $TN$ characterized by  by $L_N\circ \rho_{| \mathfrak{H}\circ\bar{\rho}^*(\bar{\cal M}^*_N)}=L_{| \mathfrak{H}\circ\bar{\rho}^*(\bar{\cal M}^*_N)}$. Moreover,  if  $S_{L_N}$ the associated semispray of $L_N$  on $\bar{\cal M}_N$  we have  $T\rho\circ\hat{\rho}({S_L}_{| \check{p}_*({\cal G})\cap {\cal M}_N})=S_{L_N}$.
\item[(iii)] An ${\cal M}$-admissible  curve $\bar{c}:I\ap {\cal M}_N$ gives rise to an integral curve $\rho(c)$  of  $\hat{\rho}(S_L)$ if and only if  the projection $\tau\circ \bar{c}$ is an integral curve of $S_{L_N}$.
%\item(iii) Each extremal of $L$ is a normal extremal.
\end{enumerate}
\end{Lem}

\begin{proof}[Proof of Lemma \ref{dualham}]${}$\\
Choose a point $m\in{\cal M}$ and let ${\cal M}_N$ the leaf of the foliation defined by  $\hat{\rho}(\T{\cal M})$  which contains $m$. Since  $\hat{\rho}(m,\;)=T_m{\cal M}_N$ then necessary $\hat{\rho}({\cal S}_L (m))$ belongs to $T_m{\cal M}_N$.\\

Since $\L_L$ is a diffeomorphism it is sufficient to prove the result locally. Consider a canonical local basis $\{{\cal X}_\a,{\cal P}^\b\}$ on $\T{\cal A}^*$  around a point $(x,\eta)\in {\cal M}^*$ associated with a coordinate system $(x^i,y^\a)$  and a local basis $\{e_\a\}$ of $\cal A$. Recall that $\O_{\cal A}= {\cal X}^\a\wedge{\cal P}_\a+\dis\frac{1}{2}\eta_\g C_{\a\b}^\a{\cal X}^\a\wedge{\cal X}^\b$  (see Point (2) in  Remark \ref{dPsbasic}).  From the local expression of $\L_L$ and $\O_L$  (see \ref{semihamO})  we have easily   $\L_L^*(\O_{\cal A})=\O_L$. Since $\L_L$ is a diffeomorphism  we get $\O_{\cal A}=  \mathfrak{H}_L^*\O_L$. But ${\cal S}_L$ is the Hamiltonian field of ${\cal L}^{\cal P}_CL-L$ according to $\O_L$. Therefore we obtain $\T\L_L({\cal S}_L)=X_{H_L}\circ\mathfrak{H}_L$. This ends the proof of  the first part of Point (ii).\\
%The last Point  is a direct consequence of the facts  that $\mathfrak{g}$ is a symplectomorphism  and that  we have $H_L\circ\mathfrak{g}=H^*_L$.\\

 Now since $({\cal A},M,\rho,[\;,\;]_{\cal A})$ is a pre-Lie algebroid, we have  $(\rho^*)(\O_L)=\O_{\cal A}$ according to Proposition \ref{characalgebroid}. Then the proof of Point (iii) ends as  the proof of the first Point (ii).\\

The annulator $(TN)^0$ of $T¨^*M_{| N}$ of $TN$ is equal to the kernel of  the restriction of $\rho^*$  to $T¨^*M_{| N}$. Since $H^*_L=H_L\circ\rho^*$, the restriction of $H^*_L$ to $(TN)^0$ is zero.  %Note that $H^*_L$ is defined  on $T^*M$.
 Therefore by projection, $H^*_L$ induces on $T^*N\equiv T^*M/(TN)^0$ a hamiltonian $H^*_{L_N}$. Now we denote by $X_{L_N}$ the hamiltonian vector field of $H^*_{L_N}$ on $T^*N$. Note that the announced property of Point (iv) is a local property.  For this proof we will use some notations used in  the proof of Proposition \ref{reducleaf}.\\
 We can choose a chart  $(U, (x^1,\cdots,x^n))$ such that $(U\cap N,(x^1,\cdots,x^q))$ is a chart for $N$ and $\cal A$ is trivializable over $U$. We can also   choose a coordinate system $(U, (x^1,\cdots,x^n))$  and a local basis $\{{e}_\a\}$ of $\cal A$ such that, over $U\cap N$ we have $\rho(e_\a)=\dis\frac{\p}{\p x^\a}$ for $\a=1,\cdots,q$ and $\rho(e_\a)=0$ for $\a=q+1,\dots,k$ .\\ Denote by $(x^1,\cdots,x^q,y^1,\cdots,y^k)$ the induced coordinates system on ${\cal U}=\tau_N^{-1}(U)$. Therefore  $\dis\frac{\p}{\p y^{q+1}},\cdots,\dis\frac{\p}{\p y^{k}}$ is a basis of $\ker T\rho$ over $\cal U$. Now let  $(x^1,\cdots,x^q,\dot{x}^1,\dots,\dot{x}^q)$ be the coordinate system on $TN$ associate to $(U\cap N,(x^1,\cdots,x^q))$. By  composition of the projection $\bar{\tau}_N:{\cal A}_N\ap {\cal A}_N/{\cal K}_N$   the  local coordinate system $(x^1,\cdots,x^n,y^1,\cdots,y^k)$ gives rise to a local coordinate system $(\bar{x}^1,\cdots,\bar{x}^q,\bar{y}^1,\cdots \bar{y}^q)$ of $ {\cal A}_N/{\cal K}_N$ on the open $\bar{U}=\bar{\tau}_N(U)$ and the local expression $\rho_N:{\cal A}_N/{\cal K}_N\ap TN$ is $x^\a=\bar{x}^\a$ and $\dot{x}^\a=\bar{y}^\a$ for $\a=1,\cdots,q$. \\
 Now if  $(x^1,\dots,x^q,\xi_1,\dots\xi_n)$ and  $(x^1,\dots,x^q,\eta_1,\dots\eta_k)$ are the associated coordinate system of $T^*M_{| U\cap N}$ and ${\cal A}^*_{| U\cap N}$ then $\rho^*$  is the map $\eta_\a=\xi_\a$ for $\a=1,\dots,q$.
 With these  notations we can   consider that  $(\bar{x}=(\bar{x}^1,\dots,\bar{x}^q),\bar{\xi}=(\bar{\xi}_1=\xi_1,\dots,\bar{\xi}_q=\xi_q))$ is a coordinate system on $T^*N$. Now, we denote by $\bar{H}^\a(\bar{x},\eta)$ the restriction of $H^\a$ to ${\cal M}_N$ for $\a=1,\dots,k$.  In these notations  we have $\bar{H}^\a\circ\rho^*(\bar{x},\xi)=\bar{H}^\a(\bar{x},\bar{\xi})$.
 Therefore  we have:\footnote{in this proof all indices  $\a,\b,\g,\dots$ affecting coordinates $(\bar{x}^1,\dots,\bar{x}^q)$ or $(\bar{\xi}_1,\dots,\bar{\xi}_q)$ vary in the set $\{1,\dots,q\}$ and also in summations with Einstein convention }
\begin{eqnarray}\label{HLN}
H^*_{L_N}(\bar{x},\bar{\xi})=\bar{\xi}_\a \bar{H}^\a(\bar{x},\bar{\xi)}-L(\bar{x},\bar{H}^\a(\bar{x},\bar{\xi})).
\end{eqnarray}

%From the definition of the function $H^\a$ recall that we have
%$$\dis\frac{\p L}{\p y^\a}(\bar{x}, H^\b(\bar{x},\rho^*\xi))=\dis\frac{\p L}{\p y^\a}(\bar{x}, \bar{H}^\b(\bar{x},\bar{\xi}))=\bar{\xi}_\a \textrm{ for } \a=1\dots,q$$
%Therefore we get
%$\dis\frac{\p ^2L}{\p y^\a \p y^\g}(\bar{x}, \bar{H}^\g(\bar{x},\bar{\xi}))\dis\frac{\p \bar{H}^\g}{\p \xi_\b}(\bar{x}, \bar{\xi})=\d_\a^\b$.
%
But in the one hand we have $\dis\frac{\p \bar{H}^\g}{\p \xi_\b}(\bar{x}, \bar{\xi})=0$ for $\g=q+1,\dots,k$ and on the other hand we have $\dis\frac{\p \bar{H}^\b}{\p \xi^\g}(\bar{x}, \bar{\xi})=\dis\frac{\p ^2H^*_{L_N}}{\p \bar{\xi}_\b \p \bar{\xi}_\g}(\bar{x},\bar{\xi})$. It follows that $H^*_{L_N}$ is a regular Hamiltonian.

Now  an integral curve of the associated Hamiltonian vector field $X_{L_N}$  satisfies the differential system
\begin{eqnarray}\label{chamreduc}
\begin{cases}
\dot{\bar{x}}^\a=\bar{H}^\a(\bar{x},\bar{\xi}) \\
\dot{\bar{\xi}}_\a=-\bar{\xi}_\b\dis\frac{\p \bar{H}^\b}{\p\bar{ x}^\b}(\bar{x},\bar{\xi})+\dis\frac{\p L}{\p\bar{ x}^\b}(\bar{x}, \bar{H}^\a(\bar{x},\bar{\xi}))+\dis\frac{\p \bar{H}^\g}{\p\bar{ x}^\b}(\bar{x},\bar{\xi})\dis\frac{\p L}{\p{ y}^\g}(\bar{x}, H^\a(\bar{x},\bar{\xi})).
\end{cases}.
\end{eqnarray}

But  in the previous coordinates, if $\bar{c}:I\ap {\cal M}_N\cap\pi^{-1}(U)$ is a $\cal M$-admissible curve such that $\rho(c)$ is an integral curve of $\hat{\rho}(S_L)$ then $c$ is a solution of the following differential system:
\begin{eqnarray}\label{creduc}
\begin{cases}
\dot{\bar{x}}^\a=y^\a,\; \a=1,\dots,q\\
g_{\b\a}\dot{y}^\a=\dis\frac{\p L}{\p\bar{ x}^\b}-\dis\frac{\p^2L}{\p \bar{x}^\g\p y^\b}  y^\g,\; \b=1,\dots,k.
\end{cases}.
\end{eqnarray}

Now if $\bar{c}:I\ap {\cal M}_N\cap\pi^{-1}(U)$  is such that $\rho(c)= (c,y)$ satisfies the differential system (\ref{creduc}), if  ${\xi}_\a=\dis\frac{\p L}{\p{ y}^\a}(c,y)$ for $\a=1,\cdots,k$ and $\bar{\xi}=(\xi_1,\dots,\xi_q)$ then $\check{c}=(c,\bar{\xi})$ satisfies the differential system (\ref{chamreduc})  since $\bar{H}^\a(c,\dis\frac{\p L}{\p{ y}^\b}(c,y))=y^\a$ and $\dis\frac{\p L}{\p{ y}^\b}(c,\bar{H}^\b(c,{\xi}))={\xi}_\a=\bar{\xi}_\a$.

Conversely, if  $\check{c}=(c,\bar{\xi})$ is a solution of  the differential system (\ref{chamreduc}), for $y^\a=\bar{H}^\a(c,\bar{\xi})$ the curve $(c,y)$ is a solution of the differential system (\ref{creduc}) by similar arguments.

\end{proof}

\begin{proof}[Proof of Lemma \ref{reducN}]${}$\\
%Point (i) is also Point (i) in Lemma \ref{dualham}.\\
We continue to use the notations used in the previous Subsections. \\

 Note that  ${\cal M}_N^*$ is open in ${\cal A}^*_N$ %and by $\bar{\rho}^*:T^*M_{| N}/\ker\rho^*\equiv T^*N\ap {\cal A}^*_N$ the quotient map induced by $\rho^*: T^*M_{| N}\ap{\cal A}_N^*$.  We denote by $\bar{\cal M}_N$ the image $\rho({\cal M}_N)$ in $TN$. Note that
 and $\bar{\cal M}_N$ is an open subset of $TN$. We set $\bar{\cal M}^*_N=(\bar{\rho}^*)^{-1}({\cal M}_N)$. Then $\bar{\cal M}^*_N$ is an open subset of $T^*N$ and  we consider the composition:
   \begin{eqnarray}\label{compoH}
\begin{matrix}
 &\bar{\rho}^*& &\mathfrak{H}&&\rho&\\
  \bar{\cal M}_N^*&\ap&{\cal M}_N^*&\ap&{\cal M}_N&\ap&TN.\\
 \end{matrix}
\end{eqnarray}
and we set $\bar{\mathfrak{H}}=\rho\circ\mathfrak{H}\circ \bar{\rho}^*$.

%We set $\bar{\cal M}^*_N=(\bar{\rho}^*)^{-1}({\cal M}_N)$. Then $\bar{\cal M}^*_N$ is an open subset of $T^*N$.
\noindent On the other hand, according to Proposition \ref{cal G} we have $\check{p}_*({\cal G})=\L_L^{-1}\circ\rho^*(T^*M)=\mathfrak{H}(\rho^*(T^*M))$. It follows that we get $\check{p}_*({\cal G})\cap {\cal M}_N=\mathfrak{H}\circ \bar{\rho}^*(\bar{\cal M}^*_N).$

\noindent {\it We first prove that $\bar{\mathfrak{H}}$ is a diffeomorphism} from $\bar{\cal M}^*_N$ to $\bar{\cal M}_N$.
  % $(\rho_N)^*:T^*N\ap ({\cal A}_N/{\cal K}_N)^*$. If ${\cal A}^*_N={\cal A}^*_{| N}$ and ${\cal K}_N^0$ is the annulator of ${\cal K}_N$ of the subbundle of ${\cal A}_N$ then we can identify  $({\cal A}_N/{\cal K}_N)^*$ with ${\cal A}^*_N/{\cal K}_N^0$.
  %The projection ${\cal M}_N^*/{\cal K}_N^0$ of ${\cal M}_N^*$ onto the quotient ${\cal A}^*_N/{\cal K}_N^0$ is an open subset.
% Since $\rho^*((TN)^0)={\cal K}_N^0$, we have $(\rho_N)^*({\cal M}_N^*/{\cal K}_N^0)={\cal N}$ where ${\cal M}_N^*/{\cal K}_N^0$ denote the projection of ${\cal M}_N^*$ onto the quotient ${\cal A}^*_N/{\cal K}_N^0$.

 Consider  the coordinate systems and notations used in the proof of Point (iv) of Lemma \ref{dualham}.  We have $\bar{\mathfrak{H}}(\bar{x},\bar{\xi})=(\bar{x},\bar{H}^1(\bar{x},\bar{\xi}),\dots,\bar{H}^q(\bar{x},\bar{\xi}))$.
Now  from the definition of the function $H^\a$ recall that we have
$$\dis\frac{\p L}{\p y^\a}(\bar{x}, H^\b(\bar{x},\rho^*\xi))=\dis\frac{\p L}{\p y^\a}(\bar{x}, \bar{H}^\b(\bar{x},\bar{\xi}))=\bar{\xi}_\a \textrm{ for } \a=1\dots,q.$$
Therefore we get
$\dis\frac{\p ^2L}{\p y^\a \p y^\g}(\bar{x}, \bar{H}^\g(\bar{x},\bar{\xi}))\dis\frac{\p \bar{H}^\g}{\p \xi_\b}(\bar{x}, \bar{\xi})=\d_\a^\b$.
But  we have $\dis\frac{\p \bar{H}^\g}{\p \xi_\b}(\bar{x}, \bar{\xi})=0$ for $\g=q+1,\dots,k$. It follows that  $\bar{\mathfrak{H}}$  is locally invertible. By construction  the range of $\bar{\mathfrak{H}}$ is $\cal N$. Assume that $\bar{\mathfrak{H}}(\bar{x},\bar{\xi})=\bar{\mathfrak{H}}(\bar{x}',\bar{\xi}')$. Then we must have $\bar{x}=\bar{\x}'$ and  so $ \bar{H}^\a(\bar{x},\bar{\xi})= \bar{H}^\a(\bar{x},\bar{\xi}')$ for  $\a=1,\dots,q$. Since    $\bar{\mathfrak{H}}$  is locally invertible it follows that $\bar{\xi}=\bar{\xi}'$ and so $\bar{\mathfrak{H}}$ is a diffeomorphism.

Now,  in the  associated coordinate system on ${\cal M}_N$, the local equations of  $\mathfrak{H}\circ\bar{\rho}^*(\bar{\cal M}^*_N)$ are $y^\a=0$ for $\a=q+1,\dots,k$ and the restriction of $\rho$ to $\mathfrak{H}\circ\bar{\rho}^*(\bar{\cal M}^*_N)$ in theses coordinates is  $y^\a=\dot{x}^\a$ for $\a=1,\dots,q$ since $\bar{\mathfrak{H}}$ is injective, it follows that $\rho$ induces a diffeomorphism from  $\mathfrak{H}(\bar{\cal M}^*_N)$ onto $\bar{\cal M}_N$. This ends the proof of Point (i). \\

\noindent From the  expression (\ref{HLN}) of the Hamiltonian $H^*_{L_N}$ on $T^*N$ we have $\dis\frac{\p \bar{H}^\b}{\p \xi^\g}(\bar{x}, \bar{\xi})=\dis\frac{\p ^2H^*_{L_N}}{\p \bar{\xi}_\b \p \bar{\xi}_\g}(\bar{x},\bar{\xi})$. Therefore  $H^*_{L_N}$ is a convex hyperregular Hamiltonian on $\bar{\cal M}^*_N$ and we get a partial convex hyperregular  Lagrangian $L_N$ defined on $\bar{\cal M}_N$  classically by
\begin{eqnarray}\label{LN}
L_N(\bar{x},\dot{\bar{ x}})=<\bar{\xi},\dot{\bar{x}}> -H^*_{L_N}(\bar{x},\bar{y}), \textrm{ with } (\bar{x},\dot{\bar{x}})=\bar{\mathfrak{H}}(\bar{x},\bar{\xi}).
\end{eqnarray}
  If $S_{L_N}$ is the semispray on $\cal N$ associated with $L_N$ we have $\bar{\mathfrak{H}}_*(X_{L_N})=S_{L_N}$.
Point (iii) is then a consequence of Point (iv) of Lemma \ref{dualham}. \\

For the proof of Point (ii), we  begin by the proof  of the relation $L_N\circ \rho_{| \mathfrak{H}\circ\bar{\rho}^*(\bar{\cal M}^*_N)}=L_{| \mathfrak{H}\circ\bar{\rho}^*(\bar{\cal M}^*_N)}$. In the previous local notations $\mathfrak{H}\circ \bar{\rho}^*(\bar{\cal M}^*_N)$, is locally the set $(\bar{x}, \bar{H}^1(\bar{x},\bar{\xi}),\dots,\bar{H}^q(\bar{x},\bar{\xi}))$. 
Since we have $L(x,y)=<\eta, y> -H_L\circ \L_L(x,y))$ it follows that   locally we have:

$L\circ \mathfrak{H}\circ\bar{\rho}^*(\bar{x},\bar{\xi})=L(\bar{x}, \bar{H}^1(\bar{x},\bar{\xi}),\dots,\bar{H}^q(\bar{x},\bar{\xi}))$

$\;\;\;\;\;\;\;\;\;\;\;\;\;\;\;\;\;\;\;\;\;\;\;\;=\bar{\xi}_\a \bar{H}^\a(\bar{x},\bar{\xi})-H_L(\bar{x},\rho^j_\a,{\xi}_j)$

$\;\;\;\;\;\;\;\;\;\;\;\;\;\;\;\;\;\;\;\;\;\;\;\;=\bar{\xi}_\a \bar{H}^\a(\bar{x},\bar{\xi})-H_L^*(\bar{x},\bar{\xi})$

According to the choice of the coordinate system we can locally  identify $\bar{\cal M}^*_N$ with $\bar{\rho}^*(\bar{\cal M}^*_N)$ and  $\mathfrak{H}\circ\bar{\rho}^*(\bar{\cal M}^*_N)$  with $\bar{\cal M}_N$ and then with this identifications, we get $\bar{\xi}_\a \bar{H}^\a(\bar{x},\bar{\xi})-H_L^*(\bar{x},\bar{\xi})=L_N(\bar{x},\bar{y})$
where $(\bar{x},\bar{y})=(\bar{x},(\bar{H}^1(\bar{x},\bar{\xi}),\dots,\bar{H}^q(\bar{x},\bar{\xi}))$. This ends the proof of the relation $L_N\circ \rho_{| \mathfrak{H}\circ\bar{\rho}^*(\bar{\cal M}^*_N)}=L_{| \mathfrak{H}\circ\bar{\rho}^*(\bar{\cal M}^*_N)}$.\\

 We now apply apply  the context of Proposition \ref{reducleaf}. %We denote by
%$$\bar{\pi}_N: \T{\cal M}_N\ap   \left(\T{\cal M}_N\right)/ \left({\bf K} {\cal A}_N\oplus(J({\bf K} {\cal A}_N)\right)\equiv T\bar{\cal M}_N$$
to the open set ${\cal M}_N$ of ${\cal A}_N$.\\
Recall that we have an isomorphism $\hat{\rho}_{TN} : \left(\T{\cal A}_N\right)/ \left({\bf K} {\cal A}_N\oplus(J({\bf K} {\cal A}_N)\right)\ap T(TN)$ over the isomorphism $\rho_N: {\cal A}_N/{\cal K}_N\ap TN$.
On the other hand, from the Diagram (\ref{compoH}), we have  $\rho_N({\cal M}_N)=\bar{\cal M}_N$. Since ${\cal M}_N$ is an open set of ${\cal A}_N$,  then via the  projection
$\T{\cal A}_N\ap \left(\T{\cal A}_N\right)/ \left({\bf K} {\cal A}_N\oplus(J({\bf K} {\cal A}_N)\right)$, the set  $\T{\cal M}_N={\T{\cal A}_N}_{| {\cal M}_N}$ gives rise to a quotient  set   denoted by
$\left(\T{\cal M}_N\right)/ \left({\bf K} {\cal A}_N\oplus(J({\bf K} {\cal A}_N)\right)$. Moreover  the restriction of $\hat{\rho}_{TN}$ to $ \left(\T{\cal M}_N\right)/ \left({\bf K} {\cal A}_N\oplus(J({\bf K} {\cal A}_N)\right)$ is an isomorphism onto $T\bar{\cal M}_N$. Denote by
$$\bar{\pi}_N: \T{\cal M}_N\ap   \left(\T{\cal M}_N\right)/ \left({\bf K} {\cal A}_N\oplus(J({\bf K} {\cal A}_N)\right)$$
 the natural projection.
Since $\rho_N$ is also a Lie algebra isomorphism on the one hand  we  can identify ${\cal A}_N/{\cal K}_N$ with $TN$.

On the other hand, if $p_N:{\cal A}_N\ap {\cal A}_N/{\cal K}_N$ is the canonical projection  we have $\rho_{| {\cal A}_N}=\rho_N\circ p_N$.   We set  $\mathcal{\check{ M}}_N=  \check{p}_*({\cal G})\cap {\cal M}_N=\mathfrak{H}\circ\bar{\rho}^*(\bar{\cal M}^*_N)$ and so,  according to our previous identification we get $p_N(\mathcal{\check{ M}}_N)\equiv \rho(\mathcal{\check{ M}}_N)=\bar{\cal M}_N$. Moreover $p_N$ induces a diffeomorphism $\bar{p}_N$ from $\mathcal{\check{ M}}_N$ onto $\bar{\cal M}_N$. With these last conventions,  we then have $L_{| \mathcal{\check{ M}}_N}=L_N\circ p_N$. Now for any section $\cal X$  such that  $\hat{\rho}({\cal X})$ is tangent to $\bar{\cal M}_N$   using  the previous local coordinates we can see that $\bar{\pi}_M({\cal X})$ is well defined and $(\bar{p}_N)^{-1}_*(\bar{\pi}_M({\cal X}))$ is tangent to $\mathcal{\check{ M}}_N$. Therefore since  $d^{\cal P} L({\cal X})= dL(\hat{\rho}({\cal X}))$ when $\hat{\rho}({\cal X})$ is tangent to $\bar{\cal M}_N$ we have
\begin{eqnarray}\label{dnLN}
d^{\cal P}L({\cal X})=d L_{| \mathcal{\check{ M}}_N}((\bar{p}_N)^{-1}_*(\bar{\pi}_M({\cal X}))=d_NL_N(\bar{\pi}_N({\cal X})).
\end{eqnarray}

%moreover, from Point (i) of Lemma \ref{reducN} we can also  identify the projection of ${\cal M}_N$ onto  ${\cal A}_N/{\cal K}_N$ with the open set  $\rho(\bar{\cal M}_N)$ in $TN$
%and $\bar{\cal M}_N$ can be also considered as   a subset of ${\cal M}_N$ (via the restriction of $\rho$ to $\mathfrak{H}\circ\bar{\rho}^*(\bar{\cal M}^*_N)$.
%With these last conventions, we have then $L_{|\mathcal{\check{ M}}_N=L_N\circ p_N$.

Now from  Proposition  \ref{reducleaf} again we can also identify $ \left(\T{\cal M}_N\right)/ \left({\bf K} {\cal A}_N\oplus(J({\bf K} {\cal A}_N)\right)$ with  $T\bar{\cal M}_N$ and then  the classical Lie bracket on $\bar{\cal M}_N$ coincides with the bracket $[\;,\;]_{TN}$ induced by $[\;,\;]_{\cal P}$ by projection onto $ \left(\T{\cal M}_N\right)/ \left({\bf K} {\cal A}_N\oplus(J({\bf K} {\cal A}_N)\right)$. We then have
$$\bar{\pi}_N: \T{\cal M}_N\ap   \left(\T{\cal M}_N\right)/ \left({\bf K} {\cal A}_N\oplus(J({\bf K} {\cal A}_N)\right)\equiv T\bar{\cal M}_N.$$
 If  $J$ and $J_N$ are the vertical endomorphisms  in $\T{\cal A}$ and $T(TN)$  we have $\bar{\pi}_N\circ J_{| \T{\cal A}_N}=J_N\circ\bar{\pi}_N$.
Since $\bar{\pi}_N$ is a Lie algebra homomorphism (see Point (ii) in Proposition \ref{reducleaf}) we obtain $\bar{\pi}_N^*\circ  d_N =d^{\cal P}\circ \bar{\pi}_N^*$ where $d_N$ denote the classical exterior  differential on $N$ and $\bar{\pi}_N^*:T^*\bar{\cal M}_N\ap(\T{\cal M}_N)^*$ is the transpose map of $\bar{\pi}_N$. Finally the symplectic form  associated with $L_N$ is
$\O_{L_N}=-d_N\circ J_N^*(d_N L_N)$.

Therefore according to Equation  (\ref{dnLN}) we have:
$$\bar{\pi}_N^*\O_{L_N}=-\bar{\pi}_N^*\circ d_N\circ J_N^*(d_NL_N)=-d_N\circ J_N^*\circ \bar{\pi}_N^*( d_NL_N)=\O_{L}.$$
It follows that we obtain $\bar{\pi}_N(S_L)=S_{L_N}$ with our identification which intrinsically means
$$T\rho(\hat{\rho}({S_L}_{|\mathcal{\check{ M}}_N}))= S_{L_N}.$$

%But we have seen that $L_N=H^*_{L_N}\circ \bar{\mathfrak{H}}^{-1}$. According to see  the local expression of $\rho$ in restriction $\mathfrak{H}\circ\bar{\rho}^*(\bar{\cal M}^*_N)$

% in these coordinates we can identify $\mathfrak{H}\circ\bar{\rho}^*(\bar{\cal M}^*_N)$ with $\bar{\cal M}_N$  and so we have proved  the equality  $L_N\circ \rho=L_{| \mathfrak{H}\circ\bar{\rho}^*(\bar{\cal M}^*_N)}$ locally. It follows that this result is true globally.
\end{proof}

\begin{proof}[Proof of Theorem \ref{extremalN}]${}$\\
Point (1)  is a consequence of Lemma  \ref{reducN}.

 For the proof of Point (2), again, we use the identifications  described in the  end of the proof of  Lemma \ref{reducN} and the associated results $L_{| \mathcal{\check{ M}}_N}=L_N\circ\bar{\pi}_N$, $\O_{L}=\bar{\pi}_N^*(\O_{L_N})$ and  $\bar{\pi}_N(S_L)=S_{L_N}$. Since $\bar{\pi}_N$ is a Lie algebra homomorphism  if $S$ and $X$ are vector fields on $\bar{\cal M}_N$ and ${\cal S}$ and $\cal X$ are $\bar{\pi}_N$-lifts of $S$ and $X$ respectively we have:
$\bar{\pi}_N\circ J([{\cal S},{\cal X}]_{\cal P})=J_N[S,X]$ and $ \bar{\pi}_N([{\cal S},J{\cal X}]_{\cal P})=[S,JX]$.
Now recall that we have ${\cal N}_{S_L}({\cal X})={ J}[S_L,{\cal X}]_{\cal P}-[S_L,  J{\cal X}]_{\cal P}$. Therefore $\overline{\cal N}_L(X)=\bar{\pi}_N\circ {\cal N}_{S_L}({\cal X})$   for some $\bar{\pi}_N$-lift  $\cal X$  of $X$ is a well defined  endomorphism of $T\bar{\cal M}_N$ and we have $\overline{\cal N}_L(X)={\cal N}_{S_{L_N}}(X)$.\\

 Now we look for the equivalences in Point (iii).\\

 The equivalence  (i)$\Leftrightarrow$(ii) is a direct consequence of Point (i) and (ii) of Theorem \ref{locmin} and Point (iii) of Lemma \ref{dualham}. In particular $c$ is smooth

 The implication (ii)$\Rightarrow$(iii) is a consequence of Theorem \ref{locmin}, Point (iii) and of Lemma \ref{dualham}, Point (iii).

 The implication (iii)$\Rightarrow$(iv) is a consequence of  the definition of a local minimizer, Theorem \ref{locmin}, Point (ii) and of Lemma \ref{reducN}, Point (ii).

 Consider the particular case of the  anchored bundle $(TN, N,Id_N)$ from the definition of an abnormal extremal, it follows  ({\it as it is well known}) that any extremal of $L_N$ is a normal extremal of $L_N$. Now from the Maximum Principle if $c$ is a local minimizer of $L_N$, then  $c$ must a normal  extremal of $L_N$. Finally from the previous equivalence (ii)$\Leftrightarrow$(iii) (already proved) applied in this particular context we obtain the implication (iv)$\Rightarrow$(v)

 The equivalence (v)$\Leftrightarrow$(i) is a consequence of Point (iii) of  Lemma \ref{reducN}.

 For the proof of Point (iii), again we use  the identifications  described in the  end of the proof of  Lemma \ref{reducN} and the associated results $L_{| \mathcal{\check{ M}}_N}=L_N\circ\bar{\pi}_N$, $\O_{L}=\bar{\pi}_N^*(\O_{L_N})$ and  $\bar{\pi}_N(S_L)=S_{L_N}$. Since $\bar{\pi}_N$ is a Lie algebra homomorphism  if $S$ and $X$ are vector fields on $\bar{\cal M}_N$ and ${\cal S}$ and $\cal X$ are $\bar{\pi}_N$-lifts of $S$ and $X$ respectively we have:
$\bar{\pi}_N\circ J([{\cal S},{\cal X}]_{\cal P})=J_N[S,X]$ and $ \bar{\pi}_N([{\cal S},J{\cal X}]_{\cal P})=[S,JX]$.
Now recall that we have ${\cal N}_{S_L}({\cal X})={ J}[S_L,{\cal X}]_{\cal P}-[S_L,  J{\cal X}]_{\cal P}$. Therefore $\overline{\cal N}_L(X)=\bar{\pi}_N\circ {\cal N}_{S_L}({\cal X})$   for some $\bar{\pi}_N$-lift  $\cal X$  of $X$ is a well defined  endomorphism of $T\bar{\cal M}_N$ and we have $\overline{\cal N}_L(X)={\cal N}_{S_{L_N}}(X)$.\\
\end{proof}

%%%%%%%%%%%%%%%%%%%%%%%%%%%%%%%%%%%%%%%%%%%%%%%%%%%%%%%%%%%%%%%%%
\section{Finsler  geometry on an anchored foliated bundle}\label{Partialfins}
%%%%%%%%%%%%%%%%%%%%%%%%%%%%%%%%%%%%%%%%%%%%%%%%%%%%%%%%%%%%%%%
Consider an anchored bundle $({\cal A},M,\rho)$.  In this section  ${\cal M}$ will be a conic open subset of  $\cal A$ without the zero section that is $\cal M$ is stable under the action of $\R^+$ on each fiber of $\cal A$  and such that  $\pi=\t_{\cal M}:{\cal M}\ap M$ is fibered manifold. We extend to this context the result of \cite{Ja}. More precisely, we will define the notion of partial Finsler structure on  $({\cal A},M,\rho)$  and, given an almost Lie bracket $[\;,\;]_{\cal A}$ on $\cal A$ such that $({\cal A},M,\rho,[\;,\;]_{\cal A})$ is a pre-Lie algebroid, we will associate to this structure a "Finsler connection" and a "Chern connection" such that   we recover the usual properties of Finsler connection and Chern connection associated with a Finsler metric on each leaf of the foliation defined by the distribution $\rho({\cal M})$. In particular these connections induced on such leaves does not depend on the choice of the bracket $[\;,\;]_{\cal A}$.  Analogous properties are also true for the curvature and the flag curvature.

%%%%%%%%%%%%%%%%%%%%%%%%%%%%%%%%%%%%%%%%%%%%%%%%%%%%%%%%%%%%%%%
\subsection{Partial  Finsler metric  and sectional linear connections }\label{subfinrap}${}$
%%%%%%%%%%%%%%%%%%%%%%%%%%%%%%%%%%%%%%%%%%%%%%%%%%%%%%%%%%%%%%%
%%%%%%%%%%%%%%%%%%%%%%%%%%%%%%%%%%%%%%%%%%%%%%%%%%%%%%%%%%%%%%%

Let  $\cal M$ be a conic  open subset of $\cal A$ (as introduced previously).

\begin{Def}\label{parfin}${}$\\
A   partial Finsler metric  on the anchored bundle  $({\cal A},M,\rho)$  is  a continuous map ${\cal F}:{\cal M}\ap[0,+\infty[$ such that:

 (i) $\cal F$ is smooth on  ${\cal M}$;

 (ii) ${\cal F}(\l u)=\l{\cal F}(u)$ for all $u$ in the fiber ${\cal M}_x$,  all $x\in M$ and all $\l>0$ (positive homogeneity)

 (iii) the quadratic form $g_{u}(v,w):=\dis\frac{1}{2}\frac{\p^2 {\cal F}^2}{\p s \p t}(u+sv+tw)_{| s,t=0}$ is definite positive for all $v,w\in {\cal A}_x$, and
   $x\in M$
  \end{Def}

\begin{Ex}\label{exfinsler}${}$

{\bf 1.} When $\cal A=TM$    then we obtain  the context of quadratic Finsler metric on $M$ as in \cite{Ja}.

{\bf 2.} When $\cal M$ is the complementary of the zero section in $\cal A$, and if $\cal F$ is the restriction to $\cal M$ of a continuous map $\bar{\cal F}:{\cal A}\ap [0,+\infty[$, we get the classical notion of Finsler metric on $\cal A$. In particular, if  we consider a Riemannian metric $g$ on $\cal A$, we get naturally a  Finsler metric on $\cal A$ defined by:
$${ \cal F}(x,u)=\sqrt{g_x(u,u)}.$$

{\bf 3.} Given a pseudo-Riemannian metric $g$ on $\cal A$, the set ${\cal A}^+$ of  vectors  $u\in {\cal A}$ such that $g(u,u)>0$ is a conic open of  $\cal A$ and so
${\cal F}(u)=\dis\frac{1}{2}\sqrt{g(u,u)}$ is a Partial Finsler metric. In a symmetric way, the set ${\cal A}^-$ of  vectors  $u\in {\cal A}$ such that $g(u,u)<0$ is a conic open of  $\cal A$ and so
${\cal F}(u)=\dis\frac{1}{2}\sqrt{-g(u,u)}$ is a partial Finsler metric.

{\bf 4.} The distribution $E=\rho({\cal A})$ on $M$ generated by $\cal A$ is integrable.  Therefore, we can define     a partial Finsler metric  $F_N$ on the tangent space  $TN$ of a leaf $N$ of the foliation defined  by $\rho({\cal A})$ in the following way:
$$F(X)=\inf\{{\cal F}(u), \textrm{ such that } \rho(u)=X,\;\; X\in \Xi(N) \}.$$
As we have already seen  $\ker \rho$ is a subbundle of ${\cal A_N}$ and  $\rho$ induces  a canonical  morphism ${\rho}_N:{\cal A}_N/{\cal K}_N\ap TN$ which is an isomorphism.  If ${\cal F}_N$ is the restriction of $\cal F$ on  ${\cal A}_N$, we obtain a Finsler metric $\bar{\cal F}_N$ on $({\cal A}_N/{\cal K}_N\,N,{\rho}_N)$ such that $F_N(\hat{\rho}(a))=\bar{\cal F}_N(a)$.
 In particular, if $\cal A$ is an integrable subbundle of $TM$,  then the Finsler metric ${F}_N$ on a leaf $N$ is nothing but the restriction of $\cal F$ to $TN$.
%  In this case we will say that  $({\cal A}, \rho,{\cal F})$  is a{\it regular sub-finslerian structure}. If $E$ is not a subbundle of $TM$ we will
 %  used the term of {\it singular sub-Finslerian structure}.\\
%\end{enumerate}
\end{Ex}

Consider an anchored bundle $({\cal A},M,\rho)$, a conic  open subset  $\cal M$ of $\cal A$ and $\cal F$ a partial finsler metric defined on $\cal M$.\\
According to the beginning of Subsection \ref{prolongM}, we have already defined the pull-back of the bundle $\t:{\cal A}\ap M$ over $\pi: {\cal M}\ap M$ so that we have the following commutative diagrams:

% We then denote by ${\t}_*:\t^*{\cal A}\ap \stackrel{\circ}{\cal A}$ the pull-back of the bundle $\t:{\cal A}\ap M$ over $\t: \stackrel{\circ}{\cal A}\ap M$. We have then a commutative diagram

 $$ \begin{array}{ccccc}
 &\tilde{\pi}&\\
\tilde{\cal A}& \longrightarrow  &{\cal A} \\
\tilde{\tau} \;  \Big\downarrow &  & \Big\downarrow \;\tau \\
{\cal M} &  \longrightarrow & { M}   \\
& \pi &
\end{array}\;\;\;\;\;\;
 \begin{array}{ccccc}
 &{\pi}_{\tilde{\cal A}}&\\
{\T}{\cal M}& \longrightarrow  &\tilde{\cal A} \\
\hat{\pi} \;  \Big\downarrow &  & \Big\downarrow \;\tilde{\tau} \\
{\cal M} &  \longrightarrow & {\cal M}   \\
& Id &
\end{array}
$$
%This Riemannian metric is called {\it the fundamental tensor} of $\cal F$ and denoted $g_{\cal F}$

  Therefore the quadratic form $g_u$ defined in (iii), gives rise to a Riemannian metric on $\tilde{\cal A}$ which  is called {\it the fundamental tensor} of $\cal F$ and  denoted $g_{\cal F}$.   We have a natural isomorphism $\vartheta$ from the bundle  $\bar{\tau}: \tilde{\cal A}\ap{\cal M}$ to the vertical bundle ${\bf V}{\cal M} \ap  {\cal M}$ given by:
$$(x,u,v) \textrm{  associates the vertical lift } (x,u,\dis\frac{d}{dt}(u+tv)_{| t=0}).$$ %{\it In this section \ref{Partialfins}, we will identify $\bf{V}{\cal M}$ and $\bar{\cal A}$ via $\pi^V$.}\\

 According to Example \ref{exsmH}\;{\bf  2}, ${\cal L}={\cal F}^2$ is a regular Lagrangian on $\cal M$ and   the associated  Riemannian metric  is exactly $g_{\cal L}$.  Via the previous isomorphism $\vartheta$ we have $g_{\cal F}=\vartheta^*g_{\cal L}$

\begin{Def}\label{Cartan}${}$\\
Let
 $\cal F$ be  a partial Finsler metric on  $({\cal A},M,\rho)$. We denote by  $\mathbb{A}_u$  is the trilinear form:
$$\mathbb{A}_u(v_1,v_2,v_3)=\dis\frac{1}{4}\dis\frac{\p^3 {\cal F}^2}{\p s_1 \p s_2 \p s_3}(u+\sum_{i=1}^3s_iv_i)_{| s_1=s_2=s_3=0}$$ for any $u\in {\cal M}_x$ and any $x\in M$
\end{Def}

\begin{Rem}\label{Au=0}${}$
\begin{enumerate}
\item From the definition of  $\mathbb{A}_u$ we also have
$ \mathbb{A}_u(v_1,v_2,v_3)=\dis\frac{1}{2}\frac{\p}{\p s}[g_{u+sv_1}(v_2,v_3)]_{| s=0}.$
\item For $u\in {\cal M}_x$, we have $\mathbb{A}_u(u,v_2,v_3)=0$ for any $v_2, v_3$ in ${\cal A}_x$
 (same argument as in \cite{Ja} Proposition 2.6).
 \item When $\cal F$ is a Finsler metric on $\cal A$,  then $g_{\cal F}$ is a Riemannian metric if and only if $\mathbb{A}_{\cal F}\equiv 0$
\end{enumerate}
\end{Rem}

The definition of   $\mathbb{A}_u$ implies   that  its value does not depend on the order $v_1, v_2$ and $v_3$.  By  the same previous  argument  $\mathbb{A}_u$ gives rises to a symmetric  $\tilde{\cal A}$-tensor of type $(3,0)$    which is called the {\it Cartan tensor} of $\cal F$ and denoted $\mathbb{A}_{\cal F}$.\\

 We now provide the anchored bundle $({\cal A},M,\rho)$ with an almost Lie bracket $[\;,\;]_{\cal A}$ such that $({\cal A},M,\rho),[\;,\;]_{\cal A}$ is a pre-Lie algebroid .

 \begin{Def}\label{partial Finsler}
 The quadruple $ ({\cal A}, M, \rho, [\;,\;]_{\cal A}, {\cal F})$) is called a partial Finsler  pre-Lie algebroid.
 \end{Def}

Consider  an open set $O$  in $M$ and $U$ a local section of  $\cal M$ over $O$. Then we introduce:

\begin{Def}\label{glinearconn}${}$\\
Consider a linear connection $\nabla$ on ${\cal A}_{| O}$. We say that:

(1) $\nabla$ is torsion-free if
$$\nabla_XY-\nabla_YX=[X,Y]_{\cal A}$$
for all sections $X$ and $Y$ of $\cal A$ defined on $O$

(2) $\nabla$ is almost $g_{\cal F}$ compatible along  $V$ if
$$\rho(X)(g_{U(x)}(Y,Z))=g_{U(x)}(\nabla_XY,Z)+g_{U(x)}(Y,\nabla_XZ)+2\mathbb{A}_{U(x)}(\nabla_XU,Y,Z)$$
for all sections $X$, $Y$ and $Z$ of $\cal A$ defined on $O$ and all $x\in O$

\end{Def}

\begin{Rem}\label{nablasg}${}$\\
As usual the connection $\nabla$ can be defined for any $\cal A$-tensor. Therefore,  the almost $g_{\cal F}$ compatibility of $\nabla$ along $V$ is equivalent to
$$\nabla_X(g_U)(Y,Z)=\mathbb{A}_U(\nabla_XU,Y,Z)$$
for all sections $X$ and $Y$ of $\cal A$ defined on $O$.
\end{Rem}

Then, as in \cite{Ja} we have

\begin{Pro}\label{exisnablas}${}$\\
Given a section $U$ of $\cal M$ over an open set $O$, there exists a unique linear connection $\nabla^U$ on  ${\cal A}_{| O}$ which is torsion-free and almost $g_{\cal F}$ compatible along $V$. Moreover for any $x\in O$, the value of the section $\nabla^U _XY$ at $x$ only depends  on the values of $U(x)$ and $X(x)$.
\end{Pro}

\begin{Def}\label{sectcoont}${}$\\
Given a section $U$ of $\cal M$ over an open set $O$ in $M$, the linear connection characterized in Proposition \ref{exisnablas} will be denoted $\nabla^U$ and called a sectional linear connection on $O$ along $U$.
\end{Def}

\smallskip
\begin{proof} ({\it cf.} Proof of Proposition 2.12 of \cite{Ja})${}$\\
 As in the classical case of linear connection which are compatible with a Riemannian metric, if we apply three times  the almost $g_{\cal F}$-compatibily on circular permutation of $X,Y$ and $Z$, we add these relations and if we take in account the torsion-free relation, we obtain  that any connection $\nabla$ which satisfies  the properties (1) and (2) of Definition \ref{glinearconn}, is characterized by:
\begin{eqnarray}\label{kozul}
2g_U(\nabla_XY,Z)=\rho(X)(g_U(Y,Z))+\rho(Y)(g_U(X,Z))-\rho(Z)(g_U(X,Y))\nonumber\\
{}\;\;\;\;\;\;\;\;\;\;\;\;\;\;\;\;\;\;\;\;\;\;\;\;\;\;\;\;\;\;\;\;\;\;\;\;\;+g_U([X,Y]_{\cal A},Z)+g_U([Z,X]_{\cal A},Y)-g_U([Y,Z]_{\cal A},X)\nonumber\\
{}\;\;\;\;\;\;\;\;\;\;\;\;\;\;\;\;\;\;\;\;\;\;\;\;\;\;\;\;\;\;\;\;\;\;\;\;\;\;\;\;\;\;\;\;\;-2(\mathbb{A}_U(\nabla_XU,Y,Z)+2\mathbb{A}_U(\nabla_YU,X,Z)-\mathbb{A}_U(\nabla_ZU,X,Y)).
\end{eqnarray}

 If we apply (\ref{kozul}) for $X=Y=U$, according to Remark \ref{Au=0} (2), we get:
 \begin{eqnarray}\label{nablaV}
 2g_U(\nabla_UU,Z)=2\rho(U)(g_U(U,Z))-\rho(Z)(g_U(U,U))+2g_U([Z,U]_{\cal A},U).
\end{eqnarray}
Therefore $\nabla_UU$ is uniquely defined. In the same way, if we apply (\ref{kozul}) for $Y=U$ we obtain:
\begin{eqnarray}\label{Y=V}
2g_U(\nabla_XU,Z)=\rho(X)(g_U(U,Z))+\rho(U)(g_U(X,Z))-\rho(Z)(g_U(X,U))\nonumber\\
{}\;\;\;\;\;\;\;\;\;\;\;\;\;\;\;\;\;\;\;\;\;\;\;\;\;\;\;\;\;\;\;\;\;\;\;\;\;+g_U([X,U]_{\cal A},Z)+g_U([Z,X]_{\cal A},U)-g_U([Y,Z]_{\cal A},X)-2\mathbb{A}_U(\nabla_UU,X,Z).
\end{eqnarray}

Again $\nabla_XU$ is then well defined and moreover,  from this relation we obtain
$\nabla_{fX}U=f\nabla_XU$ for all smooth function $f$ on $O$.\\
Finally as  $\nabla_UU$ and $\nabla_XU$ are well defined for any $X$, the relation (\ref{kozul}) defines uniquely $\nabla_XY$.\\
 Moreover  using the fact that
$\nabla_{fX}U=f\nabla_XU$ and as $\mathbb{A}_U$ is a tensor, we   also obtain  $\nabla_{fX}Y=f\nabla_XY$. Finally  using again the fact that $\mathbb{A}_U$ is tensorial,
from the relation  (\ref{kozul})  we obtain

$\nabla_X(f)Y=\rho(X)(f)Y+f\nabla_XY$ for any smooth function $f$ on $O$.\\

For the proof of the last part, we may assume that $O$  is a chart domain and $\cal A$ is a trivial bundle over $O$.
Let $(x^i)$ be  a coordinates system on $O$ and $\{e_\a\}$ a basis of $\cal A$ over $O$. Therefore we obtain a coordinate system $(x^i,y^\a)$ on $\t^{-1}(O)$.
Then we have the local decomposition:

$U=U^a e_\a$;

$[e_\a,e_\b]_{\cal A}=C_{\a\b}^\g e_\g$;

$\rho(e_\a)=\rho_\a^i\dis\frac{\p}{\p x^i}$;

$g_U(e_\a,e_\b)=g_{\a\b}(V)$;

$\nabla^U_{e_\a}e_\b=(\G^U) _{\a\b}^\g (V)e_\g$.

Now  in the relation (\ref{kozul}) we put $X=e_\a$, $Y=e_\b$ and $Z=e_\g$. Therefore we get:
\begin{eqnarray}\label{kozulloc}
(\G^U)_{\g\a\b}=g_{\l\g}(\G^U)_{\a\b}^\l=S_{\g\a\b}
-U^\l((\G^U)_{\a\l}^\mu\mathbb{A}_{\mu\b\g}+(\G^U)_{\b\l}^\mu\mathbb{A}_{\mu\a\g}-(\G^U)_{\g\l}^\mu\mathbb{A}_{\mu\a\b}),
\end{eqnarray}
where
\begin{eqnarray}\label{Sabc}
S_{\g\a\b}=\dis\frac{1}{2}[(\rho_\a^i\dis\frac{\p g_{\b\g}}{\p x^i}+\rho_\b^i\dis\frac{\p g_{\a\g}}{\p x^i}-\rho_\g^i\dis\frac{\p g_{\a\b}}{\p x^i})+(g_{\l\g}C_{\a\b}^\l+g_{\l\b}C_{\g\a}^\l+g_{\l\a}C_{\g\b}^\l)].
\end{eqnarray}

In (\ref{nablaV}), if we put $U=U^a e_\a$ and $Z=e_\g$ and we  obtain
\begin{eqnarray}\label{nablaVloc}
U^\a U^\b(\G^U)_{\a\b}^\g=g^{\g\l}U^\a U^\b(\G^U)_{\l\a\b}=g^{\g\l}U^\a U^\b S_{\l\a\b}=U^\a U^\b S_{\a\b}^\g.
\end{eqnarray}
Now, from (\ref{kozulloc}) by using (\ref{nablaVloc}) we then have:
\begin{eqnarray}
U^\b(\G^U)_{\a\b}^\g=U^\b g^{\g\l}(\G^U)_{\l\a\b}=U^\b S^\g_{\a\b}-U^\b U^\epsilon S^\mu_{\b\epsilon}g^{\g\l}\mathbb{A}_{\mu\a\l}.
\end{eqnarray}
If we set
\begin{eqnarray}\label{NU}
({\cal N}^U)_\a^\b=U^\b S^\g_{\a\b}-U^\b U^\epsilon S^\mu_{\b\epsilon}g^{\g\l}\mathbb{A}_{\mu\a\l},
\end{eqnarray}
  the relation (\ref{kozulloc}) gives rise to:
\begin{eqnarray}\label{kozullocfin}
(\G^U)_{\a\b}^\g=S_{\a\b}^\g+g^{\g\l}(-({\cal N}^U)_\a^\mu\mathbb{A}_{\mu\b\l}-({\cal N}^U)_\b^\mu\mathbb{A}_{\mu\a\l}+({\cal N}^U)_\l^\mu\mathbb{A}_{\mu\a\b}).
\end{eqnarray}
Now, from their expression the value the quantities  $S_{\a\b}^\g$, $g^{\g\l}$ and $\tilde{{\cal N}}_\a^\b$  at each $x\in O$ depend only on $x$ and $V(x)$. Therefore from
(\ref{kozullocfin}) we deduce that $(\G^U)_{\a\b}^\g$ also depends on $x$ and $V(x)$.\\
\end{proof}

%%%%%%%%%%%%%%%%%%%%%%%%%%%%%%%%%%%%%%%%%
\subsection{Finsler and Chern connections on a partial Finsler  algebroid}\label{partfin}${}$\\
%%%%%%%%%%%%%%%%%%%%%%%%%%%%%%%%%%%%%%%%%%%%
{\rm %Given an almost bracket $[\;,\;]_{\cal A}$ on an anchored bundle $(\cal A,M,\rho)$.
%see \cite{HPo} or subsection \ref{subfinrap}
{\it A partial Finsler pre-Lie algebroid}  $(\cal A,M,\rho,[\;,\;]_{\cal A},{\cal F})$ is a partial Finslerian  metric ${\cal F}$  on an anchored bundle $(\cal A,M,\rho)$  provided with an almost Lie bracket $[\;,\;]_{\cal A}$ such that $(\cal A,M,\rho,[\;,\;]_{\cal A})$ is a pre-Lie algebroid.
According to  \cite{HPo} or subsection \ref{subfinrap},  the Finslerian metric $\cal F$ is an $1$-homogeneous Lagrangian on
 $\cal M$ and  the associated $2$-homogeneous Lagrangian ${\cal L}=\dis\frac{1}{2}{\cal F}^2$ is a partial convex  hyperregular Lagrangian.
 Since  we have already seen in Subsection \ref{normalext}  from the canonical Lie structures on $\cal A$,   we get on  $\cal M$ a canonical almost tangent structure ${\cal J}$
 and an Euler section $\cal C$ a Riemannian metric $g_{\cal L}$ on
 the vertical sub-bundle   of ${\T}{\cal M}$.  On the other hand,  a  canonical  almost cotangent structure
 $\O_{\cal L}=-d^{\cal P}{\cal J}^*d^{\cal P}{\cal L}$, a canonical spray ${\cal S}_{\cal L}$, and also a canonical nonlinear connection
  ${\cal N}_{{\cal S}_{\cal L}}$ which is $g_{\cal L}$-metric Lagrangian and conservative. In particular the geodesics of ${\cal N}_{\cal L}$ are
 the integral curves of ${\cal S}_{\cal L}$. %Of course, as we have already seen the Riemannian metric $g_{\cal L}$ \\

 From now on, the partial Finsler  quasi-algebroid  $(\cal A,M,\rho,[\;,\;]_{\cal A}, {\cal  F})$ is fixed. Then we have the following geometrical objets associated with this structure:

 $\bullet$  {\it The Riemannian metric $g_{\cal F}$ called the   Finsler Riemannian metric or the fundamental tensor.}

 $\bullet$ {\it  The Cartan tensor $\mathbb{A}$   denoted $\mathbb{A}_{\cal F}$ in the sequel.}

 $\bullet$  {\it The spray ${\cal S}_{\cal L}$  called the Finsler spray and  denoted $S_{\cal F}$ in the sequel.}

 $\bullet$ {\it The nonlinear  connection  ${\cal N}_{{\cal S}_{\cal L}}$  called the Finsler connection and    denoted ${\cal N}_{\cal F}$ in the sequel.}\\

   Denote by  $\bf{H}{\cal M}$   the  horizontal space in $\T{\cal M}$ associated with ${\cal N}_{\cal F}$. According to the decomposition $\T{\cal M}=\bf{H}{\cal M}\oplus\bf{V}{\cal M}$, we have the horizontal  and   the vertical projector denoted $\bf{h}_{\cal F}$ and  $\bf{v}_{\cal F}$ respectively. We will denote by $\theta:\T{\cal M}\ap \tilde{\cal A}$ the composition $\vartheta\circ \bf{v}_{\cal F}$. The announced linear  Chern connection $\nabla$ \footnote{such a connection must satisfy the classical property of a linear connection adapted to this context:\\
 $-\;\; \nabla_{f{\cal X}}Y=f\nabla_{\cal X}Y,\;\; \nabla_{{\cal X}+{\cal Y}}Z=\nabla_{\cal X}Z+\nabla_{\cal Y}Z$\\
 $-\;\; \nabla_{\cal X}(Y+Z)=\nabla_{\cal X}Y+\nabla_{\cal X}Z$\\
 $-\;\; \nabla_{\cal X}fY=\hat{\rho}({\cal X})(f) Y+f\nabla_{\cal X}Y$\\
 for all sections ${\cal X}$ and ${\cal Y}$ of $\T{\cal M}$, all sections $Y$ and $Z$ of $\tilde{\cal A}$ and smooth function $f$ on $\cal M$.} on $\tilde{\cal A}$ and its relation with the set of sectional linear connections $\nabla^U$ is given in the following result:

  \begin{The}\label{chern}${}$
  \begin{enumerate}
 \item There exists a unique linear connection
 $$ \begin{matrix}
 \nabla:& \Xi(\T{\cal M})\times \Xi(\tilde{\cal A})&\ap& \Xi(\tilde{\cal A})\\
 &({ \cal X},\tilde{Y})&\ap& \nabla_{\cal X}\tilde{Y}\\
 \end{matrix}$$
which satisfies  the following properties  for any ${\cal X}\in \Xi(\T{\cal M})$ and $\tilde{Y}\in\Xi(\tilde{\cal A})$:
 \begin{enumerate}
 \item[(a)] (torsion free property)
 $$\nabla_{\cal X}({\pi}_{\tilde{\cal A}}({\cal Y}))-\nabla_{\cal Y}({\pi}_{\tilde{\cal A}}({\cal X}))={\pi}_{\tilde{\cal A}}([{\cal X},{\cal Y}]_{\cal P});$$
 \item[(b)] (almost $g_{\cal F}$-compatibility)
 $$\hat{\rho}({\cal X})(g_{\cal F}(\tilde{Y},\tilde{Z}))=g_{\cal F}(\nabla_{\cal X} \tilde{Y}, \tilde{Z})+g_{\cal F}(\tilde{Y},\nabla_{\cal X}\tilde{Z})+\mathbb{A}_{\cal F}(\theta({\cal X}),\tilde{Y},\tilde{Z})$$
for all sections $\cal X$ of $\T{\cal M }$ and $\tilde{Y}$ and $\tilde{Z}$ of  $\tilde{\cal A}$.
\end{enumerate}
\item  For any (local) section $Y$ of $\cal A$, we denote by $\tilde{Y}$ the associate section of $\tilde{\cal A}$ such that $\tilde{\pi}\circ \tilde{Y}=Y\circ \pi$. Given any section of $\cal M$ defined on an open $O\subset M$ and any section $X$ and $Y$ defined on $O$, we have:
\begin{eqnarray}\label{globsec}
\tilde{\pi}(\nabla_{\cal X}\tilde{Y}(x,U(x)))=\nabla^U_XY(x)
\end{eqnarray}
where   $\cal X$  is  any section of $\T{\cal M}$  which is projectable on $X$  (see Definition \ref{Mproj}).
%\item Let $N$ be a leaf of the foliation defined by $\rh({\cal A})$.
\end{enumerate}
 \end{The}

 If ${\cal A}=TM$,  and $\cal F$ is a Finsler metric on $M$, then $\cal M$ is the complementary of the zero section in $TM$ and then $\T{\cal M}=T{\cal M}$ and $\tilde{\cal A}=\pi^*TM$ where again $\pi:{\cal M}\ap M$ is the canonical projection. Then the  linear connection $\nabla$ defined in Theorem \ref{chern} is the classical Chern connection $\pi^*TM$ ({\it cf.} \cite{Dj}). Therefore we have:\\

 \begin{Def}\label{Chernnabla}${}$\\
 The linear connection $\nabla$ defined in Theorem \ref{chern} is called the Chern connection of the partial Finsler metric $\cal F$.\\
 \end{Def}

\begin{Rem}\label{relsectionchern}${}$\\
According to the decomposition $\T{\cal M}={\bf H}{\cal M}\oplus\V{\cal M}$  each ${\cal X}\in \Xi(\T{\cal M})$can be written  ${\cal X}={\cal X}^H+{\cal X}^V$. Therefore we can decompose $\nabla$ as a sum $\nabla^H+\nabla^V$ where $\nabla^H_{\cal X}=\nabla_{{\cal X}^H}$ and $\nabla^V_{\cal X}=\nabla_{{\cal X}^V}$. It follows that in Point (iii) in Theorem \ref{chern}, for any local section $X\in \Xi(M)$, there exists a unique $\pi_{\cal A}$-horizontal lift $\cal X$ of $X$. Thus,  in the Equation (\ref{globsec}) we have
$$\tilde{\pi}(\nabla^H_{\cal X}\tilde{Y}(x,U(x)))=\nabla^U_XY(x)$$
where   $\cal X$  is  any section of $\T{\cal M}$.\\
\end{Rem}
{\it We end this subsection by the proof of Theorem \ref{chern} but for this purpose we need to introduce the Bott connection on the horizontal bundle in $\T{\cal M}$ associates to the Finsler connection ${\cal N}_{\cal F}$.}\\

 \bigskip

%According to the decomposition $\T{\cal M}=\bf{H}{\cal M}\oplus\bf{V}{\cal M}$, we have the horizontal (resp. vertical) projector denoted $\bf{h}_{\cal F}$ (resp. $\bf{v}_{\cal F}$).
 We will denote by
 %$\theta:\T{\cal M}\ap \tilde{\cal M}$ the composition $\n\circ\bf{v}_{\cal F}$ and
  $\varpi :\bf{H}{\cal M}\ap \tilde{\cal A}$ the isomorphism which is  the restriction of  ${\pi}_{{\tilde{\cal A}}}$ to $\bf{H}{\cal M}$. Each section $\cal X$ of $\T{\cal M}$ can be written ${\cal X}={\cal X}^H+{\cal X}^V$ where ${\cal X}^H=\bf{h}_{\cal F}({\cal X})$ and ${\cal X}^V=\bf{v}_{\cal F}({\cal X})$. \\
 On the one hand we have a natural almost Lie bracket $[\;,\;]_{\bf{H}{\cal M}}$ on $\bf{H}{\cal M}$ defined by the projection of $[\;,\;]_{\cal P}$ on $\bf{H}{\cal M}$ by $\bf{h}_{\cal F}$ and so we get a natural structure of almost Lie algebroid on ${\bf H}{\cal M}$. Since   the  $\varpi$ is a bundle isomorphism we obtain a structure of almost Lie algebroid on $\tilde{\cal A}$. We denote by $[\;,\;]_{\tilde{\cal A}}$ the induced almost bracket on $\tilde{\cal A}$. %Note that the bracket $[\;,\;]_{\tilde{\cal A}}$ is independent of the choice of the projection $\cal P$ (see Remark \ref{barcalA} ).
 Consider  a local  basis $\{{\cal X}_\a,{\cal V}_\b\}$ of   $\T{\cal M}$ associated with a coordinate system $(x^i)$  and a local basis $\{e_\a\}$ of $\cal A$ on an open set $O$ of $M$. The horizontal component of ${\cal X}_\a$ is ${\cal X}^H_\a={\cal X}_\a-{\cal N}_\a^\b{\cal V}_\b$ and  $\{{\cal X}^H_\a,{\cal V}_\a\}$ is also a local basis of $\T{\cal M}$ such that $\{{\cal X}^H_\a\}$ is a basis of $\bf{H}{\cal M}$. Since  $({\pi}_{\tilde{\cal A}})_{| \bf{H}{\cal M}}$ is a bundle isomorphism  from $\bf{H}{\cal M}\ap {\cal M}$ on $ \tilde{\cal A}\ap {\cal M}$  we have ${\pi}_{\tilde{\cal A}}({\cal X}^H_\a)=\tilde{e}_\a$ where   $\{\tilde{e}_\a\}$ is the local basis of $\tilde{\cal A}$ over $\pi^{-1}(O)$ characterized by
  $$\tilde{e}_\a\circ \pi=\tilde{\pi}\circ e_\a\textrm{ for all } \a.$$
According to (\ref{locbracalg}) we then have:
\begin{eqnarray}\label{locbractildeA}
\bar{\pi}([\tilde{e}_\a,\tilde{e}_\b]_{\tilde{\cal A}})=C_{\a\b}^\a\tilde{e}_{\g}=[{e}_\a,e_\b]_{\cal A}.
\end{eqnarray}
 For any section ${\cal X}$ of $\T{\cal M}$, we have ${\pi}_{\tilde{\cal A}}({\cal X})={\pi}_{\tilde{\cal A}}({\cal X}^H)$ and so each section of $\T{\cal M}$ is projected by ${\pi}_{\tilde{\cal A}}$ on a section of $\tilde{\cal A}$. Note that we  have ${\pi}_{\tilde{\cal A}}({\cal X})=\vartheta\circ (J{\cal X})$ and also $\varpi=\vartheta\circ J_{| \bf{H}{\cal M}}$\\

  {\it For the sake of simplicity   we denote by $\tilde{X}$, $\tilde {Y}$, $\tilde{Z},\dots$ the projection of sections $\cal X$, $\cal Y$, ${\cal Z},\dots$ of  $\T{\cal M}$. }\\

 On the other hand we have an isomorphism ${\cal H}_{\cal F}:\bf{V}{\cal M}\ap \bf{H}{\cal M}$ ({\it cf.} Proposition \ref{JG}) and so we get a Riemannian metric $g^H$ on $\bf{H}{\cal M}$ defined by
 $$g^H({\cal X},{\cal Y})=g_{\cal F}({\cal H}_{\cal F}({\cal X}),{\cal H}_{\cal F}({\cal Y})).$$
 If we require that $\bf{V}{\cal M}$  and $\bf{H}{\cal M}$ are orthogonal we obtain a Riemannian metric $g$ on $\T{\cal M}$.  Then we have

 \begin{Lem}\label{LC}${}$\\
 There exists a unique linear connection\footnote{in the sense of section \ref{ALbracket}} ${\cal D^H}$ on the almost Lie algebroid $(\bf{H}{\cal M},{\cal M},\hat{\rho},[\;,\;]_{\cal P})$ such that for all sections $\cal X$, $\cal Y$ and $\cal Z$ of $\bf{H}{\cal M}$ we have

$\hat{\rho}({\cal X})(g^H({\cal Y},{\cal Z}))=g^H({\cal D}^H_{\cal X}{\cal Y},{\cal Z})+g^H({\cal D}^H_{\cal X}{\cal Z},{\cal Y}) $

${\cal D}^H_{\cal X}{\cal Y}-{\cal D}^H_{\cal Y}{\cal X}=[{\cal X},{\cal Y}]_{\bf{H}{\cal M}}$

This connection is called the {\it Levi-Civita} connection.
\end{Lem}

 The proof of  this Lemma is formally the same as in the context of Riemannian Lie algebroid (see \cite{Bo} section 3.1).\\

 Following the definitions of Cartan tensors of \cite{SV} section 3.1, we consider the  $\T{\cal M}$-tensor  $\hat{ \mathbb{A}}$  of type $(3,0)$ characterized by
 $$\hat{\mathbb{A}}({\cal X},{\cal Y},{\cal Z})=({\cal L}^{\cal P}_{{\cal X}^V}J^*g)({\cal Y}{\cal Z})$$
 for all ${\cal X}$, ${\cal Y}$ and $\cal Z$ in $\Xi(\T{\cal M})$.\\

 \begin{Rem}\label{hatAgF}${}$\\
 Since  $J^*g({\cal X},{\cal Y})=g_{\cal F}(J{\cal X}, J{\cal Y})$ it follows that we have
 $$\hat{\mathbb{A}}({\cal X},{\cal Y},{\cal Z})=({\cal L}^{\cal P}_{{\cal X}^V}g_{\cal F})(J{\cal Y}, J{\cal Z}).$$
 In particular if $\cal X$ is horizontal $\hat{\mathbb{A}}({\cal X},{\cal Y},{\cal Z})=0$  as soon as ${\cal Y}$ or $\cal Z$ are vertical.
  \end{Rem}

The tensor $ \hat{\mathbb{A}}$ will be called the {\it lift Cartan tensor}. The link between $\hat{ \mathbb{A}}$ and $ \mathbb{A}_{\cal F}$ is given by Point (iii) in the following Proposition:

 \begin{Pro}\label{Bott}${}$
 \begin{enumerate}
\item[(i)] The map ${\cal D}:\Xi(\T{\cal M})\times\Xi(\bf{H}{\cal M})\ap \Xi(\bf{H}{\cal M})$ is defined by :

${\cal D}_{\cal X}^H{\cal Y}={\cal D}^H_{{\cal X}^H}{\cal Y}$

${\cal D}_{\cal X}^V{\cal Y}=\bf{h}_{\cal F}[{\cal X}^V,{\cal Y}]_{\cal P}$
For all ${\cal X}\in \Xi(\T{\cal M})$ and ${\cal Y}\in \Xi(\bf{H}{\cal M})$

is a linear connection \footnote{in the obvious same sense as footnote (6)}.
  \item[(ii)] The connection $\cal D$ is the unique linear connection which satisfies  the following properties for any ${\cal X}, {\cal Y}, {\cal Z}\in \Xi(\T{\cal M})$ :
 \begin{enumerate}
 \item[(a)] (torsion free property)
 $${\cal D}_{\cal X}({\cal Y}^H)-{\cal D}_{\cal Y}({\cal X}^H)=\bf{h}_{\cal F}([{\cal X},{\cal Y}]_{\cal P});$$
\item[(b)] (horizontal  $g^H$ compatibility)
$$\hat{\rho}({\cal X}^H(g^H({\cal Y},{\cal Z})=g^H({\cal D}_{\cal X}{\cal Y},{\cal Z})+g^H({\cal D}_{\cal X}{\cal Z},{\cal Y}); $$
 \item[(c)] (almost vertical  $g^H$ compatibility)
 $$\hat{\rho}({\cal X^V})(g^H({\cal Y},{\cal Z}))=g_{H}({\cal D}_{{\cal X}^V}{\cal Y},{\cal Z})+g_{H}({\cal Y},{\cal D}_{{\cal X}^V}{\cal Z})+2\hat{\mathbb{A}}(({\cal X}, {\cal Y},{\cal Z});$$
\end{enumerate}
\item[(iii)] we have the relation
$$\hat{\mathbb{A}}({\cal X},{\cal Y},{\cal Z})=\vartheta^*\mathbb{A}_{\cal F}({\cal X}^V,J{\cal Y},J{\cal Z})$$
for all sections $\cal X$, $\cal Y$, $\cal Z$ of $\T{\cal M}$
\end{enumerate}
\end{Pro}

Note that when ${\cal A}=TM$  and $\cal F$ is a Finsler metric on $M$, we have already seen that  ${\cal M}$ is the complementary of the zero section in $\cal A$ and  $\T{\cal M}=T{\cal M}$. With these notations, $\cal D$ is the Bott connection of $\bf{H}{\cal M}$ considered as the normal bundle of the vertical foliation of $T{\cal M}$. Therefore we have:
\begin{Def}\label{defbott}${}$\\
The linear connection $\cal D$ is called the Bott connection of $\bf{H}{\cal M}$.
\end{Def}
\begin{proof}[Proof of Proposition \ref{Bott}]${}$\\
Point (i) is elementary.\\
For Point (ii) we  must show first that the Bott connection $\cal D$ satisfies properties (a), (b) and (c).\\
Since  we have $[{\cal X}^V,{\cal Y}^V]_{\cal P}$ always belongs to $\Xi({\bf{V}(\cal M})$ property (a) is a consequence of the properties of ${\cal D}^H$ and the definition of ${\cal D}_{\cal X}^V{\cal Y}$.\\
Property (b) is a direct consequence of properties of ${\cal D}^H$.\\

We have

$\hat{\rho}({\cal X^V})(g^H({\cal Y},{\cal Z}))-g^{H}({\cal D}_{{\cal X}^V}{\cal Y},{\cal Z}-g{H}({\cal Y},{\cal D}_{{\cal X}^V}{\cal Z})$\\
${}\;\;\;\;\;\;\;\;\;\;\;=\hat{\rho}({\cal X^V})(g^H({\cal Y},{\cal Z}))-g^{H}(h_{\cal F}([{\cal X}^V,{\cal Y}]_{\cal P}),{\cal Z})-g^{H}(h_{\cal F}([{\cal X}^V,{\cal Z}]_{\cal P}),{\cal Y})$\\
${}\;\;\;\;\;\;\;\;\;\;\;=\hat{\rho}({\cal X^V})(g_{\cal F}(J{\cal Y},J{\cal Z}))-g_{\cal F}(J([{\cal X}^V,{\cal Y}]_{\cal P}),J{\cal Z})-g_{\cal F}(J([{\cal X}^V,{\cal Z}]_{\cal P}),J{\cal Y})$.

From the definition of $\mathbb{A}_{\cal F}$ and the fact that ${\cal H}_{\cal F}\circ J={\bf h}_{\cal F}$ we obtain:

$2\hat{\mathbb{A}}({\cal X}, {\cal Y},{\cal Z})= \hat{\rho}({\cal X}^V)(g(J{\cal Y},J{\cal Z}))-g(J([{\cal X}^V,{\cal Y}]_{\cal P}),J{\cal Z})-g(J([{\cal X}^V,{\cal Z}]_{\cal P}),J{\cal Y}) $\\
${}\;\;\;\;\;\;\;\;\;\;\;\;\;\;\;\;\;\;\;\;\;\;\;\;= \hat{\rho}({\cal X}^V)(g^H({\cal Y},{\cal Z}))-g^H({\bf h}_{\cal F}([{\cal X}^V,{\cal Y}]_{\cal P}),{\cal Z})-g^H({\bf h}_{\cal F}([{\cal X}^V,{\cal Z}]_{\cal P}),{\cal Y}) $\\
${}\;\;\;\;\;\;\;\;\;\;\;\;\;\;\;\;\;\;\;\;\;\;\;\;= \hat{\rho}({\cal X}^V)(g^H({\cal Y},{\cal Z}))-g({\cal D}_{{\cal X}^V}{\cal Y},{\cal Z})-g^H({\cal D}_{{\cal X}^V}{\cal Z},{\cal Y}). $

For the uniqueness  assume that ${\cal D}'$ is a linear connection on $\bf{H}{\cal M}$ (in the sense of the linear  Bott connection) which satisfies the properties of Point (ii). We first decompose such a connection  in ${\cal D}'_{\cal X}={\cal D}'_{{\cal X}^H}+{\cal D}'_{{\cal X}^V}$ and we set ${{\cal D}'}^H_{\cal X}={\cal D}'_{{\cal X}^H}$ and  ${{\cal D}'}^V_{\cal X}={\cal D}'_{{\cal X}^V}$. From (a) and (b) and the uniqueness  of the Levi-Civita connection, we must have ${{\cal D}'}^H={\cal D}^H$. As classically  in the case of the Levi-Civita  connection, by standard computation (as in the proof of Proposition \ref{exisnablas}) ${{\cal D}'}^V$ must satisfies the relation:
\begin{eqnarray}\label{kozulD}
2g^H({{\cal D}'}_{{\cal X}}^V{\cal Y}^H,{\cal Z}^H)=\hat{\rho}({\cal X}^V)(g^H({\cal Y}^H,{\cal Z}^H))+\hat{\rho}({\cal Y}^V)(g^H({\cal X}^H,{\cal Z}^H))-\hat{\rho}({\cal Z}^V)(g^H({\cal X}^H,{\cal Y}^H)\;\;\;\;\;\;\;\;\;\;\;\;\;\;\;\;\;\;\;\;\;\;\;\;\;\;\;\;{}\nonumber\\
{}\;\;\;\;\;\;\;\;\;\;\;\;\;\;\;\;\;\;\;\;\;\;\;\;\;\;+g^H({\bf h}_{\cal F}([{\cal X}^V,{\cal Y}^H
]_{\cal P}),{\cal Z}^H)+g^H({\bf h}_{\cal F}([{\cal Z}^V,{\cal X}^H]_{\cal P}),{\cal Y}^H)-g^H({\bf h}_{\cal F}([{\cal Y}^V,{\cal Z}^H]_{\cal P}),{\cal X}^H)\nonumber\\
-2(\hat{\mathbb{A}}({\cal X},Y,Z)+\hat{\mathbb{A}}({\cal Y},{\cal X},{\cal Z})-\hat{\mathbb{A}}({\cal Z},{\cal X},{\cal Y})).\;\;\;\;\;\;\;\;\;\;\;\;\;\;\;\;\;\;\;\;\;\;\;\;\;\;\;\;\;\;\;\;\;\;\;\;\;\;\;\;\;\;\;\;\;\;\;\;\;\;\;\;\;\;\;\;\;\;\;\;{}
\end{eqnarray}
Since $\hat{\mathbb{A}}$ is tensorial, as in the classical case of the Levi-Civita connection, the relation define a unique differential operator which is of course ${\cal D}^V$.\\

Clearly, point (iii) is a local property. Therefore consider   a  coordinate system  $(x^i)$ on $O$ and   a local basis $\{e_\a\}$ of $\cal A$ over $O$. Let $(x^i,y^\a)$ the associated coordinate system on $\cal M$ and  $\{{\cal X}_\a,{\cal V}_\a\}$ be the associated basis on $\T{\cal M}$;\\
Since  $(\mathbb{A}_{\cal F})_{(x,u)}(w_1,w_2,w_3)=\mathbb{A}_{u}(w_1,w_2,w_3)$, from Remark \ref{Au=0} and via the isomorphism $\vartheta$ between $\tilde{\cal A}$ and $\bf{V}{\cal M}$
we have: \\

$2\vartheta^*(\mathbb{A}_{\cal F})({\cal V}_\a,{\cal V}_\b,{\cal V}_\g)=\dis\frac{\p g_{\cal F}}{\p y^\a}({\cal V}_\b,{\cal V}_\g)$\\
${}\;\;\;\;\;\;\;\;\;\;\;\;\;\;\;\;\;\;\;\;\;\;\;\;\;\;\;\;\;\;\;\;\;\;\;\;\;=({\cal L}^{\cal P}_{{\cal V}_\a} g_{\cal F})({\cal V}_\b,{\cal V}_\g)$\\
${}\;\;\;\;\;\;\;\;\;\;\;\;\;\;\;\;\;\;\;\;\;\;\;\;\;\;\;\;\;\;\;\;\;\;\;\;\;=({\cal L}^{\cal P}_{{\cal V}_\a} g^H)({\cal H}_{\cal F}{\cal V}_\b,{\cal H}_{\cal F}{\cal V}_\b)$\\
${}\;\;\;\;\;\;\;\;\;\;\;\;\;\;\;\;\;\;\;\;\;\;\;\;\;\;\;\;\;\;\;\;\;\;\;\;\;=({\cal L}^{\cal P}_{{\cal V}_\a} g_{\cal F})(J{\cal X}_\b,J{\cal X}_\b)$.\\

From Remark \ref{hatAgF} we then  obtain
$$\vartheta^*(\mathbb{A}_{\cal F})({\cal V}_\a,{\cal V}_\b,{\cal V}_\g)=\hat{\mathbb{A}}({\cal X}_\a,{\cal X}_\b,{\cal X}_\b).$$
This ends the proof of Point (iii).

\end{proof}
\bigskip
\begin{proof}[Proof of Theorem \ref{chern}]${}$\\
 According to Proposition \ref{Bott}
 we define  $\nabla: \Xi({\cal M})\times \Xi(\tilde{\cal A})\ap \Xi(\tilde{\cal A})$ by:
 $$\nabla_{\cal X}\tilde{Y}=\varpi({\cal D}_{\cal X}(\varpi^{-1}(\tilde{ Y})).$$
We denote by $\nabla^H$ and $\nabla^V$ the corresponding decomposition of ${\cal D}$ into ${\cal D}^H$ and ${\cal D}^V$. Since  $\varpi$ is an isomorphism between $\bf{H}{\cal M}$ and $\tilde{\cal A}$, according to the properties (a),  (b) and (c) of Proposition \ref{Bott}   and Point (iii) in Proposition \ref{Bott}, it follows easily that $\nabla$ is  the unique  linear connection on $\tilde{\cal A}$ which satisfies Point (1) in  Theorem \ref{chern}. More precisely, taking into account the fact that ${\cal D}^H$ is the Levi-Civita connection, the relation  (\ref{kozulD}) for ${\cal D}^v$, Remark \ref{hatAgF}, and the relation $\pi_{\tilde{\cal A}}=\vartheta\circ J$,  it follows that    $\nabla_{\cal X}\tilde{Y}$ is characterized by
 \begin{eqnarray}\label{kozulnabla}
2g_{\cal F}({\nabla}_{{\cal X}}\tilde{ Y},\tilde{ Z})=\hat{\rho}({\cal X})(g_{\cal F}(\tilde{ Y},\tilde{ Z}))+\hat{\rho}({\cal Y})(g_{\cal F}(\tilde{\cal X}, \tilde{ Z}))-\hat{\rho}({\cal Z})(g_{\cal F}(\tilde{ X},\tilde{Y})\;\;\;\;\;\;\;\;\;\;\;\;\;\;\;\;\;\;\;\;\;\;\;\;\;\;\;\;{}\nonumber\\
{}\;\;\;\;\;\;\;\;\;\;\;\;\;\;\;\;\;\;\;\;\;\;\;\;\;\;\;\;\;\;\;\;\;\;\;\;\;\;\;\;\;\;\;\;+g_{\cal F}(\pi_{\tilde{\cal A}}([{\cal X},{\cal Y}
]_{\cal P}), \tilde{ Z})-g_{\cal F}(\pi_{\tilde{\cal A}}([{\cal Y},{\cal  Z}]_{\cal P}),\tilde{X})-g_{\cal F}(\pi_{\tilde{\cal A}}([{\cal Z},{\cal X}]_{\cal P}),\tilde{Y})\nonumber\\
{}\;\;\;\;\;\;\;\;\;\;\;\;\;\;\;\;\;\;\;\;\;\;\;\;\;\;\;\;\;\;\;\;\;\;\;\;\;\;\;\;\;\;\;\;\;\;\;\;\;\;\;\;\;\;\;\;\;\;\;\;
-2({\mathbb{A}}_{\cal F}(\theta({\cal X}),\tilde{Y},\tilde{Z})+{\mathbb{A}}_{\cal F}(\theta({\cal Y}),\tilde{ X},\tilde{Z})-{\mathbb{A}}_{\cal F}(\theta({\cal Z}),\tilde{ X},\tilde{Y}))\end{eqnarray}
where $\cal Y$ and $\cal Z$ are such that $\tilde{Y}=\pi_{\tilde{\cal A}}({\cal Y})$  and $\tilde{Z}=\pi_{\tilde{\cal A}}({\cal Z})$ respectively.\\

 For the proof of the last Point we need:

 Consider   a local basis $\{{\cal X}_\a,{\cal V}_\b\}$  of $\T{\cal M}$ associated with a coordinate system $(x^i)$ on $M$ and a basis $\{e_\a\}$ of $\cal A$. Denote by $\{\tilde{e}_\a\}$  the basis of $\tilde{\cal A}$ characterized by $\tilde{\pi}\circ\tilde{e}_\a=e_\a\circ \pi$.

\begin{Lem}\label{christof}${}$
\begin{enumerate}
\item  If we consider the decomposition
$\nabla_{{\cal X}_\a}\tilde{e}_\b=\G_{\a\b}^\g\tilde{e}_\g$ then we have
 \begin{eqnarray}\label{Christoffel}
\G_{\a\b}^\g=S_{\a\b}^\g+g^{\g\l}(-{{\cal N}}_\a^\mu\mathbb{A}_{\mu\b\l}-{{\cal N}}_\b^\mu\mathbb{A}_{\mu\a\l}+{{\cal N}}_\l^\mu\mathbb{A}_{\mu\a\b})
\end{eqnarray}
where
$${}\;\;\;S_{\a\b\g}=g_{\g\l}{\cal S}^\l_{\a\b}=-\dis\frac{1}{2}[(\rho_\a^i\dis\frac{\p g_{\b\g}}{\p x^i}+\rho_\b^i\dis\frac{\p g_{\a\g}}{\p x^i}-\rho_\g^i\dis\frac{\p g_{\a\b}}{\p x^i})+(g_{\l\g}C_{\a\b}^\l+g_{\l\b}C_{\g\a}^\l+g_{\l\a}C_{\g\b}^\l)],$$
${\cal N}_\a^\b$ are the coefficients of the Finsler connection in the basis  $\{{\cal X}_\a,{\cal V}_\b\}$\\
 $\mathbb{A}_{\a\b\g}=\mathbb{A}_{\cal F}(\tilde{e}_\a,\tilde{e}_\b,\tilde{e}_\g)$.
\item Given a section $U$ of $\cal M$ defined on an open $O$ in $M$   we consider the decomposition $$\nabla^U_{e_\a}e_\b=(\G^U)_{\a\b}^\g e_\g.$$
Then we have
\begin{eqnarray}\label{kozullocfin2}
(\G^U)_{\a\b}^\g =S_{\a\b}^\g+g^{\g\l}(-({\cal N}^U)_\a^\mu\mathbb{A}_{\mu\b\l}-({\cal N}^U)_\b^\mu\mathbb{A}_{\mu\a\l}+({\cal N}^U)_\l^\mu\mathbb{A}_{\mu\a\b})
\end{eqnarray}
where $({\cal N}^U)_\a^\b(x)={{\cal N}}_\a^\b(x,U(x))$. In particular we have  $\G_{\a\b}^\g(x,U(x))=(\G^U)_{\a\b}^\g(x)$
\end{enumerate}
\end{Lem}

Now,  for any section $X$ of $\cal A$ we denote by  $\tilde{X}$ the associated section of $\tilde{\cal A}$ characterized by $\tilde{\pi}\circ \tilde{X}=X\circ \pi$.
 We will say that such  a section  $\tilde{X}$ is projectable.  Note that, if   ${\cal X}$ is a projectable section  of $\T{\cal M}$  on a section $X$ of $\cal A$ then $\pi_{\tilde{\cal A}}\circ {\cal X}=\tilde{X}$ is a  section of $\tilde{\cal A}$ which is  also  projectable on $X$. Moreover given any section $X$ of $\cal A$,  there always exists a  section $\cal X$ of $\T{\cal M}$ which is projectable  on $X$ and the difference of two such sections is vertical and so $\pi_{\tilde{\cal A}}\circ {\cal X}$ is independent of the choice of such a projectable section $\cal X$. From now to the end of this proof we only consider  projectable  sections of $\T{\cal M}$. Therefore,  in this case   for any  function $f$ on $M$  we have  $\hat{\rho}({\cal X})(f\circ \pi)=\rho(X)(f)$\\
 % On one hand,  according to notations of paragraph \ref{partfin}, if $U$ is a section of $\cal M$ over an open $O\subset M$, along $U(O)\subset {\cal M}$ we have $g_{\cal F}(\tilde{Y},\tilde{Z})=g_U(Y,Z)$ and so we obtain
%\begin{eqnarray}\label{rhoX}
%\hat{\rho}({\cal X})(g_{\cal F}(\tilde{Y},\tilde{Z}))=\rho(X)(g_U(Y,Z)) \textrm{ on } U(O).
%\end{eqnarray}
% On the other hand, according to (\ref{locbractildeA}), for any projectable sections $\tilde{Y}$ and $\tilde{Z}$ of $\tilde{\cal A}$, we have:
%\begin{eqnarray}\label{[XY]}
%[  \tilde{Y},\tilde{Z}]_{\tilde{\cal A}}=\tilde{\pi}[Y,Z]_{\cal A} \textrm{ on } U(O).
%\end{eqnarray}
%Finally, according to

 Fix some  section $U$   of $\cal M$ over an open $O$ of $M$. Given any section $X$ and $Y$  of $\cal A$ defined on $O$ we consider the decompositions
 $$U=U^\a e_\a,\;\;\;\; X=X^\a e_\a\;\;\;\; Y=Y^\a e_\a.$$
 Any   section $\cal X$ of $\T{\cal M}$ projectable on $X=X^\a e_\a$ can be written
 ${\cal X}=X^\a{\cal X}_\a+ \bar{X}^\b{\cal V}^\b$.\\
 On  the one hand we have
 $$\nabla^U_XY=[X^\a\rho_\a^i\dis\frac{\p Y^\b}{\p x^i}+(\G^U)_{\a\b}^\g X^\a Y^\b]e_\b.$$
 On the other hand,    from the construction of $\nabla$ from the Bott connection $\cal D$ we have $\nabla _{{\cal V}_\a}\tilde{e}_\b=0$.\\
Therefore we obtain:\\
 $\nabla_{\cal X}\tilde{Y}=[\hat{\rho}({\cal X})(Y^\b)+{\G}_{\a\b}^\g X^\a Y^\b] \tilde{e}_\g$.\\
 Since $Y^\b$ is a function on $M$, we get:\\
 $\hat{\rho}({\cal X})(Y^\b)=\rho(X)(Y^\b)=X^\a\rho_\a^i\dis\frac{\p Y^\b}{\p x^i}.$\\
Now, according to Lemma \ref{christof}, on $U(O)\subset {\cal M}$  we have:\\
$\G_{\a\b}^\g(x,U(x))=(\G^U)_{\a\b}^\g(x)$\\
 Therefore we obtain:\\
 $\tilde{\pi}(\nabla_{\cal X}\tilde{Y}(x,U(x)))=\nabla^U_XY(x).$\\
\end{proof}

\bigskip
\begin{proof}[Proof of Lemma \ref{christof}]${}$\\
In the context of the Lemma we have the decomposition ${\cal S}_{\cal F}=y^\a{\cal X}_\a+{\cal S}^\b{\cal V}_\b$. According to (\ref{iSO}), each component ${\cal S}_\a=g_{\a\b}{\cal S}^b$ is given by
\begin{eqnarray}\label{Sa}
-\dis\frac{1}{2}(\rho_\g^i\frac{\p g_{\a\d}}{\p x^i}+\rho_\d^i\frac{\p g_{\a\g}}{\p x^i}-\rho_\a^i\frac{\p g_{\d\g}}{\p x^i}+g_{\g\mu}C_{\a\d}^\mu+g_{\d\mu}C_{\a\g}^\mu+g_{\a\mu}C_{\g\d}^\mu)y^\g y^\d.
\end{eqnarray}
Therefore from  (\ref{Sabc}) we get $S_\a=-S_{\a\g\d}y^\g y^\d.$
Now, according to  (\ref{GSS}) the coefficient ${\cal N}_{\a\b}=g_{\a\g}{\cal N}_{\a}^\g$  of  the Finsler connection ${\cal N}_{\cal F}$ is:

$$\dis\frac{1}{2}[(\rho_\b^i\frac{\p g_{\a\d}}{\p x^i}+\rho_\d^i\frac{\p g_{\a\b}}{\p x^i}-\rho_\a^i\frac{\p g_{\b\d}}{\p x^i}+g_{\g\d}C_{\a\b}^\g+g_{\b\g}C_{\a\d}^\g+g_{\a\g}C_{\b\d}^\g)y^\d+S^\g\mathbb{A}_{\g\a\b}].$$

Thus we obtain (\ref {Christoffel}).

Now  (\ref{kozullocfin2}) is exactly (\ref{kozullocfin}) (see proof of Proposition \ref{exisnablas})

Since  $S_\a(x,y)=-S_{\a\g\d}y^\g y^\d$   according  to (\ref{NU})  we obtain ${\cal N}_\a^\b (x,U(x))=({\cal N}^U)_\a^\b(x)$. \\
\end{proof}

\begin{Def} \label{symchristoffel}${}$\\
The set of all coefficients  $\G_{\a\b}^\g$ (see (\ref {Christoffel})) are called the Christoffel symbols of  $\cal F$.
\end{Def}

 \begin{Rem}\label{Gab}${}$\\
  When ${\cal A}=TM$,  the basis  $\{e_\a\}$ of  $\cal A$ is $ \{\dis\frac{\p}{\p x^i}\}$ and classically (see \cite{BCS} for instance)  the value of  $\G_{ij}^k$  is formally the expression (\ref {Christoffel}) with $\a\equiv i,\b\equiv j, \g\equiv k$,  where  all the coefficients  of type $C_{\g\d}^\mu$ are identically zero and  all the coefficients of type $\bar{{\cal N}}_{\a}^\b$ are the coefficients of the nonlinear  Finsler connection ${\cal N}_{\cal F}$.\\
  On the other hand,  the horizontal component of ${\cal X}_\a$ is ${\cal X}^H_\a={\cal X}_\a-{\cal N}_\a^\b{\cal V}_\b$.  We denote by $\d_\a$ the vector field $\hat{\rho}({\cal X}^H_\a)$. With these notations,  the expression (\ref {Christoffel}) of  $\G_{\a\b}^\g$ becomes:
\begin{eqnarray}\label{nablad}
{}\;\;\;\;\nabla_{\g\a\b}=g_{\g\l}\nabla_{\a\b}^\l=\dis\frac{1}{2}[(\d_\a(g_{\b\g})+\d_\b( g_{\a\g})-\d_\g(g_{\a\b})+(g_{\l\g}C_{\a\b}^\l+g_{\l\b}C_{\g\a}^\l+g_{\l\a}C_{\g\b}^\l)].
\end{eqnarray}
Again when ${\cal A}=TM$ we get a classical expression of the Christoffel symbols of the Chern connection in Finsler geometry  (see \cite{BCS} for instance).

\end{Rem}}

  %%%%%%%%%%%%%%%%%%%%%%%%%%%%%%%%%%%%%%%%%%%%%%%%%%%%%%%%%%%%%%%%%%%%%%%%%%%%%
\subsection{Finsler structure on a pre-Lie algebroid and  induced structure on leaves.}\label{reduc Chen}${}$\\
%%%%%%%%%%%%%%%%%%%%%%%%%%%%%%%%%%%%%%%%%%%%%%%%%%%%%%%%%%%%%%%%%%%%%%%%%%%%
{\rm Consider a partial Finsler metric on an anchored bundle $({\cal A},M,\rho)$ defined on conic open set $\cal M$ of $\cal A$. We provide $({\cal A},M,\rho)$ with an almost Lie bracket $[\;,\;]_{\cal A}$ such that $({\cal A},M,\rho,[\;,\;]_{\cal A})$ is a pre-Lie algebroid.} When ${\cal A}=TM$ and $[\;,\;]_{\cal A}$ is the classical Lie bracket of vector fields and ${\cal M}=TM\setminus\{0\}$, we obtain the classical Finsler and Chern connection on $M$. When ${\cal M}$ is a conic open set of $TM$ we get the Finsler and Chern connection associated with a partial Finsler structure on $M$ as described  in \cite{Ja} under the term of quadric metric.
Let $N$ be a leaf of the foliation defined by $\rho({\cal A})$. According to Example \ref{exfinsler} {\bf 4.}, ${\cal F}$ induces a partial partial Finsler metric $F_N$ on $N$  defined on an open set  $\bar{\cal M}_N$ of $TN$. Therefore, by application to the previous result to the partial Finsler  algebroid $(TN,N,Id,[\;,\;], F_N)$, we get a Finsler connection ${\cal N}_{F_N}$ and  and a Chern connection $\nabla^{F_N}$ canonically associated with the  partial Finsler metric  $F_N$ induced on $N$.\\

 On the other hand, the Lagrangian ${\cal L}$ associated with $\cal F$ is a partial convex  hyperregular Lagrangian on $\cal A$ and so
 according to Subsection \ref{hyperlag}, we obtain a characterization of locally minimizers of $\cal L$. Again as in the previous cases,  when ${\cal A}=TM$  or when $N$ is leaf of the foliation defined by $\rho({\cal A})$ we recover the classical notion of geodesics of a  Finsler metric on a manifold or more generally for a partial Finsler metric (see  \cite{Ja} for more detail in this last  case). The following result gives the link between these differential objects defined on a pre-Lie algebroid and on each leaf of the foliation defined by the range of the  anchor. %Before we need to recall the context and notations:

%Consider a partial Finsler metric on an anchored bundle $({\cal A},M,\rho)$ defined on conic open set $\cal M$ of $\cal A$. We provide $({\cal A},M,\rho)$ with an almost Lie bracket $[\;,\;]_{\cal A}$ such that $({\cal A},M,\rho,[\;,\;]_{\cal A})$ is an algebroid. Let $N$ a leaf of the foliation defined by $\rho({\cal A}$.

\begin{The}\label{allpropfinsler}:${}$\\
Consider a partial Finsler metric on an anchored bundle $({\cal A},M,\rho)$ defined on a conic open set $\cal M$ of $\cal A$. We provide $({\cal A},M,\rho)$ with an almost Lie bracket $[\;,\;]_{\cal A}$ such that $({\cal A},M,\rho,[\;,\;]_{\cal A})$ is a pre-Lie algebroid. Then we have the following properties:
\begin{enumerate}
\item[(i)]  The  Finsler connection ${\cal N}_{\cal F}$ is the unique $g_{\cal F}$- metric and Lagrangian connection (relative to $\O_{\cal F}$) associated with $S_{\cal F}$. Moreover $S_{\cal F}$ is the canonical spray of  ${\cal N}_{\cal F}$. Given a leaf $N$ of the foliation defined by $\rho({\cal A})$, let  $S_{F_N}$ the canonical spray associated with $F_N$. Then ${\cal N}_{\cal F}$ induces
a nonlinear connection ${\cal N}_{F_N}$ which is the Finsler connection associated with $F_N$ via   the projection $\T{\cal A}_N\ap\left(\T{\cal A}_N\right)/ \left({\bf K} {\cal A}_N\oplus(J({\bf K} {\cal A}_N)\right)\equiv T(TN)$ (see Proposition \ref{reducleaf}). In particular the induced nonlinear connection  ${\cal N}_{F_N}$ is independent of the choice of the almost Lie bracket $[\;,\;]_{\cal A}$.
\item[(ii)] The projection  on $M$ of any geodesic of ${\cal N}_{\cal L}$ is contained in some leaf $N$ of  the foliation defined by $\rho({\cal A})$ and is also a geodesic of ${\cal N}_{F_N}$. Moreover such a curve is both  a local minimizer of ${\cal L}$ and of $L_N=\dis\frac{1}{2}F_N^2$.
\item[(iii)]Denote by $\widetilde{TN}$ the pulback of $TN\ap N$ by the map projection $\bar{\cal M}_N \ap N$. Via the associated  projection of $\T{\cal A}_N\times \tilde{\cal A}_{N}\ap T(TN)\times \widetilde{TN}_{| \bar{\cal M}_N}$, the Chern connection $\nabla^{\cal F}$ of $\cal F$ induces  the Chern connection $\nabla^{F_N}$  of $F_N$ for any leaf $N$. In particular, $\nabla^{\cal F}$ is independent of the choice of the almost Lie bracket $[\;,\;]_{\cal A}$.
\end{enumerate}
\end{The}

\begin{Rem}\label{locexpG}${}$\\
We consider the local basis $ \{\tilde{e}_\a\}$ of $\tilde{\cal A}$ and $\{{\cal X}_\a^H,{\cal V}_\b\}$  defined in Subsection \ref{partfin} after Remark \ref{relsectionchern}. Recall that we have:
 $\nabla_{{\cal X}_\a^H}\tilde{e}_\b=\G_{\a\b}^\g \tilde{e}_\g$ and $\nabla_{{\cal V}_\a}\tilde{e}_\b=0$. Moreover  if we choose the local basis of $\cal A$ and the local coordinate system as in the proof of Proposition \ref{reducleaf},  from Point (iii) of Theorem \ref{extremalN} we will have

  $\bar{\pi}_N({\cal X}_\a^H)=\dis\frac{\d}{\d x^\a}=\dis\frac{\p}{\p x^\a}-\dis\sum_{\g=1}^q {\cal N}_\a^\g\dis\frac{\p}{\p y^\g}$ for $\a=1,\dots,q$  and $\bar{\pi}_N({\cal X}_\a^H)=0$ for $\a=q+1,\dots,k$

 $\bar{\pi}_N({\cal V}_\a)=\dis\frac{\p}{\p y^\a}$  for $\a=1,\dots,q$ and $\bar{\pi}_N({\cal V}_\a)=0$ for $\a=q+1,\dots,k$

With these notations,  it follows that the Christoffel symbols of ${\nabla}^{F_N}$ are $\G_{\a\b}^\g$ for $1\leq\a,\b,\g\leq q$.
\end{Rem}

\begin{proof}${}$\\
First of all note that  the Lagrangian $\cal L$ associated with $\cal F$ is homogeneous of degree $2$ and is convex. The first consequence is that  $({\cal M}, {\cal L},0)$ is a $\cal A$-mechanical system (see  Subsection \ref{NHmeca}) and   so   the first part of Point (i) comes from Theorem \ref{mecasyst}.  Again the same argument implies that  $L_N$ is homogeneous of degree $2$ and is convex and then the last part is a consequence of  Point (iii) of Theorem \ref{extremalN}.\\

From the same arguments Point (ii)  is a consequence of Point (i) of Theorem \ref{extremalN} according to the definition of geodesics of a connection and the fact that $S_{\cal F}$ and also $S_{F_N}$ are sprays.\\

For the Point (iii)  we use again the identifications  described at the  end of the proof of Lemma  \ref{reducN} and the construction of the Chern connection (see Subsection \ref{partfin}). Recall that in the one hand we have a projection $$\bar{\pi}_N: \T{\cal M}_N\ap   \left(\T{\cal M}_N\right)/ \left({\bf K} {\cal A}_N\oplus(J({\bf K} {\cal A}_N)\right)\equiv T\bar{\cal M}_N$$
which is also a Lie algebra homomorphism relatively to the brackets $[\;,\;]_{\cal P}$ on $\T{\cal M}_N$ and $[\;,\;]$ on $T\bar{\cal M}_N$. On the other hand we have ${\cal N}_{F_N}=\bar{\pi}_N\circ {{\cal N}_{\cal F}}_{| \T{\cal M}_N}$. Therefore  if we set ${\bf H}{\cal M}_N={\bf H}{\cal M}_{| {\cal M}_N}$ we obtain $\bar{\pi}_N({\bf H}{\cal M}_N)={\bf H}\bar{\cal M}_N$ that is the horizontal bundle in $T\bar{\cal M}_N$ associated with ${\cal N}_{F_N}$.

Now  if $\widetilde{TN}$ is the pullback of $TN$ over the  open set $\bar{\cal M}_N$, we have a natural projection of $\tilde{A}_N\ap \widetilde{TN}$ which is induced by the projection of ${\cal A}_N\ap TN$. It follows that we obtain a projection of $\T{\cal M}_N\times \tilde{\cal A}_{N}\ap T{\cal M}_N\times \widetilde{TN}$. Again since ${\cal N}_{F_N}=\bar{\pi}_N\circ{ {\cal N}_{\cal F}}_{| \T{\cal M}_N}$,  we have ${\bf v}_{F_N}=\bar{\pi}_N\circ{ {\bf v}_{\cal F}}_{| \T{\cal M}_N}$. Moreover   since  $\bar{\pi}_N\circ J_{| \T{\cal A}_N}=J_N\circ\bar{\pi}_N$,
by restriction to $\T{\cal M}_N$ the projection  $\theta=\vartheta\circ \bf{v}_{\cal F}:\T{\cal M}\ap \tilde{\cal A}$  induces a projection  $\theta_N=\vartheta\circ\bar{\pi}_N\circ \bf{v}_{\cal F}:T{\cal M}_N\ap \widetilde{TN}$ so that the following diagrams are commutative:
\begin{eqnarray}\label{barpiN}
\begin{array}{ccccc}
 &\theta&\\
T{\cal M}_N& \longrightarrow  &\tilde{\cal A}_{| {\cal M}_N} \\
\bar{\pi}_N \;  \Big\downarrow &  & \Big\downarrow \tilde{p}_N \\
T\bar{\cal M}_N &  \longrightarrow &\widetilde{TN} \\
& \theta_N &
\end{array}\;\;\;\;
\begin{array}{ccccc}
 &\varpi&\\
{{\bf H}{\cal M}_N}_{| {\cal M}_N}& \longrightarrow  &\tilde{\cal A}_{| {\cal M}_N} \\
\bar{\pi}_N \;  \Big\downarrow &  & \Big\downarrow \tilde{p}_N \\
{\bf H}\bar{\cal M}_N &  \longrightarrow &\widetilde{TN} \\
& \varpi_N &
\end{array}
\end{eqnarray}
where $\varpi_N$ is the  restriction  of the natural projection of $\T\bar{\cal M}_N=T\bar{\cal M}_N\ap \widetilde{TN}$ to ${\bf H}\bar{\cal M}_N$.

On the one hand, since $L_{| \mathcal{\check{ M}}_N}=L_N\circ\bar{\pi}_N$ is the Levi-Civita connection ${\cal D}^H$ on $({\bf H}{\cal M},{\cal M},\hat{\rho},[\;,\;]_{\cal P})$ (see Lemma \ref{LC}) restricted to the $\bar{\pi}_M$-projectable section induces clearly the Levi-Civita connection of the Riemannian metric $g_N^H$ associated with $L_N$ on  $({\bf H}\bar{\cal M}_N,\bar{\cal M}_N,Id,[\;,\;])$.\\
On the other hand,  according the Diagrams (\ref{barpiN}) and  the all other  the properties of $\bar{\pi}_N$, it is easy to see   that  the The Bott connection restricted to $\bar{\pi}_N$-projectable sections, induces the Bott connection of ${\bf H}\bar{\cal M}_N$.

Finally since the Chern connection is characterized by $\nabla_{\cal X}\tilde{Y}=\varpi({\cal D}_{\cal X}(\varpi^{-1}(\tilde{ Y}))$,  the previous property of the Bott connection and the second diagram of (\ref{barpiN}) gives rise to the announced result in Point (iii).
\end{proof}

  %%%%%%%%%%%%%%%%%%%%%%%%%%%%%%%%%%%%%%%%%%%%%%%%%%%%%%%%%%%%%%%%%%%%%%%%%%%%%
\subsection{Curvature and flag curvature.}\label{curvatureChern}${}$\\
%%%%%%%%%%%%%%%%%%%%%%%%%%%%%%%%%%%%%%%%%%%%%%%%%%%%%%%%%%%%%%%%%%%%%%%%%%%%
The partial Finsler pre-Lie algebroid $({\cal A},M,\rho,[\;,\;]_{\cal A},{\cal F}$ being fixed we simply denote by $\nabla$ its Chern connection. As classically, the {\it (total) curvature} of $\nabla$ is defined by
$$\Phi({\cal X},{\cal Y})(\tilde{Z})=\nabla_{\cal X}\nabla_{\cal Y}\tilde{Z}-\nabla_{\cal Y}\nabla_{\cal X}\tilde{Z}-\nabla_{[{\cal X},{\cal Y}]_{\cal P}}\tilde{Z}$$
for all $\cal X$, ${\cal Y}\in\Xi(\T{\cal M})$ and $\tilde{Z}\in \Xi(\tilde{\cal A})$.
We can decompose $\nabla$ in $\nabla^H+\nabla^V$ where $\nabla^H_{\cal X}=\nabla_{{\cal X}^H}$ and $\nabla^V_{\cal X}=\nabla_{{\cal X}^V}$ (see Remark \ref{Chernnabla}). Therefore we get the following decomposition:
$$\Phi({\cal X},{\cal Y})=\Phi^{HH}({\cal X},{\cal Y})+\Phi^{HV}({\cal X},{\cal Y})+\Phi^{VH}({\cal X},{\cal Y})+\Phi^{VV}({\cal X},{\cal Y})$$
with:\\
 $\Phi^{HH}({\cal X},{\cal Y})=\Phi({\cal X}^H,{\cal Y}^H)=\nabla^H_{\cal X}\nabla^H_{\cal Y}-\nabla^H_{\cal Y}\nabla^H_{\cal X}-\nabla_{[{\cal X}^H,{\cal Y}^H]_{\cal P}}$;\\
$\Phi^{HV}({\cal X},{\cal Y})=\Phi({\cal X}^H,{\cal Y}^V)=\nabla^H_{\cal X}\nabla^V_{\cal Y}-\nabla^V_{\cal Y}\nabla^H_{\cal X}-\nabla_{[{\cal X}^H,{\cal Y}^V]_{\cal P}}$;\\
$\Phi^{VH}({\cal X},{\cal Y})=\Phi({\cal X}^V,{\cal Y}^H)=\nabla^V_{\cal X}\nabla^H_{\cal Y}-\nabla^H_{\cal Y}\nabla^V_{\cal X}-\nabla_{[{\cal X}^V,{\cal Y}^H]_{\cal P}}$;\\
$\Phi^{VV}({\cal X},{\cal Y})=\Phi({\cal X}^V,{\cal Y}^V)=\nabla^V_{\cal X}\nabla^V_{\cal Y}-\nabla^V_{\cal Y}\nabla^V_{\cal X}-\nabla_{[{\cal X}^V,{\cal Y}^V]_{\cal P}}$.

From these definitions, it follows that each of the four  components of $\Phi$ is a map from $\Xi(\T{\cal M})\times\Xi(\T{\cal M})$ in the set of endomorphisms of the bundle $\tilde{A}\ap {\cal M}$. Moreover as in the classical case of a Finsler manifold  we have:

\begin{Lem}\label{HV}${}$\\
.\\
$\Phi^{HV}({\cal X},{\cal Y})=-\Phi^{VH}({\cal Y},{\cal X});$\\
 $\Phi^{HH}$ is antisymmetric;\\
  $\Phi^{VV}\equiv 0$.
\end{Lem}

\begin{proof}${}$\\
The two  first relations are clearly a consequence of the previous definition of $\Phi^{VV}$ $\Phi^{HV}$ and $\Phi^{VH}$.\\
For the last part it is sufficient to prove this property locally. Consider  the notations introduced in Subsection \ref{partfin} just after Definition \ref{Chernnabla}. First of all if $\tilde{Z}=\tilde{Z}^\a\tilde{e}_\a$, for any ${\cal X}\in \Xi(\T{\cal M})$,  we have
$$\nabla_{\cal X}^V\tilde{Z}=\hat{\rho}({\cal X}^V)(\tilde{Z}^\a)\tilde{e}_\a+\tilde{Z}_\a\nabla_{\cal X}^V\tilde{e}_\a$$
From the construction of $\nabla$ we obtain:

$\nabla_{\cal X}^V\tilde{e}_\a=\varpi\circ \mathbf{h}_{\cal F}([{\cal X}^V,{\cal X}_\a^H]_{\cal P}=\varpi\circ {\bf h}_{\cal F}[{\cal X}^V,{\cal X}_\a]_{\cal P}=0,$

\noindent since the bracket of vertical sections is vertical and $[{\cal X}_\b,{\cal X}_\a]_{\cal P}=0$ this also implies that we get an analogous result for $\nabla_{[{\cal X}^V,{\cal Y}^V]_{\cal P}}$.

\noindent Therefore we have $\nabla_{{\cal X}^V}\tilde{Z}=\hat{\rho}({\cal X}^V)(\tilde{Z}^\a)\tilde{e}_\a$ for any vertical section. It follows that we have
$$\Phi^{VV}({\cal X}^V,{\cal Y}^V)(\tilde{Z})=\left (\hat{\rho}({\cal X}^V)\circ\hat{\rho}({\cal Y}^V)-\hat{\rho}({\cal Y}^V)\circ\hat{\rho}({\cal X}^V)-\hat{\rho}([{\cal X}^V,{\cal Y}^V]_{\cal P})\right)(\tilde{Z}^\a)\tilde{e}_\a.$$
Since $\hat{\rho}$ is  a Lie algebra morphism it follows that this last expression is zero.\\
\end{proof}

As usually in Finsler geometry we set
\begin{eqnarray}\label{RPQ}
\begin{cases}
R=\Phi^{HH} \;\;\;\; \;\;\;\; \;\;\;\;\;\;(hh\rm{-curvature})\cr
P=\Phi^{HV}+\Phi^{VH}\;\;(hv\rm{-curvature}).\cr
\end{cases}
\end{eqnarray}

 $R$ and $P$  are also called the {\it Riemannian curvature} and the {\it Minkowski curvature} of the partial Finsler structure. Therefore $R$ and $P$ are the two components of the total curvature $\Phi$.\\

 \begin{Rem}\label{locexp}${}$\\
 We consider the local basis $ \{\tilde{e}_\a\}$ of $\tilde{\cal A}$ and $\{{\cal X}_\a^H,{\cal V}_\b\}$  defined in Subsection \ref{partfin} after Remark \ref{relsectionchern}. Recall that we have:
 $\nabla_{{\cal X}_\a^H}\tilde{e}_\b=\G_{\a\b}^\g \tilde{e}_\g$ and $\nabla_{{\cal V}_\a}\tilde{e}_\b=0$. Therefore by a classical calculus we obtain

 $R({\cal X}_\g^H,{\cal X}_\d^H)\tilde{e}_\a=R_{\a\g\d}^{\b}\tilde{e}_\b$ with $R_{\a\g\d}^{\b}=\hat{\rho}({\cal X}_\d^H)(\G^\b_{\a\g})-\hat{\rho}({\cal X}_\g^H)(\G^\b_{\a\d})+\G^\b_{\l\g}\G^\l_{\a\d}-\G^\b_{\l\d}\G^\l_{\a\g}$

 $P({\cal X}_\g^H,{\cal V}_\d)\tilde{e}_\a=P^\b_{\a\g\d}\tilde{e}_\b$ with $P^\b_{\a\g\d}\tilde{e}_\b=-\dis\frac{\p \G_{\a\g}^\b}{\p y^\d}$
 \end{Rem}

{\it We now look for the link between the curvature of $\nabla$ and the curvature of all sectional linear connections defined in Subsection \ref{subfinrap} and we will also define the flag curvature.}\\

 According to the notation of these Subsection, given a section $U$ of $\cal M$ over an open set $O$ in $M$, we consider the associated linear connection $\nabla^U$. Therefore  as classically, we can define the curvature of $\nabla^U$ that is:
\begin{eqnarray}\label{curvectureloc}
R^U(X,Y)Z=\nabla^U_X\nabla^U_YZ-\nabla^U_Y\nabla^U_XZ-\nabla^U_{[X,Y]}Z
\end{eqnarray}
for all sections  $X,Y$ and $ Z$ of ${\cal A}_{| O}$.
According to Theorem \ref{chern} and Remark \ref{relsectionchern} for any section $X, Y$ of  of ${\cal A}_{| O}$ we have:
\begin{eqnarray}\label{RRU}
\tilde{\pi}\left(R({\cal X},{\cal Y})(x,U(x)\right)=R^U(X,Y)(x)
\end{eqnarray}
for any $x\in O$ any ${\cal X}$ and $\cal Y$ in $\Xi(\T{\cal M})$ which projects on $X$ and $Y$ respectively.\\

Again as in the classical case of a Finsler manifold we have:
\begin{Def}\label{flag}${}$
\begin{enumerate}
\item An ${\cal M}$-flag at $x\in M$ is a pair $(u,\Pi)$ of a point $u$ in the fiber  ${\cal M}_x$  and a $2$-dimensional subspace $\Pi$ in the fiber ${\cal A}_x$  which contains $u$.
\item The flag curvature of $(u,\Pi)$ is the quantity:
$$K(u,\Pi)=\dis\frac{g_{\cal F}(R({\cal U},{\cal X})u,v)}{g_{\cal F}(u,u)g_{\cal F}(v,v)-\left(g_{\cal F}(u,v)\right)^2}$$
where $v$ belongs to $\Pi$ so that $\{u,v\}$ is a basis of $\Pi$ and $\cal U$ and $\cal V$ are elements of $\T{\cal M}$ such that $\pi_{\cal A}({\cal U})=u$ and $\pi_{\cal A}({\cal V})=v$.
\end{enumerate}
\end{Def}

According to the Relation (\ref{RRU}) and notations and results in Subsection \ref{subfinrap} we also have:
$$K(u,\Pi)=\dis\frac{g_u(R^u(u,v)u,v)}{g_{u}(u,u)g_{u}(v,v)-\left(g_{u}(u,v)\right)^2}.$$
Of course $K(u,\Pi)$ is independent of the choice of $v$ such that $\{u,v\}$ is a basis of $\Pi$. In particular, we can choose $v$ such that $g_u(v,v)=1$ and $v$ is orthogonal to $u$ relatively to $g_u$. Then we get $K(u,\Pi)=g_u(R^u(u,v)u,v)$.

We end this subsection with a result which  gives a link between the curvature  and the curvature flag of partial Finsler algebroid and the same objects for the induced partial Finsler structure induced on each leaf.

\begin{Pro}\label{curvlink}${}$\\
Consider a leaf $N$ of the foliation defined by $\rho({\cal A})$.
\begin{enumerate}
\item Via the   projections $\bar{\pi}_N:\T{\cal M}_N\ap T\bar{\cal M}_N$ and $\rho_N:{\cal A}_N\ap TN$,  the Riemannian curvature $R$ and the Minkowski curvature $P$ of $\cal F$ induces  the the Riemannian curvature $R_N$ and the Minkowski curvature $P_N$ of  the partial Finsler structure on $N$ defined par $F_N$ characterized by
$$R_N(X,Y)Z=\rho_N\circ R({\cal X},{\cal Y})\tilde{Z}\;\;\;\;\; P_N(X,Y)=\rho_N\circ P({\cal X},{\cal Y})\tilde{Z}$$
where ${\cal X}$ and ${\cal Y}$ are sections of $\T{\cal M}_N$ such that $\bar{\pi}_N({\cal X})=X$,  $\bar{\pi}_N({\cal Y})=Y$ and $\tilde{\i}_N(\tilde{ Z})=Z$
\item Consider  a flag $(\bar{u},\bar{\Pi})$  at $x\in N$ ({\it i.e.} $u\in \bar{\cal M}_N$ and $\bar{\Pi}$ a $2$-plane in $T_xN$ with $\bar{u}\in\bar{\Pi}$). Then the flag curvature of $(\bar{u},\bar{\Pi})$ is given by
$$K_N(\bar{u},\bar{\Pi})=K({u}, {P})$$
where $\rho_N({u})=\bar{u}$ and $\Pi$ is a $2$-plane in ${\cal A}_N$ such that $\rho_N(P)=\bar{P}$ with ${u}\in \Pi$.
\end{enumerate}
\end{Pro}

\begin{proof}
Since $R$ and $P$ are tensorial, it is sufficient to prove Point (1)  in some well  adapted choice of a coordinate system and of  a basis $\{e_\a\}$ of $\cal A$. We choose the same context used in Remark \ref{locexpG}. In this case, for $\a=1,\dots,q$,  $\rho_N$ sends $e_\a$ on $\dis\frac{\p}{\p x^\a}$ for $\a=1,\dots,q$, and $\bar{\pi}_N$ sends ${\cal X}_\a$ on $\dis\frac{\p}{\x^\a}$ and ${\cal V}_\a$ on $\dis\frac{\p}{\p y^\a}$. Then according to Diagram \ref{barpiN}  and Remark \ref{locexpG} and Remark \ref{locexp} we obtain:

$\rho_N\circ R({\cal X}^H_\g,{\cal X}^H_\l)\tilde{e}_\a=R_N(\dis\frac{\d}{\d x^\g},\dis\frac{\d}{\d x^\l})\dis\frac{\p}{\p x^\a}$,

$\rho_N\circ P({\cal X}^H_\g,{\cal V}_\l)\tilde{e}_\a=P_N(\dis\frac{\d}{\d x^\g}^H,\dis\frac{\p}{\p y^\l}^H)\dis\frac{\p}{\p x^\a}$.

This ends the proof of Point (1).\\

Point (2) is an easy consequence of Point (1).
\end{proof}

\end{document}